\documentclass[11pt]{amsart}

\usepackage{evearticle}

\begin{document}

\title[The conformally covariant geodesic metric on simple CLE carpets]{Existence and uniqueness of the conformally covariant geodesic metric on simple conformal loop ensemble carpets}
\author{Jason Miller and Yi Tian}
\date{\today}

\begin{abstract}
We prove that for each $\kappa \in (8/3, 4)$ there exists a geodesic metric on the carpet of a $\CLE_\kappa$ which is canonical in the sense that it is characterized by a certain list of axioms. Our metric can be constructed explicitly as the scaling limit of \emph{Minkowski first passage percolation} (MFPP), i.e., the metric obtained by taking the infimum of the Lebesgue measure of the $\varepsilon$-neighborhood of all paths connecting each pair of points. Earlier work by the first co-author \cite{TightSimCLE} showed that MFPP admits nontrivial subsequential limits. The present paper shows that this subsequential limit is unique and is characterized by our list of axioms. We conjecture that our metric describes the scaling limit of the chemical distance metric for discrete loop models that converge to $\CLE_\kappa$ for $\kappa \in (8/3, 4)$ in the scaling limit, e.g., the critical Ising model for $\kappa=3$. Our argument is inspired by recent works of Gwynne and Miller \cite{ExUniLQG} and Ding and Gwynne \cite{UniCriSupercriLQGMet} on the uniqueness of Liouville quantum gravity metrics. 
\end{abstract}

\maketitle

\tableofcontents

\section{Introduction}

\subsection{Overview}
\label{subsec:overview}

\begin{figure}[ht!]
\includegraphics[width=0.49\textwidth]{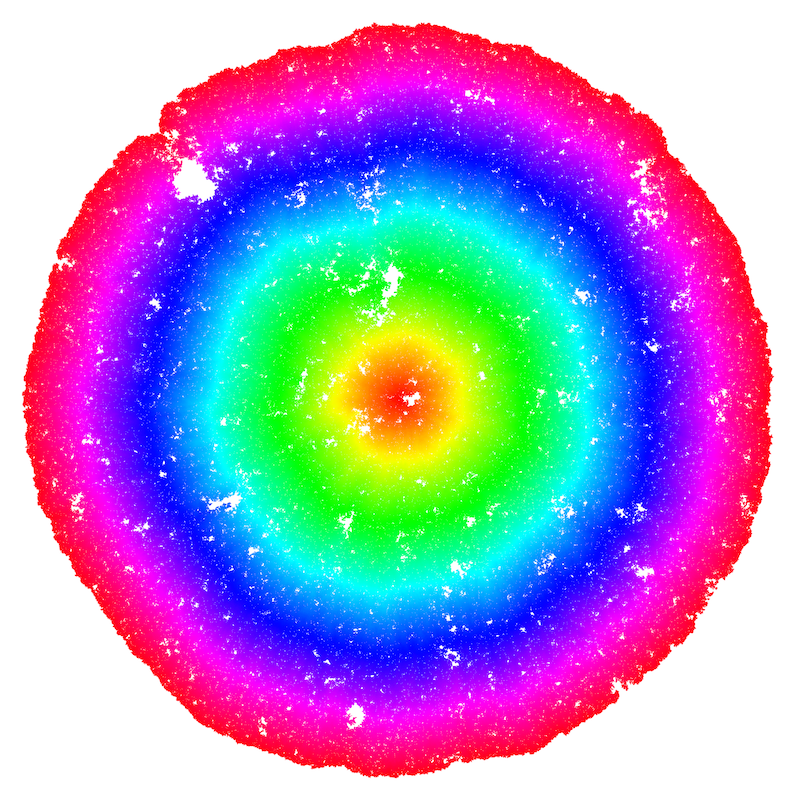} \includegraphics[width=0.49\textwidth]{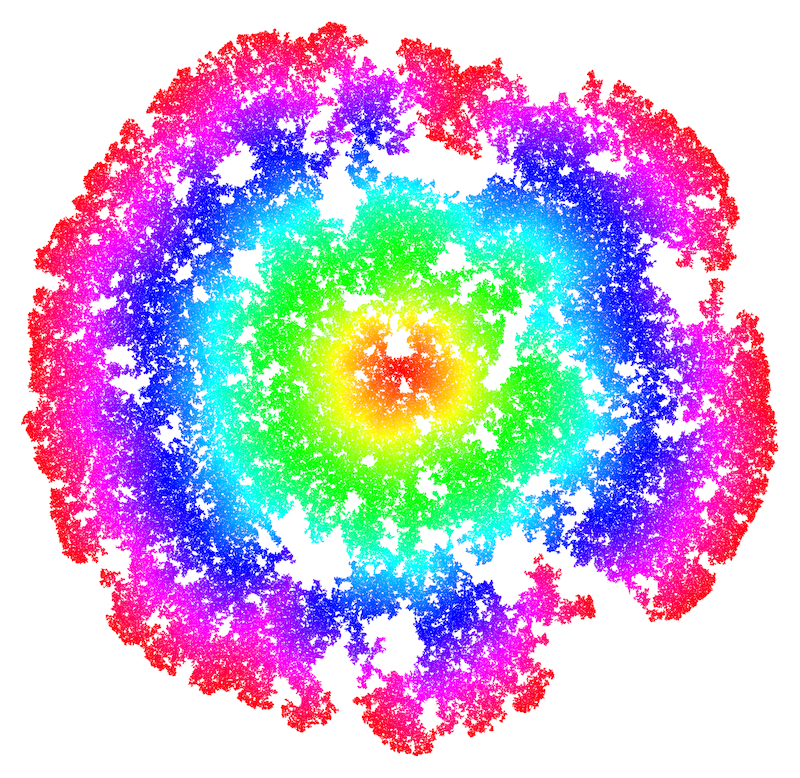}

\includegraphics[width=0.49\textwidth]{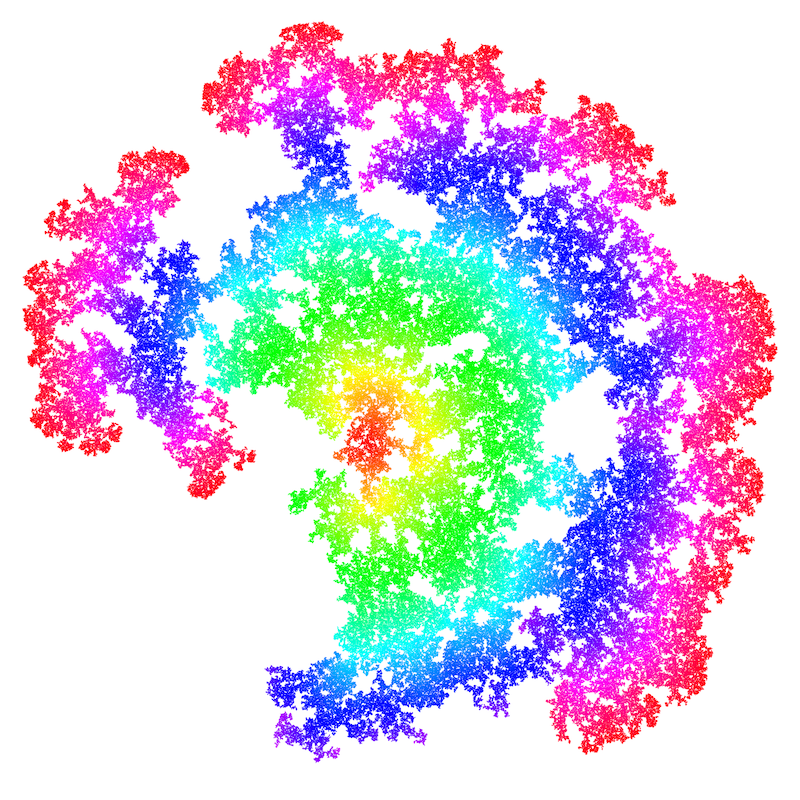} \includegraphics[width=0.49\textwidth]{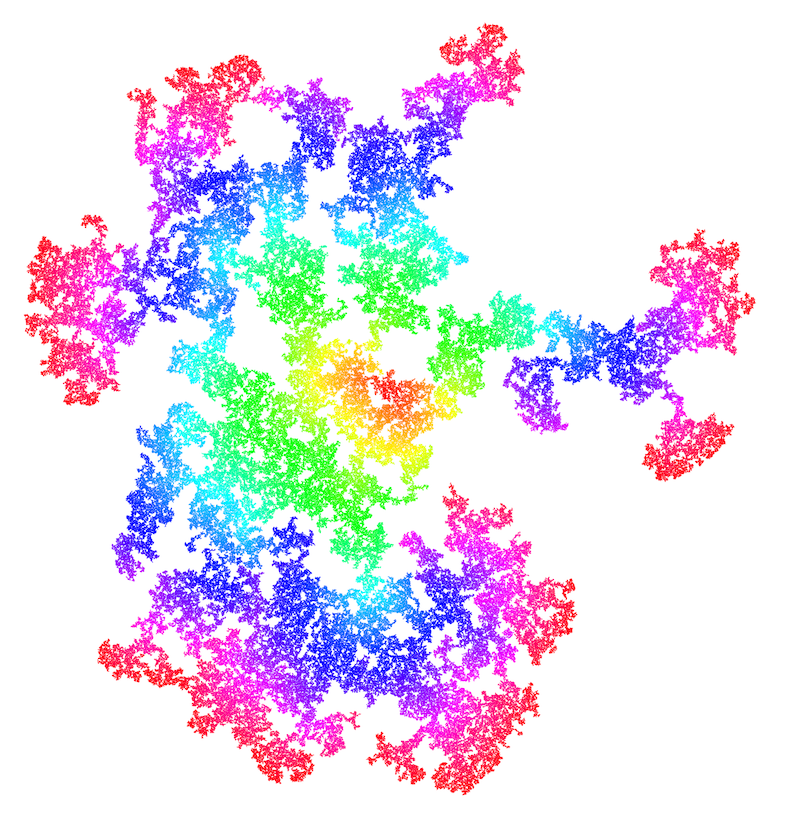}

\caption{Simulation of a metric ball associated with the $\CLE_\kappa$ metric with $\kappa \approx 2.72$ (top left), $\kappa=3$ (top right), $\kappa=10/3$ (bottom left), and $\kappa \approx 3.99$ (bottom right).  Different colors represent the distance of points in the ball to the center.  The $\kappa$ value associated with the top left simulation is close to the critical value $\kappa=8/3$ where the CLE carpet is the entire domain and the $\kappa$ value associated with the bottom right simulation is close to the critical value $\kappa=4$ above which the loops of a CLE intersect themselves, each other, and the domain boundary and the results of the present paper do not apply.}
\end{figure}

The \emph{conformal loop ensembles} ($\CLE_\kappa$, $\kappa \in (8/3,8)$) make up the canonical conformally invariant family of probability measures on countable families of loops in a simply connected domain in $\CC$ which do not cross each other or themselves \cite{TreeCLE,CLE}. They are the loop-analogs of Schramm's $\SLE_\kappa$ ($\kappa > 0$) curves \cite{s2000sle}, and, just like $\SLE_\kappa$ has either been conjectured or shown to describe the scaling limit of a single interface in a number of discrete models from statistical mechanics defined on a planar lattice \cite{s2001percolation,lsw2004lerw,ss2009dgff,s2010ising}, $\CLE_\kappa$ has been conjectured or shown to describe the joint scaling limit of all of the interfaces \cite{ScalLimCriIsingCLE3,cn2006cle,ks2019fkising,lsw2004lerw}.  In the setting of random planar lattices (i.e., random planar maps), several similar results have been proved \cite{s2016hc,kmsw2019bipolar,gm2021percolation,gm2021saw,gkmw2018active,lsw2017schnyder} using the framework from \cite{s2016zipper,dms2021mating}. Each loop in a $\CLE_\kappa$ looks locally like an $\SLE_\kappa$, so their phases coincide. In particular, for $\kappa \in (8/3,4]$ the loops are simple and do not intersect each other or the domain boundary while for $\kappa \in (4,8)$ the loops are self-intersecting and intersect each other and the domain boundary \cite{rs2005basic}.  For $\kappa \in (8/3,4]$ (resp.~ $\kappa \in (4,8)$), the set of points not surrounded by a loop is referred to as its \emph{carpet} (resp.~\emph{gasket}). The reason for the distinction in this terminology is that the $\CLE_\kappa$ carpet is a topological Sierpinski carpet, while the $\CLE_\kappa$ gasket can be thought of as a random analog of the Sierpinski gasket.

The purpose of the present paper is to complete the program of showing that one can associate with a $\CLE_\kappa$ carpet (i.e., focusing on $\kappa \in (8/3,4)$) a canonical geodesic metric.  This is a nontrivial problem since the $\CLE_\kappa$ carpet does not contain any nontrivial rectifiable paths.  The first step in this program was carried out in \cite{TightSimCLE} in which the tightness and nontriviality of the subsequential limits associated with a particular approximation scheme (Minkowski first passage percolation, MFPP) to the geodesic metric were established.  The contribution of this work is that the subsequential limits exist as true limits, and that the limit is canonical in the sense that it is singled out by a simple list of axioms.  Our metric will provide a new tool for studying analysis questions on the $\CLE_\kappa$ carpet, beyond its Hausdorff dimension \cite{ConfRadCLE,SoupCarFracDim} and canonical conformally covariant measure \cite{SimCLELQG,ExUniCoCoVolCLE}.  We also expect that our metric describes the scaling limit of the chemical distance for lattice models which converge to $\CLE_\kappa$.

Our overall approach mirrors the program used to construct the Liouville quantum gravity (LQG) metric \cite{TightLFPP,ExUniLQG,UniCriSupercriLQGMet}. Specifically, we will: (1) propose a concise set of axioms that any candidate limit must satisfy; (2) prove that all subsequential limits of MFPP obey these axioms; and (3) show that these axioms determine the metric uniquely. Precise statements appear in \Cref{ss:11}.

\subsection{Setup}\label{ss:18}

Throughout the remainder of the present paper, fix $\kappa \in (8/3, 4)$. In the context of $\CLE_\kappa$, one can refer to either its \emph{nested} or \emph{non-nested} variant.  The latter consists of a countable collection of loops which do not cross themselves or each other in a simply connected domain $U$ where no loop surrounds another.  In the case of the former, the outermost loops have the law of a non-nested $\CLE_\kappa$ and the law more generally is characterized by the fact that for each loop $\SCL$ the conditional law of the outermost loops in the component $V$ surrounded by $\SCL$ given $\SCL$ and all of the loops not surrounded by $\SCL$ is that of non-nested $\CLE_\kappa$ in $V$.  By considering a nested $\CLE_\kappa$ in $U$ and then taking an appropriate limit as $U$ increases to $\CC$, it is possible to define the so-called whole-plane nested $\CLE_\kappa$ which turns out to be invariant under translation, scaling, rotation, and inversion \cite{CLERiemSph} (see also \cite{CoInCLERiemSph} for a proof of the inversion symmetry for $\kappa \in (4,8)$).  For the remainder of the paper, we work in this setting. Once the metric is constructed here, the corresponding metrics for $\CLE_\kappa$ on general simply connected domains follow directly by restriction and/or local absolute continuity \cite{CoCoGeoCLE}.

Let $\Gamma$ be a whole-plane nested $\CLE_\kappa$.  For $\SCL \in \Gamma$, we shall write $X(\SCL) \subset \CC$ for the closed subset of points that are surrounded by $\SCL$ but by no other loop of $\Gamma$ inside $\SCL$. We observe that there is a natural mapping $\wp \colon \Gamma \to \ZZ/2\ZZ$ which is well-defined up to composition with an automorphism of $\ZZ/2\ZZ$ such that the following is true: For each pair of distinct loops $\SCL_1, \SCL_2 \in \Gamma$ such that $\SCL_2$ is surrounded by $\SCL_1$ but by no other loop of $\Gamma$ in between $\SCL_1$ and $\SCL_2$, we have $\wp(\SCL_1) \neq \wp(\SCL_2)$. Conditionally on $\Gamma$, let us sample a uniform choice among the two possible representatives of $\wp$. We shall write 
\begin{equation*}
    \Upsilon \defeq \bigcup_{\SCL \in \Gamma: \wp(\SCL) = 0} X(\SCL) \quad \text{and} \quad \Upsilon^\dag \defeq \bigcup_{\SCL \in \Gamma: \wp(\SCL) = 1} X(\SCL). 
\end{equation*}
We observe that $\Upsilon$ and $\Upsilon^\dag$ have the same law, and that $\Gamma$ is almost surely determined by $\Upsilon$. For an open subset $U \subset \CC$ and points $x, y \in U \cap \Upsilon$, we shall write $x \xleftrightarrow{U \cap \Upsilon} y$ (or, when there is no danger of confusion, $x \xleftrightarrow U y$) if $x$ and $y$ lie in the same connected component of $U \cap \Upsilon$. The use of $\Upsilon$ instead of $\coprod_{\SCL \in \Gamma} X(\SCL)$ has the advantage that it can be naturally embedded into $\CC$ as a subset, and that for any $x, y \in \Upsilon$, it is possible to perform a ``local rewiring'' of $\Gamma$ so that $x \leftrightarrow y$. (One may imagine $\Upsilon$ as corresponding to the collection of the plus clusters of a critical Ising model on $\ZZ^2$. Thus, given two vertices with plus spins, it is possible to flip the spins at some other vertices so that these two vertices are contained in the same cluster.)

Suppose that we are given mappings $D_{X(\SCL)} \colon X(\SCL) \times X(\SCL) \to \RR$ for all $\SCL \in \Gamma$ with $\wp(\SCL) = 0$. Then for each $x, y \in \Upsilon$, we shall write 
\begin{equation}\label{eq:057}
    D_\Upsilon(x, y) \defeq 
    \begin{cases}
        D_{X(\SCL)}(x, y) & \text{if } x, y \in X(\SCL) \text{ for some } \SCL \in \Gamma \text{ with } \wp(\SCL) = 0; \\
        \infty & \text{otherwise}. 
    \end{cases}
\end{equation}

\emph{Throughout the remainder of the present paper, let $\Gamma$, $\wp$, and $\Upsilon$ be as above, unless otherwise specified.}

Let us now recall some preliminary definitions concerning metric spaces. Let $(X, d)$ be a metric space. For a continuous path $P \colon [0, 1] \to X$, we shall refer to
\begin{equation*}
    \len(P; d) \defeq \sup_{0 = t_0 < t_1 < \cdots < t_n = 1} \sum_{j = 1}^n d(P(t_{j-1}), P(t_j))
\end{equation*}
as the \emph{$d$-length} of $P$, where the supremum is over all partitions of $[0, 1]$. We shall refer to $d$ as a \emph{geodesic metric} if for each $x, y \in X$, there exists a continuous path $P$ in $X$ connecting $x$ and $y$ such that $\len(P; d) = d(x, y)$. Let $Y \subset X$ be a subset. Then we shall refer to
\begin{equation*}
    d(x, y; Y) \defeq \inf_{\substack{P \subset Y\\P \colon x \to y}} \len(P; d), \quad \forall x, y \in Y,
\end{equation*}
where $P$ ranges over all continuous paths in $Y$ connecting $x$ and $y$ (with the convention that $\inf\emptyset = \infty$), as the \emph{internal metric} on $Y$. We shall refer to $d$ as a \emph{length metric} if $d(\bullet, \bullet) = d(\bullet, \bullet; X)$. 

Let $n \in \NN$. We shall write $\mathscr{Haus}_n$ for the metric space consisting of nonempty compact subsets of $\RR^n$, and equipped with the Hausdorff distance. We shall write $\mathscr{HausUni}_n$ for the metric space consisting of pairs $(X, f)$, where $X \subset \RR^n$ is a nonempty compact subset and $f \colon X \to \RR$ is a continuous function, and equipped with the metric given by
\begin{multline*}
    d_{\mathscr{HausUni}_n}((X, f), (Y, g)) \defeq \inf\{\varepsilon > 0 : \forall x \in X, \ \exists y \in Y \text{ s.t.~} \lVert x - y\rVert \le \varepsilon \text{ and } \lvert f(x) - g(y)\rvert \le \varepsilon\} \\
    \vee \inf\{\varepsilon > 0 : \forall y \in Y, \ \exists x \in X \text{ s.t.~} \lVert x - y\rVert \le \varepsilon \text{ and } \lvert f(x) - g(y)\rvert \le \varepsilon\}, \\
    \forall (X, f), (Y, g) \in \mathscr{HausUni}_n. 
\end{multline*}
One verifies easily that $\mathscr{HausUni}_n$ is a separable (but not complete) metric space. However, we recall from \cite[Theorem~5.1]{ConvProbMeas} that Prokhorov's theorem holds for all (not necessarily separable or complete) metric spaces. The following lemma follows immediately from a similar argument to the argument applied in the proof of the Arzel\`a-Ascoli theorem (cf.~\cite{TightSimCLE}).

\begin{lemma}
    Let $F \subset \mathscr{HausUni}_n$ be a subset. Suppose that the following conditions are satisfied:
    \begin{enumerate}
        \item For each $\varepsilon > 0$, there exists $\delta > 0$ such that for each $(X, f) \in F$ and $x, y \in X$ with $\lVert x - y\rVert < \delta$, we have $\lvert f(x) - f(y)\rvert < \varepsilon$. 
        \item There exists $C > 0$ such that $X \subset B_C(0)$ and $\sup_{x \in X} \lvert f(x)\rvert \le C$ for all $(X, f) \in F$. 
    \end{enumerate}
    Then the closure of $F$ in $\mathscr{HausUni}_n$ is compact.
\end{lemma}

Let $X$ be a random variable taking values in $\mathscr{Haus}_2$. Then we shall say that the law of $X$ is \emph{absolutely continuous with respect to the sum of the laws of $X(\SCL)$ for $\SCL \in \Gamma$} if
\begin{multline}\label{eq:004}
    ``\BP\lbrack X \in A\rbrack > 0'' \ \Rightarrow \ ``\BP\lbrack\text{there exists } \SCL \in \Gamma \text{ such that } X(\SCL) \in A\rbrack > 0'', \\
    \forall \text{ Borel subset } A \subset \mathscr{Haus}_2. 
\end{multline}

Let us now recall the precise definition of the \vocab{Minkowski first passage percolation (MFPP)} approximation procedure. Let $\mathop{\mathrm{Leb}}(\bullet)$ denote the Lebesgue measure. For $\varepsilon > 0$, the \emph{$\varepsilon$-MFPP metric} on $X$ is given by
\begin{equation*}
    D_X^\varepsilon(x, y) \defeq \inf_{P \colon x \to y} \mathop{\mathrm{Leb}}(B_\varepsilon(P)), \quad \forall x, y \in X, 
\end{equation*}
where $P$ ranges over all continuous paths in $X$ connecting $x$ and $y$. To extract a nontrivial limit from the metrics $D_{X}^\varepsilon$, we need to renormalize. For $\varepsilon > 0$, we shall write $\ka_\varepsilon$ for the median of the random variable 
\begin{equation}\label{eq:263}
    \sup_{x, y \in X(\SCL_1)} D_{X(\SCL_1)}^\varepsilon(x, y), 
\end{equation}
where $\SCL_1$ denotes the largest loop of $\Gamma$ contained in $B_1(0)$ and surrounding the origin.

The existence of subsequential limits of $\{(X \times X, \ka_\varepsilon^{-1}D_{X}^\varepsilon)\}_{\varepsilon > 0}$ follows essentially from \cite{TightSimCLE}. The main challenge in the present paper is to prove uniqueness of the subsequential limit. We adopt the framework used in the uniqueness proofs for the Liouville quantum gravity metric \cite{ExUniLQG,UniCriSupercriLQGMet}: we first formulate a list of axioms for the geodesic $\CLE_\kappa$ carpet metric; then we show that every subsequential limit of $\{(X \times X, \ka_\varepsilon^{-1}D_{X}^\varepsilon)\}_{\varepsilon > 0}$ satisfies these axioms; and finally we prove that at most one metric can satisfy them. Consequently, a unique geodesic $\CLE_\kappa$ carpet metric exists.

\subsection{Main results}
\label{ss:11}

\begin{definition}\label{307}
    We shall refer to as a \vocab{(strong) geodesic $\CLE_\kappa$ carpet metric} a measurable mapping 
    \begin{equation*}
        D \colon \mathscr{Haus}_2 \to \mathscr{HausUni}_4 \colon X \mapsto (X \times X, D_X)
    \end{equation*}
    such that whenever \eqref{eq:004} holds, the following conditions are satisfied:
    \begin{enumerate}[label=(\Roman*), ref=\Roman*]
        \item\label{307A} {\bfseries (Geodesic metric)} $D_X$ is almost surely a geodesic metric on $X$ inducing the Euclidean topology.
        \item\label{307B} {\bfseries (Locality)} For each deterministic open subset $U \subset \CC$, the internal metric $D_X(\bullet, \bullet; U \cap X)$ is almost surely determined by $U \cap X$. 
        \item\label{307C} {\bfseries (Translation invariance and scale covariance)} There is a deterministic constant $\theta \in \RR$ (which we shall refer to as the \vocab{(chemical) distance exponent}) such that for each deterministic $z \in \CC$ and $r > 0$, 
        \begin{equation*}
            D_{rX + z}(r\bullet + z, r\bullet + z) = r^\theta D_X(\bullet, \bullet) \quad \text{almost surely}. 
        \end{equation*}
        \item\label{307D} {\bfseries (Tightness)} There is a deterministic constant $\alpha_{\mathrm{KC}} > 0$ such that the following is true: Let $K \subset \CC$ be a deterministic compact subset. Then it holds with superpolynomially high probability as $A \to \infty$ that
        \begin{equation*}
            D_\Upsilon(x, y) \le A \left\lvert x - y\right\rvert^{\alpha_{\mathrm{KC}}}, \quad \forall x, y \in K \cap \Upsilon \text{ with } x \leftrightarrow y. 
        \end{equation*}
        (Here, we note that $D_\Upsilon$ is almost surely well-defined.)
        \item\label{307E} {\bfseries (Monotonicity)} Let $X_1, X_2 \in \mathscr{Haus}_2$ with $X_1 \subset X_2$. Then $D_{X_1}(x, y) \ge D_{X_2}(x, y)$ for all $x, y \in X_1$. 
    \end{enumerate}
\end{definition}
(The usage of the notation $\alpha_{\mathrm{KC}}$ above is that the exponent is derived in \cite{TightSimCLE} using an argument of the same flavor as in the proof of the Kolmogorov-Centsov theorem.)

In the remainder of the present subsection, let $X$ be a random variable taking values in $\mathscr{Haus}_2$ that satisfies \eqref{eq:004}.

Our main results establish the existence and uniqueness of the strong geodesic $\CLE_\kappa$ carpet metric. 

\begin{theorem}[Convergence of MFPP]\label{305}
    For each sequence of positive real numbers $\varepsilon$'s tending to zero, there exists a strong geodesic $\CLE_\kappa$ carpet metric $D$ and a subsequence along which $(X \times X, \ka_\varepsilon^{-1} D_X^\varepsilon)$ converges in probability to $(X \times X, D_X)$ with respect to $\mathscr{HausUni}_4$.
\end{theorem}

\begin{theorem}[Uniqueness]\label{320}
    Let $D$ and $\widetilde D$ be two strong geodesic $\CLE_\kappa$ carpet metrics with the same distance exponent $\theta$. Then there is a deterministic constant $M > 0$ such that $\widetilde D_X = MD_X$ almost surely.
\end{theorem}

Moreover, we prove that the distance exponent $\theta$ is uniquely determined by $\kappa$.

\begin{theorem}[Uniqueness of the distance exponent]\label{319}
    There is a unique $\theta = \theta(\kappa)$ for which the strong geodesic $\CLE_\kappa$ carpet metric with distance exponent $\theta$ exists.
\end{theorem}

Finally, we consider the Hausdorff dimension of the carpet and geodesics. We recall that for a metric space $(X, d)$,
\begin{equation*}
    \inf\left\{\alpha \ge 0 : \inf\left\{\sum_{j = 1}^\infty \diam(U_j; d)^\alpha : X \subset \bigcup_{j = 1}^\infty U_j\right\} = 0\right\}
\end{equation*}
is called the \emph{Hausdorff dimension} of $(X, d)$. 

\begin{theorem}[Hausdorff dimension]\label{302}
    Let $D$ be a strong geodesic $\CLE_\kappa$ carpet metric with distance exponent $\theta$. Then:
    \begin{enumerate}
        \item\label{302A} $\theta\dim(X; D_X) = \dim(X)$ almost surely, where $\dim(\bullet; D_X)$ (resp.~$\dim(\bullet)$) denotes the Hausdorff dimension with respect to $D_X$ (resp.~the Euclidean metric). (Recall from \cite{ConfRadCLE,SoupCarFracDim} that $\dim(X) = 1 + 2/\kappa + 3\kappa/32$ almost surely.)
        \item\label{302B} Almost surely, $\dim(P) = \theta$ for all (nontrivial) $D_X$-geodesics $P$. In particular, it follows that $1 < \theta \le 1 + 2/\kappa + 3\kappa/32 < 2$. (In fact, by \cite{HughesThesis}, we have a stronger upper bound that $\theta < 1 + \kappa/8$.)
    \end{enumerate}
\end{theorem}

\subsection{A weaker list of axioms}

It is straightforward to verify that any subsequential limit of $\{(X \times X, \ka_\varepsilon^{-1}D_{X}^\varepsilon)\}_{\varepsilon > 0}$ satisfies Axioms~\eqref{307A}, \eqref{307B}, \eqref{307D}, and \eqref{307E} of \Cref{307}, as well as the translation-invariance portion of Axiom~\eqref{307C}. However, the scale-covariance portion of Axiom~\eqref{307C} is much harder to verify. To address this, we replace it with a weaker condition that is still sufficient to prove uniqueness.

\begin{definition}\label{010}
   We shall refer to as a \emph{weak geodesic $\CLE_\kappa$ carpet metric} a measurable mapping 
    \begin{equation*}
        D \colon \mathscr{Haus}_2 \to \mathscr{HausUni}_4 \colon X \mapsto (X \times X, D_X)
    \end{equation*}
    such that whenever \eqref{eq:004} holds, the following conditions are satisfied:
    \begin{enumerate}[label=(\Roman*), ref=\Roman*]
        \item\label{010A} {\bfseries (Geodesic metric)} $D_X$ is almost surely a geodesic metric on $X$ inducing the Euclidean topology.
        \item\label{010B} {\bfseries (Locality)} For each deterministic open subset $U \subset \CC$, the internal metric $D_X(\bullet, \bullet; U \cap X)$ is almost surely determined by $U \cap X$. 
        \item\label{010C} {\bfseries (Translation invariance)} For each deterministic point $z \in \CC$, we have $D_{X + z}(\bullet + z, \bullet + z) = D_X(\bullet, \bullet)$ almost surely.
        \item\label{010D} {\bfseries (Tightness across scales)} There is a family of deterministic constants $\{\kc_r\}_{r > 0}$ satisfying the following conditions:
        \begin{enumerate}[label=(\roman*), ref=\roman*]
            \item {\bfseries (Upper bound)} There is a deterministic constant $\alpha_{\mathrm{KC}} > 0$ such that the following is true: Let $K \subset \CC$ be a deterministic compact subset. Then for each $r > 0$, it holds with superpolynomially high probability as $A \to \infty$, at a rate which is uniform in $r$, that
            \begin{equation}\label{eq:058}
                \kc_r^{-1} D_\Upsilon(x, y) \le A \left\lvert\frac{x - y}r\right\rvert^{\alpha_{\mathrm{KC}}}, \quad \forall x, y \in rK \cap \Upsilon \text{ with } x \leftrightarrow y. 
            \end{equation}
            \item {\bfseries (Lower bound)} Let $K_1, K_2 \subset \CC$ be deterministic and disjoint compact subsets. Then the laws of the random variables
            \begin{equation*}
                \kc_r/D_\Upsilon(rK_1 \cap \Upsilon, rK_2 \cap \Upsilon)
            \end{equation*}
            for $r > 0$ are tight. 
            \item There is a deterministic constant $\kK > 0$ such that 
            \begin{equation}\label{eq:340}
                \varepsilon^2 \le \kc_{\varepsilon r}/\kc_r \le \kK\varepsilon, \quad \forall r > 0, \ \forall \varepsilon \in (0, 1). 
            \end{equation}
        \end{enumerate}
        \item\label{010E} {\bfseries (Monotonicity)} Let $X_1, X_2 \in \mathscr{Haus}_2$ with $X_1 \subset X_2$. Then $D_{X_1}(x, y) \ge D_{X_2}(x, y)$ for all $x, y \in X_1$. 
    \end{enumerate}
\end{definition}

\begin{theorem}\label{321}
    For each sequence of positive real numbers $\varepsilon$'s tending to zero, there exists a weak geodesic $\CLE_\kappa$ carpet metric $D$ and a subsequence along which $(X \times X, \ka_\varepsilon^{-1} D_X^\varepsilon)$ converges in probability to $(X \times X, D_X)$ with respect to $\mathscr{HausUni}_4$.
\end{theorem}
 
\begin{theorem}\label{000}
    Let $D$ and $\widetilde D$ be two weak geodesic $\CLE_\kappa$ carpet metrics with the same distance exponent $\theta$. Then there is a deterministic constant $M > 0$ such that $\widetilde D_X = MD_X$ almost surely.
\end{theorem}

Let us now explain why \Cref{319,321,000} are sufficient to establish our main results, \Cref{305,320}. 

\begin{lemma}\label{312}
    Every weak geodesic $\CLE_\kappa$ carpet metric is a strong geodesic $\CLE_\kappa$ carpet metric. 
\end{lemma}

\begin{proof}[Proof of \Cref{312} assuming \Cref{000}]
    For $\lambda > 0$, write $D_{X}^{(\lambda)}(\bullet, \bullet) \defeq D_{\lambda X}(\lambda\bullet, \lambda\bullet)$. We \emph{claim} that $D^{(\lambda)}$ is a weak geodesic $\CLE_\kappa$ carpet metric with the same scaling constants as $D$. Indeed, it is clear that $D^{(\lambda)}$ satisfies \Cref{010}, Axioms~\eqref{010A}, \eqref{010B}, \eqref{010C}, and \eqref{010E}. Moreover, since 
    \begin{equation*}
        \kc_r^{-1} D_{rX}^{(\lambda)}(r\bullet, r\bullet) = \frac{\kc_{r\lambda}}{\kc_r} \kc_{r\lambda}^{-1} D_{r\lambda X}(r\lambda\bullet, r\lambda\bullet)
    \quad \text{and} \quad 
    \begin{cases}
            \lambda^2 \le \kc_{r\lambda}/\kc_r \le \kK\lambda & \text{if } \lambda \in (0, 1); \\
            \kK^{-1}\lambda \le \kc_{r\lambda}/\kc_r \le \lambda^2 & \text{if } \lambda \ge 1,
    \end{cases}
    \end{equation*}
    it follows that $D^{(\lambda)}$ satisfies \Cref{010}, Axiom~\eqref{010D}. Thus, by \Cref{000}, there exists a deterministic constant $M_\lambda > 0$ such that $D_{X}^{(\lambda)} = M_\lambda D_{X}$ almost surely. For $\lambda_1, \lambda_2 > 0$, 
    \begin{equation*}
        M_{\lambda_1\lambda_2}D_{X}(\bullet, \bullet) = D_{X}^{(\lambda_1\lambda_2)}(\bullet, \bullet) = D_{\lambda_2X}^{(\lambda_1)}(\lambda_2\bullet, \lambda_2\bullet) = M_{\lambda_1} D_{\lambda_2X}(\lambda_2\bullet, \lambda_2\bullet) = M_{\lambda_1}M_{\lambda_2}D_{X}(\bullet, \bullet), 
    \end{equation*}
    which implies that $M_{\lambda_1}M_{\lambda_2} = M_{\lambda_1\lambda_2}$. Since the mapping $X \mapsto (X \times X, D_{X})$ is measurable, it follows that the assignment $\lambda \mapsto M_\lambda$ is measurable. Thus, we conclude that there exists $\theta \in \RR$ such that $M_\lambda = \lambda^\theta$ for all $\lambda > 0$, i.e., for each $\lambda > 0$, we have $D_{\lambda X}(\lambda\bullet, \lambda\bullet) = \lambda^\theta D_{X}(\bullet, \bullet)$ almost surely. This completes the proof. 
\end{proof}

\begin{proof}[Proof of \Cref{305,320} assuming \Cref{319,321,000}]
    \Cref{305} follows immediately from \Cref{321,312}. By \Cref{305}, there exists a strong geodesic $\CLE_\kappa$ carpet metric with distance exponent $\theta$. By \Cref{319}, this $\theta$ lies between one and two and is the only possible distance exponent for a strong geodesic $\CLE_\kappa$ carpet metric. Thus, we conclude that every strong geodesic $\CLE_\kappa$ carpet metric satisfies \eqref{eq:340} with $\{r^\theta\}_{r > 0}$ in place of $\{\kc_r\}_{r > 0}$, hence is a weak geodesic $\CLE_\kappa$ carpet metric. Thus, \Cref{320} follows immediately from \Cref{000}. 
\end{proof}

\subsection{Related works}

\subsubsection*{Conformal covariance}

In our future paper \cite{CoCoGeoCLE}, we will prove that, for each simply connected domain $U \subset \CC$, there exists a measurable mapping 
\begin{equation*}
    D^U \colon \mathscr{Haus}_2 \to \mathscr{HausUni}_4 \colon X \mapsto (X \times X, D_X^U)
\end{equation*}
satisfying the following conditions:
\begin{enumerate}[label=(\Roman*), ref=\Roman*]
    \item {\bfseries (Length metric)} Let $\Gamma$ be a non-nested $\CLE_\kappa$ in $U$. Write $X$ for the carpet of $\Gamma$. Then $D_X^U$ is almost surely a length metric on $X \setminus \partial U$ inducing the Euclidean topology. 
    \item {\bfseries (Conformal covariance)} Let $\phi \colon U \mapsto \phi(U)$ be a deterministic conformal mapping. Then, almost surely, 
    \begin{equation*}
        D_{\phi(X)}^{\phi(U)}(\phi(x), \phi(y)) = \inf_{\substack{P \subset X \setminus \partial U\\P \colon x \to y}} \int_0^{\len(P; D_X^U)} \left\lvert\phi^\prime(P(t))\right\rvert^\theta \, \rd t, \quad \forall x, y \in X \setminus \partial U, 
    \end{equation*}
    where $P$ ranges over all continuous paths in $X \setminus \partial U$ from $x$ to $y$ parameterized by $D_X^U$-length.
    \item {\bfseries (Locality)} Let $V \subset U$ be a deterministic simply connected subdomain. Then 
    \begin{equation*}
        D_X^U(\bullet, \bullet; V^\star \cap X) = D_{V^\star \cap X}^{V^\star}(\bullet, \bullet) \quad \text{almost surely}, 
    \end{equation*}
    where $V^\star \defeq V \setminus \overline{\bigcup_{\SCL \in \Gamma : \SCL \not\subset V} \mathop{\mathrm{int}}(\SCL)}$. 
\end{enumerate}

\subsubsection*{Confluence of geodesics}

Confluence of geodesics means that if you take two shortest paths that start at different points but head to the same destination, they usually join together before they get there. This does not happen for usual smooth Riemannian manifolds. It was proved for the Brownian map (equivalently, for $\sqrt{8/3}$-LQG) in \cite{GeoBMap}. The subcritical LQG metric also has this property \cite{ConfLQG}, and that fact was crucial for the uniqueness result in \cite{ExUniLQG}. Confluence has likewise been shown for the critical and supercritical LQG metrics \cite{RegConfSuperLQG} and for the directed landscape \cite{27NetDiLand}. In our forthcoming paper \cite{ConfGeoCLE}, we will show that geodesics in the geodesic $\CLE_\kappa$ carpet metric also satisfy the confluence property.

\subsubsection*{The LQG metric on CLE carpets}

Set $\gamma \defeq \sqrt\kappa$. By a similar argument to the argument applied in \cite{TightSimCLE}, the present paper, and \cite{CoCoGeoCLE}, one may prove that, for each simply connected domain $U \subset \CC$, there exists a measurable mapping 
\begin{equation*}
    D^U \colon \SD^\prime(U) \times \mathscr{Haus}_2 \to \mathscr{HausUni}_4 \colon (\Phi, X) \mapsto (X \times X, D_{\Phi,X}^U)
\end{equation*}
(where $\SD^\prime(U)$ denotes the space of Schwartz distributions on $U$) satisfying the following conditions:
\begin{enumerate}[label=(\Roman*), ref=\Roman*]
    \item {\bfseries (Length metric)} Let $\Phi$ be a Gaussian free field on $U$. Let $\Gamma$ be an independent non-nested $\CLE_\kappa$ in $U$. Write $X$ for the carpet of $\Gamma$. Then $D_{\Phi,X}^U$ is almost surely a length metric on $X \setminus \partial U$ inducing the Euclidean topology.
    \item {\bfseries (Weyl scaling)} There is a deterministic constant $\zeta = \zeta(\kappa) \in \RR$ such that the following is true: Let $f \colon U \to \RR$ be a random continuous function that is independent of $X$. Then, almost surely, 
    \begin{equation}\label{eq:005}
        D_{\Phi + f, X}^U(x, y) = \inf_{\substack{P \subset X \setminus \partial U\\P \colon x \to y}} \int_0^{\len(P; D_{\Phi,X}^U)} \re^{\zeta f(P(t))} \, \rd t, \quad \forall x, y \in X \setminus \partial U, 
    \end{equation}
    where $P$ ranges over all continuous paths in $X \setminus \partial U$ from $x$ to $y$ parameterized by $D_{\Phi,X}^U$-length.
    \item {\bfseries (Conformal covariance)} Let $\phi \colon U \to \phi(U)$ be a deterministic conformal mapping. Then 
    \begin{equation*}
        D_{\phi_\ast(\Phi), \phi(X)}^{\phi(U)}(\phi(\bullet), \phi(\bullet)) = D_{\Phi,X}^U(\bullet, \bullet) \quad \text{almost surely}, 
    \end{equation*}
    where $\phi_\ast(\Phi) \defeq \Phi \circ \phi^{-1} + Q\log(\lvert(\phi^{-1})^\prime\rvert)$. 
    \item {\bfseries (Locality)} Let $V \subset U$ be a deterministic simply connected subdomain. Then 
    \begin{equation*}
        D_{\Phi,X}^U(\bullet, \bullet; V^\star \cap X) = D_{\Phi|_{V^\star},V^\star \cap X}^{V^\star}(\bullet, \bullet) \quad \text{almost surely}, 
    \end{equation*}
    where $V^\star \defeq V \setminus \overline{\bigcup_{\SCL \in \Gamma : \SCL \not\subset V} \mathop{\mathrm{int}}(\SCL)}$. 
\end{enumerate}

We expect that the above metric describes the scaling limit of the chemical distance for statistical mechanics models on random planar maps.

\subsubsection*{CLE(4) carpets and non-simple CLE gaskets}

We expect that the the analogs of our main results still hold for the critical value $\kappa = 4$. However, in this case, additional challenges need to be addressed since the tightness argument given in \cite{TightSimCLE} does not apply and since the four-arm exponent for $\CLE_4$ is equal to $2$.  When $\kappa \in (4, 8)$, the analogs of our main results are also true \cite{TightNonsimCLE,ExUniCoCoGeoMetNonsimCLEGas}. However, in this case, the proofs are completely different due to the fact that the loops of a $\CLE_\kappa$ in this case intersect each other.  The construction of the canonical metric for $\CLE_8$ (i.e., space-filling $\SLE_8$) and how it arises as the scaling limit of the chemical distance on the uniform spanning tree has been established in \cite{SRW2DUST,SLEMatTreeEuc}.  In this case, the metric space has the topology of the a tree where the branches are given by the $\SLE_2$ branches of a $\CLE_8$ so to define a metric one only needs to assign a length to these branches.  This length is supplied by the natural parameterization of $\SLE_2$ \cite{ls2011natural,lr2015natural} and the convergence of loop-erased random walk to $\SLE_2$ with respect to the natural parameterization \cite{lv2021natural} plays a key role in \cite{SRW2DUST}.

\subsubsection*{Other random fractal metrics} There are also several other papers that prove tightness and/or uniqueness for different random fractal metrics. Examples include the critical long-range percolation metric \cite{UniCriLRPMet} and exponential metrics built from higher-dimensional log-correlated Gaussian fields \cite{TightExpMetLogCorGFF}. We also note the \emph{directed landscape}, a random directed metric on $\RR^2$ connected to the KPZ universality class \cite{DiLand}.

\subsection{Open problems}

\subsubsection*{The Distance Exponent}

\begin{conjecture}\label{282}
    The assignment $(8/3, 4) \mapsto \RR \colon \kappa \mapsto \theta(\kappa)$ is continuous and strictly increasing. Moreover, $\theta(\kappa) \to 1$ as $\kappa \to (8/3)^+$. 
\end{conjecture}

As mentioned earlier, there are similarities between the geodesic $\CLE_\kappa$ carpet metric and the LQG metric. We suspect that computing the exact value of $\theta(\kappa)$ is highly nontrivial, in analogy with the difficulty of computing the exact value of the LQG dimension $d_\gamma$. However, we expect that \Cref{282} is tractable (cf.~\cite{DimLQG}). Let us also mention that based on Monte Carlo simulations, \cite{Deng:2009vn} predicts that $\theta(\kappa) = (9\kappa + 8)(\kappa + 8)/(128\kappa)$. 

\subsubsection*{Geodesics}

Let $X$ be the carpet of a non-nested $\CLE_\kappa$. Given $X$, let $x$ and $y$ be independent samples from the canonical conformally covariant measure $\mu_X$ (renormalized to be a probability measure) (cf.~\cite{ExUniCoCoVolCLE}). 

\begin{conjecture}\label{107}
    There is almost surely a unique $D_X$-geodesic connecting $x$ and $y$.
\end{conjecture}

\Cref{107} is related to the following. 

\begin{conjecture}\label{108}
    For each deterministic $a \ge 0$, we have $D_X(x, y) \neq a$ almost surely. 
\end{conjecture}

We remark that the analogue of \Cref{108} for the LQG metric is relatively easy to prove, since one can perturb the Gaussian free field by adding a smooth function without significantly altering its probability density. In contrast, in the case of CLE, the randomness is inherently discrete, which makes the conjecture far from obvious (see also \cite{UniDimGeoCriLRPMet}).

\begin{problem}\label{109}
    Does the $D_{X}$-geodesic connecting $x$ and $y$ almost surely touch the $\CLE_\kappa$ loops? 
\end{problem} 

\subsubsection*{Stable Carpets}

Recent works of Le Gall and Miermont \cite{ScalLimRPMLargeFace} and Curien, Miermont, and Riera \cite{StabCarGas} proved that, as the number of vertices tends to infinity, the rescaled Boltzmann planar maps of index $\alpha \in (1, 2)$ converge to an explicit metric measure space $(\SCS_\alpha, D_\alpha, \mu_\alpha)$, which is called the \emph{$\alpha$-stable carpet/gasket}, with respect to the Gromov-Hausdorff-Prokhorov topology. Boltzmann planar maps arise naturally. For instance, the carpet (resp.~gasket) of a random planar map decorated by a critical Ising model (resp.~a critical percolation model) is a Boltzmann planar map of index $11/6$ (resp.~$7/6$). This links the stable carpets/gaskets and the LQG metric on CLE carpets/gaskets. More precisely: 

\begin{conjecture}
    Set $\alpha \defeq 1/2 + 4/\kappa$. Let $(\SCD, \Phi, 0, \infty)$ be a $\gamma$-LQG disk. Let $\Gamma$ be an independent non-nested $\CLE_\kappa$ in $\SCD$. Write $X$ for the carpet of $\Gamma$. Then, given $\mu_{\Phi,X}(X) = 1$ (cf.~\cite{SimCLELQG}), the metric measure space $(X, D_{\Phi,X}, \mu_{\Phi,X})$ is conditionally a disk variant of the $\alpha$-stable carpet. 
\end{conjecture}

\begin{conjecture}\label{001}
    Recall from \eqref{eq:005} the Weyl scaling exponent $\zeta$. We have $\zeta = \gamma/4$. 
\end{conjecture}

The intuitive reason for \Cref{001} is as follows: We may scale the measure $\mu_{\Phi,X}$ by a factor of $C^\alpha > 0$ by adding $(2/\gamma)\log(C)$ to $\Phi$. This results in scaling the metric $D_{\Phi,X}$ by $C^{2\zeta/\gamma}$. The Hausdorff dimension of the $\alpha$-stable carpet is almost surely equal to $2\alpha$. Since
\begin{equation*}
    \parbox{.85\linewidth}{``the dimension times the scaling exponent of the metric is equal to the scaling exponent of the measure'',}
\end{equation*}
we expect that $(2\alpha) \cdot (2\zeta/\gamma) = \alpha$. 

\subsection{Organization of the paper}

The remainder of the paper is structured as follows. In \Cref{s:04}, we establish notational conventions and review the necessary background material. This section also contains several quantitative estimates that will be used repeatedly throughout the paper. In \Cref{s:05}, we define the concept of local metrics on CLE carpets, prove that any two such metrics are bi-Lipschitz equivalent, and show that the local metric is almost surely determined by the carpet itself. \Cref{s:01} is devoted to proving the existence of a weak geodesic $\CLE_\kappa$ carpet metric (\Cref{321}). In \Cref{s:06}, we establish that the distance exponent $\theta$ is uniquely determined by $\kappa$ (\Cref{319}) and compute the Hausdorff dimensions of both the carpet and its geodesics (\Cref{302}). Finally, \Cref{s:03}, \Cref{s:00}, and \Cref{s:02} contain the proof of the uniqueness of weak geodesic $\CLE_\kappa$ carpet metrics (\Cref{000}).

\subsection*{Acknowledgements}

J.M.~received support from ERC consolidator grant ARPF (Horizon Europe UKRI G120614).  Y.T.~was supported by Cambridge International Scholarship from Cambridge Trust. 

\section{Preliminaries}\label{s:04}

\subsection{Notation and conventions}\label{ss:12}

\subsubsection*{Numbers}

The notation $\CC$ (resp.~$\RR$; $\QQ$) will be used to denote the field of complex (resp.~real; rational) numbers. The notation $\ZZ$ (resp.~$\NN$) will be used to denote the set of integers (resp.~positive integers). For $a < b$, we shall write $[a, b]_\ZZ \defeq [a, b] \cap \ZZ$. 

\subsubsection*{Geometry}

The notation $\HH$ (resp.~$\DD$) will be used to denote the upper half-plane (resp.~the open unit disk). For $z \in \CC$ and $0 < s < t$, we shall write $B_t(z)$ for the open Euclidean ball of radius $t$ centered at $z$; we shall write $A_{s,t}(z)$ for the open Euclidean annulus of inner radius $s$ and outer radius $t$ centered at $z$. We shall refer to as a \emph{rectangle} (resp.~\emph{square}) an open subset of the form $(a, a + s) \times (b, b + t) \subset \CC$ (resp.~$(a, a + t) \times (b, b + t) \subset \CC$), where $a, b \in \RR$ and $s, t > 0$. 

Let $(X, d)$ be a metric space. Then we shall write $\diam(X; d) \defeq \sup_{x, y \in X} d(x, y)$. For $x \in X$ and $r > 0$, we shall write $B_r(x; d) \defeq \{y \in X : d(x, y) < r\}$ and $\overline B_r(x; d) \defeq \{y \in X : d(x, y) \le r\}$. 

We shall refer to as a \emph{conformal mapping} a holomorphic bijection between two Riemann surfaces. We shall refer to two Riemann surfaces as \emph{conformally equivalent} if they are biholomorphic. We shall refer to as a \emph{domain} a connected open subset of $\CC$. We shall refer to as a \emph{simply connected domain} an open subset of $\CC$ that is conformally equivalent to $\DD$. We shall refer to as a \emph{doubly connected domain} an open subset of $\CC$ that is conformally equivalent to $A_{s,t}(0)$ for some $0 < s < t$. 

We shall refer to as a \emph{dyadic square} a square with side length $2^n$ and vertices in $(2^n\ZZ)^2$ for some $n \in \ZZ$. We shall write $\mathscr{DySq}$ for the collection of dyadic squares. We shall refer to a domain $U \subset \CC$ as \emph{dyadic} if there exists a finite collection $\CS$ of dyadic squares such that $U$ is the interior of $\bigcup_{S \in \CS} \overline S$. We shall write $\mathscr{DyDo}$ for the collection of dyadic domains. 

We shall refer to as a \emph{path} a continuous mapping from $[0, 1]$ to $\CC$. We shall refer to as a \emph{parameterized loop} a continuous mapping from $\RR/\ZZ$ to $\CC$. The group of orientation-preserving homeomorphisms of $\RR/\ZZ$ acts on the space of parameterized loops via pre-composition. We shall refer to as a \emph{loop} an orbit of this action. We shall write $\mathscr{Loop}$ for the metric space consisting of loops in $\CC$, and equipped with the metric given by
\begin{equation*}
    d_\mathscr{Loop}(\SCL_1, \SCL_2) \defeq \inf_{P_1, P_2} \sup_{t \in \RR/\ZZ} \left\lvert P_1(t) - P_2(t)\right\rvert, \quad \forall \SCL_1, \SCL_2 \in \mathscr{Loop}, 
\end{equation*}
where $P_1$ (resp.~$P_2$) ranges over all parameterizations of $\SCL_1$ (resp.~$\SCL_2$). One verifies immediately that $\mathscr{Loop}$ is a separable and complete metric space. 

Let $U \subset \CC$ be an open subset. Then we shall refer to as a \emph{$U$-excursion} a path $P \colon [0, 1] \to \overline U$ such that $P(0), P(1) \in \partial U$ and $P((0, 1)) \subset U$. Let $K \subset U$ be a compact subset. Then we shall refer to as a \emph{$(U, K)$-excursion} a $U$-excursion that intersects $K$. Let $P \colon [0, 1] \to \CC$ be a path. Then we shall refer to as a \emph{$U$-excursion of $P$} times $0 \le s < t \le 1$ such that $P|_{[s, t]}$ forms a $U$-excursion. We shall refer to as a \emph{$(U, K)$-excursion of $P$} times $0 \le s < \sigma < \tau < t \le 1$ such that $P|_{[s, t]}$ forms a $U$-excursion, $\sigma$ is the first time after $s$ at which $P$ hits $K$, and $\tau$ is the last time before $t$ at which $P$ hits $K$. Let $\SCL \subset \CC$ be a loop. Then we shall refer to as a \emph{complementary $(U, K)$-excursion of $\SCL$} a connected arc of $\SCL$ that does not overlap with any $(U, K)$-excursion of $\SCL$ and is not contained in any larger connected arc of $\SCL$ with such property.

\subsubsection*{Probability}

Let $\{E_t\}_{t > 0}$ be a family of events. Then:
\begin{itemize}
    \item We shall say that $E_t$ occurs with \emph{polynomially high probability} as $t \to 0$ (resp.~$t \to \infty$) if there exists $\alpha > 0$ and $C > 0$ such that 
    \begin{equation}\label{eq:006}
        \BP\lbrack E_t\rbrack \ge 1 - Ct^\alpha \quad \text{(resp.~} \BP\lbrack E_t\rbrack \ge 1 - Ct^{-\alpha}\text{)}, \quad \forall t > 0. 
    \end{equation}
    \item We shall say that $E_t$ occurs with \emph{superpolynomially high probability} as $t \to 0$ (resp.~$t \to \infty$) if for each $\alpha > 0$, there exists $C = C(\alpha) > 0$ such that \eqref{eq:006} holds.
\end{itemize}

\subsection{Simple Conformal Loop Ensembles}\label{ss:01}

In the present subsection, we review the constructions and properties of simple conformal loop ensembles.

\begin{definition}
    For each simply connected domain $U \subset \CC$, a \emph{non-nested simple conformal loop ensemble (CLE)} in $U$ is a random collection $\Gamma_U$ of non-nested disjoint simple loops in $U$ satisfying the following conditions:
    \begin{itemize}
        \item {\bfseries (Conformal invariance)} For each conformal mapping $\phi \colon U \to \phi(U)$, the collection $\phi(\Gamma_U)$ is a non-nested CLE in $\phi(U)$. 
        \item {\bfseries (Restriction)} Let $V \subset U$ be a simply connected subdomain. Write 
        \begin{equation*}
            V^\star \defeq V \setminus \overline{\bigcup_{\SCL \in \Gamma_U : \SCL \not\subset V} \mathop{\mathrm{int}}(\SCL)}.
        \end{equation*}
        Then the restrictions of $\Gamma_U$ to the connected components of $V^\star$ are conditionally independent non-nested CLEs given $V^\star$. 
    \end{itemize}
\end{definition}

We review two different constructions of simple CLE, one using Brownian loop-soup \cite{CLE}, the other using boundary conformal loop ensembles (BCLE) \cite{CLEPerc}, both of which will be relevant for this work.

\subsubsection*{Brownian loop-soup}

The \emph{Brownian loop measure} $\mu^{\mathrm{loop}}$ is a $\sigma$-finite measure on the metric space $\mathscr{Loop}$ which is uniquely (up to a multiplicative constant) characterized by the following property \cite{BLS,CoInMeasSAL}: For each conformal mapping $\phi \colon U \to \phi(U)$, if we write $\mu_U^{\mathrm{loop}}$ for the restriction of $\mu^{\mathrm{loop}}$ on $\{\SCL \in \mathscr{Loop} : \SCL \subset U\}$, then $\phi_\ast(\mu_U^{\mathrm{loop}}) = \mu_{\phi(U)}^{\mathrm{loop}}$.

Let $c \in (0, 1]$. A \emph{Brownian loop-soup} of intensity $c$ in $U$ is a Poisson point process with intensity measure given by $c\mu_U^{\mathrm{loop}}$. Let $\Xi_U$ be a Brownian loop-soup of intensity $c$ in $U$. Write
\begin{equation*}
    \kappa(c) \defeq \frac{13 - c - \sqrt{(1 - c)(25 - c)}}3 \in (8/3, 4]. 
\end{equation*}
Write $\Gamma_U$ for the collection of the outer boundaries of the outermost clusters of the Brownian loops of $\Xi_U$. Then $\Gamma_U$ is called a \emph{non-nested CLE$_\kappa$} in $U$.

The following is proved in \cite{CLE}. 

\begin{theorem}
    Each non-nested $\CLE_\kappa$ for $\kappa \in (8/3, 4]$ is a non-nested simple CLE. Conversely, each non-nested simple CLE is a non-nested $\CLE_\kappa$ for some $\kappa \in (8/3, 4]$.
\end{theorem}

\subsubsection*{CLE percolation}

Let $\kappa \in (2, 4]$; $\rho \in (-2, \kappa - 4)$; $\Gamma$ a branching $\SLE_\kappa(\rho; \kappa - 6 - \rho)$ process in $\HH$ starting from $0$ and targeting all other boundary points. If we equip $\Gamma$ with the orientation from $0$ towards other boundary points, then we shall refer to $\Gamma$ as a \emph{BCLE$_\kappa^\circlearrowright(\rho)$}. Moreover, we shall refer to the clockwise (resp.~counterclockwise) loops formed from $\Gamma$ as the \emph{true loops} (resp.~\emph{false loops}) of the $\BCLE_\kappa^\circlearrowright(\rho)$. (Note that any point $z \in \HH \setminus \Gamma$ is surrounded by either a true loop or a false loop.) We shall refer to a $\BCLE_\kappa^\circlearrowright(\kappa - 6 - \rho)$ as a \emph{BCLE$_\kappa^\circlearrowleft(\rho)$}; to the true loops (resp.~false loops) of $\BCLE_\kappa^\circlearrowright(\kappa - 6 - \rho)$ as the \emph{false loops} (resp.~\emph{true loops}) of $\BCLE_\kappa^\circlearrowleft(\rho)$.

Let $\kappa^\prime \in (4, 8)$; $\rho^\prime \in [\kappa^\prime / 2 - 4, \kappa^\prime / 2 - 2]$; $\Gamma^\prime$ a branching $\SLE_{\kappa^\prime}(\rho^\prime; \kappa^\prime - 6 - \rho^\prime)$ process in $\HH$ starting from $0$ and targeting all other boundary points. If we equip $\Gamma^\prime$ with the orientation from $0$ towards other boundary points, then we shall refer to $\Gamma^\prime$ as a \emph{BCLE$_{\kappa^\prime}^\circlearrowright(\rho^\prime)$}. Moreover, we shall refer to the clockwise (resp.~counterclockwise) loops formed from $\Gamma^\prime$ as the \emph{true loops} (resp.~\emph{false loops}) of the $\BCLE_{\kappa^\prime}^\circlearrowright(\rho^\prime)$. (Note that any point $z \in \HH \setminus \Gamma^\prime$ is surrounded by either a true loop or a false loop.) We shall refer to a $\BCLE_{\kappa^\prime}^\circlearrowright(\kappa^\prime - 6 - \rho^\prime)$ as a \emph{BCLE$_{\kappa^\prime}^\circlearrowleft(\rho^\prime)$}; to the true loops (resp.~false loops) of $\BCLE_{\kappa^\prime}^\circlearrowright(\kappa^\prime - 6 - \rho^\prime)$ as the \emph{false loops} (resp.~\emph{true loops}) of $\BCLE_{\kappa^\prime}^\circlearrowleft(\rho^\prime)$.

Note that by "forgetting the orientations", a $\BCLE_\kappa^\circlearrowright(\rho)$ and a $\BCLE_\kappa^\circlearrowleft(\rho)$ have the same law, which we shall simply denote by \emph{BCLE$_\kappa(\rho)$}. In a similar vein, we shall use the notation \emph{BCLE$_{\kappa^\prime}(\rho^\prime)$}. 

Let $\kappa \in (2, 4]$; $\rho \in (-2, \kappa - 4)$. Then:
\begin{enumerate}
    \item The law of a $\BCLE_\kappa^\circlearrowright(\rho)$ is invariant under any conformal automorphism of $\HH$. 
    \item The image of a $\BCLE_\kappa^\circlearrowright(\rho)$ under any anticonformal automorphism of $\HH$ is a $\BCLE_\kappa^\circlearrowleft(\rho)$. 
    \item The orientation-reversal of a $\BCLE_\kappa^\circlearrowright(\rho)$ is a $\BCLE_\kappa^\circlearrowleft(\rho)$. 
    \item The collection of loops of a $\BCLE_\kappa(\rho)$ is locally finite. 
\end{enumerate}

Let $\kappa^\prime \in (4, 8)$; $\rho^\prime \in [\kappa^\prime / 2 - 4, \kappa^\prime / 2 - 2]$. Then:
\begin{enumerate}
    \item The law of a $\BCLE_{\kappa^\prime}^\circlearrowright(\rho^\prime)$ is invariant under any conformal automorphism of $\HH$. 
    \item The image of a $\BCLE_{\kappa^\prime}^\circlearrowright(\rho^\prime)$ under any anticonformal automorphism of $\HH$ is a $\BCLE_{\kappa^\prime}^\circlearrowleft(\rho^\prime)$. 
    \item The orientation-reversal of a $\BCLE_{\kappa^\prime}^\circlearrowright(\rho^\prime)$ is a $\BCLE_{\kappa^\prime}^\circlearrowleft(\rho^\prime)$. 
    \item The collection of loops of a $\BCLE_{\kappa^\prime}(\rho^\prime)$ is locally finite. 
\end{enumerate}

Let $\kappa \in (8/3, 4)$; $\beta \in [-1, 1]$. Then we shall refer to as the \emph{trunk} of an $\SLE_\kappa^\beta(\kappa - 6)$ the locus of its force point.

The following follows from \cite{CLEPerc,SimCLELQG}.

\begin{theorem}
    Let 
    \begin{equation*}
        \kappa \in (8/3, 4); \quad \kappa^\prime \defeq 16/\kappa; \quad \rho^\prime \in [\kappa^\prime - 6, 0]; \quad \rho_\rR \defeq -(\kappa / 4)(\rho^\prime + 2); \quad \rho_\rL \defeq \kappa / 2 - 4 - \rho_\rR. 
    \end{equation*}
    Let $\Gamma^\prime$ be a $\BCLE_{\kappa^\prime}^\circlearrowright(\rho^\prime)$ in $\HH$. Inside of each connected component of each true (resp.~false) loop of $\Gamma^\prime$, we sample an independent $\BCLE_\kappa^\circlearrowleft(\rho_\rR)$ (resp.~$\BCLE_\kappa^\circlearrowright(\rho_\rL)$), and we write $\Gamma$ for the collection of all the true loops of all these simple BCLEs. Write $\eta^\prime$ for the loop from $0$ to $0$ that traces the domain boundary counterclockwise except that each time we encounter a true loop of $\Gamma^\prime$ for the first time, we trace the entire loop before continuing. Write $\eta$ for the loop from $0$ to $0$ that traces $\eta^\prime$ except that each time we encounter a true loop of $\Gamma$ for the first time, we trace the entire loop before continuing. For each $x \in \partial \HH$, write $\eta_x^\prime$ (resp.~$\eta_x$) for $\eta^\prime$ (resp.~$\eta$) parameterized by the half-plane capacity seen from $x$. Let $\beta = \beta(\kappa^\prime, \rho^\prime) \in [-1, 1]$ be so that
    \begin{equation*}
        \frac{1 - \beta}2 = \frac{\sin(\pi\rho^\prime / 2)}{\sin(\pi\rho^\prime / 2) - \sin(\pi(\kappa^\prime - \rho^\prime) / 2)}. 
    \end{equation*}
    Then for each $x \in \partial \HH$, $\eta_x$ has the law of an $\SLE_\kappa^\beta(\kappa - 6)$ targeting $x$ and $\eta_x^\prime$ is the trunk of $\eta_x$; moreover, $\eta_x$ is equal to the path that traces $\eta_x^\prime$ except that each time we encounter a true loop of $\Gamma$ for the first time, we trace the entire loop before continuing. In particular, we conclude that an $\SLE_\kappa^\beta(\kappa - 6)$ is almost surely generated by a continuous curve, and its trunk has the law of an $\SLE_{\kappa^\prime}(\rho^\prime; \kappa^\prime - 6 - \rho^\prime)$.
\end{theorem}

\begin{corollary}
    Consider the following algorithm:
    \begin{itemize}
        \item One samples a $\BCLE_{\kappa^\prime}(0)$. Then, inside each connected component of each true loop of this $\BCLE_{\kappa^\prime}(0)$, one samples an independent $\BCLE_\kappa(-\kappa/2)$. 
        \item[\vdots]
        \item Inside each connected component of each false loop of the $\BCLE_{\kappa^\prime}(0)$'s and $\BCLE_\kappa(-\kappa/2)$'s of the previous step, one samples an independent $\BCLE_{\kappa^\prime}(0)$. Then, inside each connected component of each true loop of these $\BCLE_{\kappa^\prime}(0)$'s, one samples an independent $\BCLE_\kappa(-\kappa/2)$. 
        \item[\vdots]
    \end{itemize}
    Then the collection of all the true loops of all the $\BCLE_\kappa(-\kappa/2)$'s has the law of a non-nested $\CLE_\kappa$. (Here, we note that the collection of all the outermost true loops of all the $\BCLE_{\kappa^\prime}(0)$'s has the law of a non-nested $\CLE_{\kappa^\prime}$.)
\end{corollary}

\subsubsection*{Arm events}

Finally, we collect several results concerning arm events for CLE.

We shall write 
\begin{equation}\label{eq:016}
    \alpha_{\mathrm{4A}} \defeq \frac{(12 - \kappa)(4 + \kappa)}{8\kappa} > 2. 
\end{equation}
for the \emph{four-arm exponent} for $\CLE_\kappa$.

\begin{lemma}\label{264}
    Let $U$ be either $\CC$ or a simply connected domain. Let $\Gamma$ be a nested $\CLE_\kappa$ in $U$. Then:
    \begin{enumerate}
        \item\label{264A} For each $\alpha \in (0, \alpha_{\mathrm{4A}})$, there exists $C = C(\alpha) > 0$ such that the following is true: Let $z \in U$ and $0 < s < t < \dist(z, \partial U)$. Then it holds with probability at least $1 - C(s/t)^\alpha$ that there are at most two connected arcs of $\Gamma$ in $A_{s,t}(z)$ connecting its inner and outer boundaries. 
        \item\label{264B} For each $\alpha \in (0, \alpha_{\mathrm{4A}})$, there exists $\lambda_\ast = \lambda_\ast(\alpha) \in (0, 1)$ such that for each $0 < \lambda_1 < \lambda_2 < 1$ with $\lambda_1/\lambda_2 \le \lambda_\ast$, there exists $b = b(\alpha, \lambda_1, \lambda_2) \in (0, 1)$ and $C = C(\alpha, \lambda_1, \lambda_2) > 0$ such that the following is true: Let $z \in U$. Let $\{r_j\}_{j \in \NN}$ be a decreasing sequence of positive real numbers such that $r_1 \le \dist(z, \partial U)$ and $r_{j + 1}/r_j \le \lambda_1$ for all $j \in \NN$. Then for each $n \in \NN$, it holds with probability at least $1 - C(\lambda_1/\lambda_2)^{\alpha n}$ that there are at least $bn$ values of $j \in [1, n]_\ZZ$ for which there are at most two connected arcs of $\Gamma$ in $A_{\lambda_1r_j,\lambda_2r_j}(z)$ connecting its inner and outer boundaries.
    \end{enumerate}
\end{lemma}

\begin{proof}
    See \cite[Appendix~C]{TightNonsimCLE}.
\end{proof}

\begin{lemma}\label{328}
    Let $\Gamma$ be a whole-plane nested $\CLE_\kappa$. Let $z \in \CC$ and $0 < s < t$. Then it holds with superpolynomially high probability as $N \to \infty$ that there are at most $2N$ connected arcs of $\Gamma$ in $A_{s,t}(z)$ connecting its inner and outer boundaries. 
\end{lemma}

\begin{proof}
    Write $\kappa^\prime \defeq 16/\kappa$. Let $\eta^\prime$ be a whole-plane space-filling $\SLE_{\kappa^\prime}$ curve from $\infty$ to $\infty$. Recall from \cite{CLEPerc} that it is possible to couple $\eta^\prime$ with $\Gamma$ so that it traces the loops of $\Gamma$. We observe that, if there are at least $2N$ connected arcs of $\Gamma$ in $A_{s,t}(z)$ connecting its inner and outer boundaries, then $\eta^\prime$ makes at least $2N$ crossings between the inner and outer boundaries of $A_{s,t}(z)$. By \cite[Proposition~3.4 and Remark~3.9]{ASKPZSLEBM}, it holds with superpolynomially high probability as $\varepsilon \to 0$ that for each $a < b$ with $\eta^\prime([a, b]) \subset B_t(z)$ and $\diam(\eta^\prime([a, b])) \ge \varepsilon^{1/2}$, the set $\eta^\prime([a, b])$ contains a ball of radius at least $\varepsilon$. Suppose that this event occurs and $\varepsilon^{1/2} \le t - s$. Then $\eta^\prime$ makes at most $\mathop{\mathrm{Leb}}(A_{s,t}(z)) / (\pi\varepsilon^2)$ crossings between the inner and outer boundaries of $A_{s,t}(z)$. This completes the proof. 
\end{proof}

\subsection{Liouville quantum gravity}\label{ss:13}

Let $\gamma \in (0, 2)$ and $Q \defeq 2/\gamma + \gamma/2$. Then a \emph{$\gamma$-Liouville quantum gravity surface} ($\gamma$-LQG surface) is an equivalence class of pairs $(U, \Phi)$, where $U \subset \CC$ is a simply connected domain, $\Phi$ is an instance of the Gaussian free field (GFF) on $U$, and $(U_1, \Phi_1)$ and $(U_2, \Phi_2)$ are equivalent if there exists a conformal mapping $\phi \colon U_1 \to U_2$ such that 
\begin{equation}\label{eq:059}
    \Phi_1 = \Phi_2 \circ \phi + Q\log(\lvert\phi^\prime\rvert).
\end{equation}
More generally, we can consider a $\gamma$-LQG surface with marked points $x_1, \ldots, x_n \in U$. Two $\gamma$-LQG surfaces with marked points are considered to be equivalent if the fields are related as in \eqref{eq:059} where the conformal map $\phi$ takes the marked points for one surface to the marked points for the other.

The volume form (which is formally given by ``$\re^{\gamma\Phi}\rd x\rd y$'') was constructed in \cite{LQGKPZ} (see also the references therein). If $\Phi$ is (some form of) the GFF with Neumann boundary conditions, then one can also make sense of a boundary length measure.

The recent works \cite{TightLFPP,ExUniLQG} constructed the \emph{$\gamma$-LQG metric} $D_\Phi$ on the $\gamma$-LQG surface $(U, \Phi)$ which is formally given by ``$\re^{\gamma\Phi}(\rd x^2 + \rd y^2)$''. Moreover, two equivalent $\gamma$-LQG surfaces are isometric. The $\gamma$-LQG metric is conjectured to be the scaling limit of the graph distance on random planar maps with respect to the Gromov-Hausdorff topology.

The $\gamma$-LQG half-plane and $\gamma$-LQG disk are $\gamma$-LQG surfaces with two boundary marked points. They are in some sense the most natural infinite-volume and finite-volume $\gamma$-LQG surfaces, respectively, due to the following properties:
\begin{itemize}
    \item Let $(\HH, \Phi, 0, \infty)$ be a $\gamma$-LQG half-plane. Then for each deterministic $r > 0$ and $x \in \RR$ so that the boundary length measure of $[0,x]$ is equal to $r$ we have that $(\HH, \Phi, x, \infty)$ is again a $\gamma$-LQG half-plane. 
    \item Let $(\HH, \Phi, 0, \infty)$ be a $\gamma$-LQG disk. Given $\Phi$, let $x$ and $y$ be independent samples from the $\gamma$-LQG length measure on $\partial\HH$ associated with $\Phi$ (renormalized to be a probability measure). Then $(\HH, \Phi, x, y)$ is again a $\gamma$-LQG disk. 
\end{itemize}

The following are proved in \cite{SimCLELQG}.

When $\kappa = \gamma^2 \in (0, 4)$, we say that an $\SLE_\kappa$ (or a $\CLE_\kappa$) and a $\gamma$-LQG surface have the same ``central charge'', in the sense that an independent $\SLE_\kappa$ (or $\CLE_\kappa$) on a $\gamma$-LQG surface has a well-defined $\gamma$-LQG length and cuts the $\gamma$-LQG surface into independent $\gamma$-LQG surfaces \cite{CoWeld,dms2021mating,SimCLELQG}. 

\begin{theorem}
    Let $\gamma \in (\sqrt{8/3}, 2)$; $\kappa \defeq \gamma^2$; $\beta \in [-1, 1]$; $\SCH = (\HH, \Phi, 0, \infty)$ a $\gamma$-LQG half-plane; $\eta$ an independent $\SLE_\kappa^\beta(\kappa - 6)$ on $\SCH$ from $0$ to $\infty$ that is parameterized by the generalized $\gamma$-LQG length of its trunk $\eta^\prime$. Write $\Gamma_{\rL,+}$ (resp.~$\Gamma_{\rR,+}$) for the collection of pairs $(\SCS, t)$, where $\SCS$ is the $\gamma$-LQG surface parameterized by a connected component of $\HH \setminus \eta$ that is encircled by a $\CLE_\kappa$ loop that lies to the left (resp.~right) of $\eta^\prime$ and is first visited by $\eta^\prime$ at time $t$, equipped with a marked point $\eta^\prime(t)$. Write $\Gamma_{\rL,-}$ (resp.~$\Gamma_{\rR,-}$) for the collection of pairs $(\SCS, t)$, where $\SCS$ is the $\gamma$-LQG surface parameterized by a connected component of $\HH \setminus \eta$ that is disconnected from $\infty$ by $\eta^\prime$ at time $t$ and lies to the left (resp.~right) of $\eta^\prime$, equipped with a marked point $\eta^\prime(t)$. Then there are deterministic constants $a_{\rL,+}, a_{\rR,+}, a_{\rL,-}, a_{\rR,-} > 0$ with
    \begin{equation*}
        a_{\rL,+} : a_{\rR,+} : a_{\rL,-} : a_{\rR,-} = -(1 - \beta)\cos(4\pi/\kappa) : -(1 + \beta)\cos(4\pi/\kappa) : 1 : 1.
    \end{equation*}
    such that $\Gamma_{\rL,+}$ (resp.~$\Gamma_{\rR,+}$; $\Gamma_{\rL,-}$; $\Gamma_{\rR,-}$) is a Poisson point process with intensity measure given by the product of the Lebesgue measure on $\RR_{\ge 0}$ and $a_{\rL,+}$ (resp.~$a_{\rR,+}$; $a_{\rL,-}$; $a_{\rR,-}$) times the one-pointed $\gamma$-LQG disk measure. Moreover, $\Gamma_{\rL,+}$, $\Gamma_{\rR,+}$, $\Gamma_{\rL,-}$, $\Gamma_{\rR,-}$ are independent.
\end{theorem}

\begin{corollary}
    Let $\gamma \in (\sqrt{8/3}, 2)$; $\kappa \defeq \gamma^2$; $\beta \in [-1, 1]$; $\SCH = (\HH, \Phi, 0, \infty)$ a $\gamma$-LQG half-plane; $\eta$ an independent $\SLE_\kappa^\beta(\kappa - 6)$ on $\SCH$ from $0$ to $\infty$ that is parameterized by the generalized $\gamma$-LQG length of its trunk. For each $t \ge 0$, write $\SCH_t$ for the $\gamma$-LQG surface parameterized by the unbounded connected component of $\HH \setminus \eta([0, t])$, equipped with marked points $\eta(t)$ and $\infty$; $L_t^\SCH$ (resp.~$R_t^\SCH$) for the $\gamma$-LQG length of the left (resp.~right) boundary of $\SCH_t$ minus the $\gamma$-LQG length of the left (resp.~right) boundary of $\SCH$. Then $L^\SCH$ and $R^\SCH$ are independent $4/\kappa$-stable L\'evy processes.
\end{corollary}

\begin{theorem}
    Let $\gamma \in (\sqrt{8/3}, 2)$; $\kappa \defeq \gamma^2$; $\SCD = (\HH, \Phi)$ a $\gamma$-LQG disk of fixed boundary length; $\Gamma_\SCD$ an independent $\CLE_\kappa$ on $\SCD$. Then the $\gamma$-LQG surfaces encircled by the loops of $\Gamma_\SCD$ are conditionally independent $\gamma$-LQG disks given their boundary lengths.
\end{theorem}

\subsection{Multichordal SLE}\label{ss:02}

In the present subsection, we review the theory of multichordal SLE \cite{IG2,ConnProbCLE,NonsimSLERange,MultiSLE}. 

Let $U \subset \CC$ be a deterministic simply connected domain. Let $x_1, \ldots, x_{2N}$ be deterministic and distinct points of $\partial U$. A \emph{link pattern} on $(U; x_1, \ldots, x_{2N})$ is a partition $\{\{a_1, b_1\}, \ldots, \{a_N, b_N\}\}$ of $\{x_1, \ldots, x_{2N}\}$ such that there exist $N$ disjoint simple paths in $\overline U$ connecting $a_j$ and $b_j$, respectively. Let $\alpha = \{\{a_1, b_1\}, \ldots, \{a_N, b_N\}\}$ be a deterministic link pattern on $(U; x_1, \ldots, x_{2N})$. A \emph{multichordal} $\SLE_\kappa$ in $(U; x_1, \ldots, x_{2N})$ with \emph{interior} link pattern $\alpha$ is a random collection $\{\eta_1, \ldots, \eta_N\}$ of disjoint simple paths in $\overline U$ such that for each $j \in [1, N]_\ZZ$, given $\{\eta_k : k \neq j\}$, the path $\eta_j$ is conditionally a chordal $\SLE_\kappa$ in $U_j$ connecting $a_j$ and $b_j$, where $U_j$ denotes the connected component of $U \setminus \bigcup_{k : k \neq j} \eta_k$ having $a_j$ and $b_j$ on its boundary. 

\begin{theorem}
    The law of a multichordal $\SLE_\kappa$ in $(U; x_1, \ldots, x_{2N})$ with interior link pattern $\alpha$ exists and is unique. Moreover, if $\{\eta_1, \ldots, \eta_N\}$ is a multichordal $\SLE_\kappa$ in $(U; x_1, \ldots, x_{2N})$ with interior link pattern $\alpha$, then the following conditions are satisfied:
    \begin{enumerate}
        \item Let $\phi \colon U \to \phi(U)$ be a deterministic conformal mapping. Then $\{\phi(\eta_1), \ldots, \phi(\eta_N)\}$ is a multichordal $\SLE_\kappa$ in $(\phi(U); \phi(x_1), \ldots, \phi(x_{2N}))$ with interior link pattern $\phi(\alpha)$. 
        \item Let $j \in [1, N]_\ZZ$. Let $\tau$ be a stopping time for $\eta_j$. Then, given $\eta_j|_{[0, \tau]}$, the collection $\{\eta_j|_{[\tau, 1]}\} \cup \{\eta_k : k \neq j\}$ is conditionally a multichordal $\SLE_\kappa$ in $(U \setminus \eta([0, \tau]); \{x_1, \ldots, x_{2N}, \eta(\tau)\} \setminus \{\eta(0)\})$ with interior link pattern obtained from $\alpha$ by replacing $\eta(0)$ with $\eta(\tau)$. 
    \end{enumerate}
\end{theorem}

\begin{lemma}\label{271}
    Let $\{\eta_1, \ldots, \eta_N\}$ be a multichordal $\SLE_\kappa$ in $(U; x_1, \ldots, x_{2N})$ with interior link pattern $\alpha$. Let $\gamma_1, \ldots, \gamma_N$ be disjoint simple paths in $\overline U$ inducing the link pattern $\alpha$. Then for each $\varepsilon > 0$, it holds with positive probability that $\eta_j \subset B_\varepsilon(\gamma_j)$ for all $j \in [1, N]_\ZZ$. 
\end{lemma}

\begin{proof}
    We apply induction on $N$. The case where $N = 1$ follows immediately from \cite[Proposition~3.4]{IG1}. Suppose that $N \ge 2$. Write $\alpha = \{\{a_1, b_1\}, \ldots, \{a_N, b_N\}\}$ such that for each $j \in [1, N]_\ZZ$, $\gamma_j$ is a simple path connecting $a_j$ and $b_j$. By possibly decreasing $\varepsilon$, we may assume without loss of generality that $B_\varepsilon(\gamma_j) \cap B_\varepsilon(\gamma_k) = \emptyset$ for all distinct $j, k \in [1, N]_\ZZ$. One verifies immediately that there exists $j \in [1, N]_\ZZ$ such that either $(a_j, b_j)_{\partial U}^\circlearrowleft \cap \{x_1, \ldots, x_{2N}\} = \emptyset$ or $(a_j, b_j)_{\partial U}^\circlearrowright \cap \{x_1, \ldots, x_{2N}\} = \emptyset$. We may assume without loss of generality that the former is true. Again, by possibly decreasing $\varepsilon$, we may assume without loss of generality that $B_\varepsilon(\gamma_k) \cap B_\varepsilon((a_j, b_j)_{\partial U}^\circlearrowleft) = \emptyset$ for all $k \in [1, N]_\ZZ$ with $k \neq j$. First, it follows from the case where $N = 1$ that it holds with positive probability that $\eta_j \subset B_\varepsilon((a_j, b_j)_{\partial U}^\circlearrowleft)$. Then, by the induction hypothesis, it holds with positive probability that $\eta_j \subset B_\varepsilon((a_j, b_j)_{\partial U}^\circlearrowleft)$ and $\eta_k \subset B_\varepsilon(\gamma_k)$ for all $k \in [1, N]_\ZZ$ with $k \neq j$. Finally, it follows again from the case where $N = 1$ that it holds with positive probability that $\eta_k \subset B_\varepsilon(\gamma_k)$ for all $k \in [1, N]_\ZZ$. This completes the proof. 
\end{proof}

Write
\begin{equation*}
    \begin{dcases}
        Z(x) \defeq x^{2/\kappa}(1 - x)^{1 - 6/\kappa}{}_2F_1(4/\kappa, 1 - 4/\kappa, 8/\kappa; x); \\
        H(x) \defeq \frac{Z(x)}{Z(x) - 2\cos(4\pi/\kappa)Z(1 - x)}, 
    \end{dcases}
    \quad \forall x \in (0, 1),
\end{equation*}
where ${}_2F_1$ denotes the hypergeometric function. Choose $x \in (0, 1)$ so that $(U; x_1, x_2, x_3, x_4)$ is conformally equivalent to $(\HH; \infty, 0, 1 - x, 1)$. A \emph{bichordal $\SLE_\kappa$} in $(U; x_1, x_2, x_3, x_4)$ with \emph{exterior} link pattern $\{\{x_1, x_2\}, \{x_3, x_4\}\}$ is obtained by first sampling $\alpha$ with 
\begin{equation*}
    \BP\lbrack\alpha = \{\{x_1, x_4\}, \{x_2, x_3\}\}\rbrack = H(x) \quad \text{and} \quad \BP\lbrack\alpha = \{\{x_1, x_2\}, \{x_3, x_4\}\}\rbrack = 1 - H(x), 
\end{equation*}
and then sampling a bichordal $\SLE_\kappa$ in $(U; x_1, x_2, x_3, x_4)$ with interior link pattern $\alpha$. Let $\beta$ be a link pattern on $(U; x_1, \ldots, x_{2N})$. A \emph{multichordal $\SLE_\kappa$} in $(U; x_1, \ldots, x_{2N})$ with \emph{exterior} link pattern $\beta$ is a random collection $\{\eta_1, \ldots, \eta_N\}$ of disjoint simple paths in $\overline U$ such that for each pair of distinct $j, k \in [1, N]_\ZZ$, on the event that $\eta_j$ and $\eta_k$ are contained in the same connected component $U_{j,k} \subset U \setminus \bigcup_{l : l \neq j, k}$, given $\{\eta_l : l \neq j, k\}$, the pair $\{\eta_j, \eta_k\}$ is conditionally a multichordal $\SLE_\kappa$ in $(U_{j,k}; \eta_j(0), \eta_j(1), \eta_k(0), \eta_k(1))$ with the exterior link pattern induced by $\{\eta_l : l \neq j, k\}$ and $\beta$. 

\begin{theorem}
    The law of a multichordal $\SLE_\kappa$ in $(U; x_1, \ldots, x_{2N})$ with exterior link pattern $\beta$ exists and is unique. Moreover, if $\{\eta_1, \ldots, \eta_N\}$ is a multichordal $\SLE_\kappa$ in $(U; x_1, \ldots, x_{2N})$ with exterior link pattern $\beta$, then the following conditions are satisfied:
    \begin{enumerate}
        \item Let $\phi \colon U \to \phi(U)$ be a deterministic conformal mapping. Then $\{\phi(\eta_1), \ldots, \phi(\eta_N)\}$ is a multichordal $\SLE_\kappa$ in $(\phi(U); \phi(x_1), \ldots, \phi(x_{2N}))$ with exterior link pattern $\phi(\beta)$. 
        \item Let $j \in [1, N]_\ZZ$. Let $\tau$ be a stopping time for $\eta_j$. Then, given $\eta_j|_{[0, \tau]}$, the collection $\{\eta_j|_{[\tau, 1]}\} \cup \{\eta_k : k \neq j\}$ is conditionally a multichordal $\SLE_\kappa$ in $(U \setminus \eta([0, \tau]); \{x_1, \ldots, x_{2N}, \eta(\tau)\} \setminus \{\eta(0)\})$ with exterior link pattern obtained from $\beta$ by replacing $\eta(0)$ with $\eta(\tau)$. 
    \end{enumerate}
\end{theorem}

\begin{lemma}\label{012}
    Let $\{\eta_1, \ldots, \eta_N\}$ be a multichordal $\SLE_\kappa$ in $(U; x_1, \ldots, x_{2N})$ with exterior link pattern $\beta$. Then $\BP\lbrack\text{the link pattern induced by } \{\eta_1, \ldots, \eta_N\} \text{ is equal to } \alpha\rbrack > 0$. Moreover, the mapping 
    \begin{equation*}
        (U; x_1, \ldots, x_{2N}) \mapsto \BP\lbrack\text{the link pattern induced by } \{\eta_1, \ldots, \eta_N\} \text{ is equal to } \alpha\rbrack
    \end{equation*}
    is continuous with respect to the Carath\'eodory topology for simply connected domains, together with the uniform topology for boundary points. 
\end{lemma}

Let $\{\eta_1, \ldots, \eta_N\}$ be a multichordal $\SLE_\kappa$ in $(U; x_1, \ldots, x_{2N})$ with exterior link pattern $\beta$. Given $\{\eta_1, \ldots, \eta_N\}$, let $\Gamma$ be conditionally independent nested $\CLE_\kappa$'s in each connected component of $U \setminus \bigcup_j \eta_j$. Then $\{\eta_1, \ldots, \eta_N\} \cup \Gamma$ has the law of a \emph{multichordal $\CLE_\kappa$} in $(U; x_1, \ldots, x_{2N})$ with exterior link pattern $\beta$. 

\begin{theorem}\label{252}
    Let $U$ be either $\CC$ or a deterministic simply connected domain. Let $V \subset U$ be a deterministic simply connected subdomain. Let $K \subset V$ be a deterministic connected and compact subset. Let $\Gamma$ be a nested $\CLE_\kappa$ in $U$. Write $V^\star$ for the connected component containing $K$ of the open subset obtained by removing from $V$ the closure of the union of the $(U \setminus K, U \setminus V)$-excursions of $\Gamma$ and all the other loops of $\Gamma$ that intersect with $U \setminus V$. Write $X \defeq \partial V^\star \cap K$ for the collection of the endpoints of the $(U \setminus K, U \setminus V)$-excursions of $\Gamma$. Write $\beta$ for the link pattern on $(V^\star; X)$ induced by the $(U \setminus K, U \setminus V)$-excursions of $\Gamma$. Then, given $(V^\star; X)$ and $\beta$, the collection of the complementary $(U \setminus K, U \setminus V)$-excursions of $\Gamma$, together with the collection of the loops of $\Gamma$ that are contained in $V^\star$, is conditionally a multichordal $\CLE_\kappa$ in $(V^\star; X)$ with exterior link pattern $\beta$. 
\end{theorem}

\begin{lemma}\label{260}
    Let $U$ be a deterministic simply connected domain. Let $V \Subset U$ be a deterministic simply connected subdomain. Let $K \subset V$ be a deterministic connected and compact subset. Let $\Gamma$ be a multichordal $\CLE_\kappa$ in $(U; x_1, \ldots, x_{2N})$ with exterior link pattern $\beta$. Write $\Gamma(V, K)$ for the union of the collection of the $(V, K)$-excursions of $\Gamma$ and the collection of the loops of $\Gamma$ that are contained in $V$ and intersect with $K$. Then the mapping 
    \begin{equation*}
        (U; x_1, \ldots, x_{2N}) \mapsto (\text{the law of } \Gamma(V, K))
    \end{equation*}
    is continuous with respect to the Carath\'eodory topology for simply connected domains, together with the uniform topology for boundary points, and the total variation distance for probability measures. 
\end{lemma}

\subsection{Independence for CLE}

A key tool in the present paper is the fact that the restrictions of a nested $\CLE_\kappa$ to disjoint concentric annuli or disjoint balls are nearly independent \cite{MultiSLE}. 

\begin{lemma}\label{263}
    For each $\lambda \in (0, 1)$, $\alpha > 0$, and $b \in (0, 1)$, there exists $p = p(\lambda, \alpha, b) \in (0, 1)$ and $C = C(\lambda, \alpha, b) > 0$ such that the following is true: 
    \begin{equation}\label{eq:261}
        \parbox{.85\linewidth}{Let $U$ be either $\CC$ or a simply connected domain. Let $\Gamma$ be a nested $\CLE_\kappa$ in $U$. Let $z \in U$. Let $\{r_j\}_{j \in \NN}$ be a decreasing sequence of positive real numbers such that $r_1 \le \dist(z, \partial U)$ and $r_{j + 1}/r_j \le \lambda$ for all $j \in \NN$. Let $\{E_j\}_{j \in \NN}$ be a sequence of events such that 
        \begin{equation*}
            E_j \in \sigma\left(A_{r_{j + 1},r_j}(z) \cap \Gamma\right) \quad \text{and} \quad \BP\lbrack E_j\rbrack \ge p, \quad \forall j \in \NN, 
        \end{equation*}
        where $A_{r_{j + 1},r_j}(z) \cap \Gamma \defeq \{A_{r_{j + 1},r_j}(z) \cap \SCL : \SCL \in \Gamma\}$. Then 
        \begin{equation*}
            \BP\lbrack\#\{j \in [1, n]_\ZZ : E_j \text{ occurs}\} \ge b n\rbrack \ge 1 - C\re^{-\alpha n}, \quad \forall n \in \NN. 
        \end{equation*}}
    \end{equation}
\end{lemma}

\begin{lemma}\label{273}
    For each $\lambda \in (0, 1)$ and $p \in (0, 1)$, there exists $\alpha = \alpha(\lambda, p) > 0$, $b = b(\lambda, p) \in (0, 1)$, and $C = C(\lambda, p) > 0$ such that \eqref{eq:261} is true. 
\end{lemma}

\begin{lemma}\label{265}
    For each $t > 0$ and $p \in (0, 1)$, there exists $\alpha = \alpha(t, p) > 0$ and $C = C(t, p) > 0$ such that the following is true: Let $U$ be either $\CC$ or a simply connected domain. Let $\Gamma$ be a nested $\CLE_\kappa$ in $U$. Let $r > 0$. Let $z_1, \ldots, z_n \in U$ be a sequence of points such that
    \begin{equation*}
        \dist(z_j, \partial U) \ge (1 + t)r \quad \text{and} \quad \left\lvert z_j - z_k\right\rvert \ge (2 + t)r, \quad \forall \text{ distinct } j, k \in [1, n]_\ZZ.
    \end{equation*}
    Let $E_1, \ldots, E_n$ be a sequence of events such that 
    \begin{equation*}
        E_j \in \sigma(B_r(z_j) \cap \Gamma) \quad \text{and} \quad \BP\lbrack E_j\rbrack \ge p, \quad \forall j \in [1, n]_\ZZ. 
    \end{equation*}
    Then 
    \begin{equation*}
        \BP\!\left\lbrack\bigcup_{j = 1}^n E_j\right\rbrack \ge 1 - C\re^{-\alpha n}. 
    \end{equation*}
\end{lemma}

\begin{remark}\label{266}
    Let $\Gamma$ be a whole-plane nested $\CLE_\kappa$. Let $\Upsilon$ and $\Upsilon^\dag$ be as in \Cref{ss:18}. Let $U \subset \CC$ be an open subset. Since $\Upsilon$ and $\Upsilon^\dag$ have the same law, it follows that for each event $E \in \sigma(U \cap \Upsilon)$, we have 
    \begin{gather*}
        \{\Upsilon \in E\} \cap \{\Upsilon^\dag \in E\} \in \sigma(U \cap \Gamma); \quad \BP\lbrack\{\Upsilon \in E\} \cap \{\Upsilon^\dag \in E\}\rbrack \ge 2\BP\lbrack\Upsilon \in E\rbrack - 1; \\
        \{\Upsilon \in E\} \cup \{\Upsilon^\dag \in E\} \in \sigma(U \cap \Gamma); \quad \BP\lbrack\{\Upsilon \in E\} \cup \{\Upsilon^\dag \in E\}\rbrack \le 2\BP\lbrack\Upsilon \in E\rbrack. 
    \end{gather*}
    Thus, we conclude that \Cref{263,273,265} also hold with 
    \begin{equation*}
        ``E_j \in \sigma(A_{r_{j + 1},\lambda r_j}(z) \cap \Upsilon)'' \quad \text{and} \quad ``E_j \in \sigma(B_r(z_j) \cap \Upsilon)''
    \end{equation*}
    in place of
    \begin{equation*}
        ``E_j \in \sigma(A_{r_{j + 1},\lambda r_j}(z) \cap \Gamma)'' \quad \text{and} \quad ``E_j \in \sigma(B_r(z_j) \cap \Gamma)'',
    \end{equation*}
    respectively. 
\end{remark}

\subsection{Lower bounds}

Let $D$ be a weak geodesic $\CLE_\kappa$ carpet metric. In the present subsection, we prove several lower bounds for $D_\Upsilon$. First, we prove that the $D_\Upsilon$-distance between two disjoint compact subsets has finite moments of all negative orders (\Cref{309}). Next, we prove that the identity mapping from $(\Upsilon, D_\Upsilon)$ to $\Upsilon$, equipped with the Euclidean metric, is locally H\"older continuous (\Cref{314}).  

\begin{lemma}\label{309}
    Let $K_1, K_2 \subset \CC$ be deterministic and disjoint compact subsets. Then for each $\rr > 0$, it holds with superpolynomially high probability as $\varepsilon \to 0$, at a rate which is uniform in $\rr$, that 
    \begin{equation*}
        D_\Upsilon(\rr K_1 \cap \Upsilon, \rr K_2 \cap \Upsilon) \ge \varepsilon \kc_{\rr}. 
    \end{equation*}
\end{lemma}

\begin{proof}
    Fix a deterministic bounded open subset $U \subset \CC$ and $\delta > 0$ such that 
    \begin{equation*}
        K_1, K_2 \subset U, \quad \dist(K_1, K_2) \ge \delta, \quad \dist(K_1, \partial U) \ge \delta. 
    \end{equation*}
    Fix $\rr > 0$ and $\alpha > 0$. By \Cref{010}, Axioms \eqref{010C} and \eqref{010D} (translation invariance and tightness across scales), for each $p \in (0, 1)$, there exists $a = a(p) > 0$ such that
    \begin{equation*}
        \BP\lbrack D_\Upsilon(\partial B_{r/2}(z) \cap \Upsilon, \partial B_r(z) \cap \Upsilon) \ge a\kc_r\rbrack \ge p, \quad \forall z \in \CC, \ \forall r > 0. 
    \end{equation*}
    Thus, by \Cref{263} and a union bound, there exists a sufficiently small $a > 0$ such that it holds with probability at least $1 - O(\varepsilon^\alpha)$ as $\varepsilon \to 0$, at a rate which is uniform in $\rr$, that for each $z \in \left(\frac1{100}\varepsilon^2\rr\ZZ\right)^2 \cap \rr U$, there exists $r \in [\varepsilon^2\rr, \varepsilon\rr]$ such that
    \begin{equation}\label{eq:300}
        D_\Upsilon(\partial B_{r/2}(z) \cap \Upsilon, \partial B_r(z) \cap \Upsilon) \ge a\kc_r. 
    \end{equation}
    Henceforth assume that $\varepsilon \le \delta/100$ and the above event occurs. Then each path in $\Upsilon$ connecting $\rr K_1 \cap \Upsilon$ and $\rr K_2 \cap \Upsilon$ must cross between the inner and outer boundaries of $A_{r/2,r}(z)$ for some $z \in \left(\frac1{100}\varepsilon^2\rr\ZZ\right)^2 \cap \rr U$ and $r \in [\varepsilon^2\rr, \varepsilon\rr]$ for which \eqref{eq:300} is true. Combining this with \eqref{eq:340}, we obtain that
    \begin{equation*}
        D_\Upsilon(\rr K_1 \cap \Upsilon, \rr K_2 \cap \Upsilon) \ge \inf\left\{a\kc_r : r \in [\varepsilon^2\rr, \varepsilon\rr]\right\} \ge a\varepsilon^4\kc_{\rr}. 
    \end{equation*}
    Since $\alpha$ is arbitrary, this completes the proof. 
\end{proof}

\begin{lemma}\label{314}
    Let $\chi > 2$. Let $K \subset \CC$ be a deterministic compact subset. Then for each $\rr > 0$, it holds with probability tending to one as $\varepsilon \to 0$, at a rate which is uniform in $\rr$, that
    \begin{equation*}
        \kc_{\rr}^{-1}D_\Upsilon(x, y) \ge \left\lvert\frac{x - y}{\rr}\right\rvert^\chi, \quad \forall x, y \in \rr K \cap \Upsilon \text{ with } \lvert x - y\rvert \le \varepsilon\rr. 
    \end{equation*}
\end{lemma}

\begin{lemma}\label{313}
    Let $\alpha > 0$. Let $K \subset \CC$ be a deterministic compact subset. Then for each $\rr > 0$, it holds with probability tending to one as $\varepsilon_\ast \to 0$, at a rate which is uniform in $\rr$, that
    \begin{equation}\label{eq:311}
        D_\Upsilon(\partial B_{\varepsilon\rr/2}(z) \cap \Upsilon, \partial B_{\varepsilon\rr}(z) \cap \Upsilon) \ge \varepsilon^{\alpha}\kc_{\varepsilon\rr}, \quad \forall z \in \rr K, \ \forall \varepsilon \in (0, \varepsilon_\ast]. 
    \end{equation}
\end{lemma}

\begin{proof}
    By \Cref{309}, for each $r > 0$, it holds with superpolynomially high probability as $\varepsilon \to 0$, at a rate which is uniform in $r$, that 
    \begin{equation*}
        D_\Upsilon(\partial B_{5r/8}(0) \cap \Upsilon, \partial B_{7r/8}(0) \cap \Upsilon) \ge \varepsilon^\alpha\kc_r. 
    \end{equation*}
    Thus, by \Cref{010}, Axiom~\eqref{010C} (translation invariance) and a union bound, for each $\rr > 0$, it holds with probability tending to one as $\varepsilon_\ast \to 0$, at a rate which is uniform in $\rr$, that 
    \begin{multline}\label{eq:312}
        D_\Upsilon(\partial B_{5\varepsilon\rr/8}(z) \cap \Upsilon, \partial B_{7\varepsilon\rr/8}(z) \cap \Upsilon) \ge \varepsilon^\alpha\kc_{\varepsilon\rr}, \\
        \forall \varepsilon \in (0, \varepsilon_\ast] \cap \left\{\left(\frac{99}{100}\right)^k\right\}_{k \in \ZZ}, \ \forall z \in  \left(\frac1{100}\varepsilon\rr\ZZ\right)^2 \cap B_{\rr}(\rr K). 
    \end{multline}
    Now, \eqref{eq:311} follows immediately from \eqref{eq:312}, together with the fact that for each $z \in \rr K$ and $\varepsilon \in (0, \varepsilon_\ast]$, there exists $\varepsilon^\prime \in (0, \varepsilon_\ast] \cap \left\{(99/100)^k\right\}_{k \in \ZZ}$ and $z^\prime \in  \left(\frac1{100}\varepsilon^\prime\rr\ZZ\right)^2 \cap B_{\rr}(\rr K)$ such that $B_{\varepsilon\rr/2}(z) \subset B_{5\varepsilon^\prime\rr/8}(z^\prime) \subset B_{7\varepsilon^\prime\rr/8}(z^\prime) \subset B_{\varepsilon\rr}(z)$. 
\end{proof}

\begin{proof}[Proof of \Cref{314}]
    Fix $\alpha \in (0, \chi - 2)$. Then, by \Cref{313}, for each $\rr > 0$, it holds with probability tending to one as $\varepsilon_\ast \to 0$ that
    \begin{equation}\label{eq:313}
        D_\Upsilon(\partial B_{\varepsilon\rr/2}(z) \cap \Upsilon, \partial B_{\varepsilon\rr}(z) \cap \Upsilon) \ge \varepsilon^{\alpha}\kc_{\varepsilon\rr}, \quad \forall z \in \rr K, \ \forall \varepsilon \in (0, \varepsilon_\ast]. 
    \end{equation}
    Henceforth assume that \eqref{eq:313} holds. Let $x, y \in \rr K \cap \Upsilon$ with $\lvert x - y\rvert \le \varepsilon_\ast\rr$. Combining \eqref{eq:340} and \eqref{eq:313}, we obtain that
    \begin{equation*}
        D_\Upsilon(x, y) \ge D_\Upsilon(\partial B_{\lvert x - y\rvert/2}(x) \cap \Upsilon, \partial B_{\lvert x - y\rvert}(x) \cap \Upsilon) \ge \left\lvert\frac{x - y}{\rr}\right\rvert^\alpha \kc_{\lvert x - y\rvert} \ge \left\lvert\frac{x - y}{\rr}\right\rvert^{2 + \alpha} \kc_{\rr}. 
    \end{equation*}
    Since $\alpha + 2 < \chi$, this completes the proof. 
\end{proof}

\begin{remark}\label{268}
    By a similar argument to the argument applied in the proof of \Cref{314}, the following is true: Let $D$ be a strong geodesic $\CLE_\kappa$ carpet metric with distance exponent $\theta$. Let $\chi > \theta$. Let $K \subset \CC$ be a deterministic compact subset. Then it holds with probability tending to one as $\varepsilon \to 0$ that
    \begin{equation*}
         D_\Upsilon(x, y) \ge \lvert x - y\rvert^\chi, \quad \forall x, y \in K \cap \Upsilon \text{ with } \lvert x - y\rvert \le \varepsilon. 
    \end{equation*}
\end{remark}

\subsection{Upper bounds}

Let $D$ be a weak geodesic $\CLE_\kappa$ carpet metric. In the present subsection, we prove several upper bounds for $D_\Upsilon$. First, we prove that the diameter with respect to the internal metric of $D_\Upsilon$ has finite moments of positive orders (\Cref{229}). Next, we prove a result which roughly speaking states the existence of a ball whose $D_\Upsilon$-diameter is not too large, among a sequence of concentric ones (\Cref{214}). Finally, we prove that the identity mapping from $\Upsilon$, equipped with the Euclidean metric, to $(\Upsilon, D_\Upsilon)$ is locally H\"older continuous (\Cref{019}). (This is analogous to \eqref{eq:058} but concerning the internal metric of $D_\Upsilon$.)

\begin{lemma}\label{229}
    Let $V \Subset U \subset \CC$ be deterministic bounded open subsets. Then for each $\rr > 0$, it holds with superpolynomially high probability as $A \to \infty$, at a rate which is uniform in $\rr$, that
    \begin{equation*}
        \sup\left\{D_\Upsilon(x, y; \rr U \cap \Upsilon) : x, y \in \rr V \cap \Upsilon \text{ with } x \xleftrightarrow{\rr V} y\right\} \le A\kc_{\rr}.
    \end{equation*}
\end{lemma}

\begin{proof}
    Let $\alpha_{\mathrm{NET}}$ be as in \Cref{014}. By \Cref{050} (there is no circular reasoning), it holds with superpolynomially high probability as $\varepsilon_\ast \to 0$ that for each $\varepsilon \in (0, \varepsilon_\ast]$ and $x, y \in \rr V \cap \Upsilon$ with $x \xleftrightarrow{\rr V} y$, there exist points $z_1, \ldots, z_{n - 1} \in \Upsilon$ with $n \le \lfloor\varepsilon^{-\alpha_{\mathrm{NET}}}\rfloor$ such that 
    \begin{equation}\label{eq:255}
        z_{j - 1} \xleftrightarrow{\rr V} z_j \quad \text{and} \quad \lvert z_{j - 1} - z_j\rvert \le \varepsilon\rr, \quad \forall j \in [1, n]_\ZZ,
    \end{equation}
    where $z_0 \defeq x$ and $z_n \defeq y$. Write $E_1^{\rr}(\varepsilon)$ for this event. 

    By \Cref{010}, Axiom~\eqref{010D} (tightness across scales), it holds with superpolynomially high probability as $\varepsilon \to 0$, at a rate which is uniform in $\rr$, that
    \begin{equation}\label{eq:256}
        \sup\left\{\kc_{\rr}^{-1}D_\Upsilon(x, y)\left\lvert\frac{x - y}{\rr}\right\rvert^{-\alpha_{\mathrm{KC}}} : x, y \in \rr V \cap \Upsilon \text{ with } x \leftrightarrow y\right\} \le \varepsilon^{-\alpha_{\mathrm{KC}}/2}. 
    \end{equation}
    Write $E_2^{\rr}(\varepsilon)$ for this event. 

    By \Cref{309}, it holds with superpolynomially high probability as $\varepsilon \to 0$, at a rate which is uniform in $\rr$, that
    \begin{equation}\label{eq:257}
        D_\Upsilon(\rr V \cap \Upsilon, (\CC \setminus \rr U) \cap \Upsilon) \ge \varepsilon^{\alpha_{\mathrm{KC}}/2} \kc_{\rr}. 
    \end{equation}
    Write $E_3^{\rr}(\varepsilon)$ for this event. 

    Henceforth assume that $E_1^{\rr}(\varepsilon) \cap E_2^{\rr}(\varepsilon) \cap E_3^{\rr}(\varepsilon)$ occurs. Let $x, y \in \rr V \cap \Upsilon$ such that $x \xleftrightarrow{\rr V} y$. Let $x = z_0, z_1, \ldots, z_n = y$ be as in the definition of $E_1^{\rr}(\varepsilon)$. Combining \eqref{eq:255} and \eqref{eq:256}, we obtain that
    \begin{equation*}
        D_\Upsilon(z_{j - 1}, z_j) \le \varepsilon^{\alpha_{\mathrm{KC}}/2}\kc_{\rr}, \quad \forall j \in [1, n]_\ZZ. 
    \end{equation*}
    This, together with \eqref{eq:257}, implies that
    \begin{equation*}
        D_\Upsilon(z_{j - 1}, z_j; \rr U \cap \Upsilon) = D_\Upsilon(z_{j - 1}, z_j) \le \varepsilon^{\alpha_{\mathrm{KC}}/2}\kc_{\rr}, \quad \forall j \in [1, n]_\ZZ. 
    \end{equation*}
    Thus, we conclude that
    \begin{equation*}
        D_\Upsilon(x, y; \rr U \cap \Upsilon) \le \sum_{j = 1}^n D_\Upsilon(z_{j - 1}, z_j; \rr U \cap \Upsilon) \le n\varepsilon^{\alpha_{\mathrm{KC}}/2}\kc_{\rr} \le \varepsilon^{\alpha_{\mathrm{KC}}/2 - \alpha_{\mathrm{NET}}} \kc_{\rr}. 
    \end{equation*}
    This completes the proof. 
\end{proof}

\begin{lemma}\label{214}
    For each $\lambda \in (0, 1)$, $\alpha > 0$, and $b \in (0, 1)$, there exists $A = A(\lambda, \alpha, b) > 0$ and $C = C(\lambda, \alpha, b) > 0$ such that the following is true: Let $z \in \CC$. Let $\{r_j\}_{j \in \NN}$ be a decreasing sequence of positive real numbers such that $r_{j + 1}/r_j \le \lambda$ for all $j \in \NN$. Then for each $n \in \NN$, it holds with probability at least $1 - C\re^{-\alpha n}$ that there are at least $bn$ values of $j \in [1, n]_\ZZ$ for which we have
    \begin{equation*}
        D_\Upsilon(x, y; B_{2r_j}(z) \cap \Upsilon) \le A\kc_{r_j}, \quad \forall x, y \in B_{r_j}(z) \cap \Upsilon \text{ with } x \xleftrightarrow{B_{r_j}(z)} y. 
    \end{equation*}
\end{lemma}

\begin{figure}[ht!]
    \centering
    \includegraphics[width=\linewidth]{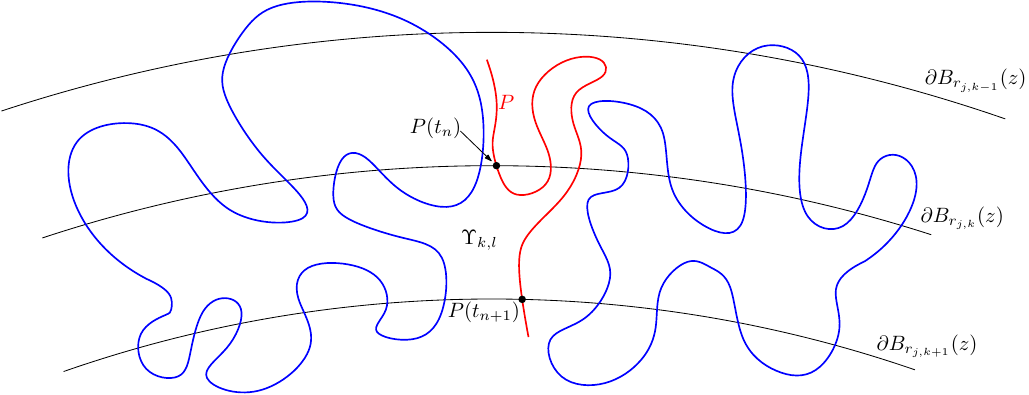}
    \caption{Illustration of the proof of \Cref{214}. The red path $P$ is contained in $B_{r_j}(z) \cap \Upsilon$. The times $\{t_n : n \in [1, N]_\ZZ\}$ are defined inductively so that for each $n \in [1, N - 1]_\ZZ$, there exists $k \in \NN_{> j}$ such that $P(t_n) \in \partial B_{r_{j,k}}(z)$ and $t_{n + 1}$ is the first time after $t_n$ at which $P$ hits $\partial A_{r_{j,k - 1},r_{j,k + 1}}(z)$. The two blue loops are two loops of $\Gamma$ that cross between $\partial B_{r_{j,k}}(z)$ and $\partial B_{r_{j,k + 1}}(z)$. The set $\Upsilon_{k,l}$ is the connected component of $\overline{A_{r_{j,k},r_{j,k + 1}}(z)} \cap \Upsilon$ that lies between the two blue loops. We may assume without loss of generality that for each $\Upsilon_{k,l}$, there exists at most one $n \in [1, N]_\ZZ$ such that $P(t_n) \in \partial B_{r_{j,k}}(z) \cap \Upsilon_{k,l}$, and at most one $n \in [1, N]_\ZZ$ such that $P(t_n) \in \partial B_{r_{j,k + 1}}(z) \cap \Upsilon_{k,l}$. Then, since for each $k \in \NN_{\ge j}$, we are able to bound the number of connected components of $\overline{A_{r_{j,k},r_{j,k + 1}}(z)} \cap \Upsilon$, and the $D_\Upsilon$-distance between any points that lie on the same such connected component, we are able to bound the $D_\Upsilon$-distance between $P(0)$ and $P(1)$.}
    \label{fig:diameter}
\end{figure}

\begin{proof}
    See \Cref{fig:diameter} for an illustration of the proof. By possibly increasing $\lambda$, we may assume without loss of generality that $\lambda \in [3/4, 1)$. Fix $A > 0$ to be chosen later. For $j, k \in \NN$ with $j \le k$, write $r_{j,k} \defeq \lambda^{k - j}r_j$ and $F_{j,k}$ for the event that the following are true: 
    \begin{enumerate}
        \item\label{it:305} We have
        \begin{multline*}
            D_\Upsilon(x, y; A_{r_{j,k + 3},r_{j,k - 2}}(z) \cap \Upsilon) \le A \lambda^{(k - j)/2} \kc_{r_j}, \\
            \forall x, y \in \overline{A_{r_{j,k + 2},r_{j,k - 1}}(z)} \cap \Upsilon \text{ with } x \xleftrightarrow{A_{r_{j,k + 2},r_{j,k - 1}}(z)} y
        \end{multline*}
        with the convention that $r_{j,j - 1} \defeq \lambda^{-1}r_j$ and $r_{j,j - 2} \defeq \lambda^{-2}r_j$.
        \item\label{it:306} There are at most $2A\lambda^{(j - k)/4}$ connected arcs of $\Gamma$ in $A_{r_{j,k + 1},r_{j,k}}(z)$ connecting its inner and outer boundaries. 
    \end{enumerate}
    By \eqref{eq:340} and \Cref{229}, for each $q > 0$, there exists $C_q > 0$ such that for each $j, k \in \NN$ with $j \le k$, condition~\eqref{it:305} holds with probability at least
    \begin{equation*}
        1 - \frac{C_q}{A^q \lambda^{(k - j)q/2}(\kc_{r_j}/\kc_{r_{j,k}})^q} \ge 1 - C_p A^{-q} \kK^q \lambda^{(k - j)q/2}. 
    \end{equation*}
    By \Cref{328} and the scale invariance of the law of $\Gamma$, for each $q > 0$, there exists $C_q^\prime > 0$ such that for each $j, k \in \NN$ with $j \le k$, condition~\eqref{it:306} holds with probability at least $1 - C_q^\prime A^{-q} \lambda^{(k - j)q/4}$. Thus, we conclude that for each $q > 0$, there exists $C_q^\pprime > 0$ such that
    \begin{equation*}
        \BP\lbrack F_{j,k}\rbrack \ge 1 - C_q^\pprime A^{-q} \lambda^{(k - j)q/4}, \quad \forall j, k \in \NN \text{ with } j \le k. 
    \end{equation*}
    Thus, it follows from a similar argument to the argument applied in the proof of \cite[Proposition~3.1]{ExUniCoCoGeoMetNonsimCLEGas} that there exists $A = A(\lambda, \alpha, b) > 0$ and $C = C(\lambda, \alpha, b) > 0$ such that 
    \begin{equation*}
        \BP\!\left\lbrack\#\left\{j \in [1, n]_\ZZ : \bigcap_{k \in \NN : k \ge j} F_{j,k} \text{ occurs}\right\} \ge bn\right\rbrack \ge 1 - C\re^{-\alpha n}, \quad \forall n \in \NN.  
    \end{equation*}

    Henceforth assume that the event $\bigcap_{k \in \NN : k \ge j} F_{j,k}$ occurs. We \emph{claim} that
    \begin{equation*}
        D_\Upsilon(x, y; B_{r_{j,j - 2}}(z) \cap \Upsilon) \le \left(\frac{2A^2}{1 - \lambda^{1/4}} + A\right) \kc_{r_j}, \quad \forall x, y \in B_{r_j}(z) \cap \Upsilon \text{ with } x \xleftrightarrow{B_{r_j}(z)} y.
    \end{equation*}
    Fix $x, y \in B_{r_j}(z) \cap \Upsilon$ and a path $P \colon [0, 1] \to B_{r_j}(z) \cap \Upsilon$ from $x$ to $y$. For each $k \in \NN_{\ge j}$, by condition~\eqref{it:306}, there are at most $A\lambda^{(j - k)/4}$ connected components of $\overline{A_{r_{j,k + 1},r_{j,k}}(z)} \cap \Upsilon$ connecting its inner and outer boundaries; write $\Upsilon_{k,1}, \Upsilon_{k,2}, \ldots$ for these connected components. Set $t_0 \defeq 0$. If $P \cap \partial B_{r_{j,k}}(z) \neq \emptyset$ for some $k \in \NN_{> j}$, then set $t_1$ to be the first time at which $P$ hits $\partial B_{r_{j,k}}(z)$ for some $k \in \NN_{> j}$; otherwise set $t_1 \defeq 1$. Then, inductively, for each $n \in \NN$, if $t_n \neq 1$, $P(t_n) \in \partial B_{r_{j,k}}(z)$, and $P|_{[t_n, 1]} \cap \partial A_{r_{j,k + 1},r_{j,k - 1}}(z) \neq \emptyset$, then set $t_{n + 1}$ to be the first time after $t_n$ at which $P$ hits $\partial A_{r_{j,k + 1},r_{j,k - 1}}(z)$; otherwise set $t_{n + 1} \defeq 1$. Write $N \defeq \inf\{n \in \NN : t_n = 1\} - 1$. By definition, for each $n \in [1, N]_\ZZ$, we have $P(t_n) \in \partial A_{r_{j,k + 1},r_{j,k}}(z) \cap \Upsilon_{k,l}$ for some $k \in \NN_{\ge j}$ and $l \in [1, A\lambda^{(j - k)/4}]_\ZZ$. Suppose that there exist $a, b \in [1, N]_\ZZ$ with $a < b$, $k \in \NN_{\ge j}$, and $l \in [1, A\lambda^{(j - k)/4}]_\ZZ$ such that either both $P(t_a)$ and $P(t_b)$ lie in $\partial B_{r_{j,k}}(z) \cap \Upsilon_{k,l}$ or both $P(t_a)$ and $P(t_b)$ lie in $\partial B_{r_{j,k + 1}}(z) \cap \Upsilon_{k,l}$. In both cases, $P(t_a)$ and $P(t_b)$ lie in the same connected component of $A_{r_{j,k + 1},r_{j,k}}(z) \cap \Upsilon$, write $P^\prime$ for the concatenation of $P|_{[0, t_a]}$, a path in $A_{r_{j,k + 1},r_{j,k}}(z) \cap \Upsilon$ from $P(t_a)$ to $P(t_b)$, and $P|_{[t_b, 1]}$. If we define $\{t_n^\prime\}_{n \in \NN}$ and $N^\prime$ in the same manner as $\{t_n\}_{n \in \NN}$ and $N$, resp.~but with $P^\prime$ in place of $P$, then one verifies immediately that
    \begin{equation*}
        P^\prime(t_n^\prime) = 
        \begin{cases}
            P(t_n) & \text{if } n \le a; \\
            P(t_{n + b - a}) & \text{otherwise}, 
        \end{cases}
    \end{equation*}
    and $N^\prime = N + a - b < N$. Thus, by performing this operation finitely many times, we may assume without loss of generality that such $a$ and $b$ do not exist. On the other hand, for each $n \in [1, N + 1]_\ZZ$, by definition, $P(t_{n - 1})$ and $P(t_n)$ lie in the same connected component of $A_{r_{j,k + 2},r_{j,k - 1}}(z) \cap \Upsilon$ for some $k \in \NN_{\ge j}$, in which case we have
    \begin{equation*}
        D_\Upsilon(P(t_{n - 1}), P(t_n); B_{r_{j,j - 2}}(z) \cap \Upsilon) \le D_\Upsilon(P(t_{n - 1}), P(t_n); A_{r_{j,k + 3},r_{j,k - 2}}(z) \cap \Upsilon) \le A \lambda^{(k - j)/2} \kc_{r_j}
    \end{equation*}
    by condition~\eqref{it:305}. Thus, 
    \begin{align*}
        D_\Upsilon(x, y; B_{r_{j,j - 1}}(z) \cap \Upsilon) &\le \sum_{n = 1}^N D_\Upsilon(P(t_{n - 1}), P(t_n); B_{r_{j,j - 1}}(z) \cap \Upsilon) \\
        &+ D_\Upsilon(P(t_N), P(1); B_{r_{j,j - 1}}(z) \cap \Upsilon) \\
        &\le \sum_{k = j}^\infty 2 A\lambda^{(j - k)/4} A \lambda^{(k - j)/2} \kc_{r_j} + A \kc_{r_j} \\
        &= \left(\frac{2A^2}{1 - \lambda^{1/4}} + A\right) \kc_{r_j}. 
    \end{align*}
    This completes the proof. 
\end{proof}

\begin{lemma}\label{019}
    Let $\alpha_{\mathrm{4A}}$ be as in \eqref{eq:016}. Let $\chi \in (0, 1 - 2/\alpha_{\mathrm{4A}})$. Let $U \subset \CC$ be a deterministic open subset. Let $K \subset U$ be a deterministic compact subset. Then for each $\rr > 0$, it holds with probability tending to one as $\varepsilon \to 0$, at a rate which is uniform in $\rr$, that the following is true: Let $x, y \in \rr K \cap \Upsilon$ with $\lvert x - y\rvert \le \varepsilon\rr$. Then either $x \xleftrightarrow{\rr U} y$ or there exists a loop of $\Gamma$ contained in $\rr U$ and disconnecting $x$ and $y$. Moreover, in the former case, we have
    \begin{equation*}
        \kc_{\rr}^{-1}D_\Upsilon(x, y; \rr U \cap \Upsilon) \le \left\lvert\frac{x - y}{\rr}\right\rvert^\chi. 
    \end{equation*}
\end{lemma}

\begin{proof}
    Fix $\chi < \mu < \nu < 1$ such that $\alpha_{\mathrm{4A}}(1 - \nu) > 2$. 
    
    By possibly shrinking $U$, we may assume without loss of generality that it is bounded. Then, by \Cref{264}, \eqref{264A} and a union bound, for each $\rr > 0$, it holds with polynomially high probability as $\varepsilon \to 0$, at a rate which is uniform in $\rr$, that for each $z \in \left(\frac1{100}\varepsilon\rr\ZZ\right)^2 \cap \rr U$, there are at most two connected arcs of $\Gamma$ in $A_{\varepsilon\rr,\varepsilon^\nu\rr}(z)$ connecting its inner and outer boundaries. Thus, by the Borel-Cantelli lemma, it holds with probability tending to one as $\varepsilon_\ast \to 0$, at a rate which is uniform in $\rr$, that the above event occurs for all $\varepsilon \in (0, \varepsilon_\ast] \cap \{2^k\}_{k \in \ZZ}$. Henceforth assume that this holds for a sufficiently small $\varepsilon_\ast \in (0, 1)$ (how small does not depend on $\rr$). 
    
    Let $x, y \in \rr K \cap \Upsilon$ with $\lvert x - y\rvert \le \varepsilon_\ast\rr/100$. Choose $\varepsilon \in (0, \varepsilon_\ast] \cap \{2^k\}_{k \in \ZZ}$ and $z \in \left(\frac1{100}\varepsilon\rr\ZZ\right)^2 \cap \rr U$ such that $x, y \in B_{\varepsilon\rr}(z)$ and $\varepsilon\rr \le 100\lvert x - y\rvert$. We \emph{claim} that either $x \xleftrightarrow{B_{\varepsilon^\nu\rr}(z)} y$ or there exists a loop of $\Gamma$ contained in $B_{\varepsilon^\nu\rr}(z)$ and disconnecting $x$ and $y$. Suppose by way of contradiction that the \emph{claim} is not true. Since both $x$ and $y$ lie in $\Upsilon$, there must exist at least two $B_{\varepsilon^\nu\rr}(z)$-excursions of $\Gamma$ disconnecting $x$ and $y$ in $B_{\varepsilon^\nu\rr}(z)$. However, this implies that there are at least four connected arcs of $\Gamma$ in $A_{\varepsilon\rr,\varepsilon^\nu\rr}(z)$ connecting its inner and outer boundaries, a contradiction. This completes the proof of the \emph{claim}. 

    By \Cref{214} and a union bound, there exists a sufficiently large $A > 0$ such that for each $\rr > 0$, it holds with polynomially high probability as $\varepsilon \to 0$, at a rate which is uniform in $\rr$, that 
    \begin{multline}\label{eq:314}
        D_\Upsilon(x, y; B_{\varepsilon^\mu\rr}(z) \cap \Upsilon) \le A\varepsilon^\mu\kc_{\rr}, \\
        \forall z \in \left(\frac1{100}\varepsilon^\nu\rr\ZZ\right)^2 \cap \rr U, \ \forall x, y \in B_{\varepsilon^\nu\rr}(z) \cap \Upsilon \text{ with } x \xleftrightarrow{B_{\varepsilon^\nu\rr}(z)} y. 
    \end{multline}
    Thus, by the Borel-Cantelli lemma, it holds with probability tending to one as $\varepsilon_\ast \to 0$, at a rate which is uniform in $\rr$, that \eqref{eq:314} is true for all $\varepsilon \in (0, \varepsilon_\ast] \cap \{2^k\}_{k \in \ZZ}$. Henceforth assume that this holds for a sufficiently small $\varepsilon_\ast \in (0, 1)$ (how small does not depend on $\rr$). Combining the \emph{claim} of the preceding paragraph and \eqref{eq:314}, we obtain that either there exists a loop of $\Gamma$ contained in $B_{\varepsilon^\nu\rr}(z)$ and disconnecting $x$ and $y$ or 
    \begin{equation*}
        D_\Upsilon(x, y; \rr U \cap \Upsilon) \le D_\Upsilon(x, y; B_{\varepsilon^\mu\rr}(z) \cap \Upsilon) \le A\varepsilon^\mu\kc_{\rr} \le A \cdot 100^\mu \cdot \left\lvert\frac{x - y}{\rr}\right\rvert^\mu \kc_{\rr}. 
    \end{equation*}
    This completes the proof. 
\end{proof}

\begin{remark}
    By a similar argument to the argument applied in the proof of \Cref{019}, the following is true: Let $D$ be a strong geodesic $\CLE_\kappa$ carpet metric with distance exponent $\theta$. Let $\chi \in (0, (1 - 2/\alpha_{\mathrm{4A}})\theta)$. Let $U \subset \CC$ be a deterministic open subset. Let $K \subset U$ be a deterministic compact subset. Then it holds with probability tending to one as $\varepsilon \to 0$ that
    \begin{equation*}
        D_\Upsilon(x, y; U \cap \Upsilon) \le \lvert x - y\rvert^\chi, \quad \forall x, y \in K \cap \Upsilon \text{ with } x \xleftrightarrow{U} y \text{ and } \lvert x - y\rvert \le \varepsilon. 
    \end{equation*}
\end{remark}

\subsection{Estimates for the scaling constants}

In the present subsection, we prove a basic estimate for the scaling constants $\{\kc_r\}_{r > 0}$ (\Cref{280}) which will be useful in \Cref{s:00}. 

\begin{lemma}\label{280}
    For each $\rr > 0$, we have $\kc_{\varepsilon\rr}/(\varepsilon\kc_{\rr}) \to 0$ as $\varepsilon \to 0$, at a rate which is uniform in $\rr$. 
\end{lemma}

\begin{lemma}\label{324}
    There exists a deterministic constant $\xi > 1$ such that the following is true: Let $\rr > 0$. Write $S_{\rr}$ for the open square of side length $\rr$ centered at $0$. For $n \in \NN$, write $\CS_{\rr}^n$ for the collection of open squares $S \subset S_{\rr}$ that have side length $2^{-n}\rr$ and vertices in $(2^{-n}\rr\ZZ)^2$. Write $\SCL_{\rr}$ for the largest loop of $\Gamma$ contained in $B_{4\rr}(0)$ and surrounding $0$, and $\Upsilon_{\rr} \subset \CC$ for the closed subset of points that are surrounded by $\SCL_{\rr}$ but by no other loop of $\Gamma$ inside $\SCL_{\rr}$. Write $F_{\rr}$ for the event that $\SCL_{\rr}$ surrounds $B_{3\rr}(0)$. Then, on the event $F_{\rr}$, there almost surely exists $\delta > 0$ such that 
    \begin{equation*}
        \inf_P \#\{S \in \CS_{\rr}^n : S \cap P \neq \emptyset\} \ge \delta 2^{\xi n}, \quad \forall n \in \NN, 
    \end{equation*}
    where $P$ ranges over all continuous paths in $S_{\rr} \cap \Upsilon_{\rr}$ that crosses between the left and right sides of $S_{\rr}$ (with the convention that $\inf\emptyset = \infty$). 
\end{lemma}

\begin{proof}
    By the scale invariance of the law of $\Gamma$, we may assume without loss of generality that $\rr = 1$. Write $\Gamma_1$ for the collection of loops of $\Gamma$ that are surrounded by $\SCL_1$ but by no other loop of $\Gamma$ inside $\SCL_1$. Then, given $\SCL_1$, $\Gamma_1$ is conditionally a non-nested $\CLE_\kappa$ inside $\SCL_1$. Given $\SCL_1$, let $\Xi_1$ be a Brownian loop-soup of intensity $c(\kappa)$ inside $\SCL_1$. Suppose that $\Gamma_1$ and $\Xi_1$ are coupled so that $\Gamma_1$ is given by the outer boundaries of the outermost clusters of the Brownian loops of $\Xi_1$ (cf.~\cite{CLE}). For $n \in \NN$ and $S \in \CS_1^n$, we set $X(S) \defeq 0$ if $S$ is surrounded by a Brownian loop of $\Xi_1$ that is contained in the union of its eight neighbors, and $X(S) \defeq 1$ otherwise. Then, given $\SCL_1$ and the event $F_1$, the random variables $X(S)$ for $n \in \NN$ and $S \in \CS_1^n$ are conditionally independent Bernoulli random variables. Thus, given $\SCL_1$ and the event $F_1$, 
    \begin{equation*}
        \SCC \defeq \overline{S_1} \setminus \bigcup_{S : X(S) = 0} S
    \end{equation*}
    is conditionally a fractal percolation (with retention probability given by a deterministic constant depending only on $\kappa$) such that $S_1 \cap \Upsilon_1 \subset \SCC$ almost surely. Thus, \Cref{324} follows immediately from \cite[(2.32)]{DimFracPerc}. 
\end{proof}

\begin{lemma}\label{281}
    In the notation of \Cref{324}, there exists $\pp \in (0, 1)$ and $A > 0$ such that the following is true: Let $\rr > 0$. Write $E_{\rr}$ for the event that $F_{\rr}$ occurs and there exists a continuous path $P \subset S_{\rr} \cap \Upsilon_{\rr}$ that crosses between the left and right sides of $S_{\rr}$ such that $\len(P; D_\Upsilon) \le A\kc_{\rr}$. Then $\BP\lbrack E_{\rr}\rbrack \ge \pp$. 
\end{lemma}

\begin{proof}
    This follows immediately from \Cref{010}, Axiom~\eqref{010D} (tightness across scales). 
\end{proof}

\begin{proof}[Proof of \Cref{280}]
    In the notation of \Cref{324,281}, fix $\nu \in (0, \xi - 1)$. By \Cref{010}, Axioms~\eqref{010C} and \eqref{010D} (translation invariance and tightness across scales), for each $p \in (0, 1)$, there exists $a = a(p) > 0$ such that 
    \begin{equation*}
        \BP\lbrack D_\Upsilon(\partial B_{r/2}(z) \cap \Upsilon, \partial B_r(z) \cap \Upsilon) \ge a\kc_r\rbrack \ge p, \quad \forall z \in \CC, \ \forall r > 0. 
    \end{equation*}
    Thus, by \Cref{263} and a union bound, there exists a sufficiently small $a > 0$ such that it holds with polynomially high probability as $\varepsilon \to 0$, at a rate which is uniform in $\rr$, that for each $z \in \left(\frac1{100}\varepsilon^{1 + \nu}\rr\ZZ\right)^2 \cap S_{\rr}$, there exists $r \in [\varepsilon^{1 + \nu}\rr, \varepsilon\rr]$ such that
    \begin{equation}\label{eq:269}
        D_\Upsilon(\partial B_{r/2}(z) \cap \Upsilon, \partial B_r(z) \cap \Upsilon) \ge a\kc_r. 
    \end{equation}
    Thus, by the Borel-Cantelli lemma, almost surely, there exists $\varepsilon_\ast \in (0, 1)$ such that the above event occurs for all $\varepsilon \in (0, \varepsilon_\ast] \cap \{2^k\}_{k \in \ZZ}$. 

    Henceforth assume that $E_{\rr}$ occurs. Then there exists a path $P \subset S_{\rr} \cap \Upsilon_{\rr}$ that crosses between the left and right sides of $S_{\rr}$ such that $\len(P; D_\Upsilon) \le A\kc_{\rr}$. On the other hand, by \Cref{324}, there exists $\delta > 0$ such that 
    \begin{equation*}
        \#\{S \in \CS_{\rr}^n : S \cap P \neq \emptyset\} \ge \delta 2^{\xi n}, \quad \forall n \in \NN. 
    \end{equation*}
    By the scale invariance of the law of $\Gamma$, we may choose $\delta$ so that its law does not depend on $\rr$. Write 
    \begin{align*}
        \CS_{\rr,0}^n &\defeq \left\{S \in \CS_{\rr}^n : (\text{the bottom-left vertex of } S) \in (0, 0) + (2^{-n + 1}\rr\ZZ)^2\right\}; \\
        \CS_{\rr,1}^n &\defeq \left\{S \in \CS_{\rr}^n : (\text{the bottom-left vertex of } S) \in (2^{-n}\rr, 0) + (2^{-n + 1}\rr\ZZ)^2\right\}; \\
        \CS_{\rr,2}^n &\defeq \left\{S \in \CS_{\rr}^n : (\text{the bottom-left vertex of } S) \in (0, 2^{-n}\rr) + (2^{-n + 1}\rr\ZZ)^2\right\}; \\
        \CS_{\rr,3}^n &\defeq \left\{S \in \CS_{\rr}^n : (\text{the bottom-left vertex of } S) \in (2^{-n}\rr, 2^{-n}\rr) + (2^{-n + 1}\rr\ZZ)^2\right\}. 
    \end{align*}
    Then $\CS_{\rr}^n = \CS_{\rr,0}^n \coprod \CS_{\rr,1}^n \coprod \CS_{\rr,2}^n \coprod \CS_{\rr,3}^n$. By the pigeonhole principle, for each $n \in \NN$, there exists $\bullet \in \{0, 1, 2, 3\}$ such that 
    \begin{equation}\label{eq:270}
        \#\{S \in \CS_{\rr,\bullet}^n : S \cap P \neq \emptyset\} \ge (\delta/4) 2^{\xi n}. 
    \end{equation}
    Note that $\dist(S_1, S_2) \ge 2^{-n}\rr$ for all distinct $S_1, S_2 \in \CS_{\rr,\bullet}^n$. Setting $\varepsilon \defeq 2^{-n - 2}$ and combining \eqref{eq:269} and \eqref{eq:270}, we obtain that 
    \begin{equation*}
        A\kc_{\rr} \ge \len(P; D_\Upsilon) \ge (\delta/4) 2^{\xi n} a \kc_{2^{-(1 + \nu)(n + 2)}\rr}
    \end{equation*}
    for all sufficiently large $n \in \NN$. This completes the proof. 
\end{proof}

\subsection{Estimates within thin annuli}

In the present subsection, we prove some estimates concerning the intersection of an annulus with $\Upsilon$ as the width of the annulus goes to zero. The results of the present subsection will be used in \Cref{s:00,s:02} to control the amount of time that a $D_\Upsilon$-geodesic can spend in a thin annulus.

\begin{lemma}\label{322}
    Let $\nu \in (0, 1)$. Let $U \subset \CC$ be a deterministic bounded simply connected domain with smooth boundary. Then it holds with superpolynomially high probability as $\varepsilon \to 0$ that every connected component of $(B_\varepsilon(U) \setminus U) \cap \Upsilon$ has Euclidean diameter at most $\varepsilon^\nu$. 
\end{lemma}

\begin{proof}
    For $z \in \CC$ and $r > 0$, write $E_r(z)$ for the event that there exists a loop of $\Gamma$ in $A_{r,2r}(z)$ disconnecting the inner and outer boundaries of $A_{r,2r}(z)$. By the scaling and translation invariance of the law of $\Gamma$, there exists $p \in (0, 1)$ such that
    \begin{equation*}
        \BP\lbrack E_r(z)\rbrack = p, \quad \forall z \in \CC, \ \forall r > 0. 
    \end{equation*}
    Since $\partial U$ is smooth, by the tubular neighborhood theorem (cf.~\cite[Theorem~6.24]{GTM218}), there exists $\varepsilon_\ast > 0$ such that the mapping $\partial U \times [0, \varepsilon_\ast) \to \CC \colon (z, \varepsilon) \mapsto z + \varepsilon n_z$ is a diffeomorphism --- where $n_z$ denotes the unit normal vector to $\partial U$ at $z$. In particular, for each $z \in \partial U$ and $\varepsilon \in (0, \varepsilon_\ast)$, we have $\partial B_\varepsilon(z) \cap \partial B_\varepsilon(U) \neq \emptyset$. Since $\partial U$ is smooth, there exists a constant $A > 0$ such that for each $\varepsilon > 0$, if we choose $z_1, \ldots, z_n \in \partial U$ in counterclockwise order such that $\len([z_j, z_{j + 1}]_{\partial U}^\circlearrowleft) = A\varepsilon$ for all $j \in [1, n - 1]_\ZZ$ and $A\varepsilon \le \len([z_n, z_1]_{\partial U}^\circlearrowleft) < 2A\varepsilon$, then $\lvert z_j - z_k\rvert \ge 5\varepsilon$ for all distinct $j, k \in [1, n]_\ZZ$. Thus, by \Cref{265}, there exists $\alpha > 0$ and $C > 0$ such that for each connected arc $I \subset \partial U$ of Euclidean length $\varepsilon^\nu/4$, we have
    \begin{equation}\label{eq:345}
        \BP\lbrack\text{there exists } j \in [1, n]_\ZZ \text{ such that } z_j \in I \text{ and } E_\varepsilon(z_j) \text{ occurs}\rbrack \ge 1 - C\exp(-\alpha\varepsilon^{\nu - 1}). 
    \end{equation}
    Since there exist $\lceil4\varepsilon^{-\nu}\len(\partial U)\rceil$ connected arcs of $\partial U$ of Euclidean length $\varepsilon^\nu/4$ such that every connected arc of $\partial U$ of Euclidean length $\varepsilon^\nu/2$ contains at least one of them, by \eqref{eq:345} and a union bound, it holds with superpolynomially high probability as $\varepsilon \to 0$ that for each connected arc $I \subset \partial U$ of Euclidean length $\varepsilon^\nu/2$, there exists $j \in [1, n]_\ZZ$ such that $z_j \in I$ and $E_\varepsilon(z_j)$ occurs. Since $\partial B_\varepsilon(z) \cap \partial B_\varepsilon(U) \neq \emptyset$ by the above discussion, it follows that the loop described in the event $E_\varepsilon(z_j)$ must cross between the inner and outer boundaries of $B_\varepsilon(U) \setminus U$. This completes the proof. 
\end{proof}

\begin{lemma}\label{323}
    There exists a deterministic constant $\zeta \in (0, 1)$ such that the following is true: Let $U \subset \CC$ be a bounded simply connected domain with smooth boundary. Let $\nu \in (0, 1)$ and $\delta > 0$. Then it holds with probability tending to one as $\varepsilon \to 0$ that there exists a subset $Z \subset \left(\frac1{100}\varepsilon^\nu\ZZ\right)^2$ with $\#Z \le \varepsilon^{-\zeta\nu}$ such that for each connected component $X \subset \Upsilon$ that has Euclidean diameter at least $\delta$ and each connected component $\SCC$ of $(B_\varepsilon(U) \setminus U) \cap X$, there exists $z \in Z$ such that $\SCC \subset B_{\varepsilon^\nu}(z)$. 
\end{lemma}

\begin{proof}
    Since the boundary $\partial U$ is bounded and smooth, for each $\varepsilon > 0$, there exists a deterministic subset $S \subset \left(\frac1{100}\varepsilon^\nu\ZZ\right)^2$ with $\#S = O(\varepsilon^{-\nu})$ as $\varepsilon \to 0$ such that
    \begin{equation}\label{eq:344}
        B_\varepsilon(U) \setminus U \subset \bigcup_{z \in S} B_{\varepsilon^\nu/2}(z). 
    \end{equation}
    For $z \in \CC$ and $r > 0$, write $E_r(z)$ for the event that there exists a loop of $\Gamma$ in $A_{r,2r}(z)$ disconnecting the inner and outer boundaries of $A_{r,2r}(z)$. By the scaling and translation invariance of the law of $\Gamma$, there exists $p \in (0, 1)$ such that
    \begin{equation*}
        \BP\lbrack E_r(z)\rbrack = p, \quad \forall z \in \CC, \ \forall r > 0. 
    \end{equation*}
    Thus, by \Cref{273}, there exists $\alpha > 0$ and $C > 0$ such that
    \begin{multline}\label{eq:343}
        \BP\!\left\lbrack\text{there exists } r \in (\varepsilon^\nu, \delta/2) \cap \{2^k\}_{k \in \ZZ} \text{ such that } E_r(z) \text{ occurs}\right\rbrack \ge 1 - C\varepsilon^{\alpha\nu}, \\
        \forall \varepsilon > 0, \ \forall z \in \CC. 
    \end{multline}
    Write 
    \begin{equation*}
        Z \defeq \left\{z \in S : \text{there does not exist } r \in (\varepsilon^\nu, \delta/2) \cap \{2^k\}_{k \in \ZZ} \text{ such that } E_r(z) \text{ occurs}\right\}. 
    \end{equation*}
    Then, by \eqref{eq:343}, we have $\BE\lbrack\#Z\rbrack \le C\varepsilon^{\alpha\nu}\#S = O(\varepsilon^{(\alpha - 1)\nu})$ as $\varepsilon \to 0$. (In particular, we must have $\alpha \in (0, 1)$.) Choose $\zeta \in (1 - \alpha, 1)$. Then it follows from Markov's inequality that $\BP\lbrack\#Z \le \varepsilon^{-\zeta\nu}\rbrack \to 1$ as $\varepsilon \to 0$. On the other hand, by \Cref{322}, it holds with probability tending to one as $\varepsilon \to 0$ that every connected component of $(B_\varepsilon(U) \setminus U) \cap \Upsilon$ has Euclidean diameter at most $\varepsilon^\nu/2$. Henceforth assume that $\#Z \le \varepsilon^{-\zeta\nu}$ and every connected component of $(B_\varepsilon(U) \setminus U) \cap \Upsilon$ has Euclidean diameter at most $\varepsilon^\nu/2$. It suffices to verify that for each connected component $X \subset \Upsilon$ that has Euclidean diameter at least $\delta$ and each connected component of $\SCC \subset (B_\varepsilon(U) \setminus U) \cap X$, there exists $z \in Z$ such that $\SCC \subset B_{\varepsilon^\nu}(z)$. By \eqref{eq:344}, we may choose $z \in S$ such that $\dist(\SCC, z) \le \varepsilon^\nu/2$. Since the Euclidean diameter of $\SCC$ is at most $\varepsilon^\nu/2$, it follows that $\SCC \subset B_{\varepsilon^\nu}(z)$. Since $\SCC$ is contained in a connected component of $\Upsilon$ that has Euclidean diameter at least $\delta$, it follows that $E_r(z)$ cannot occur for all $r \in (\varepsilon^\nu, \delta/2)$, hence that $z \in Z$. This completes the proof. 
\end{proof}

\section{Local metrics}\label{s:05}

In the present section, we consider local metrics, which are, a priori, even weaker metrics than the weak geodesic $\CLE_\kappa$ carpet metric. We first prove that any two conditionally independent local metrics are bi-Lipschitz equivalent (\Cref{045}). We then prove that any local metric is almost surely determined by the $\CLE_\kappa$ carpet (\Cref{006}). \Cref{006} will be applied in \Cref{s:01} to prove that any subsequential limit is a weak geodesic $\CLE_\kappa$ carpet metric. The argument of the present section is similar to the argument of \cite{LocMetGFF}.

\begin{definition}\label{048}
    Let $(\Upsilon, D_\Upsilon)$ be a coupling between $\Upsilon$ and a metric $D_\Upsilon$ on $\Upsilon$. Then we shall refer to $D_\Upsilon$ as a \emph{local} metric for $\Upsilon$ if Axioms~\eqref{010A}, \eqref{010C}, \eqref{010D}, \eqref{010E} of \Cref{010}, together with the following condition, are satisfied:
    \begin{enumerate}[label=(\Roman*'), ref=\Roman*']
        \setcounter{enumi}1
        \item\label{048B} Let $U \subset \CC$ be a deterministic open subset. Then $D_\Upsilon(\bullet, \bullet; U \cap \Upsilon)$ is conditionally independent of the pair $\left((\CC \setminus \overline U) \cap \Upsilon, D_\Upsilon(\bullet, \bullet; (\CC \setminus \overline U) \cap \Upsilon)\right)$ given $U \cap \Upsilon$.
    \end{enumerate}
\end{definition}

\subsection{Bi-Lipschitz equivalence}

\begin{proposition}\label{045}
    Let $D_\Upsilon$ and $\widetilde D_\Upsilon$ be local metrics for $\Upsilon$ with the same scaling constants $\{\kc_r\}_{r > 0}$. Suppose that $D_\Upsilon$ and $\widetilde D_\Upsilon$ are conditionally independent given $\Upsilon$. Then there is a deterministic constant $M \ge 1$ such that $M^{-1}D_\Upsilon \le \widetilde D_\Upsilon \le MD_\Upsilon$ almost surely. 
\end{proposition}

\begin{lemma}\label{002}
    Let $D_\Upsilon$ and $\widetilde D_\Upsilon$ be as in \Cref{045}. Then for each $\lambda \in (0, 1)$, $\alpha > 0$, and $b \in (0, 1)$, there exists $p = p(\lambda, \alpha, b) \in (0, 1)$ and $C = C(\lambda, \alpha, b) > 0$ such that the following is true: Let $z \in \CC$. Let $\{r_j\}_{j \in \NN}$ be a decreasing sequence of positive real numbers such that $r_{j + 1}/r_j \le \lambda$ for all $j \in \NN$. Let $\{E_j\}_{j \in \NN}$ be a sequence of events such that 
    \begin{equation*}
        E_j \in \sigma\left(A_{r_{j + 1},r_j}(z) \cap \Upsilon, \ D_\Upsilon(\bullet, \bullet; A_{r_{j + 1},r_j}(z) \cap \Upsilon), \ \widetilde D_\Upsilon(\bullet, \bullet; A_{r_{j + 1},r_j}(z) \cap \Upsilon)\right) \quad \text{and} \quad \BP\lbrack E_j\rbrack \ge p. 
    \end{equation*}
    for all $j \in \NN$. 
    Then 
    \begin{equation*}
        \BP\lbrack\#\left\{j \in [1, n]_\ZZ : E_j \text{ occurs}\right\} \ge bn\rbrack \ge 1 - C\re^{-\alpha n}, \quad \forall n \in \NN. 
    \end{equation*}
\end{lemma}

\begin{proof}
    This follows from \Cref{048}, Axiom~\eqref{048B}, together with the same argument as the argument applied in the proof of \Cref{263}. 
\end{proof}

\begin{lemma}\label{011}
    Let $D_\Upsilon$ and $\widetilde D_\Upsilon$ be as in \Cref{045}. Then for each $\lambda \in (0, 1)$, $\alpha > 0$, and $b \in (0, 1)$ there exists $A = A(\lambda, \alpha, b) > 0$ and $C = C(\lambda, \alpha, b) > 0$ such that the following is true: Let $z \in \CC$. Let $\{r_j\}_{j \in \NN}$ be a decreasing sequence of positive real numbers such that $r_{j + 1}/r_j \le \lambda$ for all $j \in \NN$. Then for each $n \in \NN$, it holds with probability at least $1 - C\re^{-\alpha n}$ that there are at least $bn$ values of $j \in [1, n]_\ZZ$ for which we have
    \begin{equation*}
        D_\Upsilon\left(x, y; B_{2r_j}(z) \cap \Upsilon\right) \vee \widetilde D_\Upsilon\left(x, y; B_{2r_j}(z) \cap \Upsilon\right) \le A\kc_{r_j}, \quad \forall x, y \in B_{r_j}(z) \cap \Upsilon \text{ with } x \xleftrightarrow{B_{r_j}(z)} y. 
    \end{equation*}
\end{lemma}

\begin{proof}
    This follows from the same argument as the argument applied in the proof of \Cref{214}. 
\end{proof}

\begin{proof}[Proof of \Cref{045}]
    By symmetry, it suffices to consider the inequality ``$\widetilde D_\Upsilon \le MD_\Upsilon$''. 

    By \Cref{010}, Axiom~\eqref{010D} (tightness across scales), for each $p \in (0, 1)$, there exists $a = a(p) > 0$ such that for each $r > 0$, it holds with probability at least $p$ that
    \begin{equation*}
        D_\Upsilon(\partial B_{r/2}(0) \cap \Upsilon, \partial B_r(0) \cap \Upsilon) \ge a\kc_r. 
    \end{equation*}
    Thus, by \Cref{010}, Axiom~\eqref{010C} (translation invariance), \Cref{002,011}, and a union bound, there exists a sufficiently small $a > 0$ and a sufficiently large $A > 0$ such that it holds with polynomially high probability as $\varepsilon \to 0$ that for each $z \in \left(\frac1{100}\varepsilon^2\ZZ\right)^2 \cap B_{1/\varepsilon}(0)$, there exists $r \in [\varepsilon^2, \varepsilon]$ such that the following are true:
    \begin{enumerate}
        \item\label{it:301} $D_\Upsilon(\partial B_{r/2}(z) \cap \Upsilon, \partial B_r(z) \cap \Upsilon) \ge a\kc_r$. 
        \item\label{it:302} We have
        \begin{equation*}
            \widetilde D_\Upsilon(x, y) \le A\kc_r, \quad \forall x, y \in B_r(z) \cap \Upsilon \text{ with } x \xleftrightarrow{B_r(z)} y. 
        \end{equation*}
    \end{enumerate}
    Thus, by the Borel-Cantelli lemma, almost surely, there exists $\varepsilon_\ast \in (0, 1)$ such that the above event occurs for all $\varepsilon \in (0, \varepsilon_\ast] \cap \{2^k\}_{k \in \ZZ}$. 

    Let $x, y \in \Upsilon$ with $x \leftrightarrow y$. Let $P \colon [0, 1] \to \Upsilon$ be a $D_\Upsilon$-geodesic from $x$ to $y$. Choose a sufficiently small $\varepsilon \in (0, \varepsilon_\ast] \cap \{2^k\}_{k \in \ZZ}$ such that $P \subset B_{1/\varepsilon}(0)$. Set $t_0 \defeq 0$ and choose $z_1 \in \left(\frac1{100}\varepsilon^2\ZZ\right)^2 \cap B_{1/\varepsilon}(0)$ and $r_1 \in [\varepsilon^2, \varepsilon]$ such that $x \in B_{\varepsilon^2/2}(z_1)$ and conditions~\eqref{it:301} and \eqref{it:302} (with $(z_1, r_1)$ in place of $(z, r)$) hold. Then, inductively, for each $j \in \NN$, if $y \notin B_{r_j}(z_j)$, then set $t_j$ to be the first time after $t_{j - 1}$ at which $P$ hits $\partial B_{r_j}(z_j)$, and choose $z_{j + 1} \in \left(\frac1{100}\varepsilon^2\ZZ\right)^2 \cap B_{1/\varepsilon}(0)$ and $r_{j + 1} \in [\varepsilon^2, \varepsilon]$ such that $P(t_j) \in B_{\varepsilon^2/2}(z_{j + 1})$ and conditions~\eqref{it:301} and \eqref{it:302} (with $(z_{j + 1}, r_{j + 1})$ in place of $(z, r)$) hold; otherwise set $t_j \defeq 1$. Write $n \defeq \inf\{j \in \NN : t_j = 1\}$. Since $P(t_{j - 1}) \in B_{r_j/2}(z_j)$ and $P(t_j) \in \partial B_{r_j}(z_j)$ for all $j \in [1, n - 1]_\ZZ$, it follows from condition~\eqref{it:301} that
    \begin{equation}\label{eq:335}
        D_\Upsilon(P(t_{j - 1}), P(t_j)) \ge a\kc_{r_j}, \quad \forall j \in [1, n - 1]_\ZZ. 
    \end{equation}
    Since $P([t_{j - 1}, t_j]) \subset B_{r_j}(z_j)$ for all $j \in [1, n]_\ZZ$, it follows from condition~\eqref{it:302} that
    \begin{equation}\label{eq:336}
        \widetilde D_\Upsilon(P(t_{j - 1}), P(t_j)) \le A\kc_{r_j}, \quad \forall j \in [1, n]_\ZZ. 
    \end{equation}
    Combining \eqref{eq:340}, \eqref{eq:335}, and \eqref{eq:336}, we obtain that
    \begin{multline*}
        \widetilde D_\Upsilon(x, y) \le \sum_{j = 1}^n \widetilde D_\Upsilon(P(t_{j - 1}), P(t_j)) \le \sum_{j = 1}^n A\kc_{r_j} \le \frac Aa \sum_{j = 1}^{n - 1} D_\Upsilon(P(t_{j - 1}), P(t_j)) + A\kc_{r_n} \\
        \le \frac Aa D_\Upsilon(x, y) + A\kK\varepsilon\kc_1. 
    \end{multline*}
    Since $\varepsilon$ may be chosen to be arbitrarily small, we conclude that $\widetilde D_\Upsilon(x, y) \le (A/a)D_\Upsilon(x, y)$. This completes the proof. 
\end{proof}

\subsection{Measurability}

\begin{proposition}\label{006}
    Let $D_\Upsilon$ be a local metric for $\Upsilon$. Then $D_\Upsilon$ is almost surely determined by $\Upsilon$. In particular, $D$ determines a weak geodesic $\CLE_\kappa$ carpet metric. 
\end{proposition}

The remainder of the present subsection is devoted to the proof of \Cref{006}.

\begin{lemma}\label{046}
    There is a deterministic constant $M \ge 1$ such that the following is true: Suppose that distinct points $x, y \in \Upsilon$ are chosen in a manner that is almost surely determined by $\Upsilon$ such that $x \leftrightarrow y$ almost surely. Then 
    \begin{equation*}
        M^{-1}\BE\lbrack D_\Upsilon(x, y) \mid \Upsilon\rbrack \le D_\Upsilon(x, y) \le M\BE\lbrack D_\Upsilon(x, y) \mid \Upsilon\rbrack \quad \text{almost surely}. 
    \end{equation*}
\end{lemma}

\begin{proof}
    Given $\Upsilon$, let $\widetilde D_\Upsilon$ be a conditionally independent copy of $D_\Upsilon$. Then, by \Cref{045}, there exists a deterministic constant $M \ge 1$ such that, almost surely, 
    \begin{equation}\label{eq:020}
        M^{-1}D_\Upsilon(u, v) \le \widetilde D_\Upsilon(u, v) \le MD_\Upsilon(u, v), \quad \forall u, v \in \Upsilon \text{ with } u \leftrightarrow v. 
    \end{equation}
    Write 
    \begin{equation*}
        a(\Upsilon) \defeq \inf\{t > 0 : \BP\lbrack D_\Upsilon(x, y) \le t \mid \Upsilon\rbrack > 0\}; \quad A(\Upsilon) \defeq \sup\{t > 0 : \BP\lbrack D_\Upsilon(x, y) \ge t \mid \Upsilon\rbrack > 0\}.
    \end{equation*}
    Then $a(\Upsilon)$ and $A(\Upsilon)$ are almost surely determined by $\Upsilon$, and 
    \begin{equation*}
        D_\Upsilon(x, y), \BE\lbrack D_\Upsilon(x, y) \mid \Upsilon\rbrack \in [a(\Upsilon), A(\Upsilon)] \quad \text{almost surely}. 
    \end{equation*}
    Thus, it suffices to show that $A(\Upsilon) / a(\Upsilon) \le M$ almost surely. (In particular, $a(\Upsilon) > 0$ and $A(\Upsilon) < \infty$ almost surely.) Choose $\lambda \in (0, A(\Upsilon) / a(\Upsilon))$, in a manner that is almost surely determined by $\Upsilon$. Then there exists $a^\prime > a(\Upsilon)$ and $A^\prime < A(\Upsilon)$ such that $A^\prime / a^\prime \ge \lambda$. Then
    \begin{multline}\label{eq:337}
        \BP\!\left\lbrack\widetilde D_\Upsilon(x, y) / D_\Upsilon(x, y) \ge \lambda \ \middle\vert \ \Upsilon\right\rbrack \ge \BP\!\left\lbrack\widetilde D_\Upsilon(x, y) \ge A^\prime , \ D_\Upsilon(x, y) \le a^\prime \ \middle\vert \ \Upsilon\right\rbrack \\
        = \BP\!\left\lbrack\widetilde D_\Upsilon(x, y) \ge A^\prime \ \middle\vert \ \Upsilon\right\rbrack \BP\lbrack D_\Upsilon(x, y) \le a^\prime \mid \Upsilon\rbrack > 0. 
    \end{multline}
    Combining \eqref{eq:020} and \eqref{eq:337}, we obtain that $\lambda \le M$ almost surely. By letting $\lambda \to A(\Upsilon) / a(\Upsilon)$, we conclude that $A(\Upsilon) / a(\Upsilon) \le M$ almost surely. This completes the proof. 
\end{proof}

\begin{lemma}\label{047}
    There is a deterministic constant $M \ge 1$ such that the following is true: Suppose that $D_\Upsilon$ and $\widetilde D_\Upsilon$ have the same conditional law given $\Upsilon$. ($D_\Upsilon$ and $\widetilde D_\Upsilon$ are not necessarily conditionally independent given $\Upsilon$.) Then $M^{-1}D_\Upsilon \le \widetilde D_\Upsilon \le MD_\Upsilon$ almost surely.
\end{lemma}

\begin{proof}
    Let $M$ be as in \Cref{046}. Label the connected components of $\Upsilon$ as $\Upsilon_1, \Upsilon_2, \ldots$, in a manner that is almost surely determined by $\Upsilon$. For each $j \in \NN$, choose a pairwise distinct sequence $x_{j,1}, x_{j,2}, \ldots \in \Upsilon_j$, in a manner that is almost surely determined by $\Upsilon$, such that $\{x_{j,k}\}_{k \in \NN}$ is almost surely dense in $\Upsilon_j$. Then it follows from \Cref{046} that, almost surely, 
    \begin{multline*}
        M^{-1}\BE\!\left\lbrack D_\Upsilon(x_{j,k}, x_{j,l}) \mid \Upsilon\right\rbrack \le D_\Upsilon(x_{j,k}, x_{j,l}) \le M\BE\!\left\lbrack D_\Upsilon(x_{j,k}, x_{j,l}) \mid \Upsilon\right\rbrack \\
        \text{and} \quad M^{-1}\BE\!\left\lbrack D_\Upsilon(x_{j,k}, x_{j,l}) \mid \Upsilon\right\rbrack \le \widetilde D_\Upsilon(x_{j,k}, x_{j,l}) \le M\BE\!\left\lbrack D_\Upsilon(x_{j,k}, x_{j,l}) \mid \Upsilon\right\rbrack, \quad \forall j, k, l \in \NN. 
    \end{multline*}
    This implies that, almost surely, 
    \begin{equation*}
        M^{-2}D_\Upsilon(x_{j,k}, x_{j,l}) \le \widetilde D_\Upsilon(x_{j,k}, x_{j,l}) \le M^2D_\Upsilon(x_{j,k}, x_{j,l}), \quad \forall j, k, l \in \NN. 
    \end{equation*}
    Thus, we conclude from the continuity of $D_\Upsilon$ and $\widetilde D_\Upsilon$ that, almost surely, 
    \begin{equation*}
        M^{-2}D_\Upsilon(x, y) \le \widetilde D_\Upsilon(x, y) \le M^2D_\Upsilon(x, y), \quad \forall x, y \in \Upsilon \text{ with } x \leftrightarrow y. 
    \end{equation*}
    This completes the proof. 
\end{proof}

For $O \in \CC / \ZZ^2$ and $\varepsilon \in (0, 1)$, we shall write 
\begin{equation*}
    \GG_\varepsilon(O) \defeq \left\{z \in \CC : \Re(z - \varepsilon O) \in \varepsilon\ZZ \text{ or } \Im(z - \varepsilon O) \in \varepsilon\ZZ\right\}, 
\end{equation*}
and $\CS_\varepsilon(O)$ for the collection of connected components of $\CC \setminus \GG_\varepsilon(O)$. 

Let $O$ be sampled uniformly from $\CC / \ZZ^2$, independently of everything else. 

\begin{lemma}\label{041}
    Let $P \colon [0, 1] \to \Upsilon$ be a path of finite $D_\Upsilon$-length, chosen in a manner that is almost surely determined by $\Upsilon$ and $D_\Upsilon$. Then for each $\varepsilon \in (0, 1)$, 
    \begin{equation}\label{eq:052}
        \len(P; D_\Upsilon) = \sum_{S \in \CS_\varepsilon(O)} \len(P \cap S; D_\Upsilon) \quad \text{almost surely}.
    \end{equation}
\end{lemma}

\begin{proof}
    Since $O$ is independent of $P$, it follows that for each deterministic $t \in [0, 1]$, we have $\BP\lbrack P(t) \in \GG_\varepsilon(O) \mid \Upsilon, \ D_\Upsilon\rbrack = 0$. This implies that the Lebesgue measure of the set $\{t \in [0, 1] : P(t) \in \GG_\varepsilon(O)\}$ is almost surely zero, which implies \eqref{eq:052}.
\end{proof}

\begin{lemma}\label{042}
    For each $\varepsilon \in (0, 1)$, the metric $D_\Upsilon$ is almost surely determined by $\Upsilon$, $O$, and the internal metrics $D_\Upsilon(\bullet, \bullet; S \cap \Upsilon)$ for $S \in \CS_\varepsilon(O)$. 
\end{lemma}

\begin{proof}
    Given $\Upsilon$, $O$, and the internal metrics $D_\Upsilon(\bullet, \bullet; S \cap \Upsilon)$ for $S \in \CS_\varepsilon(O)$, let $\widetilde D_\Upsilon$ be a conditionally independent copy of $D_\Upsilon$. Thus, it suffices to show that $D_\Upsilon = \widetilde D_\Upsilon$ almost surely. By definition, almost surely, 
    \begin{equation}\label{eq:031}
        D_\Upsilon(\bullet, \bullet; S \cap \Upsilon) = \widetilde D_\Upsilon(\bullet, \bullet; S \cap \Upsilon), \quad \forall S \in \CS_\varepsilon(O). 
    \end{equation}
    By \Cref{047}, there exists a deterministic constant $M \ge 1$ such that, almost surely, 
    \begin{equation}\label{eq:030}
        M^{-1}D_\Upsilon(u, v) \le \widetilde D_\Upsilon(u, v) \le MD_\Upsilon(u, v), \quad \forall u, v \in \Upsilon \text{ with } u \leftrightarrow v. 
    \end{equation}
    Choose $x, y \in \Upsilon$ with $x \leftrightarrow y$ and a $D_\Upsilon$-geodesic $P$ from $x$ to $y$, in a manner that is almost surely determined by $\Upsilon$ and $D_\Upsilon$. By \eqref{eq:030}, the $\widetilde D_\Upsilon$-length of $P$ is almost surely finite. Thus, we conclude from \Cref{041} and \eqref{eq:031} that, almost surely,  
    \begin{equation*}
        D_\Upsilon(x, y) = \sum_{S \in \CS_\varepsilon(O)} \len(P \cap S; D_\Upsilon) = \sum_{S \in \CS_\varepsilon(O)} \len(P \cap S; \widetilde D_\Upsilon) = \len(P; \widetilde D_\Upsilon) \ge \widetilde D_\Upsilon(x, y). 
    \end{equation*}
    By symmetry, it follows that $D_\Upsilon(x, y) \le \widetilde D_\Upsilon(x, y)$ almost surely, hence that $D_\Upsilon(x, y) = \widetilde D_\Upsilon(x, y)$ almost surely. Let $\{x_{j,k}\}_{j,k \in \NN}$ be as in the proof of \Cref{047}. Then it follows that, almost surely, $D_\Upsilon(x_{j,k}, x_{j,l}) = \widetilde D_\Upsilon(x_{j,k}, x_{j,l})$ for all $j, k, l \in \NN$. Thus, we conclude from the continuity of $D_\Upsilon$ and $\widetilde D_\Upsilon$ that $D_\Upsilon = \widetilde D_\Upsilon$ almost surely. This completes the proof. 
\end{proof}

\begin{lemma}\label{043}
    For each $\varepsilon \in (0, 1)$, the internal metrics $D_\Upsilon(\bullet, \bullet; S \cap \Upsilon)$ for $S \in \CS_\varepsilon(O)$ are conditionally independent given $\Upsilon$ and $O$. 
\end{lemma}

\begin{proof}
    This follows immediately from \Cref{048}, Axiom~\eqref{048B}. 
\end{proof}

Before proceeding, we recall the following well-known  concentration result.

\begin{lemma}[Efron-Stein]\label{044}
    Let $X_1, \ldots, X_n$ be independent random variables and $(X_1^\prime, \ldots, X_n^\prime)$ be an independent copy of $(X_1, \ldots, X_n)$. Then
    \begin{equation*}
        \Var\lbrack F(X_1, \ldots, X_n)\rbrack \le \frac12 \sum_{j = 1}^n \BE\!\left\lbrack\left(F(X_1, \ldots, X_n) - F(X_1, \ldots, X_{j - 1}, X_j^\prime, X_{j + 1}, \ldots, X_n)\right)^2\right\rbrack
    \end{equation*}
    for all measurable functionals $F$. 
\end{lemma}

\begin{proof}[Proof of \Cref{006}]
    Fix $\varepsilon > 0$. By \Cref{042}, there exists a measurable mapping $F$ such that 
    \begin{equation*}
        D_\Upsilon = F(\Upsilon, O, \{D_\Upsilon(\bullet, \bullet; S \cap \Upsilon) : S \in \CS_\varepsilon(O)\}) \quad \text{almost surely}. 
    \end{equation*}
    Given $\Upsilon$ and $O$, let $\{D_\Upsilon^\prime(\bullet, \bullet; S \cap \Upsilon) : S \in \CS_\varepsilon(O)\}$ be a conditionally independent copy of $\{D_\Upsilon(\bullet, \bullet; S \cap \Upsilon) : S \in \CS_\varepsilon(O)\}$. Given $\Upsilon$ and $O$, for each $S \in \CS_\varepsilon(O)$, write 
    \begin{equation*}
        D_\Upsilon^S \defeq F\left(\Upsilon, O, \{D_\Upsilon^\prime(\bullet, \bullet; S \cap \Upsilon)\} \cup \{D_\Upsilon(\bullet, \bullet; \widetilde S) : \widetilde S \in \CS_\varepsilon(O), \ \widetilde S \neq S\}\right). 
    \end{equation*}
    Choose $x, y \in \Upsilon$ with $x \leftrightarrow y$, in a manner that is almost surely determined by $\Upsilon$. Then it follows from \Cref{043,044} that, almost surely, 
    \begin{equation}\label{eq:029}
        \Var\lbrack D_\Upsilon(x, y) \mid \Upsilon\rbrack = \Var\lbrack D_\Upsilon(x, y) \mid \Upsilon, \ O\rbrack \le \frac12 \sum_{S \in \CS_\varepsilon(O)} \BE\!\left\lbrack\left(D_\Upsilon^S(x, y) - D_\Upsilon(x, y)\right)^2 \ \middle\vert \ \Upsilon, \ O\right\rbrack. 
    \end{equation}

    We \emph{claim} that, almost surely, 
    \begin{equation}\label{eq:028}
        \sum_{S \in \CS_\varepsilon(O)} \BE\!\left\lbrack\left(D_\Upsilon^S(x, y) - D_\Upsilon(x, y)\right)^2 \ \middle\vert \ \Upsilon, \ O\right\rbrack \to 0 \quad \text{as } \varepsilon \to 0. 
    \end{equation}
    By \Cref{047}, there exists a deterministic constant $M \ge 1$ such that, almost surely, 
    \begin{equation}\label{eq:021}
        M^{-1}D_\Upsilon(u, v) \le D_\Upsilon^S(u, v) \le MD_\Upsilon(u, v), \quad \forall S \in \CS_\varepsilon(O), \ \forall u, v \in \Upsilon \text{ with } u \leftrightarrow v. 
    \end{equation}
    Choose a $D_\Upsilon$-geodesic $P$ from $x$ to $y$, in a manner that is almost surely determined by $\Upsilon$ and $D_\Upsilon$. By \Cref{041} and \eqref{eq:021}, we have
    \begin{multline}\label{eq:022}
        D_\Upsilon^S(x, y) \le \len(P; D_\Upsilon^S) = \sum_{\widetilde S \in \CS_\varepsilon(O)} \len(P \cap \widetilde S; D_\Upsilon^S) \\
        = \sum_{\widetilde S \in \CS_\varepsilon(O) \setminus \{S\}} \len(P \cap \widetilde S; D_\Upsilon) + \len(P \cap S; D_\Upsilon^S) \le D_\Upsilon(x, y) + M\len(P \cap S; D_\Upsilon). 
    \end{multline}
    By symmetry, we have
    \begin{equation}\label{eq:023}
        \BE\!\left\lbrack\left(D_\Upsilon^S(x, y) - D_\Upsilon(x, y)\right)^2 \ \middle\vert \ \Upsilon, \ O\right\rbrack = 2\BE\!\left\lbrack\left(D_\Upsilon^S(x, y) - D_\Upsilon(x, y)\right)_+^2 \ \middle\vert \ \Upsilon, \ O\right\rbrack, 
    \end{equation}
    where $(D_\Upsilon^S(x, y) - D_\Upsilon(x, y))_+ \defeq (D_\Upsilon^S(x, y) - D_\Upsilon(x, y)) \vee 0$. Combining \eqref{eq:022} and \eqref{eq:023}, we obtain that
    \begin{align}\label{eq:024}
        &\sum_{S \in \CS_\varepsilon(O)} \BE\!\left\lbrack\left(D_\Upsilon^S(x, y) - D_\Upsilon(x, y)\right)^2 \ \middle\vert \ \Upsilon, \ O\right\rbrack \\
        &= 2\sum_{S \in \CS_\varepsilon(O)} \BE\!\left\lbrack\left(D_\Upsilon^S(x, y) - D_\Upsilon(x, y)\right)_+^2 \ \middle\vert \ \Upsilon, \ O\right\rbrack \notag \\
        &\le 2M^2 \sum_{S \in \CS_\varepsilon(O)} \BE\!\left\lbrack\len(P \cap S; D_\Upsilon)^2 \ \middle\vert \ \Upsilon, \ O\right\rbrack \notag \\
        &\le 2M^2 \BE\!\left\lbrack \left(\max_{S \in \CS_\varepsilon(O)} \len(P \cap S; D_\Upsilon)\right) \sum_{S \in \CS_\varepsilon(O)} \len(P \cap S; D_\Upsilon) \ \middle\vert \ \Upsilon, \ O\right\rbrack \notag \\
        &\le 2M^2 \BE\!\left\lbrack \left(\max_{S \in \CS_\varepsilon(O)} \len(P \cap S; D_\Upsilon)\right)^2 \ \middle\vert \ \Upsilon, \ O\right\rbrack^{1/2} \BE\!\left\lbrack D_\Upsilon(x, y)^2 \ \middle\vert \ \Upsilon\right\rbrack^{1/2}. \notag 
    \end{align}
    By \Cref{046} and possibly increasing $M$, we may assume without loss of generality that 
    \begin{equation}\label{eq:026}
        M^{-1}\BE\lbrack D_\Upsilon(x, y) \mid \Upsilon\rbrack \le D_\Upsilon(x, y) \le M\BE\lbrack D_\Upsilon(x, y) \mid \Upsilon\rbrack \quad \text{almost surely}. 
    \end{equation}
    This implies that 
    \begin{equation}\label{eq:025}
        \BE\!\left\lbrack D_\Upsilon(x, y)^2 \ \middle\vert \ \Upsilon\right\rbrack \le M^2\BE\lbrack D_\Upsilon(x, y) \mid \Upsilon\rbrack^2 \le M^4D_\Upsilon(x, y) < \infty \quad \text{almost surely}, 
    \end{equation}
    and
    \begin{equation}\label{eq:338}
        \max_{S \in \CS_\varepsilon(O)} \len(P \cap S; D_\Upsilon) \le D_\Upsilon(x, y) \le M\BE\lbrack D_\Upsilon(x, y) \mid \Upsilon\rbrack \quad \text{almost surely}. 
    \end{equation}
    On the other hand, it follows from the continuity of $D_\Upsilon$ that
    \begin{equation}\label{eq:339}
        \lim_{\varepsilon \to 0} \max_{S \in \CS_\varepsilon(O)} \len(P \cap S; D_\Upsilon) = 0 \quad \text{almost surely}. 
    \end{equation}
    Combining \eqref{eq:338} and \eqref{eq:339}, it follows from the bounded convergence theorem that, almost surely, 
    \begin{equation}\label{eq:027}
        \BE\!\left\lbrack \left(\max_{S \in \CS_\varepsilon(O)} \len(P \cap S; D_\Upsilon)\right)^2 \ \middle\vert \ \Upsilon, \ O\right\rbrack \to 0 \quad \text{as } \varepsilon \to 0. 
    \end{equation}
    Thus, we conclude from \eqref{eq:024}, \eqref{eq:025}, and \eqref{eq:027} that \eqref{eq:028} is true. This completes the proof of the \emph{claim}. 

    Finally, we conclude from \eqref{eq:029} and \eqref{eq:028} that 
    \begin{equation*}
        \Var\lbrack D_\Upsilon(x, y) \mid \Upsilon\rbrack = 0 \quad \text{almost surely}.
    \end{equation*}
    This implies that $D_\Upsilon(x, y)$ is almost surely determined by $\Upsilon$. Let $\{x_{j,k}\}_{j,k \in \NN}$ be as in the proof of \Cref{047}. Then it follows that, almost surely, $D_\Upsilon(x_{j,k}, x_{j,l})$ is determined by $\Upsilon$ for all $j, k, l \in \NN$. Thus, we conclude from the continuity of $D_\Upsilon$ that $D_\Upsilon$ is almost surely determined by $\Upsilon$. This completes the proof. 
\end{proof}

\section{Convergence of MFPP}\label{s:01}

In the present section, we aim to deduce \Cref{321} from the results of \cite{TightSimCLE}. In \Cref{ss:00}, we prove that the laws of the rescaled $\varepsilon$-MFPP metrics $(X \times X, \ka_\varepsilon^{-1}D_{X}^\varepsilon)$ are tight (where $X$ is as in \eqref{eq:004}), and that any subsequential limit satisfies Axioms~\eqref{307A} and \eqref{307D} of \Cref{307}. In \Cref{ss:03}, we prove that any subsequential limit satisfies Axiom~\eqref{010D} of \Cref{010}. In \Cref{ss:04}, we prove that any subsequential limit satisfies Axiom~\eqref{048B} of \Cref{048}. We then conclude the proof of \Cref{321} using \Cref{006}.

\subsection{Subsequential limits}\label{ss:00}

In the present subsection, we explain why the results of \cite{TightSimCLE} imply that the family $\{(X \times X, \ka_\varepsilon^{-1}D_{X}^\varepsilon)\}_{\varepsilon > 0}$ is tight, and that any subsequential limit satisfies Axioms~\eqref{307A} and \eqref{307D} of \Cref{307}. This is not entirely immediate, since \cite{TightSimCLE} deals only with non-nested $\CLE_\kappa$'s, whereas Axiom~\eqref{307D} concerns the whole-plane nested $\CLE_\kappa$. We also include the corresponding results for the internal metrics on dyadic domains, which will be needed in \Cref{ss:04}.

Recall from \Cref{ss:12} the definition of dyadic domains. For $n \in \NN$, we shall write $\mathscr{DyDo}_n$ for the collection of sequences $(W_1, \ldots, W_n) \in \mathscr{DyDo}^n$ such that $W_j \Subset W_{j + 1}$ for all $j \in [1, n - 1]_\ZZ$. For an open subset $U \subset \CC$, we shall write $\mathscr{DyDo}_n(U) \defeq \{(W_1, \ldots, W_n) \in \mathscr{DyDo}_n : W_n \Subset U\}$.

We shall write $\mathscr{HausUniEns}_4$ for the metric space consisting of discrete subsets of $\mathscr{HausUni}_4$, and equipped with the metric given by the Hausdorff distance induced by $d_{\mathscr{HausUni}_4}$.

\begin{proposition}\label{331}
    Let $\Gamma_\DD$ be a non-nested $\CLE_\kappa$ in $\DD$. Write $\SCL$ for the loop of $\Gamma_\DD$ surrounding $0$. Given $\SCL$, let $\Gamma_{\SCL}$ be a non-nested $\CLE_\kappa$ inside $\SCL$. Write $X_\SCL$ for the carpet of $\Gamma_{\SCL}$. For $\varepsilon > 0$, write $\kb_\varepsilon$ for the median of the random variable $\sup_{x, y \in \SCL} D_{X_\SCL}^\varepsilon(x, y)$. Then for every sequence of positive real numbers $\varepsilon$'s tending to zero, there exists $D_{X_\SCL}$, $\{D_{X_\SCL,W_1,W_2} : (W_1, W_2) \in \mathscr{DyDo}_2(\DD)\}$, and a subsequence $\SE$ along which we have the convergence of the joint laws
    \begin{multline*}
        \Biggl(\Bigl(X_\SCL \times X_\SCL, \kb_\varepsilon^{-1}D_{X_\SCL}^\varepsilon\Bigr), \biggl\{\Bigl((\overline{W_1} \cap X_\SCL) \times (\overline{W_1} \cap X_\SCL), \\
        \left.\kb_\varepsilon^{-1}D_{X_\SCL}^\varepsilon(\bullet, \bullet; W_2 \cap X_\SCL)\right|_{(\overline{W_1} \cap X_\SCL) \times (\overline{W_1} \cap X_\SCL)}\Bigr) : (W_1, W_2) \in \mathscr{DyDo}_2(\DD)\biggr\}\Biggr) \\
        \to \biggl(\bigl(X_\SCL \times X_\SCL, D_{X_\SCL}\bigr), \Bigl\{\bigl((\overline{W_1} \cap X_\SCL) \times (\overline{W_1} \cap X_\SCL), D_{X_\SCL,W_1,W_2}(\bullet, \bullet)\bigr) : (W_1, W_2) \in \mathscr{DyDo}_2(\DD)\Bigr\}\biggr), 
    \end{multline*}
    with respect to the metric space $\mathscr{HausUniEns}_4$. Moreover:
    \begin{enumerate}
        \item $D_{X_\SCL}$ is almost surely a geodesic metric on $X_\SCL$ inducing the Euclidean topology. 
        \item Almost surely, 
        \begin{equation*}
            D_{X_\SCL,W_1,W_2}(\bullet, \bullet) = \left.D_{X_\SCL}(\bullet, \bullet; W_2 \cap X_\SCL)\right|_{(\overline{W_1} \cap X_\SCL) \times (\overline{W_1} \cap X_\SCL)}, \quad \forall (W_1, W_2) \in \mathscr{DyDo}_2(\DD). 
        \end{equation*}
    \end{enumerate}
\end{proposition}

\begin{proof}
    See \cite[Theorem~1.1 and Section~7]{TightSimCLE}. 
\end{proof}

In the remainder of the present subsection, let $\{\kb_\varepsilon\}_{\varepsilon > 0}$ be as in \Cref{331}.

Set $\gamma \defeq \sqrt\kappa$ and $Q \defeq 2/\gamma + \gamma/2$. Let $\ell > 0$. Let $(\SCD, \Phi, x, y)$ be a $\gamma$-LQG disk of boundary length $\ell$ (cf.~\Cref{ss:13}). We shall write $\mu_\Phi$ (resp.~$\nu_\Phi$) for the $\gamma$-LQG area (resp.~length) measure associated with $\Phi$. Suppose that the law of $\partial\SCD$ is absolutely continuous with respect to the $\SLE_\kappa$ loop measure, i.e., 
\begin{equation*}
    ``\BP\lbrack\partial\SCD \in A\rbrack > 0'' \ \Rightarrow \ ``\BP\lbrack\text{there exists } \SCL \in \Gamma \text{ such that } \SCL \in A\rbrack > 0'', \quad \forall \text{ Borel subset } A \subset \mathscr{Loop}
\end{equation*}
(where $\Gamma$ is a whole-plane nested $\CLE_\kappa$). Given $\SCD$, let $\Gamma_\SCD$ be a nested $\CLE_\kappa$ in $\SCD$, conditionally independent of $\Phi$. We shall write $\Gamma_\SCD^\circ$ for the collection of the outermost loops of $\Gamma_\SCD$. Thus, $\Gamma_\SCD^\circ$ is conditionally a non-nested $\CLE_\kappa$ in $\SCD$ given $\SCD$. We shall write $\widehat\Gamma_\SCD \defeq \Gamma_\SCD \cup \{\partial\SCD\}$. For $\SCL \in \widehat\Gamma_\SCD$, we shall write $X(\SCL) \subset \overline\SCD$ for the closed subset consisting of points that are surrounded by $\SCL$ but by no other loop of $\Gamma_\SCD$ inside $\SCL$. To lighten notation, we shall write $X_\SCD \defeq X(\partial\SCD)$.

\begin{lemma}\label{052}
    There are deterministic constants $0 < \alpha_{\mathrm{UBD}} < 2 < \alpha_{\mathrm{LBD}}$ such that the following is true: Let $(\DD, \Phi_\DD, -\ri, \ri)$ be a $\gamma$-LQG disk of unit boundary length. Then it holds with superpolynomially high probability as $\varepsilon_\ast \to 0$ that 
    \begin{equation}\label{eq:047}
        \varepsilon^{\alpha_{\mathrm{LBD}}} \le \mu_{\Phi_\DD}(B_\varepsilon(z)) \le \varepsilon^{\alpha_{\mathrm{UBD}}}, \quad \forall z \in \DD, \ \forall \varepsilon \in (0, \dist(z, \partial\DD) \wedge \varepsilon_\ast]. 
    \end{equation}
\end{lemma}

\begin{proof}
    See \cite[Lemma~A.6 and Lemma~B.3]{TightSimCLE}. 
\end{proof}

\begin{lemma}\label{332}
    Let $\alpha_{\mathrm{UBD}}$ and $\alpha_{\mathrm{LBD}}$ be as in \Cref{052}. Then there almost surely exists $\varepsilon_\ast \in (0, 1)$ such that 
    \begin{equation}\label{eq:349}
        \varepsilon^{\alpha_{\mathrm{LBD}}} \le \mu_\Phi(B_\varepsilon(z)) \le \varepsilon^{\alpha_{\mathrm{UBD}}}, \quad \forall z \in \SCD, \ \forall \varepsilon \in (0, \dist(z, \partial\SCD) \wedge \varepsilon_\ast]. 
    \end{equation}
\end{lemma}

\begin{proof}
    Let $\SCL$ be as in \Cref{331}. It suffices to consider the case where $\SCD$ is the domain surrounded by $\SCL$. Let $(\DD, \Phi_\DD, -\ri, \ri)$ be a $\gamma$-LQG disk of unit boundary length. Then \Cref{332} follows immediately from \Cref{052}, together with the fact that $(\SCD, \Phi_\DD|_\SCD, x, y)$ is conditionally a $\gamma$-LQG disk given its boundary length (cf.~\cite{SimCLELQG}). 
\end{proof}

For each $\varepsilon_\ast \in (0, 1)$, we shall write $E(\varepsilon_\ast)$ for the event that \eqref{eq:349} holds.

\begin{lemma}\label{014}
    There is a deterministic constant $\alpha_{\mathrm{NET}} > 0$ such that the following is true: Let $(\DD, \Phi_\DD, -\ri, \ri)$ be a $\gamma$-LQG disk of unit boundary length. Let $\Gamma_\DD$ be an independent non-nested $\CLE_\kappa$ in $\DD$. Write $X_\DD$ for the carpet of $\Gamma_\DD$. Given $\Phi_\DD$ and $\Gamma_\DD$, let $\{x_n\}_{n \in \NN}$ be conditionally independent samples from $\mu_{\Phi_\DD,X_\DD}$ (renormalized to be a probability measure) (cf.~\cite{SimCLELQG}). Then it holds with superpolynomially high probability as $\delta_\ast \to 0$ that
    \begin{equation*}
        X_\DD \subset \bigcup_{j \in [1, \delta^{-\alpha_{\mathrm{NET}}}]_\ZZ} B_\delta(x_j), \quad \forall \delta \in (0, \delta_\ast]. 
    \end{equation*}
\end{lemma}

\begin{proof}
    See \cite[Appendix~B]{TightSimCLE}. 
\end{proof}

\begin{lemma}\label{051}
    Let $\alpha_{\mathrm{NET}}$ be as in \Cref{014}. Then there is a deterministic constant $\alpha_{\mathrm{PERIM}} > 0$ such that the following is true: Given $\Phi$ and $\Gamma_\SCD$, let $\{x_n\}_{n \in \NN}$ be conditionally independent samples from $\mu_{\Phi,X_\SCD}$ (renormalized to be a probability measure). Let $\varepsilon_\ast \in (0, 1)$. Then, on the event $E(\varepsilon_\ast)$, it holds with superpolynomially high probability as $\delta_\ast \to 0$, at a rate depending only on $\varepsilon_\ast$, that
    \begin{equation}\label{eq:049}
        X_\SCD \subset \bigcup_{j \in [1, \delta^{-\alpha_{\mathrm{NET}}}]_\ZZ} B_{\ell^{\alpha_{\mathrm{PERIM}}}\delta}(x_j), \quad \forall \delta \in (0, \delta_\ast]. 
    \end{equation}
\end{lemma}

\begin{proof}
    Let $\phi$ be a conformal mapping from $\SCD$ onto $\DD$. Write 
    \begin{equation*}
        \Phi_\DD \defeq \Phi \circ \phi^{-1} + Q\log(\lvert (\phi^{-1})^\prime\rvert) - (2/\gamma)\log(\ell); \quad \Gamma_\DD \defeq \phi(\Gamma_\SCD); \quad X_\DD \defeq \phi(X_\SCD). 
    \end{equation*}
    Then $(\DD, \Phi_\DD, -\ri, \ri)$ is a $\gamma$-LQG disk of unit boundary length, and $\Gamma_\DD$ is an independent non-nested $\CLE_\kappa$ in $\DD$. Write $E_\DD(\delta_\ast)$ for the event that \eqref{eq:047} holds. Then it follows from \Cref{052} that $E_\DD(\delta_\ast)$ holds with superpolynomially high probability as $\delta_\ast \to 0$. Write $F_\DD(\delta_\ast)$ for the event that 
    \begin{equation}\label{eq:048}
        X_\DD \subset \bigcup_{j \in [1, \delta^{-\alpha_{\mathrm{NET}}}]_\ZZ} B_\delta(\phi(x_j)), \quad \forall \delta \in (0, \delta_\ast]. 
    \end{equation}
    Since $\{\phi(x_j)\}_{j \in \NN}$ has the law of conditionally independent samples from $\mu_{\Phi_\DD,X_\DD}$ (renormalized to be a probability measure), it follows from \Cref{051} that $F_\DD(\delta_\ast)$ holds with superpolynomially high probability as $\delta_\ast \to 0$.

    Now, suppose that we are working on the event $E(\varepsilon_\ast) \cap E_\DD(\delta_\ast) \cap F_\DD(\delta_\ast)$ for some $\delta_\ast \in (0, \varepsilon_\ast]$. For each $\delta \in (0, \delta_\ast]$, by Koebe's quarter theorem and the definitions of $E(\varepsilon_\ast)$ and $E_\DD(\delta_\ast)$, we have
    \begin{align*}
        \delta^{\alpha_{\mathrm{UBD}}} &\ge \mu_{\Phi_\DD}(B_\delta(z)) = \ell^{-2}\mu_\Phi\left(\phi^{-1}(B_\delta(z))\right) \\
        &\ge \ell^{-2}\mu_\Phi\left(B_{\lvert(\phi^{-1})^\prime(z)\rvert\delta/4}(z)\right) \ge \ell^{-2}\left(\lvert(\phi^{-1})^\prime(z)\rvert\delta/4\right)^{\alpha_{\mathrm{LBD}}}, \quad \forall z \in B_{1 - \delta}(0). 
    \end{align*}
    This implies that 
    \begin{equation*}
        \sup_{z \in B_{1 - \delta}(0)} \left\lvert(\phi^{-1})^\prime(z)\right\rvert \le 4\ell^{2/\alpha_{\mathrm{LBD}}}  \delta^{\alpha_{\mathrm{UBD}}/\alpha_{\mathrm{LBD}} - 1}, \quad \forall \delta \in (0, \delta_\ast].
    \end{equation*}
    Since $\alpha_{\mathrm{UBD}}/\alpha_{\mathrm{LBD}} - 1 > -1$, this implies that $\phi^{-1}$ extends to an $\alpha_{\mathrm{UBD}}/\alpha_{\mathrm{LBD}}$-H\"older continuous mapping from $\overline\DD$ onto $\overline\SCD$ with H\"older coefficient equal to $\ell^{2/\alpha_{\mathrm{LBD}}}$ times some constant depending only on $\alpha_{\mathrm{UBD}}$ and $\alpha_{\mathrm{LBD}}$. Thus, by decreasing $\alpha_{\mathrm{NET}}$, \eqref{eq:049} follows immediately from \eqref{eq:048}. This completes the proof.
\end{proof}

\begin{proposition}\label{333}
    There are deterministic constants $\alpha_{\mathrm{KC}}, \alpha_{\mathrm{PERIM}} > 0$ such that the following is true: Let $\varepsilon_\ast \in (0, 1)$. Then, on the event $E(\varepsilon_\ast)$, it holds with superpolynomially high probability as $A \to 0$, at a rate depending only on $\varepsilon_\ast$, that
    \begin{equation*}
        \kb_\varepsilon^{-1} D_{X_\SCD}^\varepsilon(u, v) \le A \ell^{\alpha_{\mathrm{PERIM}}} \lvert u - v\rvert^{\alpha_{\mathrm{KC}}}, \quad \forall u, v \in X_\SCD, \ \forall \varepsilon \in (0, \varepsilon_\ast]. 
    \end{equation*}
\end{proposition}

\begin{proof}
    See \cite[Section~6]{TightSimCLE}. 
\end{proof}

\begin{lemma}\label{334}
    There is a constant $\kk > 0$ such that for each connected and compact subset $K \subset \CC$, 
    \begin{equation*}
        \mathop{\mathrm{Leb}}(B_{\varepsilon r}(K)) \ge \kk\varepsilon\mathop{\mathrm{Leb}}(B_r(K)), \quad \forall r > 0, \ \forall \varepsilon \in (0, 1). 
    \end{equation*}
    In particular, 
    \begin{equation*}
        \ka_{\varepsilon r}/\ka_r \ge \kk\varepsilon \quad \text{and} \quad \kb_{\varepsilon r}/\kb_r \ge \kk\varepsilon, \quad \forall r > 0, \ \forall \varepsilon \in (0, 1). 
    \end{equation*}
\end{lemma}

\begin{proof}
    See {\cite[Lemma~5.4]{TightSimCLE}}. 
\end{proof}

\begin{lemma}\label{335}
    \begin{enumerate}
        \item\label{335A} $\BE\lbrack\mu_\Phi(\SCD)^p\rbrack < \infty$ for all $p \in (-\infty, 4/\gamma^2)$.
        \item\label{335B} $\BE\!\left\lbrack\sum_{\SCL \in \Gamma_\SCD} \nu_\Phi(\SCL)^p\right\rbrack < \infty$ for all $p \in (2, 8/\gamma^2)$.
        \item\label{335C} Let $\varepsilon_\ast \in (0, 1)$. Then 
        \begin{equation*}
            \BP\!\left\lbrack\sup_{\SCL \in \Gamma_\SCD} \nu_\Phi(\SCL) \ge A \ \middle\vert \ \mu_\Phi(\SCD) \le 1/\varepsilon_\ast\right\rbrack = O(A^{-\infty}) \quad \text{as } A \to \infty. 
        \end{equation*}
        \item\label{335D} Let $\varepsilon_\ast \in (0, 1)$. Then
        \begin{equation*}
            \BE\!\left\lbrack\sum_{\SCL \in \Gamma_\SCD} \nu_\Phi(\SCL)^p \ \middle\vert \ \mu_\Phi(\SCD) \le 1/\varepsilon_\ast\right\rbrack < \infty, \quad \forall p > 2. 
        \end{equation*}
    \end{enumerate}
\end{lemma}

\begin{proof}
    Assertion~\eqref{335A} follows, e.g., from \cite{LQG-Boundary-MoT}. 
    
    Next, we consider assertion~\eqref{335B}. Since $X_\SCD$ is conditionally independent of $\Phi$ given $\SCD$ and has zero Lebesgue measure, it follows that $\mu_\Phi(X_\SCD) = 0$ almost surely. Thus, almost surely, 
    \begin{equation*}
        \mu_\Phi(\SCD) = \sum_{\SCL \in \Gamma_\SCD^\circ} \mu_\Phi(\SCD(\SCL)), 
    \end{equation*}
    where $\SCD(\SCL)$ denotes the domain surrounded by $\SCL$. This implies that, almost surely, 
    \begin{equation}\label{eq:357}
        \mu_\Phi(\SCD)^p > \sum_{\SCL \in \Gamma_\SCD^\circ} \mu_\Phi(\SCD(\SCL))^p, \quad \forall p > 1. 
    \end{equation}
    By \cite{SimCLELQG}, for each $\SCL \in \Gamma_\SCD^\circ$, the $\gamma$-LQG surface parameterized by $\SCD(\SCL)$ is conditionally a $\gamma$-LQG disk given its boundary length. Thus, it follows from assertion~\eqref{335A} that 
    \begin{equation}\label{eq:358}
        \BE\lbrack\mu_\Phi(\SCD(\SCL))^p \mid \nu_\Phi(\SCL)\rbrack = \nu_\Phi(\SCL)^{2p}\BE\lbrack\mu_\Phi(\SCD)^p\rbrack < \infty, \quad \forall \SCL \in \Gamma_\SCD^\circ, \ \forall p \in (-\infty, 4/\gamma^2). 
    \end{equation}
    Combining \eqref{eq:357} and \eqref{eq:358}, we obtain that
    \begin{equation}\label{eq:359}
        \BE\!\left\lbrack\sum_{\SCL \in \Gamma_\SCD^\circ} \nu_\Phi(\SCL)^{2p}\right\rbrack < 1, \quad \forall p \in (1, 4/\gamma^2). 
    \end{equation}
    Now, assertion~\eqref{335B} follows immediately from \eqref{eq:359}, together with the fact that the $\gamma$-LQG lengths of the loops of $\Gamma_\SCD$ form a multiplicative cascade with offspring distribution given by the law of the collection $\{\nu_\Phi(\SCL) : \SCL \in \Gamma_\SCD^\circ\}$. 

    Next, we consider assertion~\eqref{335C}. By Bayes' theorem, 
    \begin{multline}\label{eq:014}
        \BP\!\left\lbrack\sup_{\SCL \in \Gamma_\SCD} \nu_\Phi(\SCL) \ge A \ \middle\vert \ \mu_\Phi(\SCD) \le 1/\varepsilon_\ast\right\rbrack \\
        = \left.\BP\!\left\lbrack\mu_\Phi(\SCD) \le 1/\varepsilon_\ast \ \middle\vert \ \sup_{\SCL \in \Gamma_\SCD} \nu_\Phi(\SCL) \ge A\right\rbrack \BP\!\left\lbrack\sup_{\SCL \in \Gamma_\SCD} \nu_\Phi(\SCL) \ge A\right\rbrack \middle/ \BP\lbrack\mu_\Phi(\SCD) \le 1/\varepsilon_\ast\rbrack\right. \\
        \le \left.\BP\!\left\lbrack\mu_\Phi(\SCD) \le 1/\varepsilon_\ast \ \middle\vert \ \sup_{\SCL \in \Gamma_\SCD} \nu_\Phi(\SCL) \ge A\right\rbrack \middle/ \BP\lbrack\mu_\Phi(\SCD) \le 1/\varepsilon_\ast\rbrack\right.. 
    \end{multline}
    Since for each $\SCL \in \Gamma_\SCD$, the $\gamma$-LQG surface parameterized by the domain surrounded by $\SCL$ is conditionally a $\gamma$-LQG disk given its boundary length, it follows from Markov's inequality that
    \begin{equation}\label{eq:015}
        \BP\!\left\lbrack\mu_\Phi(\SCD) \le 1/\varepsilon_\ast \ \middle\vert \ \sup_{\SCL \in \Gamma_\SCD} \nu_\Phi(\SCL) \ge A\right\rbrack \le \BP\lbrack\mu_\Phi(\SCD) \le 1/(A^2\varepsilon_\ast)\rbrack \le \BE\lbrack\mu_\Phi(\SCD)^{-p}\rbrack (A^2\varepsilon_\ast)^{-p}. 
    \end{equation}
    for all $p > 0$. Combining \eqref{eq:014}, \eqref{eq:015}, and assertion~\eqref{335A}, we complete the proof of assertion~\eqref{335C}.

    Finally, assertion~\eqref{335D} follows immediately from assertions~\eqref{335B} and \eqref{335C}.
\end{proof}

\begin{lemma}\label{325}
    Let $\alpha_{\mathrm{KC}}$ be as in \Cref{333}. Then for every sequence of positive real numbers $\varepsilon$'s tending to zero, there exists $\{D_{X(\SCL)} : \SCL \in \widehat\Gamma_\SCD\}$ and a subsequence $\SE$ along which we have the convergence of the joint laws
    \begin{equation}\label{eq:013}
        \left\{\left(X(\SCL) \times X(\SCL), \kb_\varepsilon^{-1} D_{X(\SCL)}^\varepsilon\right) : \SCL \in \widehat\Gamma_\SCD\right\} \to \left\{\left(X(\SCL) \times X(\SCL), D_{X(\SCL)}\right) : \SCL \in \widehat\Gamma_\SCD\right\}, 
    \end{equation}
    with respect to the metric space $\mathscr{HausUniEns}_4$. Moreover:
    \begin{enumerate}
        \item\label{325A} Almost surely, for each $\SCL \in \widehat\Gamma_\SCD$, $D_{X(\SCL)}$ is a geodesic metric on $X(\SCL)$ inducing the Euclidean topology. 
        \item\label{325B} Given $\SCD$, it holds with superpolynomially high conditional probability as $A \to \infty$, at a rate which is uniform in the choice of the subsequential limit, that 
        \begin{equation*}
            D_{X(\SCL)}(u, v) \le A\lvert u - v\rvert^{\alpha_{\mathrm{KC}}}, \quad \forall \SCL \in \widehat\Gamma_\SCD. 
        \end{equation*}
    \end{enumerate}
\end{lemma}

\begin{proof} 
    It follows immediately from \Cref{331} that for each $\SCL \in \widehat\Gamma_\SCD$, the family $\{(X(\SCL) \times X(\SCL), \kb_\varepsilon^{-1} D_{X(\SCL)}^\varepsilon)\}_{\varepsilon > 0}$ is tight. Thus, by Tikhonov's theorem and Prokhorov's theorem, the family of the joint laws 
    \begin{equation*}
        \left\{\left(X(\SCL) \times X(\SCL), \kb_\varepsilon^{-1} D_{X(\SCL)}^\varepsilon\right) : \SCL \in \widehat\Gamma_\SCD\right\}
    \end{equation*}
    for $\varepsilon > 0$ is tight. Thus, we conclude that there exists $\{D_{X(\SCL)} : \SCL \in \widehat\Gamma_\SCD\}$ and a subsequence $\SE$ along which \eqref{eq:013} holds. Moreover, assertion~\eqref{325A} follows immediately from \Cref{331}. Thus, it suffices to verify assertion~\eqref{325B}. 

    To lighten notation, we may assume that $\SCD$ is deterministic. Let $\varepsilon_\ast > 0$. Write $F(\varepsilon_\ast)$ for the event that $\mu_\Phi(\SCD) \le 1/\varepsilon_\ast$. By \Cref{333}, for each $p > 0$, there exists $C_p = C_p(\varepsilon_\ast) > 0$ such that
    \begin{equation*}
        \BP\!\left\lbrack E(\varepsilon_\ast) \cap \left\{\sup_{u, v \in X_\SCD} \frac{\kb_\varepsilon^{-1} D_{X_\SCD}^\varepsilon(u, v)}{\lvert u - v\rvert^{\alpha_{\mathrm{KC}}}} \ge A\right\}\right\rbrack \le C_p A^{-p} \ell^{\alpha_{\mathrm{PERIM}}p}, \quad \forall \varepsilon \in (0, \varepsilon_\ast], \ \forall A > 0, 
    \end{equation*}
    hence that
    \begin{multline}\label{eq:352}
        \BP\!\left\lbrack E(\varepsilon_\ast) \cap F(\varepsilon_\ast) \cap \left\{\sup_{u, v \in X_\SCD} \frac{\kb_\varepsilon^{-1} D_{X_\SCD}^\varepsilon(u, v)}{\lvert u - v\rvert^{\alpha_{\mathrm{KC}}}} \ge A\right\}\right\rbrack \le C_p A^{-p} \ell^{\alpha_{\mathrm{PERIM}}p}, \\
        \forall \varepsilon \in (0, \varepsilon_\ast], \ \forall A > 0. 
    \end{multline}
    Let $\SCL \in \widehat\Gamma_\SCD$. Write $\SCD(\SCL)$ for the simply connected domain surrounded by $\SCL$. By \cite{SimCLELQG}, the $\gamma$-LQG surface parameterized by $\SCD(\SCL)$ is conditionally a $\gamma$-LQG disk given its boundary length. Note that if $E(\varepsilon_\ast) \cap F(\varepsilon_\ast)$ occurs, then $E(\varepsilon_\ast)  \cap F(\varepsilon_\ast)$ also occurs with $\SCD(\SCL)$ and $\Phi|_{\SCD(\SCL)}$ in place of $\SCD$ and $\Phi$, respectively. Thus, we conclude from \eqref{eq:352} that
    \begin{multline}\label{eq:355}
        \BP\!\left\lbrack E(\varepsilon_\ast) \cap F(\varepsilon_\ast) \cap \left\{\sup_{u, v \in X(\SCL)} \frac{\kb_\varepsilon^{-1} D_{X(\SCL)}^\varepsilon(u, v)}{\lvert u - v\rvert^{\alpha_{\mathrm{KC}}}} \ge A\right\} \ \middle\vert \ \nu_\Phi(\SCL)\right\rbrack \le C_p A^{-p} \nu_\Phi(\SCL)^{\alpha_{\mathrm{PERIM}}p}, \\
        \forall \varepsilon \in (0, \varepsilon_\ast], \ \forall p > 0, \ \forall A > 0.
    \end{multline}
    Thus, we conclude from \eqref{eq:355} that 
    \begin{align*}
        &\BP\!\left\lbrack E(\varepsilon_\ast) \cap F(\varepsilon_\ast) \cap \bigcup_{\SCL \in \widehat\Gamma_\SCD} \left\{\sup_{u, v \in X(\SCL)} \frac{D_{X(\SCL)}(u, v)}{\lvert u - v\rvert^{\alpha_{\mathrm{KC}}}} \ge A\right\}\right\rbrack \\
        &\le \sum_{\SCL \in \widehat\Gamma_\SCD} \BP\!\left\lbrack E(\varepsilon_\ast) \cap F(\varepsilon_\ast) \cap \left\{\sup_{u, v \in X(\SCL)} \frac{D_{X(\SCL)}(u, v)}{\lvert u - v\rvert^{\alpha_{\mathrm{KC}}}} \ge A\right\}\right\rbrack \\
        &\le \sum_{\SCL \in \widehat\Gamma_\SCD} \liminf_{\SE \ni \varepsilon \to 0} \BP\!\left\lbrack E(\varepsilon_\ast) \cap F(\varepsilon_\ast) \cap \left\{\sup_{u, v \in X(\SCL)} \frac{\kb_\varepsilon^{-1} D_{X(\SCL)}^\varepsilon(u, v)}{\lvert u - v\rvert^{\alpha_{\mathrm{KC}}}} \ge A\right\}\right\rbrack \\
        &\le C_p A^{-p} \left(\ell + \BE\!\left\lbrack\sum_{\SCL \in \Gamma_\SCD} \nu_\Phi(\SCL)^{\alpha_{\mathrm{PERIM}}p} \ \middle\vert \ E(\varepsilon_\ast) \cap F(\varepsilon_\ast)\right\rbrack\right), \quad \forall p > 0, \ \forall A > 0.
    \end{align*}
    Combining this with \Cref{335}, \eqref{335D}, we obtain that
    \begin{equation}\label{eq:360}
        \BP\!\left\lbrack E(\varepsilon_\ast) \cap F(\varepsilon_\ast) \cap \bigcup_{\SCL \in \widehat\Gamma_\SCD} \left\{\sup_{u, v \in X(\SCL)} \frac{D_{X(\SCL)}(u, v)}{\lvert u - v\rvert^{\alpha_{\mathrm{KC}}}} \ge A\right\}\right\rbrack = O(A^{-\infty}) \quad \text{as } A \to \infty, 
    \end{equation}
    at a rate which is uniform in the choice of the subsequential limit. Since there almost surely exists $\varepsilon_\ast \in (0, 1)$ such that $E(\varepsilon_\ast) \cap F(\varepsilon_\ast)$ occurs (cf.~\Cref{332}), and $\Phi$ and $\Gamma_\SCD$ are conditionally independent given $\SCD$, we conclude from \eqref{eq:360} that
    \begin{equation*}
        \BP\!\left\lbrack\bigcup_{\SCL \in \widehat\Gamma_\SCD} \left\{\sup_{u, v \in X(\SCL)} \frac{D_{X(\SCL)}(u, v)}{\lvert u - v\rvert^{\alpha_{\mathrm{KC}}}} \ge A\right\}\right\rbrack = O(A^{-\infty}) \quad \text{as } A \to \infty.
    \end{equation*}
    This completes the proof. 
\end{proof}

\begin{remark}\label{050}
    It follows from \Cref{051}, together with a similar argument to the argument applied in the proof of \Cref{325}, that the following is true: Let $\alpha_{\mathrm{NET}}$ be as in \Cref{014}. Then, given $\SCD$, it holds with superpolynomially high conditional probability as $\varepsilon_\ast \to \infty$ that for each $\varepsilon \in (0, \varepsilon_\ast]$ and $\SCL \in \widehat\Gamma_\SCD$, there exists a finite subset $Z \subset X(\SCL)$ with $\#Z \le \lfloor\varepsilon^{-\alpha_{\mathrm{NET}}}\rfloor$ such that $X(\SCL) \subset \bigcup_{z \in Z} B_\varepsilon(z)$.
\end{remark}

\begin{lemma}\label{049}
    Let $\alpha_{\mathrm{KC}}$ be as in \Cref{333}. Then for every sequence of positive real numbers $\varepsilon$'s tending to zero, there exists $\{D_{X(\SCL)} : \SCL \in \Gamma\}$ and a subsequence $\SE$ along which we have the convergence of the joint laws
    \begin{equation}\label{eq:050}
        \left\{\left(X(\SCL) \times X(\SCL), \kb_\varepsilon^{-1} D_{X(\SCL)}^\varepsilon\right) : \SCL \in \Gamma\right\} \to \left\{\left(X(\SCL) \times X(\SCL), D_{X(\SCL)}\right) : \SCL \in \Gamma\right\}, 
    \end{equation}
    with respect to the metric space $\mathscr{HausUniEns}_4$. (Here, we recall from \Cref{ss:18} the definition of $X(\SCL)$ for $\SCL \in \Gamma$.) Moreover:
    \begin{enumerate}
        \item\label{049A} Almost surely, for each $\SCL \in \Gamma$, $D_{X(\SCL)}$ is a geodesic metric on $X(\SCL)$ inducing the Euclidean topology. 
        \item\label{049B} Let $K \subset \CC$ be a deterministic compact subset. Then it holds with superpolynomially high probability as $A \to \infty$, at a rate which is uniform in the choice of the subsequential limit, that 
        \begin{equation*}
            D_{X(\SCL)}(u, v) \le A\lvert u - v\rvert^{\alpha_{\mathrm{KC}}}, \quad \forall \SCL \in \Gamma, \ \forall u, v \in K \cap X(\SCL). 
        \end{equation*}
    \end{enumerate}
\end{lemma}

\begin{proof}
    This follows immediately from \Cref{325} by letting $\SCD \to \CC$.  
\end{proof}

\begin{lemma}\label{276}
    $\ka_\varepsilon \asymp \kb_\varepsilon$ as $\varepsilon \to 0$. (Here, we recall from \eqref{eq:263} the definition of $\ka_\varepsilon$.)
\end{lemma}

\begin{proof}
    For $\varepsilon > 0$, write 
    \begin{equation*}
        X_\varepsilon \defeq \sup_{u, v \in X(\SCL_1)} D_{X(\SCL_1)}^\varepsilon(u, v),
    \end{equation*}
    where $\SCL_1$ denotes the larges loop of $\Gamma$ contained in $B_1(0)$ and surrounding the origin. By definition, 
    \begin{equation}\label{eq:264}
        \BP\!\left\lbrack\ka_\varepsilon^{-1}X_\varepsilon \ge 1\right\rbrack \ge 1/2 \quad \text{and} \quad \BP\!\left\lbrack\ka_\varepsilon^{-1}X_\varepsilon \le 1\right\rbrack \ge 1/2. 
    \end{equation}
    Suppose that there exists a sequence $\{\varepsilon_n\}_{n \in \NN}$ of positive real numbers with $\lim_{n \to \infty} \varepsilon_n = 0$ such that $\lim_{n \to \infty} \kb_{\varepsilon_n}/\ka_{\varepsilon_n} = 0$. By possibly passing to a subsequence, we may assume without loss of generality that \eqref{eq:050} holds. This implies that $\kb_\varepsilon^{-1}X_\varepsilon$ converges in law to a positive and finite random variable. Thus, 
    \begin{equation*}
        \BP\!\left\lbrack\ka_{\varepsilon_n}^{-1}X_{\varepsilon_n} \ge 1\right\rbrack = \BP\!\left\lbrack(\kb_{\varepsilon_n}/\ka_{\varepsilon_n})\kb_{\varepsilon_n}^{-1}X_{\varepsilon_n} \ge 1\right\rbrack \to 0 \quad \text{as } n \to \infty, 
    \end{equation*}
    in contradiction to \eqref{eq:264}. Similarly, if $\lim_{n \to \infty} \kb_{\varepsilon_n}/\ka_{\varepsilon_n} = \infty$, then 
    \begin{equation*}
        \BP\!\left\lbrack\ka_{\varepsilon_n}^{-1}X_{\varepsilon_n} \le 1\right\rbrack = \BP\!\left\lbrack(\kb_{\varepsilon_n}/\ka_{\varepsilon_n})\kb_{\varepsilon_n}^{-1}X_{\varepsilon_n} \le 1\right\rbrack \to 0 \quad \text{as } n \to \infty. 
    \end{equation*}
    This completes the proof.
\end{proof}

\begin{proposition}\label{275}
    Let $\alpha_{\mathrm{KC}}$ be as in \Cref{333}. Then for every sequence of positive real numbers $\varepsilon$'s tending to zero, there exists $\{D_{X(\SCL)} : \SCL \in \Gamma\}$ and a subsequence $\SE$ along which we have the convergence of the joint laws
    \begin{equation}\label{eq:051}
        \left\{\left(X(\SCL) \times X(\SCL), \ka_\varepsilon^{-1} D_{X(\SCL)}^\varepsilon\right) : \SCL \in \Gamma\right\} \to \left\{\left(X(\SCL) \times X(\SCL), D_{X(\SCL)}\right) : \SCL \in \Gamma\right\}, 
    \end{equation}
    with respect to the metric space $\mathscr{HausUniEns}_4$. (Here, we recall from \Cref{ss:18} the definition of $X(\SCL)$ for $\SCL \in \Gamma$.) Moreover:
    \begin{enumerate}
        \item Almost surely, for each $\SCL \in \Gamma$, $D_{X(\SCL)}$ is a geodesic metric on $X(\SCL)$ inducing the Euclidean topology. 
        \item Let $K \subset \CC$ be a deterministic compact subset. Then it holds with superpolynomially high probability as $A \to \infty$, at a rate which is uniform in the choice of the subsequential limit, that 
        \begin{equation*}
            D_{X(\SCL)}(u, v) \le A\lvert u - v\rvert^{\alpha_{\mathrm{KC}}}, \quad \forall \SCL \in \Gamma, \ \forall u, v \in K \cap X(\SCL). 
        \end{equation*}
    \end{enumerate}
\end{proposition}

\begin{proof}
    This follows immediately from \Cref{325,276}. 
\end{proof}

\begin{lemma}\label{277}
    For every sequence of positive real numbers $\varepsilon$'s tending to zero, there exists 
    \begin{equation*}
        \left\{\left(D_{X(\SCL)}, D_{X(\SCL),W_1,W_2}\right) : \SCL \in \Gamma, \ (W_1, W_2) \in \mathscr{DyDo}_2\right\}
    \end{equation*}
    and a subsequence $\SE$ along which we have the convergence of the joint laws
    \begin{multline*}
        \Biggl(\biggl\{\Bigl(X(\SCL) \times X(\SCL), \ka_\varepsilon^{-1} D_{X(\SCL)}^\varepsilon\Bigr) : \SCL \in \Gamma\biggr\}, \biggl\{\Bigl((\overline{W_1} \cap X(\SCL)) \times (\overline{W_1} \cap X(\SCL)), \\
        \ka_\varepsilon^{-1}D_{X(\SCL)}^\varepsilon(\bullet, \bullet; W_2 \cap X(\SCL))|_{(\overline{W_1} \cap X(\SCL)) \times (\overline{W_1} \cap X(\SCL))}\Bigr) : \SCL \in \Gamma, \ (W_1, W_2) \in \mathscr{DyDo}_2\biggr\}\Biggr) \\
        \to \biggl(\Bigl\{\bigl(X(\SCL) \times X(\SCL), D_{X(\SCL)}\bigr) : \SCL \in \Gamma\Bigr\}, \Bigl\{\bigl((\overline{W_1} \cap X(\SCL)) \times (\overline{W_1} \cap X(\SCL)), \\
        D_{X(\SCL),W_1,W_2}(\bullet, \bullet)\bigr) : \SCL \in \Gamma, \ (W_1, W_2) \in \mathscr{DyDo}_2\Bigr\}\biggr), 
    \end{multline*}
    with respect to the metric space $\mathscr{HausUniEns}_4$. Moreover, almost surely, 
    \begin{multline*}
        D_{X(\SCL),W_1,W_2}(\bullet, \bullet) = D_{X(\SCL)}(\bullet, \bullet; W_2 \cap X(\SCL))|_{(\overline{W_1} \cap X(\SCL)) \times (\overline{W_1} \cap X(\SCL))}, \\
        \forall \SCL \in \Gamma, \ \forall (W_1, W_2) \in \mathscr{DyDo}_2. 
    \end{multline*}
\end{lemma}

\begin{proof}
    This follows immediately from \Cref{331,276}. 
\end{proof}

\subsection{Tightness across scales}\label{ss:03}

In the present subsection, we prove that any subsequential limit satisfies Axiom~\eqref{010D} of \Cref{010}. 

\begin{lemma}\label{326}
    Let $\kk$ be as in \Cref{334}. Let $\SE$ be a sequence of positive real numbers $\varepsilon$'s tending to zero along which \eqref{eq:051} holds. Then the limits
    \begin{equation*}
        \kc_r = \kc_r(\SE) \defeq \lim_{\SE \ni \varepsilon \to 0} \frac{r^2\ka_{\varepsilon/r}}{\ka_\varepsilon}, \quad \forall r > 0
    \end{equation*}
    exist and satisfy \eqref{eq:340} with $\kK \defeq \kk^{-1}$.
\end{lemma}

\begin{proof}
    Fix $r > 0$. We observe that 
    \begin{align*}
        &\left\{\left(X(\SCL) \times X(\SCL), \ka_{\varepsilon/r}^{-1}D_{X(\SCL)}^{\varepsilon/r}(\bullet, \bullet)\right) : \SCL \in \Gamma\right\} \\
        &= \left\{\left(X(\SCL) \times X(\SCL), \frac{\ka_\varepsilon}{r^2\ka_{\varepsilon/r}} \ka_\varepsilon^{-1} D_{rX(\SCL)}^\varepsilon(r\bullet, r\bullet)\right) : \SCL \in \Gamma\right\} \\
        &\overset d= \left\{\left((X(\SCL)/r) \times (X(\SCL)/r), \frac{\ka_\varepsilon}{r^2\ka_{\varepsilon/r}} \ka_\varepsilon^{-1} D_{X(\SCL)}^\varepsilon(r\bullet, r\bullet)\right) : \SCL \in \Gamma\right\}, \quad \forall \varepsilon > 0.
    \end{align*}
    In particular, by the definition of $\ka_{\varepsilon/r}$ (cf.~\eqref{eq:263}), 
    \begin{equation}\label{eq:265}
        1 = \ka_{\varepsilon/r}^{-1} \mathop{\mathrm{Med}}\left(\sup_{x, y \in X(\SCL_1)} D_{X(\SCL_1)}^{\varepsilon/r}(x, y)\right) = \frac{\ka_\varepsilon}{r^2\ka_{\varepsilon/r}} \ka_\varepsilon^{-1} \mathop{\mathrm{Med}}\left(\sup_{x, y \in X(\SCL_r)} D_{X(\SCL_r)}^{\varepsilon/r}(x, y)\right), 
    \end{equation}
    where $\SCL_r$ denotes the largest loop of $\Gamma$ contained in $B_r(0)$ and surrounding the origin. Write $\kc_r$ for the median of the random variable
    \begin{equation*}
        \sup_{x, y \in X(\SCL_r)} D_{X(\SCL_r)}^{\varepsilon/r}(x, y). 
    \end{equation*}
    Then it follows from \eqref{eq:265} that $r^2\ka_{\varepsilon/r}/\ka_\varepsilon \to \kc_r$ as $\SE \ni \varepsilon \to 0$. This completes the proof of the first assertion.  
    
    Let $\delta \in (0, 1)$. Then 
    \begin{equation*}
        \kc_{\delta r} = \lim_{\SE \ni \varepsilon \to 0} \frac{(\delta r)^2\ka_{\varepsilon/(\delta r)}}{\ka_\varepsilon} \ge \lim_{\SE \ni \varepsilon \to 0} \frac{(\delta r)^2\ka_{\varepsilon/r}}{\ka_\varepsilon} = \delta^2\kc_r
    \end{equation*}
    and
    \begin{equation*}
        \kc_{\delta r} = \lim_{\SE \ni \varepsilon \to 0} \frac{(\delta r)^2\ka_{\varepsilon/(\delta r)}}{\ka_\varepsilon} \le \lim_{\SE \ni \varepsilon \to 0} \frac{\kk^{-1}\delta r^2\ka_{\varepsilon/r}}{\ka_\varepsilon} = \kk^{-1}\delta\kc_r. 
    \end{equation*}
    This completes the proof of the second assertion. 
\end{proof}

\begin{proposition}\label{327}
    Let $\SE$ be a sequence of positive real numbers $\varepsilon$'s tending to zero along which \eqref{eq:051} holds. Let $\{\kc_r\}_{r > 0}$ be as in \Cref{326}. Then for each $r > 0$, the law of the collection
    \begin{equation}\label{eq:056}
        \left\{\left(X(\SCL) \times X(\SCL), \kc_r^{-1}D_{rX(\SCL)}(r\bullet, r\bullet)\right) : \SCL \in \Gamma\right\}
    \end{equation}
    is contained in the closure of the family of the collections
    \begin{equation}\label{eq:053}
        \left\{\left\{\left(X(\SCL) \times X(\SCL), \ka_\varepsilon^{-1} D_{X(\SCL)}^\varepsilon\right) : \SCL \in \Gamma\right\} : \varepsilon > 0\right\}
    \end{equation}
    with respect to the Prokhorov topology for probability measures on $\mathscr{HausUniEns}_4$. 
\end{proposition}

\begin{proof}
    To lighten notation, write $\SF$ for the family \eqref{eq:053} and $\overline\SF$ for the closure of $\SF$ with respect to the Prokhorov topology for probability measures on $\mathscr{HausUniEns}_4$. Note that 
    \begin{align*}
        &\left\{\left(X(\SCL) \times X(\SCL), \ka_{\varepsilon/r}^{-1}D_{X(\SCL)}^{\varepsilon/r}(\bullet, \bullet)\right) : \SCL \in \Gamma\right\} \\
        &= \left\{\left(X(\SCL) \times X(\SCL), \frac{\ka_\varepsilon}{r^2\ka_{\varepsilon/r}} \ka_\varepsilon^{-1} D_{rX(\SCL)}^\varepsilon(r\bullet, r\bullet)\right) : \SCL \in \Gamma\right\} \\
        &\overset d= \left\{\left((X(\SCL)/r) \times (X(\SCL)/r), \frac{\ka_\varepsilon}{r^2\ka_{\varepsilon/r}} \ka_\varepsilon^{-1} D_{X(\SCL)}^\varepsilon(r\bullet, r\bullet)\right) : \SCL \in \Gamma\right\}, \quad \forall r > 0, \ \forall \varepsilon > 0.
    \end{align*}
    This implies that for each $r > 0$ and $\varepsilon > 0$, the law of the collection
    \begin{equation}\label{eq:054}
        \left\{\left((X(\SCL)/r) \times (X(\SCL)/r), \frac{\ka_\varepsilon}{r^2\ka_{\varepsilon/r}} \ka_\varepsilon^{-1} D_{X(\SCL)}^\varepsilon(r\bullet, r\bullet)\right) : \SCL \in \Gamma\right\}
    \end{equation}
    is contained in $\SF$. Fix $r > 0$. It follows from \eqref{eq:051} and the definition of $\kc_r$ that \eqref{eq:054} converges in law to 
    \begin{equation}\label{eq:055}
        \left\{\left((X(\SCL)/r) \times (X(\SCL)/r), \kc_r^{-1} D_{X(\SCL)}(r\bullet, r\bullet)\right) : \SCL \in \Gamma\right\}
    \end{equation}
    as $\varepsilon \to 0$.
    Thus, we conclude that the law of \eqref{eq:055} is contained in $\overline\SF$. Since \eqref{eq:056} and \eqref{eq:055} have the same law, we complete the proof. 
\end{proof}

\subsection{Locality}\label{ss:04}

In the present subsection, we prove that any subsequential limit satisfies Axiom~\eqref{048B} of \Cref{048}, and then conclude the proof of \Cref{321} using \Cref{006}. 

\begin{proposition}\label{274}
    Let $\SE$ be a sequence of positive real numbers $\varepsilon$'s tending to zero along which \eqref{eq:051} holds. Then $D_\Upsilon$ satisfies Axiom~\eqref{048B} of \Cref{048}. (Here, $D_\Upsilon$ is defined as in \eqref{eq:057}.)
\end{proposition}

The main difficulty in the proof of \Cref{274} is that, in general, conditional independence is not preserved under taking weak limits. This difficulty will be resolved using \Cref{260}, together with the following lemma. 

\begin{lemma}\label{278}
    Let $\{X_n\}_{n \in \NN}$, $X$, $\{Y_n\}_{n \in \NN}$, $Y$, $Z$ be random variables taking values in complete and separable metric spaces $(\SX, d_\SX)$, $(\SY, d_\SY)$, $(\SZ, d_\SZ)$, respectively. Suppose that the following conditions are satisfied:
    \begin{enumerate}
        \item\label{278A} For each $n \in \NN$, $X_n$ and $Y_n$ are conditionally independent given $Z$. 
        \item\label{278B} $(X_n, Y_n, Z)$ converges in law to $(X, Y, Z)$ as $n \to \infty$. 
        \item\label{278C} For each $\varepsilon > 0$, there exists $\delta > 0$ such that the following is true: For each $n \in \NN$ and $z_1, z_2 \in \SZ$ with $d_\SZ(z_1, z_2) \le \delta$, the total variation distance between the conditional law of $X_n$ given $Z = z_1$ and the conditional law of $X_n$ given $Z = z_2$ is at most $\varepsilon$, and the same is true with $Y_n$ in place of $X_n$. 
    \end{enumerate}
    Then $X$ and $Y$ are conditionally independent given $Z$. 
\end{lemma}

\begin{proof}
    See \cite[Corollary~B.5 and Lemma~B.7]{TightNonsimCLE}. 
\end{proof}

\begin{proof}[Proof of \Cref{274}]
    By \Cref{277} and possibly passing to a subsequence, we may assume without loss of generality that along $\SE$ we have the convergence of the joint laws
    \begin{multline}\label{eq:267}
        \Biggl(\biggl\{\Bigl(X(\SCL) \times X(\SCL), \ka_\varepsilon^{-1} D_{X(\SCL)}^\varepsilon\Bigr) : \SCL \in \Gamma\biggr\}, \biggl\{\Bigl((\overline{W_1} \cap X(\SCL)) \times (\overline{W_1} \cap X(\SCL)), \\
        \ka_\varepsilon^{-1}D_{X(\SCL)}^\varepsilon(\bullet, \bullet; W_2 \cap X(\SCL))|_{(\overline{W_1} \cap X(\SCL)) \times (\overline{W_1} \cap X(\SCL))}\Bigr) : \SCL \in \Gamma, \ (W_1, W_2) \in \mathscr{DyDo}_2\biggr\}\Biggr) \\
        \to \biggl(\Bigl\{\bigl(X(\SCL) \times X(\SCL), D_{X(\SCL)}\bigr) : \SCL \in \Gamma\Bigr\}, \Bigl\{\bigl((\overline{W_1} \cap X(\SCL)) \times (\overline{W_1} \cap X(\SCL)), \\
        D_{X(\SCL)}(\bullet, \bullet; W_2 \cap X(\SCL))|_{(\overline{W_1} \cap X(\SCL)) \times (\overline{W_1} \cap X(\SCL))}\bigr) : \SCL \in \Gamma, \ (W_1, W_2) \in \mathscr{DyDo}_2\Bigr\}\biggr). 
    \end{multline}
    Let $U \subset \CC$ be a deterministic open subset. Fix 
    \begin{equation*}
        \left(W_1, W_2, W_3, W_4\right) \in \mathscr{DyDo}_4(U) \quad \text{and} \quad \left(W_1^\prime, W_2^\prime, W_3^\prime, W_4^\prime\right) \in \mathscr{DyDo}_4(\CC \setminus \overline U).
    \end{equation*}
    Write $U^\star$ for the connected component containing $W_4$ of the open subset obtained by removing from $U$ the closure of the union of the $(\CC \setminus \overline{W_4}, \CC \setminus U)$-excursions of $\Gamma$ and all the other loops of $\Gamma$ that intersect with $\CC \setminus U$. Write $X \defeq \partial U^\star \cap \partial W_{4}$ for the collection of the endpoints of the $(\CC \setminus \overline{W_4}, \CC \setminus U)$-excursions of $\Gamma$. Write $\alpha$ for the link pattern on $(U^\star; X)$ induced by the $(\CC \setminus \overline{W_4}, \CC \setminus U)$-excursions of $\Gamma$. Let $(V^\star; Y, \beta)$ be defined in the same manner as $(U^\star; X, \alpha)$ but with $(W_4^\prime, \CC \setminus \overline U)$ in place of $(W_4, U)$. Then it follows from \Cref{252} that, given $(U^\star; X, \alpha)$ and $(V^\star; Y, \beta)$, the following two collections are conditionally independent:
    \begin{enumerate}
        \item\label{it:207} The collection of the complementary $(\CC \setminus \overline{W_4}, \CC \setminus U)$-excursions of $\Gamma$, together with the collection of the loops of $\Gamma$ that are contained in $U^\star$. 
        \item\label{it:208} The collection of the complementary $(\CC \setminus \overline{W_4^\prime}, \overline U)$-excursions of $\Gamma$, together with the collection of the loops of $\Gamma$ that are contained in $V^\star$. 
    \end{enumerate}
    Write $\Gamma(W_3, W_2)$ for the union of the collection of the $(W_3, W_2)$-excursions of $\Gamma$ and the collection of the loops of $\Gamma$ that are contained in $W_3$ and intersect with $W_2$. Let $\Gamma(W_3^\prime, W_2^\prime)$ be defined in the same manner. Since $\Gamma(W_3, W_2)$ (resp.~$\Gamma(W_3^\prime, W_2^\prime)$) is contained in the collection of \eqref{it:207} (resp.~\eqref{it:208}), it follows that
    \begin{equation}\label{eq:268}
        \parbox{.85\linewidth}{$\Gamma(W_3, W_2)$ and $\Gamma(W_3^\prime, W_2^\prime)$ are conditionally independent given $(U^\star; X, \alpha)$ and $(V^\star; Y, \beta)$.}
    \end{equation}
    On the other hand, it follows from \Cref{260} that the mappings
    \begin{multline}\label{eq:266}
        (U^\star; X, \alpha) \mapsto \left(\text{the conditional law of } \Gamma(W_3, W_2) \text{ given } (U^\star; X, \alpha)\right) \\
        \text{and} \quad (V^\star; Y, \beta) \mapsto \left(\text{the conditional law of } \Gamma(W_3^\prime, W_2^\prime) \text{ given } (V^\star; Y, \beta)\right)
    \end{multline}
    are continuous with respect to the Carath\'eodory topology for simply connected domains, together with the uniform topology for boundary points, and the total variation distance for conditional laws. Now, we observe that for each $\varepsilon > 0$, the metrics 
    \begin{equation*}
        \left\{D_{X(\SCL)}^\varepsilon(\bullet, \bullet; W_2 \cap X(\SCL)) : \SCL \in \Gamma\right\} \quad \text{and} \quad \left\{D_{X(\SCL)}^\varepsilon(\bullet, \bullet; W_2^\prime \cap X(\SCL)) : \SCL \in \Gamma\right\}
    \end{equation*}
    are almost surely determined by $\Gamma(W_3, W_2)$ and $\Gamma(W_3^\prime, W_2^\prime)$, respectively. Since the mappings of \eqref{eq:266} do not depend on $\varepsilon$, it follows that the mappings 
    \begin{multline*}
        (U^\star; X, \alpha) \mapsto \biggl(\text{the conditional law of } \\
        \left\{\ka_\varepsilon^{-1}D_{X(\SCL)}^\varepsilon(\bullet, \bullet; W_2 \cap X(\SCL))|_{(\overline{W_1} \cap X(\SCL)) \times (\overline{W_1} \cap X(\SCL))} : \SCL \in \Gamma\right\} \text{ given } (U^\star; X, \alpha)\biggr)
    \end{multline*}
    and
    \begin{multline*}
        (V^\star; Y, \beta) \mapsto \biggl(\text{the conditional law of } \\
        \left\{\ka_\varepsilon^{-1}D_{X(\SCL)}^\varepsilon(\bullet, \bullet; W_2^\prime \cap X(\SCL))|_{(\overline{W_1^\prime} \cap X(\SCL)) \times (\overline{W_1^\prime} \cap X(\SCL))} : \SCL \in \Gamma\right\} \text{ given } (V^\star; Y, \beta)\biggr)
    \end{multline*}
    for $\varepsilon > 0$ are uniformly equicontinuous (i.e., satisfy condition~\eqref{278C} of \Cref{278}). Combining this with \eqref{eq:267}, \eqref{eq:268}, and \Cref{278}, we conclude that 
    \begin{multline*}
        \left(\Gamma(W_3, W_2), \left\{D_{X(\SCL)}(\bullet, \bullet; W_2 \cap X(\SCL))|_{(\overline{W_1} \cap X(\SCL)) \times (\overline{W_1} \cap X(\SCL))} : \SCL \in \Gamma\right\}\right) \\
        \text{and} \quad \left(\Gamma(W_3^\prime, W_2^\prime), \left\{D_{X(\SCL)}(\bullet, \bullet; W_2^\prime \cap X(\SCL))|_{(\overline{W_1^\prime} \cap X(\SCL)) \times (\overline{W_1^\prime} \cap X(\SCL))} : \SCL \in \Gamma\right\}\right)
    \end{multline*}
    are conditionally independent given $(U^\star; X, \alpha)$ and $(V^\star; Y, \beta)$. This implies that, given the $(\CC \setminus \overline{W_4^\prime}, \overline U)$-excursions of $\Gamma$ and all the other loops of $\Gamma$ that intersect with $\overline U$, 
    \begin{multline*}
        \left\{D_{X(\SCL)}(\bullet, \bullet; W_2 \cap X(\SCL))|_{(\overline{W_1} \cap X(\SCL)) \times (\overline{W_1} \cap X(\SCL))} : \SCL \in \Gamma\right\} \\
        \text{and} \quad \left(\Gamma(W_3^\prime, W_2^\prime), \left\{D_{X(\SCL)}(\bullet, \bullet; W_2^\prime \cap X(\SCL))|_{(\overline{W_1^\prime} \cap X(\SCL)) \times (\overline{W_1^\prime} \cap X(\SCL))} : \SCL \in \Gamma\right\}\right)
    \end{multline*}
    are conditionally independent. By increasing $W_1$, $W_2$, $W_3$, $W_4$ to $U$ and $W_1^\prime$, $W_2^\prime$, $W_3^\prime$, $W_4^\prime$ to $\CC \setminus \overline U$, we complete the proof. 
\end{proof}

\begin{proof}[Proof of \Cref{321}]
    By \Cref{275,327}, for every sequence of positive real numbers $\varepsilon$'s tending to zero, there exists $\{D_{X(\SCL)} : \SCL \in \Gamma\}$ and a subsequence $\SE$ along which we have the convergence of the joint laws
    \begin{equation*}
        \left\{\left(X(\SCL) \times X(\SCL), \ka_\varepsilon^{-1} D_{X(\SCL)}^\varepsilon\right) : \SCL \in \Gamma\right\} \to \left\{\left(X(\SCL) \times X(\SCL), D_{X(\SCL)}\right) : \SCL \in \Gamma\right\}, 
    \end{equation*}
    and Axioms~\eqref{010A} and \eqref{010D} of \Cref{010} are satisfied (with $D_\Upsilon$ defined as in \eqref{eq:057}). By \Cref{274}, the coupling $(\Upsilon, D_\Upsilon)$ satisfies Axiom~\eqref{048B} of \Cref{048}. The translation invariance and the monotonicity (cf.~Axioms~\eqref{010C} and \eqref{010E} of \Cref{010}) are immediate. Thus, we conclude from \Cref{006} that $D$ determines a weak geodesic $\CLE_\kappa$ carpet metric. This completes the proof. 
\end{proof}

\section{Distance exponent and Hausdorff dimension}\label{s:06}

In the present section, we prove \Cref{319,302}.

\subsection{Proof of \Cref{319} (uniqueness of the distance exponent)}

Suppose by way of contradiction that there exist $\theta < \widetilde\theta$ and strong geodesic $\CLE_\kappa$ carpet metrics $D$ and $\widetilde D$ with distance exponent $\theta$ and $\widetilde\theta$, respectively. By \Cref{321,312}, there exists a weak geodesic $\CLE_\kappa$ carpet metric which is also strong. Moreover, by \eqref{eq:340}, the distance exponent of this metric $> 0$. Thus, we may assume without loss of generality that $\widetilde\theta > 0$. By \Cref{307}, Axiom~\eqref{307C} (translation invariance and scale covariance), for each $p \in (0, 1)$, there exists $a = a(p) > 0$ such that
\begin{equation*}
    \BP\!\left\lbrack D_\Upsilon(\partial B_{r/2}(z) \cap \Upsilon, \partial B_r(z) \cap \Upsilon) \ge ar^\theta\right\rbrack \ge p, \quad \forall z \in \CC, \ \forall r > 0. 
\end{equation*}
Thus, by applying \Cref{263,214} with $\{r^{\widetilde\theta}\}_{r > 0}$ in place of $\{\kc_r\}_{r > 0}$ and a union bound, there exists a sufficiently small $a > 0$ and a sufficiently large $A > 0$ such that it holds with polynomially high probability as $\varepsilon \to 0$ that for each $z \in \left(\frac1{100}\varepsilon^2\ZZ\right)^2 \cap B_{1/\varepsilon}(0)$, there exists $r \in [\varepsilon^2, \varepsilon]$ such that the following are true:
\begin{enumerate}
    \item\label{it:303} $D_\Upsilon(\partial B_{r/2}(z) \cap \Upsilon, \partial B_r(z) \cap \Upsilon) \ge ar^\theta$. 
    \item\label{it:304} We have
    \begin{equation*}
        \widetilde D_\Upsilon(x, y) \le Ar^{\widetilde\theta}, \quad \forall x, y \in B_r(z) \cap \Upsilon \text{ with } x \xleftrightarrow{B_r(z)} y. 
    \end{equation*}
\end{enumerate}
Thus, by the Borel-Cantelli lemma, almost surely, there exists $\varepsilon_\ast \in (0, 1)$ such that the above event occurs for all $\varepsilon \in (0, \varepsilon_\ast] \cap \{2^k\}_{k \in \ZZ}$. 

Let $x, y \in \Upsilon$ with $x \leftrightarrow y$. Let $P \colon [0, 1] \to \Upsilon$ be a $D_\Upsilon$-geodesic from $x$ to $y$. Choose a sufficiently small $\varepsilon \in (0, \varepsilon_\ast] \cap \{2^k\}_{k \in \ZZ}$ so that $P \subset B_{1/\varepsilon}(0)$. Set $t_0 \defeq 0$ and choose $z_1 \in \left(\frac1{100}\varepsilon^2\ZZ\right)^2 \cap B_{1/\varepsilon}(0)$ and $r_1 \in [\varepsilon^2, \varepsilon]$ such that $x \in B_{\varepsilon^2/2}(z_1)$ and conditions~\eqref{it:303} and \eqref{it:304} (with $(z_1, r_1)$ in place of $(z, r)$) hold. Then, inductively, for each $j \in \NN$, if $y \notin B_{r_j}(z_j)$, then set $t_j$ to be the first time $t \ge t_{j - 1}$ at which $P$ hits $\partial B_{r_j}(z_j)$, and choose $z_{j + 1} \in \left(\frac1{100}\varepsilon^2\ZZ\right)^2 \cap B_{1/\varepsilon}(0)$ and $r_{j + 1} \in [\varepsilon^2, \varepsilon]$ such that $P(t_j) \in B_{\varepsilon^2/2}(z_{j + 1})$ and conditions~\eqref{it:303} and \eqref{it:304} (with $(z_{j + 1}, r_{j + 1})$ in place of $(z, r)$) hold; otherwise set $t_j \defeq 1$. Write $n \defeq \inf\{j \in \NN : t_j = 1\}$. Since $P(t_{j - 1}) \in B_{r_j/2}(z_j)$ and $P(t_j) \in \partial B_{r_j}(z_j)$ for all $j \in [1, n - 1]_\ZZ$, it follows from condition~\eqref{it:303} that
\begin{equation}\label{eq:341}
    D_\Upsilon(P(t_{j - 1}), P(t_j)) \ge ar_j^\theta, \quad \forall j \in [1, n - 1]_\ZZ. 
\end{equation}
Since $P([t_{j - 1}, t_j]) \subset B_{r_j}(z_j)$ for all $j \in [1, n]_\ZZ$, it follows from condition~\eqref{it:304} that
\begin{equation}\label{eq:342}
    \widetilde D_\Upsilon(P(t_{j - 1}), P(t_j)) \le Ar_j^{\widetilde\theta}, \quad \forall j \in [1, n]_\ZZ. 
\end{equation}
Combining \eqref{eq:341} and \eqref{eq:342}, we obtain that
\begin{multline*}
    \widetilde D_\Upsilon(x, y) \le \sum_{j = 1}^n \widetilde D_\Upsilon(P(t_{j - 1}), P(t_j)) \le \sum_{j = 1}^n Ar_j^{\widetilde\theta} \le \frac Aa \sum_{j = 1}^{n - 1} r_j^{\widetilde\theta - \theta} D_\Upsilon(P(t_{j - 1}), P(t_j)) + Ar_n^{\widetilde\theta} \\
    \le \frac Aa\varepsilon^{\widetilde\theta - \theta} D_\Upsilon(x, y) + A\varepsilon^{\widetilde\theta}. 
\end{multline*}
Since $\widetilde\theta > \theta \vee 0$ and $\varepsilon$ may be chosen to be arbitrarily small, we must have $\widetilde D_\Upsilon(x, y) = 0$, a contradiction. This completes the proof. \qed

\subsection{Hausdorff dimension of CLE carpet}

In the present subsection, we prove \Cref{302}, \eqref{302A}. We will respectively prove the upper and lower bounds in \Cref{315,316}.  Let $D$ be a strong geodesic metric with distance exponent $\theta$. Let $X$ be as in \eqref{eq:004}.

\begin{lemma}\label{315}
    We have $\theta\dim(X; D_X) \le \dim(X)$ almost surely. 
\end{lemma}

\begin{lemma}\label{308}
    Let $Y$ be a non-negative random variable. Then 
    \begin{equation*}
        \BE\lbrack Y\mathbf 1(Y \ge a)\rbrack \le a^{1 - p} \BE\lbrack Y^p\rbrack, \quad \forall a > 0, \ \forall p \ge 1. 
    \end{equation*}
\end{lemma}

\begin{proof}
    This follows immediately from the inequality
    \begin{equation*}
        Y\mathbf 1(Y \ge a) \le Y \left(\frac Ya\right)^{p - 1} \mathbf 1(Y \ge a) = a^{1 - p} Y^p \mathbf 1(Y \ge a) \le a^{1 - p} Y^p. 
    \end{equation*}
\end{proof}

\begin{proof}[Proof of \Cref{315}]
    Write $\SCL_1$ for the largest loop of $\Gamma$ contained in $B_1(0)$ and surrounding the origin. We may assume without loss of generality that $X = X(\SCL_1)$. (Recall from \Cref{ss:18} the definition of $X(\SCL_1)$.) Fix $\ell > \theta^{-1}\dim(X)$. It suffices to show that
    \begin{equation}\label{eq:310}
        \dim(X; D_{X}) \le \ell. 
    \end{equation}
    Choose $d > 2$ to be sufficiently close to $2$, $c \in (0, 2 - \dim(X))$ to be sufficiently close to $2 - \dim(X)$, and $\vartheta \in (0, \theta)$ to be sufficiently close to $\theta$ so that 
    \begin{equation}\label{eq:271}
        \vartheta\ell + c > d. 
    \end{equation}
    For each $p > 0$, write 
    \begin{equation}\label{eqn:cp_def}
        C_p \defeq \sup_{z \in \CC} \sup_{r > 0} \BE\!\left\lbrack\left(r^{-\theta}\sup\left\{D_{X}(x, y) : x, y \in B_r(z) \cap X\right\}\right)^p\right\rbrack. 
    \end{equation}
    By \Cref{307}, Axioms~\eqref{307C} and \eqref{307D}, we have $C_p < \infty$ for all $p > 0$. Choose a sufficiently large $p > 0$ so that 
    \begin{equation}\label{eq:272}
        (\theta - \vartheta)\ell p + \vartheta\ell > d. 
    \end{equation}
    Fix a sufficiently small $\varepsilon \in (0, 1)$. Since $d > 2$, there is a (deterministic) covering $\DD \subset \bigcup_{j = 1}^\infty B_{r_j}(z_j)$ such that 
    \begin{equation}\label{eq:306}
        \sum_{j = 1}^\infty r_j^d \le \varepsilon. 
    \end{equation}
    Note that $r_j \le \varepsilon^{1/d}$ for all $j \in \NN$. Since $c < 2 - \dim(X)$, if $\varepsilon$ is sufficiently small, then 
    \begin{equation}\label{eq:307}
        \BP\lbrack B_{r_j}(z_j) \cap X \neq \emptyset\rbrack \le r_j^c, \quad \forall j \in \NN
    \end{equation}    
    (cf.~\cite{ConfRadCLE}). Write $J \defeq \{j \in \NN : B_{r_j}(z_j) \cap X \neq \emptyset\}$. Then $X \subset \bigcup_{j \in J} B_{r_j}(z_j)$. For each $j \in \NN$, write $D_j \defeq \sup\left\{D_{X}(x, y) : x, y \in B_{r_j}(z_j) \cap X\right\}$. We have
    \begin{align}\label{eq:309}
        \BE\!\left\lbrack\sum_{j \in J} D_j^\ell\right\rbrack 
        &= \sum_{j = 1}^\infty \BE\!\left\lbrack D_j^\ell \mathbf 1(B_{r_j}(z_j) \cap X \neq \emptyset)\right\rbrack \\
        &\le \sum_{j = 1}^\infty \left(r_j^{\vartheta\ell + c} + \BE\!\left\lbrack D_j^\ell \mathbf 1(D_j \ge r_j^\vartheta)\right\rbrack\right) \quad \text{(by~\eqref{eq:307})} \notag \\
        &\le \varepsilon + \sum_{j = 1}^\infty \BE\!\left\lbrack D_j^\ell \mathbf 1(D_j \ge r_j^\vartheta)\right\rbrack \quad \text{(by \eqref{eq:271}, \eqref{eq:306})}. \notag
    \end{align}
    On the other hand, we have
    \begin{align}\label{eq:308}
        \sum_{j = 1}^\infty \BE\!\left\lbrack D_j^\ell \mathbf 1(D_j \ge r_j^\vartheta)\right\rbrack &\le \sum_{j = 1}^\infty r_j^{\vartheta\ell(1 - p)} \BE\!\left\lbrack D_j^{\ell p}\right\rbrack \quad \text{(by \Cref{308})} \\
        &\le C_{\ell p} \sum_{j = 1}^\infty r_j^{(\theta - \vartheta)\ell p + \vartheta\ell} \quad \text{(by~\eqref{eqn:cp_def})} \notag \\
        &\le C_{\ell p}\varepsilon \quad\text{(by \eqref{eq:272}, \eqref{eq:306})}. \notag 
    \end{align}
    Combining \eqref{eq:309} and \eqref{eq:308}, we obtain that
    \begin{equation*}
        \BE\!\left\lbrack\sum_{j \in J} D_j^\ell \right\rbrack \le (1 + C_{\ell p}) \varepsilon. 
    \end{equation*}
    This implies that
    \begin{equation*}
        \BP\!\left\lbrack\sum_{j \in J} D_j^\ell \ge \varepsilon^{1/2}\right\rbrack \le (1 + C_{\ell p}) \varepsilon^{1/2}. 
    \end{equation*}
    By setting $\varepsilon = 2^{-k}$ for all sufficiently large $k \in \NN$ and applying the Borel-Cantelli lemma, we have completed the proof of \eqref{eq:310}. 
\end{proof}

\begin{lemma}\label{316}
    We have $\theta\dim(X; D_X) \ge \dim(X)$ almost surely. 
\end{lemma}

\begin{lemma}\label{310}
    Let $\vartheta > \theta$. Let $K \subset \CC$ be a deterministic compact subset. Then it holds with probability tending to one as $r_\ast \to 0$ that
    \begin{equation*}
        D_\Upsilon(\partial B_{r/2}(z) \cap \Upsilon, \partial B_r(z) \cap \Upsilon) \ge r^{\vartheta}, \quad \forall z \in K, \ \forall r \in (0, r_\ast]. 
    \end{equation*}
\end{lemma}

\begin{proof}
    This follows immediately from \Cref{313}, together with the fact that we may take $\kc_r = r^\theta$ for all $r > 0$. 
\end{proof}

\begin{proof}[Proof of \Cref{316}]
    We may assume without loss of generality that $X = X(\SCL_1)$ is as in the proof of \Cref{315}. Fix $\ell > \dim(X; D_{X})$ and $\vartheta > \theta$. Then it suffices to show that
    \begin{equation}\label{eq:323}
        \dim(X) \le \vartheta\ell. 
    \end{equation}
    By \Cref{310}, almost surely, there exists $r_\ast > 0$ such that 
    \begin{equation}\label{eq:324}
        D_{X}(\partial B_{r/2}(z) \cap X, \partial B_r(z) \cap X) \ge r^{\vartheta}, \quad \forall z \in \DD, \ \forall r \in (0, r_\ast]. 
    \end{equation}
    Henceforth assume that \eqref{eq:324} holds. Since $\ell > \dim(X; D_{X})$, for each $\varepsilon > 0$, there exists a covering of $X$ by $D_{X}$-balls $\{B_{t_j}(z_j; D_{X})\}_{j \in \NN}$ such that $\sum_{j = 1}^\infty t_j^\ell \le \varepsilon$. If $\varepsilon$ is sufficiently small, then it follows from \eqref{eq:324} that
    \begin{equation*}
        B_{t_j}(z_j; D_{X}) \subset B_{t_j^{1/\vartheta}}(z_j), \quad \forall j \in \NN. 
    \end{equation*}
    Thus, there exists a covering $X \subset \bigcup_{j = 1}^\infty B_{t_j^{1/\vartheta}}(z_j)$ such that $\sum_{j = 1}^\infty t_j^\ell \le \varepsilon$. By sending $\varepsilon \to 0$, we obtain \eqref{eq:323}. This completes the proof. 
\end{proof} 

\begin{proof}[Proof of \Cref{302}, \eqref{302A}]
    This follows immediately from \Cref{315,316}. 
\end{proof}

\subsection{Hausdorff dimension of geodesics}

In the present subsection, we prove \Cref{302}, \eqref{302B}. We will break the proof into \Cref{269,270}, which respectively give the lower and upper bounds. Let $D$ be a strong geodesic $\CLE_\kappa$ carpet metric with distance exponent $\theta$. Let $X$ be as in \eqref{eq:004}.

\begin{lemma}\label{269}
    Almost surely, $\dim(P) \ge \theta$ for all $D_X$-geodesics $P$. 
\end{lemma}

\begin{proof}
    Let $X = X(\SCL_1)$ be as in the proof of \Cref{315}. Fix $\vartheta \in (0, \theta)$. Then it suffices to show that, almost surely, $\dim(P) \ge \vartheta$ for all $D_{X}$-geodesics $P \subset X$. By the mass distribution principle (cf.~\cite[Theorem~4.19]{BM}), it suffices to show that, almost surely, for each $D_{X}$-geodesic $P \colon [0, 1] \to X$ parameterized with constant speed, there exists $C > 0$ and $\delta > 0$ such that 
    \begin{equation*}
        \mathop{\mathrm{Leb}}(V) \le C\diam(P(V))^\vartheta, \quad \forall \text{ closed subsets } V \subset [0, 1] \text{ with } \diam(P(V)) \le \delta. 
    \end{equation*}
    Choose a sufficiently small $\nu > 0$ so that $\vartheta(1 + \nu) < \theta$. By the translation and scale invariance of the law of $\Gamma$, for each $p \in (0, 1)$, there exists $A = A(p) > 0$ such that for each $z \in \CC$ and $r > 0$, it holds with probability at least $p$ that there are at most $2A$ connected arcs of $\Gamma$ in $A_{r/2,r}(z)$ connecting its inner and outer boundaries (in which case there are at most $A$ connected components of $A_{r/2,r}(z) \cap \Upsilon$ connecting $\partial B_{r/2}(z)$ and $\partial B_r(z)$). Thus, by \Cref{263,214} and a union bound, there exists a sufficiently large $A = A(\nu) > 0$ such that it holds with polynomially high probability as $\varepsilon \to 0$ that for each $z \in \left(\frac1{100}\varepsilon^{1 + \nu}\ZZ\right)^2 \cap \DD$, there exists $r \in [\varepsilon^{1 + \nu}, \varepsilon]$ such that the following are true:
    \begin{enumerate}
        \item\label{it:205} There are at most $A$ connected components of $A_{r/2,r}(z) \cap X$ connecting $\partial B_{r/2}(z)$ and $\partial B_r(z)$ (in which case there are at most $A$ connected components of $B_r(z) \cap X$ intersecting with $B_{r/2}(z)$). 
        \item\label{it:206} We have
        \begin{equation*}
            D_{X}(x, y) \le Ar^\theta, \quad \forall x, y \in B_r(z) \cap X \text{ with } x \xleftrightarrow{B_r(z) \cap X} y. 
        \end{equation*}
    \end{enumerate}
    Thus, by the Borel-Cantelli lemma, almost surely, there exists $\varepsilon_\ast \in (0, 1)$ such that the above event occurs for all $\varepsilon \in (0, \varepsilon_\ast] \cap \{2^k\}_{k \in \ZZ}$. By possibly decreasing $\varepsilon_\ast$, we may assume without loss of generality that $\diam(X) \ge 2\varepsilon_\ast$. 

    Fix a $D_{X}$-geodesic $P \colon [0, 1] \to X$ parameterized with constant speed. Let $V \subset [0, 1]$ be a closed subset such that $\diam(P(V)) \le \varepsilon_\ast^{1 + \nu}/100$. Then there exists $\varepsilon \in (0, \varepsilon_\ast] \cap \{2^k\}_{k \in \ZZ}$ such that 
    \begin{equation}\label{eq:259}
        (\varepsilon/2)^{1 + \nu}/4 < \diam(P(V)) \le \varepsilon^{1 + \nu}/4. 
    \end{equation}
    This implies that there exists $z \in \left(\frac1{100}\varepsilon^{1 + \nu}\ZZ\right)^2 \cap \DD$ and $r \in [\varepsilon^{1 + \nu}, \varepsilon]$ such that $P(V) \subset B_{r/2}(z)$ and conditions~\eqref{it:205} and \eqref{it:206} hold. Write $\SCC_1, \ldots, \SCC_n$ for the connected components of $B_r(z) \cap X$ that intersect with $B_{r/2}(z)$. Then $n \le A$ by condition~\eqref{it:205} and $P(V) \subset \bigcup_{j = 1}^n \SCC_j$. By condition~\eqref{it:206}, we have 
    \begin{equation}\label{eq:258}
        D_{X}(x, y) \le Ar^\theta, \quad \forall j \in [1, n]_\ZZ, \ \forall x, y \in \SCC_j. 
    \end{equation}
    For each $j \in [1, n]_\ZZ$, write $s_j$ for the first time at which $P$ enters $\SCC_j$ and $t_j$ for the last time at which $P$ exits $\SCC_j$. Then it follows from \eqref{eq:258} that 
    \begin{equation}\label{eq:260}
        t_j - s_j \le \frac{Ar^\theta}{\len(P; D_{X})}, \quad \forall j \in [1, n]_\ZZ. 
    \end{equation}
    Combining \eqref{eq:259} and \eqref{eq:260}, we obtain that
    \begin{multline*}
        \mathop{\mathrm{Leb}}(V) \le \sum_{j = 1}^n (t_j - s_j) \le \frac{A^2r^\theta}{\len(P; D_{X})} \le \frac{A^2\varepsilon^\theta}{\len(P; D_{X})} \le \frac{2^{100}A^2\diam(P(V))^{\theta/(1 + \nu)}}{\len(P; D_{X})} \\
        \le \frac{2^{100}A^2\diam(P(V))^\vartheta}{\len(P; D_{X})}. 
    \end{multline*}
    This completes the proof. 
\end{proof}

\begin{lemma}\label{270}
    Almost surely, $\dim(P) \le \theta$ for all $D_X$-geodesics $P$. 
\end{lemma}

\begin{proof}
    This follows immediately from the fact that for each $\chi > \theta$, almost surely, all $D_X$-geodesics are H\"older continuous of order $1/\chi$ (cf.~\Cref{268}). 
\end{proof}

\begin{proof}[Proof of \Cref{302}, \eqref{302B}]
    This follows immediately from \Cref{269,270}. (Since every nontrivial path in $X$ has Hausdorff dimension $> 1$ (cf., e.g., \cite{DimFracPerc} and the proof of \Cref{324}), it follows that $\theta > 1$.)
\end{proof}

\section{Optimal bi-Lipschitz constants}\label{s:03}

\subsection{Setup and statements}

Fix two weak geodesic $\CLE_\kappa$ carpet metrics $D$ and $\widetilde D$ with the same scaling constants $\{\kc_r\}_{r > 0}$. 

We shall write 
\begin{align}
    M_\ast &\defeq \sup\left\{M > 0 : \BP\!\left\lbrack\exists x, y \in \Upsilon \text{ such that } x \leftrightarrow y \text{ and } \frac{\widetilde D_\Upsilon(x, y)}{D_\Upsilon(x, y)} \le M\right\rbrack = 0\right\}; \label{eq:321} \\
    M^\ast &\defeq \inf\left\{M > 0 : \BP\!\left\lbrack\exists x, y \in \Upsilon \text{ such that } x \leftrightarrow y \text{ and } \frac{\widetilde D_\Upsilon(x, y)}{D_\Upsilon(x, y)} \ge M\right\rbrack = 0\right\}. \label{eq:322}
\end{align}
It follows from \Cref{045} that $M_\ast$ and $M^\ast$ are positive and finite. By definition, for each $M \in (0, M^\ast)$ (resp.~$M > M_\ast$), 
\begin{multline}\label{eq:018}
    \BP\!\left\lbrack\text{there are distinct } x, y \in \Upsilon \text{ such that } x \leftrightarrow y \text{ and } \frac{\widetilde D_\Upsilon(x, y)}{D_\Upsilon(x, y)} \ge M\right\rbrack > 0 \\
    \text{(resp.~} \BP\!\left\lbrack\text{there are distinct } x, y \in \Upsilon \text{ such that } x \leftrightarrow y \text{ and } \frac{\widetilde D_\Upsilon(x, y)}{D_\Upsilon(x, y)} \le M\right\rbrack > 0\text{)}. 
\end{multline}
The primary objective of this section is to establish several quantitative refinements of \eqref{eq:018} --- specifically, to demonstrate that there exist many values of $\rr > 0$ for which
\begin{equation}\label{eq:017}
    \parbox{.85\linewidth}{it holds with uniformly positive probability that there are points $x, y \in \Upsilon$ with $x \leftrightarrow y$ such that $\lvert x\rvert$, $\lvert y\rvert$, and $\lvert x - y\rvert$ are all of order $\rr$ and $\widetilde D_\Upsilon(x, y) / D_\Upsilon(x, y)$ is close to $M^\ast$ (resp.~$M_\ast$).}
\end{equation}
Such pairs of points will be called ``shortcuts''. Our results will be formulated in terms of two types of events, defined in \Cref{025,003} below. In \Cref{s:00}, we will use the results of the present section to show that, with high probability, an annulus of suitable radius contains many ``shortcuts''. In \Cref{s:02}, under the assumption that $M_\ast < M^\ast$, we will find many values of $\rr > 0$ satisfying a condition that contradicts \eqref{eq:017}, leading to the conclusion that $M_\ast = M^\ast$. The argument in the present section follows the same general structure as in \cite[Section~3]{ExUniLQG} and \cite[Section~3]{UniCriSupercriLQGMet}, though the definitions of the relevant events are necessarily different.

\begin{definition}\label{025}
    Let $M, \rr > 0$ and $b \in (0, 1)$. Then:
    \begin{enumerate}
        \item We shall write $\overline G_{\rr}(M, b)$ for the event that there exist $x, y \in \overline{B_{\rr}(0)} \cap \Upsilon$ such that 
        \begin{equation*}
            \lvert x - y\rvert \ge b\rr, \quad x \leftrightarrow y, \quad \text{and} \quad \widetilde D_\Upsilon(x, y) \ge MD_\Upsilon(x, y). 
        \end{equation*}
        \item We shall write $\underline G_{\rr}(M, b)$ for the event that there exist $x, y \in \overline{B_{\rr}(0)} \cap \Upsilon$ such that 
        \begin{equation*}
            \lvert x - y\rvert \ge b\rr, \quad x \leftrightarrow y, \quad \text{and} \quad \widetilde D_\Upsilon(x, y) \le MD_\Upsilon(x, y). 
        \end{equation*}
    \end{enumerate}
\end{definition}

The other version of events has a more complicated definition, and includes several regularity conditions on the ``shortcut''. See \Cref{fig:shortcut} for an illustration. 

\begin{definition}\label{003}
    Let $M, r > 0$ and $0 < \alpha < a < 1$. Then:
    \begin{enumerate}
        \item We shall write $\overline H_r(M, a, \alpha)$ for the event that there exists a $D_\Upsilon$-geodesic 
        \begin{equation*}
            P \colon [0, 1] \to \overline{A_{(1 - a)r,(1 + a)r}(0)} \cap \Upsilon
        \end{equation*}
        from $\partial B_{(1 - a)r}(0) \cap \Upsilon$ to $\partial B_{(1 + a)r}(0) \cap \Upsilon$ and times $0 < s < t < 1$ satisfying the following conditions:
        \begin{enumerate}
            \item $\len(P; D_\Upsilon) \le a^{-1}\kc_r$. 
            \item The Euclidean diameter of $P$ is at most $r/100$. 
            \item $P(s) \in \partial B_{(1 - \alpha)r}(0) \cap \Upsilon$, $P(t) \in \partial B_r(0) \cap \Upsilon$, and $P|_{[s, t]} \subset \overline{A_{(1 - \alpha)r,r}(0)} \cap \Upsilon$. 
            \item $D_\Upsilon(P(s), P(t)) \ge \alpha^3\kc_r$ and $\widetilde D_\Upsilon(P(s), P(t)) \ge M D_\Upsilon(P(s), P(t))$. 
        \end{enumerate}
        \item We shall write $\underline H_r(M, a, \alpha)$ for the event that there exists a $\widetilde D_\Upsilon$-geodesic
        \begin{equation*}
            \widetilde P \colon [0, 1] \to \overline{A_{(1 - a)r,(1 + a)r}(0)} \cap \Upsilon
        \end{equation*}
        from $\partial B_{(1 - a)r}(0) \cap \Upsilon$ to $\partial B_{(1 + a)r}(0) \cap \Upsilon$ and times $0 < s < t < 1$ satisfying the following conditions:
        \begin{enumerate} 
            \item\label{003BA} $\len(\widetilde P; \widetilde D_\Upsilon) \le a^{-1}\kc_r$. 
            \item\label{003BB} The Euclidean diameter of $\widetilde P$ is at most $r/100$. 
            \item\label{003BC} $\widetilde P(s) \in \partial B_{(1 - \alpha)r}(0) \cap \Upsilon$, $\widetilde P(t) \in \partial B_r(0) \cap \Upsilon$, and $\widetilde P|_{[s, t]} \subset \overline{A_{(1 - \alpha)r,r}(0)} \cap \Upsilon$.  
            \item\label{003BD} $\widetilde D_\Upsilon(\widetilde P(s), \widetilde P(t)) \ge \alpha^3\kc_r$ and $\widetilde D_\Upsilon(\widetilde P(s), \widetilde P(t)) \le M D_\Upsilon(\widetilde P(s), \widetilde P(t))$. 
        \end{enumerate}
    \end{enumerate}
\end{definition}

\begin{figure}[ht!]
    \centering
    \includegraphics[width=0.6\linewidth]{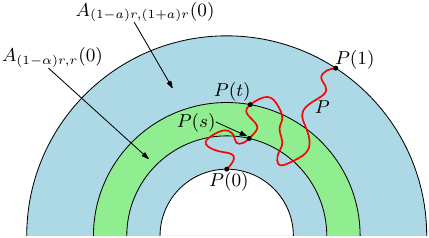}
    \caption{Illustration of the event $\overline H_r(M, a, \alpha)$ of \Cref{003}. $P$ is a $D_\Upsilon$-geodesic in $\overline{A_{(1 - a)r,(1 + a)r}(0)}$ from $\partial B_{(1 - a)r}(0) \cap \Upsilon$ to $\partial B_{(1 + a)r}(0) \cap \Upsilon$. There are times $0 < s < t < 1$ such that $P(s) \in \partial B_{(1 - \alpha)r}(0) \cap \Upsilon$, $P(t) \in \partial B_r(0) \cap \Upsilon$, and $P|_{[s, t]} \subset \overline{A_{(1 - \alpha)r,r}(0)} \cap \Upsilon$. Also, $P|_{[s, t]}$ is a ``shortcut'', in the sense that $\widetilde D_\Upsilon(P(s), P(t)) \ge M D_\Upsilon(P(s), P(t))$.}
    \label{fig:shortcut}
\end{figure}

\begin{lemma}\label{007}
    \begin{enumerate}
        \item\label{007A} For each $M \in (0, M^\ast)$ and $\rr > 0$, there exists $b = b(M, \rr) \in (0, 1)$ such that $\BP\lbrack\overline G_{\rr}(M, b)\rbrack \ge b$. 
        \item\label{007B} For each $M > M_\ast$ and $\rr > 0$, there exists $b = b(M, \rr) \in (0, 1)$ such that $\BP\lbrack\underline G_{\rr}(M, b)\rbrack \ge b$. 
    \end{enumerate}
\end{lemma}

\begin{proof}
    First, we consider assertion~\eqref{007A}. Suppose by way of contradiction that $\BP\lbrack\overline G_{\rr}(M, b)\rbrack < b$ for all $b \in (0, 1)$. Then, almost surely, 
    \begin{equation}\label{eq:000}
        \widetilde D_\Upsilon(x, y) \le MD_\Upsilon(x, y), \quad \forall x, y \in \overline{B_{\rr}(0)} \cap \Upsilon \text{ with } x \leftrightarrow y. 
    \end{equation}
    By \Cref{010}, Axiom~\eqref{010C} (translation invariance), \eqref{eq:000} implies that, almost surely, 
    \begin{equation*}
        \widetilde D_\Upsilon(x, y) \le MD_\Upsilon(x, y), \quad \forall z \in \QQ^2, \ \forall x, y \in \overline{B_{\rr}(z)} \cap \Upsilon \text{ with } x \leftrightarrow y. 
    \end{equation*}
    hence that, almost surely, 
    \begin{equation}\label{eq:320}
        \widetilde D_\Upsilon(x, y) \le MD_\Upsilon(x, y), \quad \forall x, y \in \Upsilon \text{ with } \lvert x - y\rvert \le \rr/2 \text{ and } x \leftrightarrow y. 
    \end{equation}
    Since for each $x, y \in \Upsilon$ with $x \leftrightarrow y$ and each $D_\Upsilon$-geodesic $P \colon [0, 1] \to \Upsilon$ from $x$ to $y$, there exist times $0 = t_0 < t_1 < \cdots < t_n = 1$ such that $\lvert P(t_{j - 1}) - P(t_j)\rvert \le \rr/2$ for all $j \in [1, n]_\ZZ$, it follows from \eqref{eq:320} that, almost surely, 
    \begin{multline*}
        \widetilde D_\Upsilon(x, y) \le \sum_{j = 1}^n \widetilde D_\Upsilon(P(t_{j - 1}), P(t_j)) \le M\sum_{j = 1}^n D_\Upsilon(P(t_{j - 1}), P(t_j)) = MD_\Upsilon(x, y), \\
        \forall x, y \in \Upsilon \text{ with } x \leftrightarrow y. 
    \end{multline*}
    This contradicts the assumption that $M < M^\ast$ and the definition of $M^\ast$ (cf.~\eqref{eq:322}). This completes the proof of assertion~\eqref{007A}. Assertion~\eqref{007B} from an argument which is analogous to that used to prove assertion~\eqref{007A}.
\end{proof}

\begin{proposition}\label{004}
    For each $0 < \mu < \nu$ and $\Lambda > 1$, there exists $p = p(\mu, \nu, \Lambda) \in (0, 1)$ and $a = a(\mu, \nu, \Lambda) \in (0, 1)$ such that for each $\alpha \in (0, a)$ and $M \in (0, M^\ast)$, there exists $M^\prime = M^\prime(\mu, \nu, \Lambda, \alpha, M) \in (M, M^\ast)$ such that for each $b \in (0, 1)$, there exists $\varepsilon_\ast = \varepsilon_\ast(\mu, \nu, \Lambda, \alpha, M, b) \in (0, 1)$ such that for each $\rr > 0$ for which $\BP\lbrack\overline G_{\rr}(M^\prime, b)\rbrack \ge b$ and each $\varepsilon \in (0, \varepsilon_\ast]$, the following is true:
    \begin{equation}\label{eq:002}
        \parbox{.80\linewidth}{There are at least $\mu\log_\Lambda(1/\varepsilon)$ values of $r \in [\varepsilon^{1 + \nu}\rr, \varepsilon\rr] \cap \{\Lambda^k\}_{k \in \ZZ}$ for which $\BP\lbrack\overline H_r(M, a, \alpha)\rbrack \ge p$.}
    \end{equation}
\end{proposition}

We will give the proof of \Cref{004} in \Cref{subsec:proofof64} just below.

\begin{proposition}\label{005}
    For each $0 < \mu < \nu$ and $\Lambda > 1$, there exists $p = p(\mu, \nu, \Lambda) \in (0, 1)$ and $a = a(\mu, \nu, \Lambda) \in (0, 1)$ such that for each $\alpha \in (0, a)$, $M \in (0, M^\ast)$, and $\rr > 0$, there exists $\varepsilon_\ast = \varepsilon_\ast(\mu, \nu, \Lambda, \alpha, M, \rr) \in (0, 1)$ such that \eqref{eq:002} holds for all $\varepsilon \in (0, \varepsilon_\ast]$. 
\end{proposition}

\begin{proof}
    This follows immediately from \Cref{007,004}. 
\end{proof}

\begin{proposition}\label{008}
    For each $0 < \mu < \nu$ and $\Lambda > 1$, there exists $b = b(\mu, \nu, \Lambda) \in (0, 1)$ such that for each $M \in (0, M^\ast)$ and $\rr > 0$, there exists $\varepsilon_\ast = \varepsilon_\ast(\mu, \nu, \Lambda, M, \rr) \in (0, 1)$ such that for each $\varepsilon \in (0, \varepsilon_\ast]$, the following is true:
    \begin{equation*}
        \parbox{.80\linewidth}{There are at least $\mu\log_\Lambda(1/\varepsilon)$ values of $r \in [\varepsilon^{1 + \nu}\rr, \varepsilon\rr] \cap \{\Lambda^k\}_{k \in \ZZ}$ for which $\BP\lbrack\overline G_r(M, b)\rbrack \ge b$.}
    \end{equation*}
\end{proposition}

\begin{proof}
    Let $p$ and $a$ be as in \Cref{005}. Set $b \defeq (a/2) \wedge p$. Then \Cref{008} follows immediately from \Cref{005}, together with the fact that $\overline H_r(M, a, b) \subset \overline G_r(M, b)$. 
\end{proof}

By interchanging the roles of $D$ and $\widetilde D$, we obtain from \Cref{004,008} the following statements. 

\begin{proposition}\label{030}
    For each $0 < \mu < \nu$ and $\Lambda > 1$, there exists $p = 
    p(\mu, \nu, \Lambda) \in (0, 1)$ and $a = a(\mu, \nu, \Lambda) \in (0, 1)$ such that for each $\alpha \in (0, a)$ and $M > M_\ast$, there exists $M^\prime = M^\prime(\mu, \nu, \Lambda, \alpha, M) \in (M_\ast, M)$ such that for each $b \in (0, 1)$, there exists $\varepsilon_\ast = \varepsilon_\ast(\mu, \nu, \Lambda, \alpha, M, b) \in (0, 1)$ such that for each $\rr > 0$ for which $\BP\lbrack\underline G_{\rr}(M^\prime, b)\rbrack \ge b$ and each $\varepsilon \in (0, \varepsilon_\ast]$, the following is true:
    \begin{equation*}
        \parbox{.80\linewidth}{There are at least $\mu\log_{\Lambda}(1/\varepsilon)$ values of $r \in [\varepsilon^{1 + \nu}\rr, \varepsilon\rr] \cap \{\Lambda^k\}_{k \in \ZZ}$ for which $\BP\lbrack\underline H_r(M, a, \alpha)\rbrack \ge p$.}
    \end{equation*}
\end{proposition}

\begin{proposition}\label{031}
    For each $0 < \mu < \nu$ and $\Lambda > 1$, there exists $b = b(\mu, \nu, \Lambda) \in (0, 1)$ such that for each $M > M_\ast$ and $\rr > 0$, there exists $\varepsilon_\ast = \varepsilon_\ast(\mu, \nu, \Lambda, M, \rr) \in (0, 1)$ such that for each $\varepsilon \in (0, \varepsilon_\ast]$, the following is true:
    \begin{equation*}
        \parbox{.80\linewidth}{There are at least $\mu\log_\Lambda(1/\varepsilon)$ values of $r \in [\varepsilon^{1 + \nu}\rr, \varepsilon\rr] \cap \{\Lambda^k\}_{k \in \ZZ}$ for which $\BP\lbrack\underline G_r(M, b)\rbrack \ge b$.}
    \end{equation*}
\end{proposition}

\subsection{Proof of \Cref{004}}\label{subsec:proofof64}

In order to prove \Cref{004}, it suffices to prove the following contrapositive. 

\begin{lemma}\label{029}
    For each $0 < \mu < \nu$ and $\Lambda > 1$, there exists $p = p(\mu, \nu, \Lambda) \in (0, 1)$ and $a = a(\mu, \nu, \Lambda) \in (0, 1)$ such that for each $\alpha \in (0, a)$ and $M \in (0, M^\ast)$, there exists $M^\prime = M^\prime(\mu, \nu, \Lambda, \alpha, M) \in (M, M^\ast)$ such that for each $b \in (0, 1)$, there exists $\varepsilon_\ast = \varepsilon_\ast(\mu, \nu, \Lambda, \alpha, M, b) \in (0, 1)$ such that for each $\rr > 0$, if there exists $\varepsilon \in (0, \varepsilon_\ast]$ satisfying \eqref{eq:001} just below, then $\BP\lbrack\overline G_{\rr}(M^\prime, b)\rbrack < b$.
    \begin{equation}\label{eq:001}
        \parbox{.85\linewidth}{There are at least $(\nu - \mu)\log_\Lambda(1/\varepsilon)$ values of $r \in [\varepsilon^{1 + \nu}\rr, \varepsilon\rr] \cap \{\Lambda^k\}_{k \in \ZZ}$ for which we have $\BP\lbrack\overline H_r(M, a, \alpha)\rbrack < p$.}
    \end{equation}
\end{lemma}

The main idea behind the proof of \Cref{029} is as follows. Assume that \eqref{eq:001} holds for a sufficiently small universal constant $p \in (0,1)$. Then, by applying \Cref{263,214} together with a union bound, we can cover the plane with balls of the form $B_{r/2}(z)$ such that the complement of the event $\overline H_r(M, a, \alpha)$ occurs and the $D_\Upsilon$-diameter of $B_{2r}(z) \cap \Upsilon$ remains controlled. In particular, for every pair of points $u \in \partial B_{(1 - \alpha)r}(z) \cap \Upsilon$ and $v \in \partial B_r(z) \cap \Upsilon$ that are connected within $A_{(1 - \alpha)r,r}(z)$ and satisfy certain regularity properties, we have $\widetilde D_\Upsilon(P(s), P(t)) \le M D_\Upsilon(P(s), P(t))$. Consequently, for any $D_\Upsilon$-geodesic $P$ whose $D_\Upsilon$-length is not too small, by examining the moments when it crosses such annuli $A_{(1 - \alpha)r,r}(z)$, we can deduce that $\widetilde D_\Upsilon(P(0), P(1)) \le M^\prime D_\Upsilon(P(0), P(1))$ for some suitable constant $M^\prime \in (M, M^\ast)$.

In the remainder of the present subsection, fix $M \in (0, M^\ast)$. 

Let us now define the event to which we will apply \Cref{263}. For $z \in \CC$, $r > 0$, and $0 < \alpha < a < 1$, write $E_r(z) = E_r(z; M, a, \alpha)$ for the event that the following are true:
\begin{enumerate}
    \item\label{it:000} We have
    \begin{multline*}
        D_\Upsilon(x, y) \le \frac{M_\ast}{M^\ast} \left(D_\Upsilon(\partial B_{r/2}(z) \cap \Upsilon, \partial B_{3r/4}(z) \cap \Upsilon) \wedge D_\Upsilon(\partial B_{3r/2}(z) \cap \Upsilon, \partial B_{2r}(z) \cap \Upsilon)\right), \\
        \forall x, y \in A_{(1 - a)r,(1 + a)r}(z) \cap \Upsilon \text{ with } x \xleftrightarrow{A_{(1 - a)r,(1 + a)r}(z)} y. 
    \end{multline*}
    \item\label{it:001} Let $P \colon [0, 1] \to \overline{A_{(1 - a)r,(1 + a)r}(z)} \cap \Upsilon$ be a $D_\Upsilon$-geodesic connecting $\partial B_{(1 - a)r}(z) \cap \Upsilon$ and $\partial B_{(1 + a)r}(z) \cap \Upsilon$. Suppose that there are times $0 < s < t < 1$ such that the following conditions are satisfied:
    \begin{enumerate}
        \item\label{it:0010} $\len(P; D_\Upsilon) \le a^{-1}\kc_r$. 
        \item\label{it:0011} The Euclidean diameter of $P$ is at most $r/100$. 
        \item\label{it:0012} $P(s) \in \partial B_{(1 - \alpha)r}(z) \cap \Upsilon$, $P(t) \in \partial B_r(z) \cap \Upsilon$, and $P|_{[s, t]} \subset \overline{A_{(1 - \alpha)r,r}(z)} \cap \Upsilon$. 
        \item\label{it:0013} $D_\Upsilon(P(s), P(t)) \ge \alpha^3\kc_r$.
    \end{enumerate}
    Then $\widetilde D_\Upsilon(P(s), P(t)) \le M D_\Upsilon(P(s), P(t))$. 
    \item\label{it:002} $D_\Upsilon(\partial B_{(1 - \alpha)r}(z) \cap \Upsilon, \partial B_r(z) \cap \Upsilon) \ge \alpha^3\kc_r$. 
    \item\label{it:300} Every connected component of $A_{(1 - a)r,(1 + a)r}(z) \cap \Upsilon$ has Euclidean diameter at most $r/100$. 
\end{enumerate}

The key requirement in the definition of $E_r(z)$ is condition~\eqref{it:001}. Note that this condition is precisely the complement of the event $\overline H_r(M, a, \alpha)$. Condition~\eqref{it:000} is included to guarantee that $E_r(z)$ is almost surely determined by the configuration of $A_{r/2,2r}(z) \cap \Upsilon$. Finally, conditions~\eqref{it:002} and \eqref{it:300} are introduced so that conditions~\eqref{it:0013} and \eqref{it:0011}, respectively, can be verified later.

\begin{lemma}\label{009}
    Let $z \in \CC$ and $r > 0$. Then $E_r(z)$ is almost surely determined by $A_{r/2,2r}(z) \cap \Upsilon$. 
\end{lemma}

\begin{proof}
    It is clear that conditions~\eqref{it:000}, \eqref{it:002}, \eqref{it:300} are almost surely determined by $A_{r/2,2r}(z) \cap \Upsilon$. If condition~\eqref{it:000} holds, then for each $x, y \in A_{(1 - a)r,(1 + a)r}(z) \cap \Upsilon$ with $x \xleftrightarrow{A_{(1 - a)r,(1 + a)r}(z)} y$, any path in $\Upsilon$ between $x$ and $y$ that exits $A_{r/2,2r}(z)$ is neither a $D_\Upsilon$-geodesic nor a $\widetilde D_\Upsilon$-geodesic. This implies that, on the event that condition~\eqref{it:000} holds, condition~\eqref{it:001} is almost surely determined by $A_{r/2,2r}(z) \cap \Upsilon$. This completes the proof. 
\end{proof}

\begin{lemma}\label{013}
    For each $p \in (0, 1)$, there exists $a = a(p) \in (0, 1)$ such that for each $\alpha \in (0, a)$ and each $r > 0$ for which $\BP\lbrack\overline H_r(M, a, \alpha)\rbrack < p$, we have
    \begin{equation*}
        \BP\lbrack E_r(z)\rbrack \ge 1 - 2p, \quad \forall z \in \CC. 
    \end{equation*} 
\end{lemma}

\begin{proof}
    By \Cref{010}, Axiom~\eqref{010C} (translation invariance), it suffices to consider the case $z = 0$. It follows immediately from the definition of $\overline H_r(M, a, \alpha)$ that, if $\BP\lbrack\overline H_r(M, a, \alpha)\rbrack < p$, then condition~\eqref{it:001} holds with probability at least $1 - p$. Thus, it suffices to show that there exists $a = a(p) \in (0, 1)$ such that for each $\alpha \in (0, a)$ and $r > 0$, conditions~\eqref{it:000}, \eqref{it:002}, \eqref{it:300} hold with probability at least $1 - p$. First, we consider condition~\eqref{it:000}. By \Cref{010}, Axiom~\eqref{010D} (tightness across scales), we may choose a sufficiently small $\delta > 0$ such that for each $r > 0$, it holds with probability at least $1 - p/4$ that 
    \begin{multline}\label{eq:330}
        D_\Upsilon(x, y) \le \frac{M_\ast}{M^\ast}\left(D_\Upsilon(\partial B_{r/2}(0) \cap \Upsilon, \partial B_{3r/4}(0) \cap \Upsilon) \wedge D_\Upsilon(\partial B_{3r/2}(0) \cap \Upsilon, \partial B_{2r}(0) \cap \Upsilon)\right), \\
        \forall x, y \in A_{3r/4,3r/2}(0) \cap \Upsilon \text{ with } x \leftrightarrow y \text{ and } \lvert x - y\rvert \le \delta r. 
    \end{multline}
    We may choose $a \in (0, 1)$ to be sufficiently small so that for each $r > 0$, it holds with probability at least $1 - p/4$ that
    \begin{equation}\label{eq:331}
        \parbox{.85\linewidth}{the Euclidean diameter of each connected component of $A_{(1 - a)r,(1 + a)r}(z) \cap \Upsilon$ is at most $\delta r$.}
    \end{equation}
    Combining \eqref{eq:330} and \eqref{eq:331}, we obtain that for each $r > 0$, condition~\eqref{it:000} holds with probability at least $1 - p/2$. Next, it follows immediately from \Cref{314} that there exists $a = a(p) \in (0, 1)$ such that for each $\alpha \in (0, a)$ and $r > 0$, condition~\eqref{it:002} holds with probability at least $1 - p/4$. Finally, by \eqref{eq:331}, for each $r > 0$, condition~\eqref{it:300} holds with probability at least $1 - p/4$. This completes the proof. 
\end{proof}

\begin{lemma}\label{028}
    For each $0 < \mu < \nu$ and $\Lambda > 1$, there exists $p = p(\mu, \nu, \Lambda) \in (0, 1)$ and $a = a(\mu, \nu, \Lambda) \in (0, 1)$ such that for each $\alpha \in (0, a)$ and each bounded open subset $U \subset \CC$, the following is true: For each $\varepsilon \in (0, 1)$ and $\rr > 0$ for which \eqref{eq:001} holds, it holds with probability tending to one as $\varepsilon \to 0$, at a rate which is uniform in $\rr$, that for each $z \in \left(\frac1{100}\varepsilon^{1 + \nu}\rr\ZZ\right)^2 \cap (\rr U)$, there exists $r \in [\varepsilon^{1 + \nu}\rr, \varepsilon\rr] \cap \{\Lambda^k\}_{k \in \ZZ}$ such that the following are true:
    \begin{enumerate}
        \item\label{028A} $E_r(z)$ occurs. 
        \item\label{028B} We have
        \begin{equation*}
            D_\Upsilon(x, y) \le a^{-1}\kc_r, \quad \forall x, y \in B_{2r}(z) \cap \Upsilon \text{ with } x \xleftrightarrow{B_{2r}(z)} y. 
        \end{equation*}
    \end{enumerate}
\end{lemma}

\begin{proof}
    This follows immediately from \Cref{263,214,009,013}, together with a union bound. 
\end{proof}

\begin{figure}[ht!]
    \centering
    \includegraphics[width=0.8\linewidth]{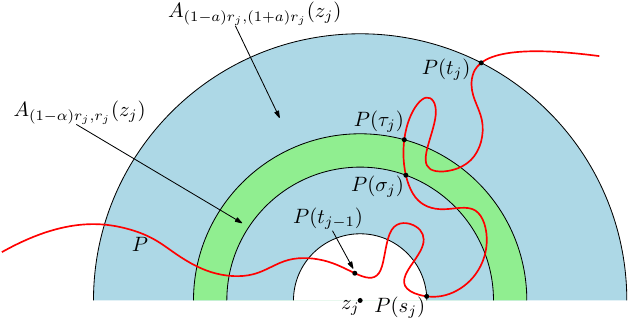}
    \caption{Illustration of the definition of the times $s_j$, $\sigma_j$, $\tau_j$, $t_j$ in the proof of \Cref{029}. $t_j$ is the first time after $t_{j - 1}$ at which $P$ hits $\partial B_{(1 + a)r_j}(z_j)$; $s_j$ is the last time before $t_j$ at which $P$ hits $\partial B_{(1 - a)r_j}(z_j)$; $\sigma_j$ is the last time before $t_j$ at which $P$ hits $\partial B_{(1 - \alpha)r_j}(z_j)$; $\tau_j$ is the first time after $\sigma_j$ at which $P$ hits $\partial B_{r_j}(z_j)$.}
    \label{fig:times}
\end{figure}

\begin{proof}[Proof of \Cref{029}]
    Let $p$ and $a$ be as in \Cref{028}. Fix $\alpha \in (0, a)$. Choose
    \begin{equation*}
        M^\prime \in \left(M^\ast - \alpha^4(M^\ast - M), M^\ast\right). 
    \end{equation*}
    It suffices to show that for each $b \in (0, 1)$, there exists $\varepsilon_\ast = \varepsilon_\ast(\mu, \nu, \Lambda, \alpha, M, b) \in (0, 1)$ such that for each $\varepsilon \in (0, \varepsilon_\ast]$ and $\rr > 0$ for which \eqref{eq:001} holds, it holds with probability greater than $1 - b$ that 
    \begin{equation}\label{eq:008}
        \widetilde D_\Upsilon(x, y)_\Upsilon < M^\prime D_\Upsilon(x, y), \quad \forall x, y \in \overline{B_{\rr}(0)} \cap \Upsilon \text{ with } \lvert x - y\vert \ge b\rr \text{ and } x \leftrightarrow y. 
    \end{equation}

    For $\varepsilon \in (0, 1)$, $\rr > 0$, and $R > 1$, write $F_\varepsilon^{\rr} = F_\varepsilon^{\rr}(R)$ for the event that the following are true:
    \begin{enumerate}
        \item\label{it:003} There exists a loop of $\Gamma$ in $A_{\rr,R\rr}(0)$ disconnecting the inner and outer boundaries of $A_{\rr,R\rr}(0)$. In particular, every $D_\Upsilon$-geodesic between two points of $\overline{B_{\rr}(0)} \cap \Upsilon$ is contained in $B_{R\rr}(0)$. 
        \item\label{it:004} The event of \Cref{028}, with $B_R(0)$ in place of $U$, occurs. 
        \item\label{it:034} We have
        \begin{equation*}
            D_\Upsilon(x, y) \ge \varepsilon^{1/2}\kc_{\rr}, \quad \forall x, y \in \overline{B_{\rr}(0)} \cap \Upsilon \text{ with } \lvert x - y\vert \ge b\rr \text{ and } x \leftrightarrow y.
        \end{equation*}
    \end{enumerate}
    We \emph{claim} that there exists $R = R(b) > 1$ and $\varepsilon_\ast = \varepsilon_\ast(\mu, \nu, \Lambda, \alpha, M, b, R) \in (0, 1)$ such that for each $\varepsilon \in (0, \varepsilon_\ast]$ and $\rr > 0$, we have $\BP\lbrack F_\varepsilon^{\rr}\rbrack > 1 - b$. Since the law of $\Upsilon$ is scaling invariant, it follows that there exists $R = R(b) > 1$ such that for each $\rr > 0$, condition~\eqref{it:003} holds with probability greater than $1 - b/3$. By \Cref{028}, there exists $\varepsilon_0 \in (0, 1)$ such that for each $\varepsilon \in (0, \varepsilon_0]$ and $\rr > 0$, condition~\eqref{it:004} holds with probability greater than $1 - b/3$. By \Cref{314}, there exists $\varepsilon_1 \in (0, 1)$ such that for each $\rr > 0$, condition~\eqref{it:034} holds with probability greater than $1 - b/3$. By setting $\varepsilon_\ast \defeq \varepsilon_0 \wedge \varepsilon_1$, we complete the proof of the \emph{claim}. 

    Henceforth assume that $F_\varepsilon^{\rr}$ occurs. Let $x, y \in \overline{B_{\rr}(0)} \cap \Upsilon$ with $\lvert x - y\vert \ge b\rr$ and $x \leftrightarrow y$. Let $P \colon [0, 1] \to \Upsilon$ be a $D_\Upsilon$-geodesic from $x$ to $y$. It follows from condition~\eqref{it:003} that $P$ is contained in $B_{R\rr}(0)$. Set $t_0 \defeq 0$ and choose $z_1 \in \left(\frac1{100}\varepsilon^{1 + \nu}\rr\ZZ\right)^2 \cap B_{R\rr}(0)$ and $r_1 \in [\varepsilon^{1 + \nu}\rr, \varepsilon\rr] \cap \{\Lambda^k\}_{k \in \ZZ}$ such that $x \in B_{r_1/2}(z_1)$ and conditions~\eqref{028A} and \eqref{028B} of \Cref{028} (with $(z_1, r_1)$ in place of $(z, r)$) hold. Then, inductively, for each $j \in \NN$, if $y \notin B_{(1 + a)r_j}(z_j)$, then set $t_j$ to be the first time after $t_{j - 1}$ at which $P$ hits $\partial B_{(1 + a)r_j}(z_j)$, and choose $z_{j + 1} \in \left(\frac1{100}\varepsilon^{1 + \nu}\rr\ZZ\right)^2 \cap B_{R\rr}(0)$ and $r_{j + 1} \in [\varepsilon^{1 + \nu}\rr, \varepsilon\rr] \cap \{\Lambda^k\}_{k \in \ZZ}$ such that $P(t_j) \in B_{r_{j + 1}/2}(z_{j + 1})$ and conditions~\eqref{028A} and \eqref{028B} in \Cref{028} (with $(z_{j + 1}, r_{j + 1})$ in place of $(z, r)$) hold; otherwise set $t_j \defeq 1$. Write $n \defeq \inf\{j \in \NN : t_j = 1\}$. For each $j = [1, n - 1]_\ZZ$, write
    \begin{itemize}
        \item $s_j$ for the last time before $t_j$ at which $P$ hits $\partial B_{(1 - a)r_j}(z_j)$;
        \item $\sigma_j$ for the last time before $t_j$ at which $P$ hits $\partial B_{(1 - \alpha)r_j}(z_j)$; 
        \item $\tau_j$ for the first time after $\sigma_j$ at which $P$ hits $\partial B_{r_j}(z_j)$. 
    \end{itemize}
    See \Cref{fig:times} for an illustration. Thus $t_{j - 1} < s_j < \sigma_j < \tau_j < t_j$ for all $j \in [1, n - 1]_\ZZ$. Fix $j \in [1, n - 1]_\ZZ$. It follows immediately from the above definitions that condition~\eqref{it:0012} is satisfied with $z_j$, $r_j$, $P|_{[s_j,t_j]}$, $\sigma_j$, $\tau_j$ in place of $z$, $r$, $P$, $s$, $t$, respectively. It follows immediately from condition~\eqref{it:002} (resp.~\eqref{it:300}) that condition~\eqref{it:0013} (resp.~\eqref{it:0011}) is satisfied with $z_j$, $r_j$, $P|_{[s_j,t_j]}$, $\sigma_j$, $\tau_j$ in place of $z$, $r$, $P$, $s$, $t$, respectively. It follows immediately from condition~\eqref{028B} in \Cref{028} that condition~\eqref{it:0010} is satisfied with $z_j$, $r_j$, $P|_{[s_j,t_j]}$, $\sigma_j$, $\tau_j$ in place of $z$, $r$, $P$, $s$, $t$, respectively. Thus, we conclude from condition~\eqref{it:001} that
    \begin{equation}\label{eq:009}
        \widetilde D_\Upsilon(P(\sigma_j), P(\tau_j)) \le M D_\Upsilon(P(\sigma_j), P(\tau_j)), \quad \forall j \in [1, n - 1]_\ZZ. 
    \end{equation}
    By condition~\eqref{it:002}, we have
    \begin{equation}\label{eq:332}
        D_\Upsilon(P(\sigma_j), P(\tau_j)) \ge \alpha^3\kc_{r_j}, \quad \forall j \in [1, n - 1]_\ZZ. 
    \end{equation}
    By condition~\eqref{028B} of \Cref{028}, we have
    \begin{equation}\label{eq:010}
        D_\Upsilon(P(t_{j - 1}), P(t_j)) \le \alpha^{-1}\kc_{r_j}, \quad \forall j \in [1, n]_\ZZ.
    \end{equation}
    Combining \eqref{eq:332} and \eqref{eq:010}, we obtain that
    \begin{equation}\label{eq:333}
        D_\Upsilon(P(\sigma_j), P(\tau_j)) \ge \alpha^4 D_\Upsilon(P(t_{j - 1}), P(t_j)), \quad \forall j \in [1, n - 1]_\ZZ. 
    \end{equation}
    By \eqref{eq:340}, condition~\eqref{it:034}, and \eqref{eq:010}, we have
    \begin{equation}\label{eq:334}
        D_\Upsilon(P(t_{n - 1}), P(t_n)) \le \alpha^{-1}\kc_{r_n} \le \alpha^{-1}\kK\varepsilon\kc_{\rr} \le \alpha^{-1}\kK\varepsilon^{1/2} D_\Upsilon(x, y). 
    \end{equation}
    Combining \eqref{eq:009}, \eqref{eq:333}, and \eqref{eq:334}, we obtain that
    \begin{align*}
        \widetilde D_\Upsilon(x, y) &\le \sum_{j = 1}^{n - 1} \left(\widetilde D_\Upsilon(P(t_{j - 1}), P(\sigma_j)) + \widetilde D_\Upsilon(P(\sigma_j), P(\tau_j)) + \widetilde D_\Upsilon(P(\tau_j), P(t_j))\right) \\
        &+ \widetilde D_\Upsilon(P(t_{n - 1}), P(t_n)) \\
        &\le \sum_{j = 1}^{n - 1} \left(M^\ast D_\Upsilon(P(t_{j - 1}), P(\sigma_j)) + M D_\Upsilon(P(\sigma_j), P(\tau_j)) + M^\ast D_\Upsilon(P(\tau_j), P(t_j))\right) \\
        &+ M^\ast D_\Upsilon(P(t_{n - 1}), P(t_n)) \\
        &= \sum_{j = 1}^{n - 1} \left(M^\ast D_\Upsilon(P(t_{j - 1}), P(t_j)) - (M^\ast - M) D_\Upsilon(P(\sigma_j), P(\tau_j))\right) + M^\ast D_\Upsilon(P(t_{n - 1}), P(t_n)) \\
        &\le \left(M^\ast - \alpha^4(M^\ast - M)\right)\sum_{j = 1}^{n - 1} D_\Upsilon(P(t_{j - 1}), P(t_j)) + M^\ast D_\Upsilon(P(t_{n - 1}), P(t_n)) \\
        &\le \left(M^\ast - \alpha^4(M^\ast - M)\right)D_\Upsilon(x, y) + M^\ast D_\Upsilon(P(t_{n - 1}), P(t_n)) \\
        &\le \left(M^\ast - \alpha^4(M^\ast - M) + \alpha^{-1}\kK\varepsilon^{1/2}\right)D_\Upsilon(x, y). 
    \end{align*}
    Recall that $M^\prime > M^\ast - \alpha^4(M^\ast - M)$. Thus, we conclude that, if $F_\varepsilon^{\rr}$ occurs and $\varepsilon$ is sufficiently small so that $M^\ast - \alpha^4(M^\ast - M) + \alpha^{-1}\kK\varepsilon^{1/2} \le M^\prime$, then \eqref{eq:008} holds. Combining this with the \emph{claim} of the preceding paragraph, we complete the proof.  
\end{proof}

\section{Preliminary constructions}\label{s:00}

\subsection{Setup and outline}\label{ss:14}

The goal of this section is to construct events that satisfy certain conditions. Then, in \Cref{s:02}, we will use these objects to prove \Cref{000}. Fix two weak geodesic $\CLE_\kappa$ carpet metrics $D$ and $\widetilde D$ with the same scaling constants $\{\kc_r\}_{r > 0}$. 

Let $M_\ast$ and $M^\ast$ be as in \eqref{eq:321} and \eqref{eq:322}, respectively. Our ultimate goal is to prove that $M_\ast = M^\ast$. Suppose by way of eventual contradiction that $M_\ast < M^\ast$. To state the conditions which our events need to satisfy, we shall write 
\begin{equation*}
    M_1 \defeq \left(2M_\ast + M^\ast\right)/3 \quad \text{and} \quad M_2 \defeq \left(M_\ast + 2M^\ast\right)/3. 
\end{equation*}
Fix a sufficiently small deterministic constant $\lambda > 0$, in a manner depending only on $M_\ast$ and $M^\ast$, whose particular value is unimportant. Fix parameters 
\begin{equation}\label{eq:327}
    \pp \in (0, 1), \quad \fA > 0, \quad 0 < \mu < \nu, \quad \Lambda \in \NN
\end{equation}
to be chosen later (cf.~\Cref{223}), in a manner depending only on $D$ and $\widetilde D$. (In particular, $\pp$ will be chosen to be close to one, $\fA$, $\nu$, and $\Lambda$ will be chosen to be large, and $\mu$ will be chosen to be close to $\nu$.) Let $\underline p = p(\mu, \nu, \Lambda)$ and $a = a(\mu, \nu, \Lambda)$ be as in \Cref{030}. 

In the present section, we construct a set of admissible radii $\SR \subset \{\Lambda^k\}_{k \in \ZZ}$ and, for each $r \in \rho^{-1}\SR$ (where $\rho \in (0, 1)$ is a deterministic constant), 
\begin{itemize}
    \item events $\fE_r$ and $\fF_r$, 
    \item random open subsets $V_r \subset U_r \subset A_{2r,4r}(0)$, and
    \item a resampling $\Upsilon_r$ of $\Upsilon$
\end{itemize} 
that roughly speaking satisfy the following conditions:
\begin{itemize}
    \item The event $\fE_r$ is almost surely determined by $A_{r,5r}(0) \cap \Upsilon$, and $\BP\lbrack\fE_r\rbrack \ge \pp$. 
    \item The number of possibilities for $(U_r, V_r)$ is finite.
    \item Let $(U_r^\prime, V_r^\prime)$ be sampled uniformly from all the possibilities for $(U_r, V_r)$, independently of everything else. Write $\Gamma(\CC \setminus \overline{V_r^\prime}, \CC \setminus U_r^\prime)$ for the union of the collection of the $(\CC \setminus \overline{V_r^\prime}, \CC \setminus U_r^\prime)$-excursions of $\Gamma$ and the collection of the loops of $\Gamma$ that are contained in $\CC \setminus \overline{V_r^\prime}$ and intersect with $\CC \setminus U_r^\prime$. Then $\Upsilon_r$ is a conditionally independent copy of $\Upsilon$ given $\Gamma(\CC \setminus \overline{V_r^\prime}, \CC \setminus U_r^\prime)$ and $\wp$. The purpose of the use of $(U_r^\prime, V_r^\prime)$ is to ensure that $\Upsilon_r$ also has the law of a whole-plane nested $\CLE_\kappa$. We will ultimately only be interested in the case where $(U_r, V_r) = (U_r^\prime, V_r^\prime)$.
    \item There is a deterministic constant $\fa_9 \in (0, 1)$ such that $\BP\lbrack\fF_r \mid \fE_r, \ \Upsilon\rbrack \ge \fa_9$ almost surely. 
    \item Write $\Gamma_r$ for the whole-plane nested $\CLE_\kappa$ corresponding to $\Upsilon_r$. Then, on the event $\fE_r \cap \fF_r$, there is a long narrow ``tube'' between loops of $\Gamma_r$ such that if $P_r$ is a $D_{\Upsilon_r}$-geodesic that passes though the ``tube'', then there are times $\sigma < \tau$ such that 
    \begin{equation}\label{eq:093}
        P_r([\sigma, \tau]) \subset B_{4r}(0) \cap \Upsilon_r \quad \text{and} \quad \widetilde D_{\Upsilon_r}(P_r(\sigma), P_r(\tau)) \le M_2 D_{\Upsilon_r}(P_r(\sigma), P_r(\tau)). 
    \end{equation}
\end{itemize}

The existence of these objects will be established through a detailed and explicit construction. To provide some intuition, we now give a rough overview of how this construction works; see \Cref{fig:section7} for an illustration. We begin by considering a large collection of ``test balls'' within $A_{2r,4r}(0)$. Inside each of these test balls, there is a uniformly positive probability that one can find a ``good'' pair of points $u, v \in \Upsilon$ with $u \leftrightarrow v$ that satisfy the following conditions:
\begin{itemize}
    \item $\widetilde D_\Upsilon(u, v) \le M_1 D_\Upsilon(u, v)$. 
    \item $\lvert u - v\rvert$ is bounded below by a constant times $r$. 
    \item There is a $\widetilde D_\Upsilon$-geodesic from $u$ to $v$ that is contained in this ``test ball'', and there are two connected arcs of $\Gamma$ --- an ``inner'' one and an ``outer'' one --- such that this $\widetilde D_\Upsilon$-geodesic lies between them. 
\end{itemize}
The event $\fE_r$ will correspond, roughly speaking, to the event that many of the ``test balls'' succeed. The random open subset $U_r$ will then be the union of long narrow ``tubes'' that link the ``test balls'' that succeed, and $V_r$ will be a slightly smaller open subset of $U_r$. (Thus, the number of possibilities for $(U_r, V_r)$ is finite.) Moreover, the ``good'' pairs of points described above, together with the $\widetilde D_\Upsilon$-geodesics connecting them, will not intersect $U_r$. The event $\fF_r$ will correspond, roughly speaking, to the event that the following are true:
\begin{itemize}
    \item We have $(U_r, V_r) = (U_r^\prime, V_r^\prime)$. In this case, $\Upsilon_r$ is a resampling of $\Upsilon$ within $U_r$, i.e., $(\CC \setminus \overline{U_r}) \cap \Upsilon_r = (\CC \setminus \overline{U_r}) \cap \Upsilon$. 
    \item The ``inner'' connected arcs of $\Gamma$ mentioned above are linked together to form a single loop $\SCL_r$ of $\Gamma_r$. The ``outer'' connected arcs of $\Gamma$ mentioned above are linked together to form two connected arcs $\gamma_r$ and $\gamma_r^\prime$ of $\Gamma_r$. Thus, the ``good'' pairs of points described above lie in the long narrow ``tube'' between $\SCL_r$, $\gamma_r$, and $\gamma_r^\prime$. 
    \item The $D_{\Upsilon_r}$-distances inside the intersection of the ``tube'' between $\SCL_r$, $\gamma_r$, and $\gamma_r^\prime$ with each connected component of $U_r$ will be very small.  
\end{itemize}
The existence of the times $\sigma$ and $\tau$ of \eqref{eq:093} will be an immediate consequence of the definitions of $\fE_r$ and $\fF_r$.

\begin{figure}[ht!]
    \centering
    \includegraphics[width=0.8\linewidth]{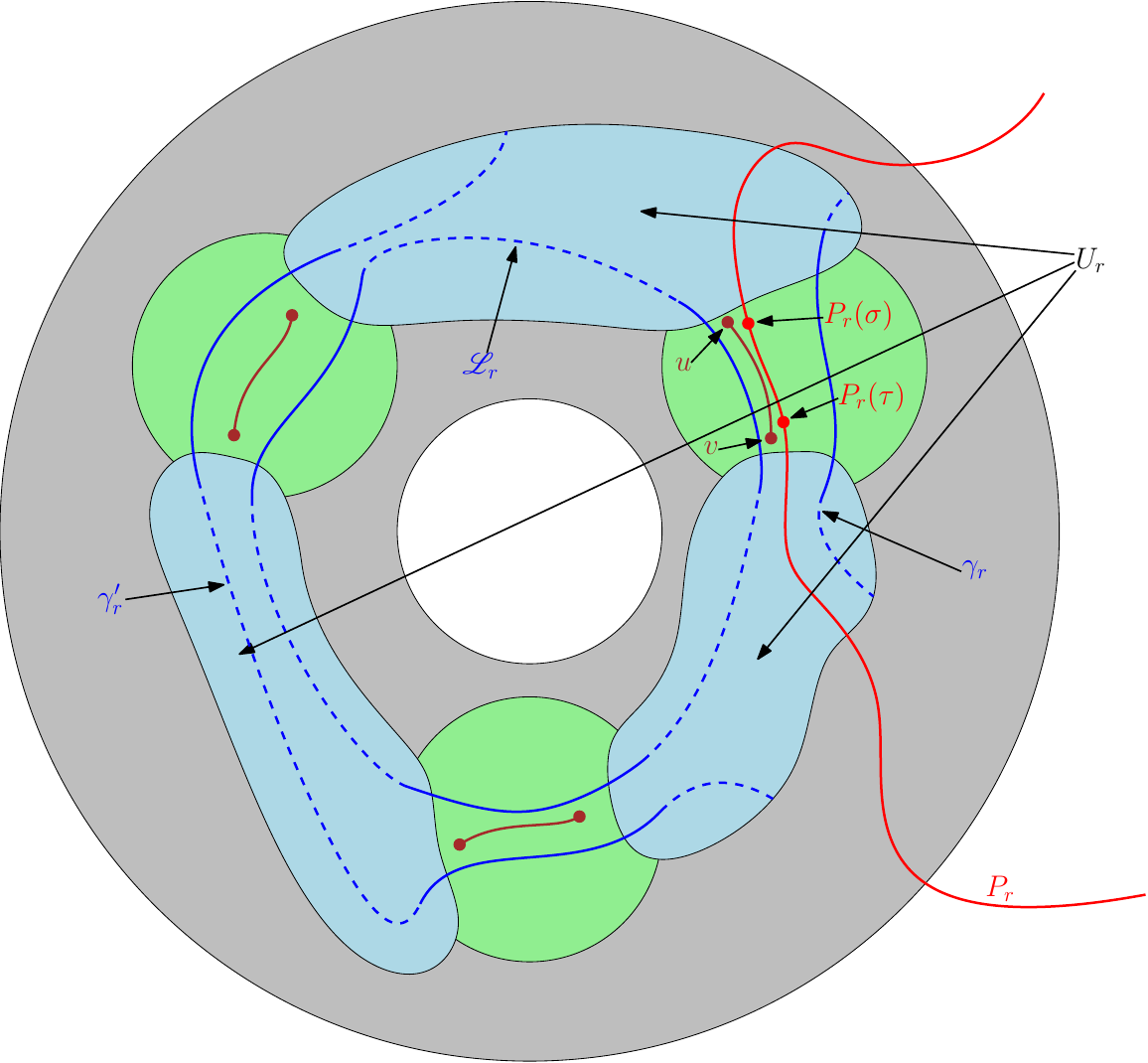}
    \caption{Illustration of the objects constructed in this section. The light green balls are the ``test balls'' that succeed, i.e., each of the light green balls contains a ``good'' pair of points $u$ and $v$, together with the $\widetilde D_\Upsilon$-geodesic connecting them (brown), and two connected arcs of $\Gamma$ (blue and solid). These connected arcs of $\Gamma$ are ``rewired'' in $\Gamma_r$ (blue and dashed) to form $\SCL_r$, $\gamma_r$, and $\gamma_r^\prime$. The red path $P_r$ is a $D_{\Upsilon_r}$-geodesic that passes though the ``tube'' between $\SCL_r$, $\gamma_r$, and $\gamma_r^\prime$. The times $\sigma$ and $\tau$ of \eqref{eq:093} are so that $P_r(\sigma)$ is close to $u$ and $P_r(\tau)$ is close to $v$.}
    \label{fig:section7}
\end{figure}

The main ideas in this section are broadly similar to those in \cite[Section~5]{ExUniLQG} and \cite[Section~5]{UniCriSupercriLQGMet}, which carry out an analogous construction for the Liouville quantum gravity (LQG) metric. However, there are two key differences. First, since a $D_{\Upsilon_r}$-geodesic cannot cross the loops of $\Gamma_r$, once it enters the ``tube'' bounded by $\SCL_r$, $\gamma_r$, and $\gamma_r'$, it must remain within this tube. In the LQG setting, by contrast, additional work is needed to show that a geodesic entering the tube will spend a long time there. Second, the LQG metric enjoys the Weyl scaling property, which precisely describes how the metric changes when a smooth function is added to the underlying field. This allows one to subtract a large bump function supported on the tube to ``attract'' the geodesic into it. In our case, the absence of the Weyl scaling property introduces significant new challenges. Roughly speaking, our method of attracting the geodesic relies instead on removing Brownian loops from the underlying Brownian loop-soup (cf.~\Cref{010}, Axiom~\eqref{010E} (monotonicity)).

We now provide a more detailed overview of our construction.

In \Cref{ss:05}, we will introduce an event involving a single ``good'' pair of points $u$ and $v$. A precise definition is given in \Cref{099}, and an illustration is provided in \Cref{fig:G}. Specifically, we require that $u \in \partial B_{(1 - \alpha)r}(0) \cap \Upsilon$, $v \in \partial B_r(0) \cap \Upsilon$, $u \leftrightarrow v$, $\widetilde D_\Upsilon(u, v) \le M_1 D_\Upsilon(u, v)$, and that the $\widetilde D_\Upsilon$-geodesic from $u$ to $v$ is contained in $\fH_r \cap \overline{A_{(1 - \alpha)r,r}(0)} \cap \Upsilon$, where $\fH_r \subset A_{(1 - a)r,(1 + a)r}(0)$ is a \emph{deterministic} half-annulus. We also require that there exist two connected arcs of $\Gamma$ in $\fH_r$ joining $\partial B_{\alpha r}(\fx_r) \cap \partial B_{(1 - a)r}(0)$ to $\partial B_{\alpha r}(\fy_r) \cap \partial B_{(1 + a)r}(0)$ such that the $\widetilde D_\Upsilon$-geodesic from $u$ to $v$ lies between them, where $\fx_r \in \partial B_{(1 - a)r}(0)$ and $\fy_r \in \partial B_{(1 + a)r}(0)$ are \emph{deterministic}. This deterministic condition is imposed so that, later, when we define the long narrow ``tube'' $U_{\rho^{-1}r}$ with $B_{\alpha r}(\fx_r)$ and $B_{\alpha r}(\fy_r)$ serving as the ``entry'' and ``exit'' of $U_{\rho^{-1}r}$, respectively, the number of possible choices for $U_{\rho^{-1}r}$ remains finite. These two connected arcs of $\Gamma$ will be obtained using the coupling between $\CLE_\kappa$ and the Brownian loop-soup by adding Brownian loops. Furthermore, we require that the ``tube'' between these arcs has ``bottlenecks'' at $u$ and $v$, which will allow us to upper-bound the $\widetilde D_\Upsilon$-distance from a $D_\Upsilon$-geodesic to $u$ (resp.~$v$), once we have forced the geodesic to pass through this ``tube''. 

This event is closely related to the event $\underline H_r(M_1, a, \alpha)$ defined in \Cref{003}. Using the results from \Cref{s:03}, we will show that for each
\begin{equation*}
    r \in \SR \defeq \left\{r \in \{\Lambda^k\}_{k \in \ZZ} : \BP\lbrack\underline H_r(M_1, a, \alpha)\rbrack \ge p\right\},
\end{equation*}
the probability of this event is bounded below by a deterministic constant $\fp \in (0, 1)$. We then define $\fG_{z,r}$, for $z \in \CC$ and $r \in \SR$, to be the translated version of the event in \Cref{099}, so that we work with annuli centered at $z$ rather than at the origin (cf.~\Cref{097}).

In \Cref{ss:15}, we will give a careful definition of the tubes that we will force the geodesic to pass through.

In \Cref{ss:06}, we will introduce several positive parameters 
\begin{equation*}
    1 \gg \fa_1 \gg \fa_2 \gg \fa_3 \gg \fa_4 \gg \fa_5 \gg \fa_6 \gg \fa_7 \gg \fa_8 \gg \fa_9,
\end{equation*}
to be chosen later, including the parameter $\rho \defeq \fa_6 \in (0, 1)$ mentioned above. We then define the random open subsets $U_r$ and $V_r$, the resampling $\Upsilon_r$ of $\Upsilon$, and the events $\fE_r$ and $\fF_r$ for $r \in \rho^{-1}\SR$ in terms of these parameters. More precisely:
\begin{itemize}
    \item The set $\fZ_r \subset \partial B_{3r}(0)$ (the collection of ``test points'') will be a deterministic finite subset, and
    \begin{equation*}
        Z_r \defeq \{z \in \fZ_r : \fG_{z,\rho r} \text{ occurs}\}. 
    \end{equation*}
    \item The open subset $V_r$ will be defined as the union of finitely many long, narrow ``tubes'' linking the half-annuli $\fH_{z,\rho r}$ for $z \in Z_r$ into an annular region, in a way that is almost surely determined by $Z_r$. The set $U_r$ will be a small Euclidean neighborhood of $V_r$. The sets $U_r$ and $V_r$ are chosen so that, for each $z \in Z_r$, the deterministic balls $B_{\alpha\rho r}(\fx_{z,\rho r})$ and $B_{\alpha\rho r}(\fy_{z,\rho r})$ introduced earlier are contained in $V_r$, while the points $u, v \in \fH_{z,\rho r}$ and the $\widetilde D_\Upsilon$-geodesic connecting them remain disjoint from $U_r$.
    \item Let $(U_r^\prime, V_r^\prime)$ be sampled uniformly from all possible pairs $(U_r, V_r)$, independently of everything else. Write $\Gamma(\CC \setminus \overline{V_r^\prime}, \CC \setminus U_r^\prime)$ for the union of the $(\CC \setminus \overline{V_r^\prime}, \CC \setminus U_r^\prime)$-excursions of $\Gamma$ and the loops of $\Gamma$ contained in $\CC \setminus \overline{V_r^\prime}$ that intersect $\CC \setminus U_r^\prime$. Then $\Upsilon_r$ is defined as a conditionally independent copy of $\Upsilon$ given $\Gamma(\CC \setminus \overline{V_r^\prime}, \CC \setminus U_r^\prime)$ and $\wp$. The use of $(U_r^\prime, V_r^\prime)$ ensures that $\Upsilon_r$ also has the law of a whole-plane nested $\CLE_\kappa$. Ultimately, we will focus on the case $(U_r, V_r) = (U_r^\prime, V_r^\prime)$.
    \item The event $\fE_r$ will include the condition that $\fG_{z,\rho r}$ occurs for many points $z \in \fZ_r$, along with several additional regularity requirements. The event $\fF_r$ will include the condition $(U_r, V_r) = (U_r^\prime, V_r^\prime)$, a condition on the link pattern induced by the complementary $(\CC \setminus \overline{V_r}, \CC \setminus U_r)$-excursions of $\Gamma_r$ (see the blue and dashed arcs in \Cref{fig:section7}), and a condition roughly stating that the $D_{\Upsilon_r}$-distances inside the ``tubes'' between these blue and dashed arcs are very small. The event $\fE_r$ will also require that $\BP[\fF_r \mid \fE_r,\, \Upsilon] \ge \fa_9$. Note that $\fE_r$ depends only on $\Upsilon$, while $\fF_r$ depends on $\Upsilon$, $\Upsilon_r$, and $(U_r^\prime, V_r^\prime)$. 
\end{itemize}

In \Cref{ss:08}, we will show that we can choose the parameters of \Cref{ss:06} in such a way that $\BP\lbrack\fE_r\rbrack \ge \pp$. In \Cref{ss:07}, we will check that our objects satisfy \eqref{eq:093}. 

\subsection{Existence of a shortcut with positive probability}\label{ss:05}

In this subsection, we will prove that for each $r \in \SR$ (cf.~\eqref{eq:003}), there is a uniformly positive probability that a ``good'' pair of points $u, v \in \overline{B_r(0)} \cap \Upsilon$ exists with $u \leftrightarrow v$, such that $\widetilde D_\Upsilon(u, v) \le M_1 D_\Upsilon(u, v)$ and certain additional regularity conditions are satisfied. In later subsections, we will use the long-range independence property of CLE (cf.~\Cref{265}) to show that, with high probability, many such pairs of points occur within an annulus at a larger scale.

We shall refer to as a \emph{half-annulus} an open subset of the form
\begin{multline*}
    \left\{x \in A_{s,t}(z) : \Re(x) > \Re(z)\right\}, \quad \left\{x \in A_{s,t}(z) : \Im(x) > \Im(z)\right\}, \quad \left\{x \in A_{s,t}(z) : \Re(x) < \Re(z)\right\}, \\
    \text{or} \quad \left\{x \in A_{s,t}(z) : \Im(x) < \Im(z)\right\}, 
\end{multline*}
where $z \in \CC$ and $0 < s < t$.

Recall that we let $\underline p = p(\mu, \nu, \Lambda)$ and $a = a(\mu, \nu, \Lambda)$ be as in \Cref{030}.

\begin{lemma}\label{099}
    There exists $\alpha \in (0, \lambda a]$ and $\fp \in (0, 1)$ such that for each 
    \begin{equation}\label{eq:003}
        r \in \SR \defeq \left\{r \in \{\Lambda^k\}_{k \in \ZZ} : \BP\lbrack\underline H_r(M_1, a, \alpha)\rbrack \ge p\right\},
    \end{equation}
    there are
    \begin{itemize}
        \item a deterministic half-annulus $\fH_r$ of inner (resp.~outer) radius $(1 - a)r$ (resp.~$(1 + a)r$) centered at $0$, and
        \item deterministic points 
        \begin{equation}\label{eq:273}
            \fx_r \in \left\{(1 - a)r\re^{\ri k} : k \in \lambda\alpha\ZZ\right\} \cap \partial\fH_r \quad \text{and} \quad \fy_r \in \left\{(1 + a)r\re^{\ri k} : k \in \lambda\alpha\ZZ\right\} \cap \partial\fH_r
        \end{equation}
    \end{itemize}
    such that it holds with probability at least $\fp$ that there exists $u \in \fH_r \cap \partial B_{(1 - \alpha)r}(0) \cap \Upsilon$ and $v \in \fH_r \cap \partial B_r(0) \cap \Upsilon$ satisfying the following conditions:
    \begin{enumerate}
        \item\label{099C} There is a $\widetilde D_\Upsilon$-geodesic $\widetilde P$ from $u$ to $v$ that is contained in $\fH_r \cap \overline{A_{(1 - \alpha)r,r}(0)} \cap \Upsilon$.
        \item\label{099E} $\widetilde D_\Upsilon(u, v) \ge \alpha^3\kc_r$ and $\widetilde D_\Upsilon(u, v) \le M_1 D_\Upsilon(u, v)$.
        \item\label{099A} There are two connected arcs $\eta_\rL$ and $\eta_\rR$ of $\Gamma$ in $\fH_r$ from $B_{\alpha r}(\fx_r) \cap \partial B_{(1 - a)r}(0)$ to $B_{\alpha r}(\fy_r) \cap \partial B_{(1 + a)r}(0)$ such that
        \begin{itemize}
            \item $\eta_\rL$ lies to the left of $\eta_\rR$, 
            \item there is no other connected arc of $\Gamma$ in $\fH_r$ from $B_{\alpha r}(\fx_r) \cap \partial B_{(1 - a)r}(0)$ to $B_{\alpha r}(\fy_r) \cap \partial B_{(1 + a)r}(0)$ that lies between $\eta_\rL$ and $\eta_\rR$, 
            \item the right side of $\eta_\rL$ and the left side of $\eta_\rR$ are contained in $\Upsilon$, and
            \item $\widetilde P$ lies between $\eta_\rL$ and $\eta_\rR$.
        \end{itemize}
        \item\label{099B} Write $O_r$ for the connected component of $\fH_r \setminus (\eta_\rL \cup \eta_\rR)$ containing $\widetilde P$. Then 
        \begin{equation*}
            \widetilde D_\Upsilon(\partial B_{(1 - a)r}(0) \cap \partial O_r \cap \Upsilon, \partial B_{(1 + a)r}(0) \cap \partial O_r \cap \Upsilon; O_r \cap \Upsilon) \le 2a^{-1}\kc_r. 
        \end{equation*}
        \item\label{099G} There are paths $P_s$ and $P_t$ in $\overline{\fH_r} \cap \Upsilon$ from $\eta_\rL$ to $\eta_\rR$ that have $\widetilde D_\Upsilon$-lengths at most $\lambda\alpha^3\kc_r$, and there are times $0 < s < t < 1$ such that $\widetilde P(s) \in P_s$, $\widetilde P(t) \in P_t$, and $\widetilde D_\Upsilon(u, \widetilde P(s)) \vee \widetilde D_\Upsilon(\widetilde P(t), v) \le \lambda\alpha^3\kc_r$.
        \item\label{099H} We have $D_\Upsilon(u, v) \le \lambda D_\Upsilon(u, \partial A_{(1 - a/2)r,(1 + a/2)r}(0) \cap \Upsilon)$.
    \end{enumerate}
\end{lemma}

\begin{figure}[ht!]
    \centering
    \includegraphics[width=0.8\linewidth]{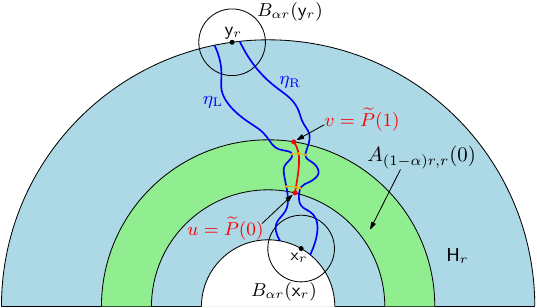}
    \caption{Illustration of the objects involved in \Cref{099}. $\fH_r$ is a deterministic half-annulus.  Both $\fx_r \in \partial B_{(1 - a)r}(0) \cap \partial\fH_r$ and $\fy_r \in \partial B_{(1 + a)r}(0) \cap \partial\fH_r$ are deterministic points. The blue paths $\eta_\rL$ and $\eta_\rR$ are connected arcs of $\Gamma$ in $\fH_r$ from $B_{\alpha r}(\fx_r) \cap \partial B_{(1 - a)r}(0)$ to $B_{\alpha r}(\fy_r) \cap \partial B_{(1 + a)r}(0)$. The red path $\widetilde P$ is a $\widetilde D_\Upsilon$-geodesic in $\fH_r \cap \overline{A_{(1 - \alpha)r,r}(0)} \cap \Upsilon$ from $u \in \partial B_{(1 - \alpha)r}(0) \cap \Upsilon$ to $v \in \partial B_r(0) \cap \Upsilon$ that lies between $\eta_\rL$ and $\eta_\rR$. The pair $(u, v)$ is a ``shortcut'' in the sense that $\widetilde D_\Upsilon(u, v) \le M_1 D_\Upsilon(u, v)$. There are times $0 < s < t < 1$ and an orange path $P_s$ (resp.~$P_t$) in $\Upsilon$ connecting $\eta_\rL$, $\widetilde P(s)$, and $\eta_\rR$ (resp.~$\eta_\rL$, $\widetilde P(t)$, and $\eta_\rR$). The $\widetilde D_\Upsilon$-lengths of $P_s$, $P_t$, $\widetilde P|_{[0, s]}$, and $\widetilde P|_{[t, 1]}$ are much smaller than the $\widetilde D_\Upsilon$-length of $\widetilde P$. This implies that every path in $\Upsilon$ that passes between $\eta_\rL$ and $\eta_\rR$ will have to get close to the ``shortcut''.}
    \label{fig:G}
\end{figure}

An illustration of the objects appearing in \Cref{099} can be found in \Cref{fig:G}. The rest of this subsection is primarily concerned with establishing \Cref{099}. Before presenting the proof itself, we begin by explaining the motivation behind the formulation of the statement.

In \Cref{ss:06}, we will ``link up'' the half-annuli $\fH_{\rho r} + z$ for varying choices of $z \in \partial B_{3r}(0)$ with long narrow ``tubes''. If we label these $z$'s in counterclockwise order as $z_1, z_2, \ldots$, then the long narrow ``tube'' linking $\fH_{\rho r} + z_j$ and $\fH_{\rho r} + z_{j + 1}$ is given by the Euclidean neighborhood of a deterministic smooth simple path connecting $\fy_{\rho r} + z_j$ and $\fx_{\rho r} + z_{j + 1}$. We want there to be only finitely many possibilities for the set $r^{-1}U_r$, which allows us to get certain estimates trivially by taking a maximum over the possibilities. This is why we require that $\fH_r$ is either vertical or horizontal and the points $\fx_r$ and $\fy_r$ belong to the finite sets in \eqref{eq:273}. We will then resample the configurations inside $U_r$. The requirement that $\widetilde P \subset \fH_r \cap \overline{A_{(1 - \alpha)r,r}(0)} \cap \Upsilon$ in condition~\eqref{099C}, together with condition~\eqref{099H}, ensures that the condition that $\widetilde P$ is a $\widetilde D_\Upsilon$-geodesic and $\widetilde D_\Upsilon(u, v) \le M_1 D_\Upsilon(u, v)$ is preserved after resampling. Condition~\eqref{099H} also ensures that the event in the lemma statement depends locally on $\Upsilon$ (cf.~\Cref{100}). Condition~\eqref{099B} is needed to upper-bound the amount of time a $D_\Upsilon$-geodesic needs to spend to pass through the the ``tube'' between $\eta_\rL$ and $\eta_\rR$. The purpose of condition~\eqref{099G} is to ensure the existence of ``bottlenecks'' at $u$ and $v$, so that any path in $\Upsilon$ that passes through the ``tube'' between $\eta_\rL$ and $\eta_\rR$ will have to get close to the ``shortcut''. 

We now proceed to the proof of \Cref{099}.

\begin{figure}[ht!]
    \centering
    \includegraphics[width=0.8\linewidth]{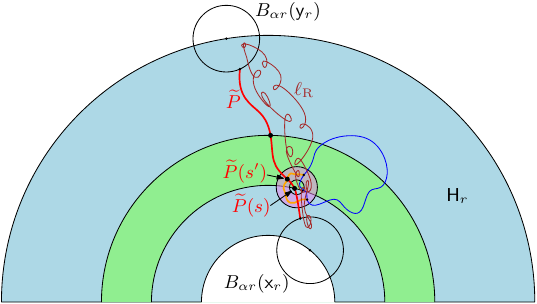}
    \caption{Illustration of the proof of \Cref{099}. The red path $\widetilde P$ is obtained by deforming the $\widetilde D_\Upsilon$-geodesic described in the definition of $\underline H_r(M_1, a, \alpha)$ (cf.~\Cref{003}) relative to $\widetilde P|_{[s, t]}$. There exists a violet annulus $A_{\delta r,\delta^\omega r}(u)$ and an orange path $P_u$ such that $\widetilde P(s) \in B_{\delta r}(u)$ and every path in $\Upsilon$ that crosses between the inner and outer boundaries of $A_{\delta r,\delta^\omega r}(u)$ intersects $P_u$. The time $s^\prime$ denotes the first time after $s$ at which $\widetilde P$ hits $P_u$. Similarly, we have $A_{\delta r,\delta^\omega r}(v)$ and $P_v$ (not shown). The brown loop $\ell_\rR$ is a Brownian loop in $\fH_r$ that intersects $\partial B_{\alpha r}(\fx_r)$, $\partial B_{\alpha r}(\fy_r)$, $\partial B_{\delta r}(u)$, and $\partial B_{\delta r}(v)$, and lies to the right of $\widetilde P$. Similarly, we have $\ell_\rL$ (not shown). After adding these two Brownian loops, the existence of $\eta_\rL$ and $\eta_\rR$ follows from the coupling between $\CLE_\kappa$ and Brownian loop-soup (cf.~\cite{CLE}). We remark that $\widetilde P|_{[s, t]}$, $s^\prime$, and $t^\prime$ in the proof will play the role of $\widetilde P$, $s$, and $t$ in the lemma statement, respectively.}
    \label{fig:Gproof}
\end{figure}

\begin{proof}[Proof of \Cref{099}]
    See \Cref{fig:Gproof} for an illustration. By \Cref{264}, \eqref{264A} and a union bound, we may choose $\delta \in (0, \lambda a]$ to be sufficiently small so that for each $r > 0$, it holds with probability at least $1 - \underline p/2$ that the following is true. 
    \begin{equation}\label{eq:315}
        \parbox{.85\linewidth}{For each $z \in \{r\re^{\ri k} : k \in \lambda\delta\ZZ \cap [0, 2\pi)\}$, there are at most two connected arcs of $\Gamma$ in $A_{\delta r,a r/2}(z)$ connecting its inner and outer boundaries.}
    \end{equation}
    Moreover, by \Cref{010}, Axiom~\eqref{010D} (tightness across scales) and possibly decreasing $\delta$, we may arrange that for each $r > 0$, it holds with probability at least $1 - \underline p/4$ that the following is true. 
    \begin{equation}\label{eq:247}
        D_\Upsilon(x, y) \le \lambda D_\Upsilon(y, \partial B_{a r/2}(y) \cap \Upsilon), \quad \forall x, y \in \overline{B_r(0)} \cap \Upsilon \text{ with } x \leftrightarrow y \text{ and } \lvert x - y\rvert \le \frac12\delta r. 
    \end{equation}
    By the scale invariance of the law of $\Upsilon$, we may choose $\alpha \in (0, \lambda\delta]$ to be sufficiently small so that for each $r > 0$, it holds with probability at least $1 - \underline p/8$ that the following is true. 
    \begin{equation}\label{eq:246}
        \parbox{.90\linewidth}{The Euclidean diameter of each connected component of $A_{(1 - \alpha)r,r}(0) \cap \Upsilon$ is at most $\delta r/2$.}
    \end{equation}
    Let $\alpha_{\mathrm{4A}}$ be as in \eqref{eq:016}. Fix $\omega \in (0, 1)$ such that $\alpha_{\mathrm{4A}}(1 - \omega) > 1$. Then, by \Cref{264}, \eqref{264B}, \Cref{263}, \eqref{eq:340}, and a union bound, we may choose a sufficiently small $\delta \in (0, \lambda\alpha]$ and a sufficiently large $A > 0$ such that for each $r > 0$, it holds with probability at least $1 - \underline p/16$ that the following is true. 
    \begin{equation}\label{eq:316}
        \parbox{.85\linewidth}{For each $z \in \{(1 - \alpha)r\re^{\ri k} : k \in \lambda\delta\ZZ \cap [0, 2\pi)\} \cup \{r\re^{\ri k} : k \in \lambda\delta\ZZ \cap [0, 2\pi)\}$, there are at most two connected arcs of $\Gamma$ in $A_{\delta r,\delta^\omega r}(z)$ connecting its inner and outer boundaries, and there is a path in $A_{\delta r,\delta^\omega r}(z) \cap \Upsilon$ that has $\widetilde D_\Upsilon$-length at most $A\delta^\omega\kc_r$ such that every path in $\Upsilon$ that crosses between the inner and outer boundaries of $A_{\delta r,\delta^\omega r}(z)$ intersects $P$.}
    \end{equation}
    Moreover, by possibly decreasing $\delta$, we may arrange so that $A\delta^\omega \le \lambda\alpha^3$. 

    By the definition of the event $\underline H_r(M_1, a, \alpha)$ (cf.~\Cref{003}) and the definition of $\SR$ (cf.~\eqref{eq:003}), for each $r \in \SR$, it holds with probability at least $\underline p$ that the following is true. 
    \begin{equation}\label{eq:274}
        \parbox{.85\linewidth}{There exists a $\widetilde D_\Upsilon$-geodesic $\widetilde P \colon [0, 1] \to \overline{A_{(1 - a)r,(1 + a)r}(0)} \cap \Upsilon$ from $\partial B_{(1 - a)r}(0) \cap \Upsilon$ to $\partial B_{(1 + a)r}(0) \cap \Upsilon$ and times $0 < s < t < 1$ satisfying conditions~\eqref{003BA}, \eqref{003BB}, \eqref{003BC}, and \eqref{003BD} of \Cref{003}.}
    \end{equation}
    Thus, we conclude that for each $r \in \SR$, it holds with probability at least $\underline p/16$ that \eqref{eq:315}, \eqref{eq:247}, \eqref{eq:246}, \eqref{eq:316}, and \eqref{eq:274} are true. Since the Euclidean diameter of $\widetilde P$ is at most $r/100$ (cf.~condition~\eqref{003BB}), there is a half-annulus $H_r$ of inner (resp.~outer) radius $(1 - a)r$ (resp.~$(1 + a)r$) centered at $0$, and points 
    \begin{equation*}
        x_r \in \left\{(1 - a)r\re^{\ri k} : k \in \lambda\alpha\ZZ\right\} \cap \partial H_r \quad \text{and} \quad y_r \in \left\{(1 + a)r\re^{\ri k} : k \in \lambda\alpha\ZZ\right\} \cap \partial H_r
    \end{equation*}
    such that $\widetilde P \subset \overline{H_r} \cap \Upsilon$, $\widetilde P(0) \in B_{\alpha r}(x_r) \cap \partial B_{(1 - a)r}(0) \cap \Upsilon$, and $\widetilde P(1) \in B_{\alpha r}(y_r) \cap \partial B_{(1 + a)r}(0) \cap \Upsilon$. Since the number of possible choices for such $H_r$, $x_r$, $y_r$ is at most a deterministic constant not depending on $r$, it follows that there exists $p^\prime \in (0, 1)$ such that for each $r \in \SR$, there exists a deterministic choice $\fH_r$, $\fx_r$, $\fy_r$ of $H_r$, $x_r$, $y_r$, resp.~such that it holds with probability at least $p^\prime$ that \eqref{eq:315}, \eqref{eq:247}, \eqref{eq:246}, \eqref{eq:316}, and \eqref{eq:274} are true, $\widetilde P \subset \overline{\fH_r} \cap \Upsilon$, $\widetilde P(0) \in B_{\alpha r}(\fx_r) \cap \partial B_{(1 - a)r}(0) \cap \Upsilon$, and $\widetilde P(1) \in B_{\alpha r}(\fy_r) \cap \partial B_{(1 + a)r}(0) \cap \Upsilon$. 
    
    By \eqref{eq:316}, there exists $u \in \{(1 - \alpha)r\re^{\ri k} : k \in \lambda\delta\ZZ \cap [0, 2\pi)\}$ and $v \in \{r\re^{\ri k} : k \in \lambda\delta\ZZ \cap [0, 2\pi)\}$ such that $\widetilde P(s) \in B_{\delta r}(u)$ and $\widetilde P(t) \in B_{\delta r}(v)$, and there is a path $P_u$ in $A_{\delta r,\delta^\omega r}(u) \cap \Upsilon$ (resp.~$P_v$ in $A_{\delta r,\delta^\omega r}(v) \cap \Upsilon$) that has $\widetilde D_\Upsilon$-length at most $A\delta^\omega\kc_r \le \lambda\alpha^3\kc_r$ such that every path in $\Upsilon$ that crosses between the inner and outer boundaries of $A_{\delta r,\delta^\omega r}(u)$ (resp.~$A_{\delta r,\delta^\omega r}(v)$) intersects $P_u$ (resp.~$P_v$). Write $s^\prime$ for the first time after $s$ at which $\widetilde P$ hits $P_u$, and $t^\prime$ for the last time before $t$ at which $\widetilde P$ hits $P_v$. Since $\widetilde P$ is a $\widetilde D_\Upsilon$-geodesic and also hits $P_u$ before $s$ (resp.~$P_v$ after $t$), it follows that 
    \begin{align*}
        \widetilde D_\Upsilon(\widetilde P(s), \widetilde P(s^\prime)) &\le \len(P_u, \widetilde D_\Upsilon) \le A\delta^\omega\kc_r \le \lambda\alpha^3\kc_r \quad \text{and}  \\
        \widetilde D_\Upsilon(\widetilde P(t^\prime), \widetilde P(t)) &\le \len(P_v, \widetilde D_\Upsilon) \le A\delta^\omega\kc_r \le \lambda\alpha^3\kc_r. 
    \end{align*}
    Moreover, there are at most two connected arcs of $\Gamma$ in $A_{\delta r,\delta^\omega r}(u)$ (resp.~$A_{\delta r,\delta^\omega r}(v)$) connecting its inner and outer boundaries. If there are exactly two such arcs, write $\SCL_s$ (resp.~$\SCL_t$) for the loop of $\Gamma$ containing them; otherwise write $\SCL_s \defeq \emptyset$ (resp.~$\SCL_t \defeq \emptyset$). (At this point, we remark that, eventually, $\widetilde P|_{[s, t]}$, $s^\prime$, and $t^\prime$ will play the role of $\widetilde P$, $s$, and $t$ in the lemma statement, respectively.)
    
    We observe that for each $\varepsilon > 0$, it is always possible to deform $\widetilde P$ relative to $\widetilde P|_{[s, t]}$ so that we obtain a path (not necessarily a $\widetilde D_\Upsilon$-geodesic) $\widetilde P^\prime \colon [0, 1] \to \overline{\fH_r} \cap \Upsilon$ from $B_{\alpha r}(\fx_r) \cap \partial B_{(1 - a)r}(0) \cap \Upsilon$ to $B_{\alpha r}(\fx_r) \cap \partial B_{(1 + a)r}(0) \cap \Upsilon$ satisfying the following conditions:
    \begin{itemize}
        \item $\widetilde P^\prime|_{[s, t]} \equiv \widetilde P|_{[s, t]}$. 
        \item $\lvert\widetilde P^\prime(\tau) - \widetilde P(\tau)\rvert \le \varepsilon r$ for all $\tau \in [0, 1]$. 
        \item $\len(\widetilde P^\prime; \widetilde D_\Upsilon) \le (1 + \varepsilon)\len(\widetilde P; \widetilde D_\Upsilon)$. 
        \item Every loop of $\Gamma$ that touches $\widetilde P^\prime|_{[0, s]}$ or $\widetilde P^\prime|_{[t, 1]}$ has Euclidean diameter at most $\varepsilon r$. 
    \end{itemize}
    We note that conditions~\eqref{003BC} and \eqref{003BD} of \Cref{003} are preserved when replacing $\widetilde P$ with $\widetilde P^\prime$, and it follows from condition~\eqref{003BA} that $\len(\widetilde P^\prime, \widetilde D_\Upsilon) \le 2a^{-1}\kc_r$. Write $\SCL$ for the smallest loop surrounding $\widetilde P$ (hence also $\widetilde P^\prime$). Write $\mu^{\mathrm{loop}}$ for the Brownian loop measure (cf.~\cite{BLS}). Write $A_\rL$ (resp.~$A_\rR$) for the collection of Brownian loops $\ell$ satisfying the following conditions:
    \begin{enumerate}[label=(\alph*), ref=\alph*]
        \item\label{it:213} $\ell$ is surrounded by $\SCL$ and contained in $\fH_r$. 
        \item\label{it:214} $\ell$ intersects $\partial B_{\alpha r}(\fx_r)$, $\partial B_{\alpha r}(\fy_r)$, $\partial B_{\delta r}(u)$, and $\partial B_{\delta r}(v)$. 
        \item\label{it:210} $\ell$ lies to the left (resp.~right) of $\widetilde P^\prime$. 
        \item\label{it:211} $\ell \cap B_{\varepsilon r}(\widetilde P^\prime) = \emptyset$. 
        \item\label{it:209} $\ell \cap \SCL_s \subset B_{\delta^\omega r}(u)$, and $\ell \cap \SCL_s \neq \emptyset$ if and only if $\SCL_s \neq \emptyset$ and crosses between the inner and outer boundaries of $A_{\delta r,\delta^\omega r}(u)$ on the left (resp.~right) side of $\widetilde P$.
        \item\label{it:215} $\ell \cap \SCL_t \subset B_{\delta^\omega r}(v)$, and $\ell \cap \SCL_t \neq \emptyset$ if and only if $\SCL_t \neq \emptyset$ and crosses between the inner and outer boundaries of $A_{\delta r,\delta^\omega r}(v)$ on the left (resp.~right) side of $\widetilde P$. 
        \item\label{it:212} Every loop of $\Gamma$ that intersects $\ell$ except possibly $\SCL_s$ and $\SCL_t$ has Euclidean diameter at most $\varepsilon r/100$. 
    \end{enumerate}
    We observe that, almost surely, $\mu^{\mathrm{loop}}(A_\rL) \wedge \mu^{\mathrm{loop}}(A_\rR) > 0$ whenever $\varepsilon$ is sufficiently small. Thus, by the scale invariance of the law of $\Upsilon$ and replacing $\widetilde P$ with $\widetilde P^\prime$, we conclude that there exists a sufficiently small $\varepsilon > 0$ such that for each $r \in \SR$, it holds with probability at least $p^\prime/2$ that \eqref{eq:315}, \eqref{eq:247}, \eqref{eq:246}, \eqref{eq:316}, together with the following \eqref{eq:276}, are true. 
    \begin{equation}\label{eq:276}
        \parbox{.85\linewidth}{There is a path (not necessarily a $\widetilde D_\Upsilon$-geodesic) $\widetilde P \colon [0, 1] \to \overline{\fH_r} \cap \Upsilon$ from $B_{\alpha r}(\fx_r) \cap \partial B_{(1 - a)r}(0) \cap \Upsilon$ to $B_{\alpha r}(\fx_r) \cap \partial B_{(1 + a)r}(0) \cap \Upsilon$ with $\len(\widetilde P, \widetilde D_\Upsilon) \le 2a^{-1}\kc_r$ and times $0 < s < t < 1$ satisfying conditions~\eqref{003BC} and \eqref{003BD} of \Cref{003}. Moreover, every loop of $\Gamma$ that touches $\widetilde P|_{[0, s]}$ or $\widetilde P|_{[t, 1]}$ has Euclidean diameter at most $\varepsilon r$, and if we define $A_\rL$ and $A_\rR$ in the same manner as above but with $\widetilde P$ in place of $\widetilde P^\prime$, then $\mu^{\mathrm{loop}}(A_\rL) \wedge \mu^{\mathrm{loop}}(A_\rR) > \varepsilon$.}
    \end{equation}
    
    We \emph{claim} that for each $\ell_\rL \in A_\rL$ and $\ell_\rR \in A_\rR$, there is no loop of $\Gamma$ that intersects both $\ell_\rL$ and $\ell_\rR$. First, we observe that $\SCL_s$ (resp.~$\SCL_t$) cannot intersect both $\ell_\rL$ and $\ell_\rR$. Indeed, this follows immediately from conditions~\eqref{it:209}, \eqref{it:215}, together with the fact that $\SCL_s$ (resp.~$\SCL_t$) can only cross between the inner and outer boundaries of $A_{\delta r,\delta^\omega r}(u)$ on either the left or the right side of $\widetilde P$, but not both (otherwise there would be at least four connected arcs of $\Gamma$ in $A_{\delta r,\delta^\omega r}(u)$ (resp.~$A_{\delta r,\delta^\omega r}(v)$) connecting its inner and outer boundaries). Now, since $\ell_\rL \cap B_{\varepsilon r}(\widetilde P) = \emptyset$ and $\ell_\rL \cap B_{\varepsilon r}(\widetilde P) = \emptyset$ (cf.~condition~\eqref{it:211}) and $\ell_\rL$ and $\ell_\rR$ lie to the different sides of $\widetilde P$, it is easy to see that $\dist(\ell_\rL, \ell_\rR) > \varepsilon r/100$, which implies that every loop of $\Gamma$ with Euclidean diameter at most $\varepsilon r/100$ cannot intersect both $\ell_\rL$ and $\ell_\rR$. Thus, by condition~\eqref{it:212}, we complete the proof of the \emph{claim}. 

    Let $\Xi$ be a Brownian loop-soup of intensity $c(\kappa)$ inside $\SCL$. Suppose that $\Gamma$ and $\Xi$ are coupled so that the outermost loops of $\Gamma$ inside $\SCL$ are given by the outer boundaries of the outermost clusters of the Brownian loops of $\Xi$ \cite{CLE}. Given $A_\rL$ (resp.~$A_\rR$), let $\ell_\rL$ (resp.~$\ell_\rR$) be sampled uniformly from $\mu^{\mathrm{loop}}|_{A_\rL}$ (resp.~$\mu^{\mathrm{loop}}|_{A_\rR}$), independently of everything else. Write $\Xi^\prime \defeq \Xi \cup \{\ell_\rL, \ell_\rR\}$. Then, (by, e.g., Mecke's formula) on the event that $\mu^{\mathrm{loop}}(A_\rL) \wedge \mu^{\mathrm{loop}}(A_\rR) > \varepsilon$, the law of $\Xi^\prime$ is absolutely continuous with respect to the law of $\Xi$ and the Radon-Nikodym derivative is bounded above by a deterministic constant.  
    
    Since for each $r \in \SR$, it holds with probability at least $p^\prime/2$ that \eqref{eq:315}, \eqref{eq:247}, \eqref{eq:246}, \eqref{eq:316}, and \eqref{eq:276} are true, it suffices to show that, on this event, conditions~\eqref{099C}, \eqref{099E}, \eqref{099A}, \eqref{099B}, \eqref{099G}, and \eqref{099H} are satisfied for $\Xi^\prime$. By the above \emph{claim} and the fact that $\ell_\rL \cap \widetilde P = \emptyset$ and $\ell_\rR \cap \widetilde P = \emptyset$ (cf.~condition~\eqref{it:211}), after adding $\ell_\rL$ and $\ell_\rR$, the path $\widetilde P$ is still contained in the carpet of the non-nested $\CLE_\kappa$ inside $\SCL$. One verifies immediately that conditions~\eqref{099C} and \eqref{099E} are satisfied with $\widetilde P|_{[s, t]}$, $\widetilde P(s)$, and $\widetilde P(t)$ in place of $\widetilde P$, $u$, and $v$, respectively. Condition~\eqref{099B} follows immediately from the fact that $\len(\widetilde P, \widetilde D_\Upsilon) \le 2a^{-1}\kc_r$. Condition~\eqref{099H} follows immediately from \eqref{eq:247}. By conditions~\eqref{it:213}, \eqref{it:214}, and \eqref{it:210}, it is easy to see that condition~\eqref{099A} is satisfied after adding $\ell_\rL$ and $\ell_\rR$. Thus, it suffices to verify that condition~\eqref{099G} is satisfied for $\Xi^\prime$. Indeed, write $P_s$ (resp.~$P_t$) for the connected component of $P_u \setminus (\ell_\rL \cup \ell_\rR)$ (resp.~$P_v \setminus (\ell_\rL \cup \ell_\rR)$) that contains $\widetilde P(s)$ (resp.~$\widetilde P(t)$). Then 
    \begin{equation*}
        \len(P_s; \widetilde D_\Upsilon) \le \len(P_u; \widetilde D_\Upsilon) \le \lambda\alpha^3\kc_r \quad \text{and} \quad \len(P_t; \widetilde D_\Upsilon) \le \len(P_v; \widetilde D_\Upsilon) \le \lambda\alpha^3\kc_r. 
    \end{equation*}
    Moreover, by conditions~\eqref{it:214}, \eqref{it:209}, and \eqref{it:215}, one verifies immediately that $P_s$ and $P_t$ must connect the arcs $\eta_\rL$ and $\eta_\rR$ of condition~\eqref{099A}. This completes the proof. 
\end{proof}

The event of \Cref{099} will be the ``building block'' for the event $\fE_r$. In the remainder of the present section, let $\alpha$, $\fp$, $\SR$, and $\{(\fH_r, \fx_r, \fy_r) : r \in \SR\}$ be as in \Cref{099}. For $z \in \CC$ and $r \in \SR$, we shall write 
\begin{equation*}
    \fH_{z,r} \defeq \fH_r + z; \quad \fx_{z,r} \defeq \fx_r + z; \quad \fy_{z,r} \defeq \fy_r + z. 
\end{equation*}

\begin{definition}\label{097}
    Let $z \in \CC$ and $r \in \SR$. Then we shall write $\fG_{z,r}$ for the event of \Cref{099} with $\Upsilon - z$ in place of $\Upsilon$, i.e., the event that there exists $u \in \fH_{z,r} \cap \partial B_{(1 - \alpha)r}(z) \cap \Upsilon$ and $v \in \fH_{z,r} \cap \partial B_r(z) \cap \Upsilon$ satisfying the following conditions:
    \begin{enumerate}
        \item\label{097C} There is a $\widetilde D_\Upsilon$-geodesic $\widetilde P$ from $u$ to $v$ that is contained in $\fH_{z,r} \cap \overline{A_{(1 - \alpha)r,r}(z)} \cap \Upsilon$.
        \item\label{097E} $\widetilde D_\Upsilon(u, v) \ge \alpha^3\kc_r$ and $\widetilde D_\Upsilon(u, v) \le M_1 D_\Upsilon(u, v)$.
        \item\label{097A} There are two connected arcs $\eta_\rL$ and $\eta_\rR$ of $\Gamma$ in $\fH_{z,r}$ from $B_{\alpha r}(\fx_{z,r}) \cap \partial B_{(1 - a)r}(z)$ to $B_{\alpha r}(\fy_{z,r}) \cap \partial B_{(1 + a)r}(z)$ such that
        \begin{itemize}
            \item $\eta_\rL$ lies to the left of $\eta_\rR$, 
            \item there is no other connected arc of $\Gamma$ in $\fH_{z,r}$ from $B_{\alpha r}(\fx_{z,r}) \cap \partial B_{(1 - a)r}(z)$ to $B_{\alpha r}(\fy_{z,r}) \cap \partial B_{(1 + a)r}(z)$ that lies between $\eta_\rL$ and $\eta_\rR$, 
            \item the right side of $\eta_\rL$ and the left side of $\eta_\rR$ are contained in $\Upsilon$, and
            \item $\widetilde P$ lies between $\eta_\rL$ and $\eta_\rR$.
        \end{itemize}
        \item\label{097B} Write $O_{z,r}$ for the connected component of $\fH_{z,r} \setminus (\eta_\rL \cup \eta_\rR)$ containing $\widetilde P$. Then 
        \begin{equation*}
            \widetilde D_\Upsilon(\partial B_{(1 - a)r}(z) \cap \partial O_{z,r} \cap \Upsilon, \partial B_{(1 + a)r}(z) \cap \partial O_{z,r} \cap \Upsilon; O_{z,r} \cap \Upsilon) \le 2a^{-1}\kc_r. 
        \end{equation*}
        \item\label{097G} There are paths $P_s$ and $P_t$ in $\overline{\fH_{z,r}} \cap \Upsilon$ from $\eta_\rL$ to $\eta_\rR$ that have $\widetilde D_\Upsilon$-lengths at most $\lambda\alpha^3\kc_r$, and there are times $0 < s < t < 1$ such that $\widetilde P(s) \in P_s$, $\widetilde P(t) \in P_t$, and $\widetilde D_\Upsilon(u, \widetilde P(s)) \vee \widetilde D_\Upsilon(\widetilde P(t), v) \le \lambda\alpha^3\kc_r$.
        \item\label{097H} We have $D_\Upsilon(u, v) \le \lambda D_\Upsilon(u, \partial A_{(1 - a/2)r,(1 + a/2)r}(z) \cap \Upsilon)$.
    \end{enumerate}
\end{definition}

By \Cref{010}, Axiom~\eqref{010C} (translation invariance) and \Cref{099},
\begin{equation}\label{eq:101}
    \BP\lbrack\fG_{z,r}\rbrack \ge \fp, \quad \forall z \in \CC, \ \forall r \in \SR.
\end{equation}

\begin{lemma}\label{100}
    Let $z \in \CC$ and $r \in \SR$. Then the event is almost surely determined by $B_{(1 + a)r}(z) \cap \Upsilon$.
\end{lemma}

\begin{proof}
    It is clear that conditions~\eqref{097A}, \eqref{097B}, \eqref{097G}, and \eqref{097H} are almost surely determined by $B_{(1 + a)r}(z) \cap \Upsilon$. Note that if condition~\eqref{097H} holds, then any path from $u$ to $v$ that exits $A_{(1 - a/2)r,(1 + a/2)r}(z)$ is neither a $D_\Upsilon$-geodesic nor a $\widetilde D_\Upsilon$-geodesic. This implies that, on the event that condition~\eqref{097H} holds, conditions~\eqref{097C} and \eqref{097E} are almost surely determined by $B_{(1 + a)r}(z) \cap \Upsilon$. This completes the proof. 
\end{proof}

\subsection{Definition of the tubes}\label{ss:15}

Recall the parameters $\lambda$, $\pp$, $\fA$, and $a$ from \Cref{ss:14}, and $\alpha$ and $\fp$ from \Cref{ss:05}. The definitions of $U_r$, $V_r$, $\Upsilon_r$, $\fE_r$, and $\fF_r$ will depend on sufficiently small positive parameters
\begin{equation}\label{eq:094}
    1 \gg \fa_1 \gg \fa_2 \gg \fa_3 \gg \fa_4 \gg \fa_5 \gg \fa_6 \gg \fa_7 \gg \fa_8 \gg \fa_9
\end{equation}
which will be chosen later in \Cref{ss:08}. These parameters will be chosen in a manner so that each $\fa_j$ depends only on $\lambda$, $\pp$, $\fA$, $a$, $\alpha$, $\fp$, and $\fa_1, \ldots, \fa_{j - 1}$. To lighten notation, we shall write $\rho \defeq \fa_6$. 

In the present subsection, we define the ``tubes'' $U_r$ and $V_r$. We will then define the events $\fE_r$ and $\fF_r$ in the next subsection.

Fix $r \in \rho^{-1}\SR$. We shall write 
\begin{equation}\label{eq:042}
    \fZ_r \defeq \left\{3r\re^{\ri k} : k \in (\fa_5\ZZ) \cap [0, 2\pi)\right\}. 
\end{equation}
for the set of ``test points''. The event $\fE_r$ will include the condition that the event $\fG_{z,\rho r}$ occurs for ``many'' of the points $z \in \fZ_r$. To quantify this, we shall write
\begin{equation}\label{eq:279}
    Z_r \defeq \left\{z \in \fZ_r : \fG_{z,\rho r} \text{ occurs}\right\}.
\end{equation}

Recall from \Cref{ss:05} the half-annuli $\fH_{z,\rho r}$, and the points $\fx_{z,\rho r}$ and $\fy_{z,\rho r}$. We will now construct the ``tubes'' which link up the half-annuli $\fH_{z,\rho r}$ for $z \in \fZ_r$. Let 
\begin{equation*}
    \left\{\fP_{w_1,w_2} : w_1, w_2 \in \fZ_r, \ w_1 \neq w_2\right\}
\end{equation*}
be a family of deterministic smooth simple paths satisfying the following conditions:
\begin{enumerate}[label=(\Alph*), ref=\Alph*]
    \item\label{it:102} Let $w_1, w_2 \in \fZ_r$ be distinct. Then $\fP_{w_1,w_2}$ is a smooth simple path from $\fy_{w_1,\rho r}$ to $\fx_{w_2,\rho r}$ that does not intersect $\fH_{z,\rho r}$ (except at the endpoints). Moreover, if we write $I_{w_1,w_2}$ for the counterclockwise connected arc of $\partial B_{3r}(0)$ from $w_1$ to $w_2$, then $\fP_{w_1,w_2} \subset B_{100\rho r}(I_{w_1,w_2})$. 
    \item Let $w_1, w_2, w_3 \in \fZ_r$ be distinct points in counterclockwise order. Then the Euclidean distance between $\fP_{w_1,w_2}$ and $\fP_{w_2,w_3}$ is at least $\alpha\rho r$. The Euclidean distance between $\fP_{w_1,w_2}$ and $\fH_{w_3,\rho r}$ is at least $\fa_5r$. 
    \item Let $w_1, w_2, w_3, w_4 \in \fZ_r$ be distinct points in counterclockwise order. Then the Euclidean distance between $\fP_{w_1,w_2}$ and $\fP_{w_3,w_4}$ is at least $\fa_5r$. 
    \item\label{it:103} The set 
    \begin{equation*}
        \left\{r^{-1}\fP_{w_1,w_2} : w_1, w_2 \in \fZ_r, \ w_1 \neq w_2\right\}
    \end{equation*}
    depends only on $\fa_5$, $\rho$, $(\rho r)^{-1}\fH_{\rho r}$, $(\rho r)^{-1}\fx_{\rho r}$, $(\rho r)^{-1}\fy_{\rho r}$ (but not on $r$). Thus, the number of possibilities for this set is at most a constant depending only on $\lambda$, $a$, $\alpha$, $\fa_5$, $\rho$. 
\end{enumerate}
We shall write 
\begin{equation}\label{eq:041}
    \fU_{w_1,w_2} \defeq B_{(1 + \fa_7)\alpha\rho r}(\fP_{w_1,w_2}) \quad \text{and} \quad \fV_{w_1,w_2} \defeq B_{\alpha\rho r}(\fP_{w_1,w_2}), \quad \forall \text{ distinct } w_1, w_2 \in \fZ_r. 
\end{equation}

Let $z_1, \ldots, z_{\#Z_r} \in \partial B_{3r}(0)$ be the points of $Z_r$ in counterclockwise order. Let $j \in [1, \#Z_r]_\ZZ$. Then we shall write 
\begin{equation}\label{eq:044}
    P_j \defeq \fP_{z_j,z_{j + 1}}; \quad U_j \defeq \fU_{z_j,z_{j + 1}}; \quad V_j \defeq \fV_{z_j,z_{j + 1}}, 
\end{equation}
with the convention that $z_{\#Z_r + 1} \defeq z_1$. We shall write (by abuse of notation)
\begin{equation*}
    U_r \defeq \bigcup_{j = 1}^{\#Z_r} U_j \quad \text{and} \quad V_r \defeq \bigcup_{j = 1}^{\#Z_r} V_j. 
\end{equation*}

See \Cref{fig:UVW} for an illustration of $P_j$, $U_j$, and $V_j$. 

\begin{figure}[ht!]
    \centering
    \includegraphics[width=0.8\linewidth]{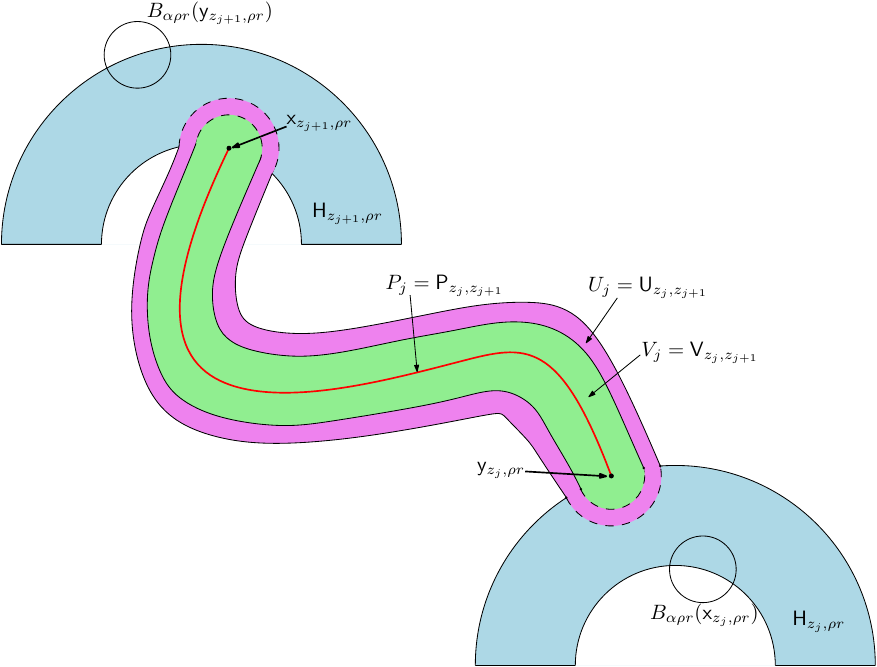}
    \caption{Illustration of $P_j = \fP_{z_{j},z_{j + 1}}$, $U_j = \fU_{z_{j},z_{j + 1}}$, and $V_j = \fV_{z_{j},z_{j + 1}}$, where $z_{j}$ and $z_{j + 1}$ are neighboring points in $Z_r$. The path $P_j$ is a smooth simple path from $\fy_{z_{j},\rho r}$ to $\fx_{z_{j + 1},\rho r}$ that are almost surely determined by $z_j$ and $z_{j + 1}$. The open subset $U_j$ (resp.~$V_j$) is the Euclidean $(1 + \fa_7)\alpha\rho r$- (resp.~$\alpha\rho r$-) neighborhood of $P_j$.}
    \label{fig:UVW}
\end{figure}

\subsection{A prerequisite event and a postrequisite event}\label{ss:06}

In the present subsection, we define the events $\fE_r$ and $\fF_r$. Recall the parameters of \eqref{eq:094}.

Before proceeding, we introduce some notation. See \Cref{fig:linkpattern} for an illustration. Let $r \in \rho^{-1}\SR$ and $j \in [1, \#Z_r]_\ZZ$. Recall the point $z_j$ and the open subsets $U_j$ and $V_j$ from \Cref{ss:15}. We shall write $U_j^\star$ for the connected component containing $V_j$ of the open subset obtained by removing from $U_j$ the closure of the union of all the $(\CC \setminus \overline{V_j}, \CC \setminus U_j)$-excursions of $\Gamma$ and the loops of $\Gamma$ that are contained in $\CC \setminus \overline{V_j}$ and intersect $\CC \setminus U_j$. We shall write $X_j \defeq \partial U_j^\star \cap \partial V_j$ for the collection of the endpoints of the $(\CC \setminus \overline{V_j}, \CC \setminus U_j)$-excursions of $\Gamma$. We shall write
\begin{equation}\label{eq:040}
    U_r^\star \defeq \bigcup_{j = 1}^{\#Z_r} U_j^\star.
\end{equation}
Let $\eta_{j,\rL}$ and $\eta_{j,\rR}$ be as in the definition of $\fG_{z_j,\rho r}$ (cf.~\Cref{097}). We shall write $x_{j,\rL} \in X_{j - 1}$ and $y_{j,\rL} \in X_j$ (resp.~$x_{j,\rR} \in X_{j - 1}$ and $y_{j,\rR} \in X_j$) for the endpoints of $\eta_{j,\rL}$ (resp.~$\eta_{j,\rR}$). Since the right side of $\eta_{j,\rL}$ and the left side of $\eta_{j,\rR}$ are contained in $\Upsilon$, it follows that
\begin{align*}
    &\#\left(X_j \cap [y_{j,\rL}, y_{j,\rR}]_{\partial U_j^\star}^\circlearrowleft\right), \quad \#\left(X_j \cap [y_{j,\rR}, x_{j + 1,\rR}]_{\partial U_j^\star}^\circlearrowleft\right),\\&\#\left(X_j \cap [x_{j + 1,\rR}, x_{j + 1,\rL}]_{\partial U_j^\star}^\circlearrowleft\right), \quad
    \#\left(X_j \cap [x_{j + 1,\rL}, y_{j,\rL}]_{\partial U_j^\star}^\circlearrowleft\right)
\end{align*}
(i.e., the number of boundary marked points of $(U_j^\star; X_j)$ that lie on the counterclockwise arc of $\partial U_j^\star$ from $y_{j,\rL}$ to $y_{j,\rR}$ (resp.~from $y_{j,\rR}$ to $x_{j + 1,\rR}$; from $x_{j + 1,\rR}$ to $x_{j + 1,\rL}$; from $x_{j + 1,\rL}$ to $y_{j,\rL}$)) are all even. Thus, we may label the points of $X_j$ in counterclockwise order as
\begin{multline}\label{eq:347}
    y_{j,\rL}, \ y_{j,1}, \ldots, \ y_{j,2n_j}, \ y_{j,\rR}, \ w_{j,1}, \ \ldots, \ w_{j,2r_j}, \ x_{j + 1,\rR}, \ x_{j + 1,1}, \ \ldots, \ x_{j + 1,2m_j}, \ x_{j + 1,\rL}, \\
    z_{j,1}, \ \ldots, \ z_{j,2l_j}. 
\end{multline}
In this situation, we shall introduce the link patterns $\alpha_j$ and $\beta_j$ (see \Cref{fig:linkpattern} for an illustration).  The link pattern $\alpha_j$ serves to create an opening in the chain of tubes which the geodesic is able to enter and the link pattern $\beta_j$ serves to create a tube which the geodesic cannot leave.  When we perform the resampling procedure later on, we aim to create a chain of tubes where two of them end up with link pattern $\alpha_j$ (so the geodesic has a place to enter and also a place to leave) and the rest end up with link pattern $\beta_j$ (so the geodesic has no option but to traverse the tube once it has entered).  Concretely, we let
\begin{multline}\label{eq:043}
    \alpha_j \defeq \left\{(x_{j + 1,\rL}, y_{j,\rL})\right\} \cup \left\{(z_{j,2k - 1}, z_{j,2k}) : k \in [1, l_j]_\ZZ\right\} \\
    \cup \left\{(y_{j,\rR}, w_{j,1}), (w_{j,2r_j}, x_{j + 1,\rR})\right\} \cup \left\{(w_{j,2k}, w_{j,2k + 1}) : k \in [1, r_j - 1]_\ZZ\right\} \\
    \cup \left\{(y_{j,2k - 1}, y_{j,2k}) : k \in [1, n_j]_\ZZ\right\} \cup \left\{(x_{j + 1,2k - 1}, x_{j + 1,2k}) : k \in [1, m_j]_\ZZ\right\}
\end{multline}
(i.e., $\alpha_j$ is the link pattern on $(U_j^\star; X_j)$ that ``links up'' the endpoints of $\eta_{j,\rL}$ and $\eta_{j + 1,\rL}$, and all other boundary marked points of $(U_j^\star; X_j)$ to their nearest neighbors);
\begin{multline*}
    \beta_j \defeq \left\{(x_{j + 1,\rL}, y_{j,\rL}), (y_{j,\rR}, x_{j + 1,\rR})\right\} \cup \left\{(z_{j,2k - 1}, z_{j,2k}) : k \in [1, l_j]_\ZZ\right\} \cup \left\{(w_{j,2k - 1}, w_{j,2k}) : k \in [1, r_j]_\ZZ\right\} \\
    \cup \left\{(y_{j,2k - 1}, y_{j,2k}) : k \in [1, n_j]_\ZZ\right\} \cup \left\{(x_{j + 1,2k - 1}, x_{j + 1,2k}) : k \in [1, m_j]_\ZZ\right\}
\end{multline*}
(i.e., $\beta_j$ is the link pattern on $(U_j^\star; X_j)$ that ``links up'' the endpoints of $\eta_{j,\rL}$ and $\eta_{j + 1,\rL}$ and the endpoints of $\eta_{j,\rR}$ and $\eta_{j + 1,\rR}$, resp.~and all other boundary marked points of $(U_j^\star; X_j)$ to their nearest neighbors). We shall also introduce collections of counterclockwise arcs of $\partial U_j^\star$ with endpoints contained in $X_j$
\begin{multline*}
    \SI_j \defeq \left\{[y_{j,\rL}, y_{j,1}]_{\partial U_j^\star}^\circlearrowleft, [y_{j,2n_j}, y_{j,\rR}]_{\partial U_j^\star}^\circlearrowleft\right\} \cup \left\{[y_{j,2k}, y_{j,2k + 1}]_{\partial U_j^\star}^\circlearrowleft : k \in [1, n_j - 1]_\ZZ\right\} \\
    \cup \left\{[x_{j + 1,\rR}, x_{j + 1,1}]_{\partial U_j^\star}^\circlearrowleft, [x_{j + 1,2m_j}, x_{j + 1,\rL}]_{\partial U_j^\star}^\circlearrowleft\right\} \cup \left\{[x_{j + 1,2k}, x_{j + 1,2k + 1}]_{\partial U_j^\star}^\circlearrowleft : k \in [1, m_j - 1]_\ZZ\right\} \\
    \cup \left\{[z_{j,2k - 1}, z_{j,2k}]_{\partial U_j^\star}^\circlearrowleft : k \in [1, l_j]_\ZZ\right\} \cup \left\{[w_{j,2k - 1}, w_{j,2k}]_{\partial U_j^\star}^\circlearrowleft : k \in [1, r_j]_\ZZ\right\};
\end{multline*}
(i.e., $\SI_j$ is the collection of all counterclockwise arcs of $\partial U_j^\star$ such that their endpoints are nearest neighbors of $X^j$ and they are contained in $\Upsilon$ when viewed from the inside of $U_j^\star$, this implies that if $P$ is a path in $\Upsilon$ that enters $U_j^\star$ from outside of $U_j^\star$, then the first point at which $P$ enters $U_j^\star$ must belong to $I \cap \partial U_j$ for some $I \in \SI_j$);
\begin{multline*}
    \SI_j^\alpha \defeq \left\{[y_{j,\rL}, y_{j,1}]_{\partial U_j^\star}^\circlearrowleft, [y_{j,2n_j}, y_{j,\rR}]_{\partial U_j^\star}^\circlearrowleft\right\} \cup \left\{[y_{j,2k}, y_{j,2k + 1}]_{\partial U_j^\star}^\circlearrowleft : k \in [1, n_j - 1]_\ZZ\right\} \\
    \cup \left\{[x_{j + 1,\rR}, x_{j + 1,1}]_{\partial U_j^\star}^\circlearrowleft, [x_{j + 1,2m_j}, x_{j + 1,\rL}]_{\partial U_j^\star}^\circlearrowleft\right\} \cup \left\{[x_{j + 1,2k}, x_{j + 1,2k + 1}]_{\partial U_j^\star}^\circlearrowleft : k \in [1, m_j - 1]_\ZZ\right\} \\
    \cup \left\{[w_{j,2k - 1}, w_{j,2k}]_{\partial U_j^\star}^\circlearrowleft : k \in [1, r_j]_\ZZ\right\} \subset \SI_j;
\end{multline*}
(i.e., $\SI_j^\alpha$ is the subcollection of $\SI_j$ consisting of the arcs that are not disconnected from the ``center'' of $U_j^\star$ by the link pattern $\alpha_j$);
\begin{multline}\label{eq:319}
    \SI_j^\beta \defeq \left\{[y_{j,\rL}, y_{j,1}]_{\partial U_j^\star}^\circlearrowleft, [y_{j,2n_j}, y_{j,\rR}]_{\partial U_j^\star}^\circlearrowleft\right\} \cup \left\{[y_{j,2k}, y_{j,2k + 1}]_{\partial U_j^\star}^\circlearrowleft : k \in [1, n_j - 1]_\ZZ\right\} \\
    \cup \left\{[x_{j + 1,\rR}, x_{j + 1,1}]_{\partial U_j^\star}^\circlearrowleft, [x_{j + 1,2m_j}, x_{j + 1,\rL}]_{\partial U_j^\star}^\circlearrowleft\right\} \cup \left\{[x_{j + 1,2k}, x_{j + 1,2k + 1}]_{\partial U_j^\star}^\circlearrowleft : k \in [1, m_j - 1]_\ZZ\right\} \subset \SI_j^\alpha;
\end{multline}
(i.e., $\SI_j^\beta$ is the subcollection of $\SI_j$ consisting of the arcs that are not disconnected from the ``center'' of $U_j^\star$ by the link pattern $\beta_j$);
\begin{equation}\label{eq:102}
    \SI_j^{\alpha \setminus \beta} \defeq \left\{[w_{j,2k - 1}, w_{j,2k}]_{\partial U_j^\star}^\circlearrowleft : k \in [1, r_j]_\ZZ\right\} = \SI_j^\alpha \setminus \SI_j^\beta. 
\end{equation}

\begin{figure}[ht!]
    \centering
    \includegraphics[width=.8\linewidth]{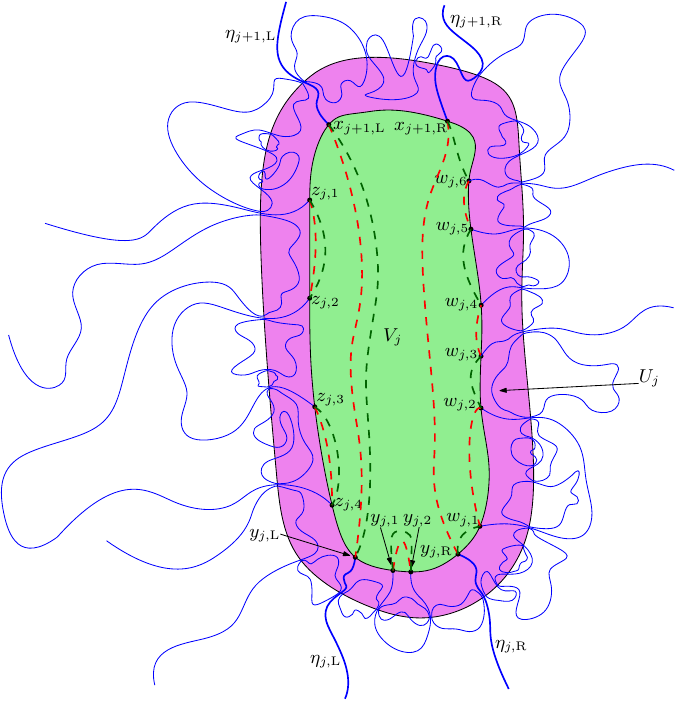}
    \caption{Illustration of $U_j$, $V_j$, $X_j$, and the labeling of $X_j$ as in \eqref{eq:347} for $j \in [1, \#Z_r]_\ZZ$. The blue paths are (connected arcs of) loops of $\Gamma$. The dark green dashed paths induce the link pattern $\alpha_j$, and the red dashed paths induce the link pattern $\beta_j$.  The purpose of the link pattern $\alpha_j$ is to create an opening that the geodesic can either enter or leave (and we aim to have two of these in the chain of tubes) and the purpose of the link pattern $\beta_j$ is to make it impossible for the geodesic to leave.}
    \label{fig:linkpattern}
\end{figure}

We will now perform a resampling operation. Let $(U_r^\prime, V_r^\prime)$ be sampled uniformly from all the possibilities for $(U_r, V_r)$, independently of everything else. We shall write $\Gamma(\CC \setminus \overline{V_r^\prime}, \CC \setminus U_r^\prime)$ for the union of the collection of the $(\CC \setminus \overline{V_r^\prime}, \CC \setminus U_r^\prime)$-excursions of $\Gamma$ and the collection of the loops of $\Gamma$ that are contained in $\CC \setminus \overline{V_r^\prime}$ and intersect $\CC \setminus U_r^\prime$. Let $\Upsilon_r$ be a conditionally independent copy of $\Upsilon$ given $\Gamma(\CC \setminus \overline{V_r^\prime}, \CC \setminus U_r^\prime)$ and $\wp$. (In particular, we have $(\CC \setminus \overline{U_r^\prime}) \cap \Upsilon_r = (\CC \setminus \overline{U_r^\prime}) \cap \Upsilon$ almost surely.) The purpose of the use of $(U_r^\prime, V_r^\prime)$ is to ensure that $\Upsilon_r$ also has the law of a whole-plane nested $\CLE_\kappa$. We will ultimately only be interested in the case where $(U_r, V_r) = (U_r^\prime, V_r^\prime)$. It follows from \eqref{it:103} that the number of possibilities for $r^{-1}(U_r, V_r)$ is at most a deterministic constant not depending on $r$. Thus, the probability that $(U_r, V_r) = (U_r^\prime, V_r^\prime)$ is uniformly positive. 

For a simply connected domain $U \subset \CC$, we shall write
\begin{equation*}
    d_U(x, y) \defeq \inf\left\{\diam(X) : X \text{ is a connected subset of } \overline U \text{ with } x, y \in X\right\}, \quad \forall x, y \in \overline U. 
\end{equation*}

We will now define the events $\fE_r$ and $\fF_r$. 

\begin{definition}\label{098}
    Let $r \in \rho^{-1}\SR$. Then we shall write $\fE_r$ for the event that the following conditions are satisfied:
    \begin{enumerate}
        \item\label{098A} We have
        \begin{equation*}
            \min\bigl(D_\Upsilon(\partial B_r(0) \cap \Upsilon, \partial B_{2r}(0) \cap \Upsilon), D_\Upsilon(\partial B_{4r}(0) \cap \Upsilon, \partial B_{5r}(0) \cap \Upsilon)\bigr) \ge \fa_1\kc_r. 
        \end{equation*}
        \item\label{098B} Let $x, y \in A_{2r,4r}(0) \cap \Upsilon$ with $\lvert x - y\rvert \le \fa_3r$. Then either $x \xleftrightarrow{A_{r,5r}(0)} y$ or there exists a loop of $\Gamma$ contained in $A_{r,5r}(0)$ and disconnecting $x$ and $y$. Moreover, in the former case, we have
        \begin{equation*}
            D_\Upsilon(x, y; A_{r,5r}(0) \cap \Upsilon) \le \fa_2\kc_r. 
        \end{equation*}
        \item\label{098C} Let $x, y \in A_{2r,4r}(0) \cap \Upsilon$ with $\lvert x - y\rvert \le 100\fa_4r$. Then either $x \xleftrightarrow{A_{r,5r}(0)} y$ or there exists a loop of $\Gamma$ with Euclidean diameter at most $\fa_3r/2$, contained in $A_{r,5r}(0)$, and disconnecting $x$ and $y$. 
        \item\label{098D} For each connected arc $I \subset \partial B_{3r}(0)$ of Euclidean length at least $\fa_4r$, we have $I \cap Z_r \neq \emptyset$. 
        \item\label{098E} There exists a finite subset $\CS \subset A_{2r,4r}(0) \times (0, \lambda r]$ such that the following are true:
        \begin{enumerate}
            \item For each connected component $\Upsilon_\bullet \subset \Upsilon$ with $(\CC \setminus B_{4r}(0)) \cap \Upsilon_\bullet \neq \emptyset$ and each connected component $\SCC \subset (U_r \setminus V_r) \cap \Upsilon_\bullet$, there exists $(z, t) \in \CS$ such that $\SCC \subset B_t(z)$. 
            \item There exists $N \in \NN$ such that for each $(z, t) \in \CS$, there are at most $2N$ connected arcs of $\Gamma$ in $A_{t,2t}(z)$ connecting its inner and outer boundaries. 
            \item We have
            \begin{equation*}
                N\sum_{(z, t) \in \CS} \sup\left\{D_\Upsilon(x, y; A_{r,5r}(0) \cap \Upsilon) : x, y \in B_{2t}(z), \ x \xleftrightarrow{B_{2t}(z)} y\right\} \le \fa_2\kc_r. 
            \end{equation*}
        \end{enumerate}
        \item\label{098F} For each $j \in [1, \#Z_r]_\ZZ$, the following are true:
        \begin{enumerate}
            \item There are at most $1/\fa_8$ connected arcs of $\Gamma$ in $U_j \setminus V_j$ connecting $\partial U_j$ and $\partial V_j$. 
            \item Each pair of distinct points of $X_j$ has Euclidean distance at least $\fa_8r$.  
            \item For each $I \in \SI_j$, $I \cap \partial U_j$ and $P_j$ are contained in the same connected component of $U_j^\star \setminus \SCB_{100\fa_8r}(\partial U_j^\star \setminus I; d_{U_j^\star})$, where $\SCB_{100\fa_8r}(\partial U_j^\star \setminus I; d_{U_j^\star}) \subset U_j^\star$ denotes the $100\fa_8r$-neighborhood of $\partial U_j^\star \setminus I$ with respect to the metric $d_{U_j^\star}$. (Here, we recall the definition of $P_j$ from \eqref{eq:044}.)
        \end{enumerate}
        \item\label{098G} For each subset $J_r \subset [1, \#Z_r]_\ZZ$ with $\#J_r = 2$, the following is true: Given $\Upsilon$ and the event that condition~\eqref{098F} holds, the following conditions~\eqref{035H}, \eqref{035I}, and \eqref{035K} of \Cref{035} hold with conditional probability at least $\fa_9$. (Here, we observe that condition~\eqref{098F} depends only on $\Gamma(\CC \setminus \overline{V_r}, \CC \setminus U_r)$ and $\wp$. Thus, if condition~\eqref{098F} holds for $\Upsilon$, then it also holds for $\Upsilon_r$.)
    \end{enumerate}
\end{definition}

\begin{definition}\label{035}
    Let $J_r \subset [1, \#Z_r]_\ZZ$ be a subset with $\#J_r = 2$ (chosen in a manner that is almost surely determined by $\Upsilon$). Then we shall write $\fF_r(J_r)$ for the event that the following conditions are satisfied:
    \begin{enumerate}
        \setcounter{enumi}7
        \item\label{035J} We have $(U_r, V_r) = (U_r^\prime, V_r^\prime)$.
        \item\label{035H} For each $j \in J_r$ (resp.~$j \in [1, \#Z_r]_\ZZ \setminus J_r$), the link pattern induced by the complementary $(\CC \setminus \overline{V_j}, \CC \setminus U_j)$-excursions of $\Gamma_r$ is given by $\alpha_j$ (resp.~$\beta_j$). 
        \item\label{035I} For each $j \in J_r$ and $(x, y) \in \alpha_j$ (resp.~$j \in [1, \#Z_r]_\ZZ \setminus J_r$ and $(x, y) \in \beta_j$), the complementary $(\CC \setminus \overline{V_j}, \CC \setminus U_j)$-excursion of $\Gamma_r$ connecting $x$ and $y$ is contained in $\SCB_{\fa_8r}([x, y]_{\partial U_j^\star}^\circlearrowleft; d_{U_j^\star})$.  
        \item\label{035K} We have
        \begin{equation*}
            D_{\Upsilon_r}(x_1, x_2; A_{r,5r}(0) \cap \Upsilon_r) \le \rho\kc_r, \quad \forall j \in J_r, \ \forall I_1, I_2 \in \SI_j^\alpha, \ \forall x_1 \in I_1 \cap \partial U_j, \ \forall x_2 \in I_2 \cap \partial U_j. 
        \end{equation*}
        and
        \begin{multline*}
            D_{\Upsilon_r}(x_1, x_2; A_{r,5r}(0) \cap \Upsilon_r) \le \rho\kc_r, \\
            \forall j \in [1, \#Z_r]_\ZZ \setminus J_r, \ \forall I_1, I_2 \in \SI_j^\beta, \ \forall x_1 \in I_1 \cap \partial U_j, \ \forall x_2 \in I_2 \cap \partial U_j. 
        \end{multline*}
    \end{enumerate}
\end{definition}

We note that, by condition~\eqref{098G}, for each $r \in \rho^{-1}\SR$ and each $J_r \subset [1, \#Z_r]_\ZZ$ with $\#J_r = 2$ chosen in a manner that is almost surely determined by $\Upsilon$, we have
\begin{equation}\label{eq:286}
    \BP\!\left\lbrack\fF_r(J_r) \ \middle\vert \ \fE_r, \ \Upsilon\right\rbrack \ge \fa_9 \quad \text{almost surely}. 
\end{equation}

We next explain the roles of the different conditions appearing in \Cref{098,035}. Condition~\eqref{098A} occurs with high probability due to \Cref{010}, Axiom~\eqref{010D} (tightness across scales). This condition can be used to lower-bound the amount of time a $D_\Upsilon$-geodesic between points outside of $B_{5r}(0)$ has to spend in $A_{r,5r}(0)$ if it enters $B_r(0)$. Condition~\eqref{098B} occurs with high probability due to \Cref{019}. Condition~\eqref{098B} together with condition~\eqref{098A} implies that if a $D_\Upsilon$-geodesic between points outside of $B_{5r}(0)$ enters $B_r(0)$, then the points at which it enters and exits $B_{3r}(0)$ cannot be too close to each other. Condition~\eqref{098C} occurs with high probability due to a similar reason to the proof of \Cref{019}. Condition~\eqref{098E} occurs with high probability due to \Cref{323}. Conditions~\eqref{098C} and \eqref{098E} will be used to show that a $D_\Upsilon$-geodesic cannot spend a long time in $U_r \setminus V_r$ (cf.~\Cref{318}). Condition~\eqref{098D} is in some sense the most important condition in the definition of $\fE_r$. Due to the definition of $\fG_{z,\rho r}$ from \Cref{ss:05}, this condition provides a large collection of ``good'' pairs of points $u, v \in A_{2r,4r}(0)$ such that $\widetilde D_\Upsilon(u, v) \le M_1 D_\Upsilon(u, v)$. The fact that we consider the event $\fG_{z,\rho r}$ in this condition is the reason why we need to require that $r \in \rho^{-1}\SR$. We will apply \Cref{265} to say that condition~\eqref{098D} occurs with high probability. Condition~\eqref{098F} occurs with high probability due the scale invariance of the law of $\Upsilon$. The purpose of condition~\eqref{098F} is to ensure that conditions~\eqref{035H}, \eqref{035I}, \eqref{035K} occur with uniformly positive conditional probability (i.e., condition~\eqref{098G} occurs). The purpose of conditions~\eqref{035H}, \eqref{035I}, \eqref{035K} is to ensure that the ``tubes'' $U_j$ for $j \in [1, \#Z_r]_\ZZ$ are ``penetrated'', and it takes a $D_{\Upsilon_r}$-geodesic a small amount of time to run though them. The ``tubes'' $U_j$ for $j \in J_r$ (i.e., the ``tubes'' $U_j$ for which the corresponding link pattern is given by $\alpha_j$ (cf.~condition~\eqref{035H})) will be the place where the $D_{\Upsilon_r}$-geodesic enters and exits. Note that in the definition of $\alpha_j$ (cf.~\eqref{eq:043}), the endpoints of $\eta_{j,\rR}$ and $\eta_{j + 1,\rR}$ are not ``linked up''. This will allow the $D_{\Upsilon_r}$-geodesic to enter and exit the ``tube''. While in the ``tubes'' $U_j$ for $j \in [1, \#Z_r]_\ZZ \setminus J_r$ (i.e., the ``tubes'' $U_j$ for which the corresponding link pattern is given by $\beta_j$), the the complementary $(\CC \setminus \overline{V_j}, \CC \setminus U_j)$-excursions of $\Gamma_r$ that ``link up'' the endpoints of $\eta_{j,\rL}$ and $\eta_{j + 1,\rL}$ and the endpoints of $\eta_{j,\rR}$ and $\eta_{j + 1,\rR}$, resp.~will be the ``wall'' that prevents the $D_{\Upsilon_r}$-geodesic from exiting the ``tube'' (see \Cref{ss:07} for more details).

\subsection{Properties of the events}\label{ss:08}

We begin by verifying that $\fE_r$ satisfies the required measurability property.

\begin{lemma}
    For each $r \in \rho^{-1}\SR$, the event $\fE_r$ is almost surely determined by $A_{r,5r}(0) \cap \Upsilon$. 
\end{lemma}

\begin{proof}
    It follows immediately from the definition that conditions~\eqref{098A}, \eqref{098B}, \eqref{098C}, \eqref{098E}, \eqref{098F}, \eqref{098G} are almost surely determined by $A_{r,5r}(0) \cap \Upsilon$. It follows from \Cref{100} that condition~\eqref{098D} is almost surely determined by $A_{r,5r}(0) \cap \Upsilon$. 
\end{proof}

The remainder of the present subsection is devoted to proving the following.

\begin{lemma}\label{104}
    One may choose the parameters of \eqref{eq:094} in such a way that
    \begin{equation*}
        \BP\lbrack\fE_r\rbrack \ge \pp, \quad \forall r \in \rho^{-1}\SR. 
    \end{equation*}
\end{lemma}

To establish \Cref{104}, we will address the conditions of \Cref{098} sequentially.

\begin{lemma}\label{132}
    One may choose the parameters $\fa_1$, $\fa_2$, $\fa_3$, and $\fa_4$ of \eqref{eq:094} in such a way that for each $r > 0$, the probability of each of conditions~\eqref{098A}, \eqref{098B}, and \eqref{098C} of \Cref{098} is at least $1 - (1 - \pp)/100$. 
\end{lemma}

\begin{proof}
    By \Cref{010}, Axiom~\eqref{010D} (tightness across scales), if $\fa_1 \in (0, 1)$ is chosen to be sufficiently small, then for each $r > 0$, condition~\eqref{098A} holds with probability at least $1 - (1 - \pp)/100$. Choose $\fa_2 \in (0, \fa_1)$ to be sufficiently small ($\fa_2 \defeq \lambda\fA^{-1}\fa_1^2$ would be sufficient for later use). By \Cref{019}, if $\fa_3 \in (0, \fa_2)$ is chosen to be sufficiently small, then for each $r > 0$, condition~\eqref{098B} holds with probability at least $1 - (1 - \pp)/100$. Next, we consider condition~\eqref{098C}. By \Cref{264}, \eqref{264A} and a union bound, if $\fa_4 \in (0, \fa_3)$ is chosen to be sufficiently small, then for each $r > 0$, it holds with probability at least $1 - (1 - \pp)/100$ for each $z \in (\fa_4r\ZZ)^2 \cap B_{4r}(0)$, there are at most two connected arcs of $\Gamma$ in $A_{200\fa_4r,\fa_3r/4}(z)$ connecting its inner and outer boundaries. Suppose that this event occurs. Let $x, y \in A_{2r,4r}(0) \cap \Upsilon$ with $\lvert x - y\rvert \le 100\fa_4r$. Choose $z \in (\fa_4r\ZZ)^2 \cap B_{4r}(0)$ such that $x, y \in B_{200\fa_4r}(z)$. Suppose by way of contradiction that $x$ and $y$ do not lie in the same connected component of $B_{\fa_3r/4}(z) \cap \Upsilon$. Then either there is a loop of $\Gamma$ contained in $B_{\fa_3r/4}(z)$ and disconnecting $x$ and $y$, or there are at least two $B_{\fa_3r/4}(z)$-excursions of $\Gamma$ disconnecting $x$ and $y$ in $B_{\fa_3r/4}(z)$. However, if the latter is true, since $x, y \in B_{200\fa_4r}(z)$, there would be at least four connected arcs of $\Gamma$ in $A_{200\fa_4r,\fa_3r/4}(z)$ connecting its inner and outer boundaries, a contradiction. Thus, we conclude that the former is true. This completes the proof. 
\end{proof}

Henceforth fix $\fa_1$, $\fa_2$, $\fa_3$, and $\fa_4$ as in \Cref{132}. We next consider condition~\eqref{098D}. 

\begin{lemma}\label{133}
    One may choose the parameters $\fa_5$ and $\rho$ of \eqref{eq:094} in such a way that for each $r \in \rho^{-1}\SR$, the probability of condition~\eqref{098D} of \Cref{098} is at least $1 - (1 - \pp)/100$. 
\end{lemma}

\begin{proof}
    For each $r \in \rho^{-1}\SR$, we may choose a deterministic collection $\SJ_r$ of connected arcs of $\partial B_{3r}(0)$ satisfying the following conditions:
    \begin{enumerate}
        \item $\#\SJ_r$ is a constant depending only on $\fa_4$. 
        \item Each arc $J \in \SJ_r$ has Euclidean length $\fa_4 r/2$. 
        \item Each connected arc $I \subset \partial B_{3r}(0)$ of Euclidean length at least $\fa_4 r$ contains come $J \in \SJ_r$. 
    \end{enumerate}
    Then, by \Cref{265}, \eqref{eq:101}, and a union bound, if $\fa_5 \in (0, \fa_4)$ and $\rho \in (0, \fa_5)$ are chosen to be sufficiently small, then for each $r \in \rho^{-1}\SR$, it holds with probability at least $1 - (1 - \pp)/100$ that $J \cap Z_r \neq \emptyset$ for all $J \in \SJ_r$, in which case condition~\eqref{098D} holds. 
\end{proof}

Henceforth fix $\fa_5$ and $\rho$ as in \Cref{133}. We next deal with conditions~\eqref{098E} and \eqref{098F}. 

\begin{lemma}\label{134}
    One may choose the parameters $\fa_7$ and $\fa_8$ of \eqref{eq:094} in such a way that for each $r \in \rho^{-1}\SR$, the probability of each of conditions~\eqref{098E} and \eqref{098F} of \Cref{098} is at least $1 - (1 - \pp)/100$. 
\end{lemma}

\begin{proof}
    Let $\zeta$ be as in \Cref{323}. Then, by \Cref{323}, for each $r > 0$, it holds with probability tending to one as $\fa_7 \to 0$, at a rate which is uniform in $r$, that there exists a subset $Z \subset \left(\frac1{100}\fa_7^{1/2}\alpha\rho r\ZZ\right)^2 \cap A_{2r,4r}(0)$ with $\#Z \le \fa_7^{-\zeta/2}$ such that for each connected component $\Upsilon_\bullet \subset \Upsilon$ with $(\CC \setminus A_{2r,4r}(0)) \cap \Upsilon_\bullet \neq \emptyset$ and each connected component $\SCC \subset (U_r \setminus V_r) \cap \Upsilon_\bullet$, there exists $z \in Z$ such that $\SCC \subset B_{\fa_7^{1/2}\alpha\rho r}(z)$. Fix $\xi \in (\zeta, 1)$. By \Cref{263,214} and a union bound, we may choose a sufficiently large $A > 0$ such that for each $r > 0$, it holds with probability tending to one as $\fa_7 \to 0$, at a rate which is uniform in $r$, that for each $z \in \left(\frac1{100}\fa_7^{1/2}\alpha\rho r\ZZ\right)^2 \cap B_{4r}(0)$, there exists $t_z \in [\fa_7^{1/2}\alpha\rho r, \fa_7^{\xi/2}\alpha\rho r] \cap \{2^k\}_{k \in \ZZ}$ such that there are at most $2A$ connected arcs of $\Gamma$ in $A_{t_z,2t_z}(z)$ connecting its inner and outer boundaries, and
    \begin{equation*}
        D_\Upsilon(x, y; A_{r,5r}(0) \cap \Upsilon) \le A\kc_{t_z}, \quad \forall x, y \in B_{2t_z}(z) \text{ with } x \xleftrightarrow{B_{2t_z}(z)} y.
    \end{equation*}
    Assume that the above two events occur. Write $\CS \defeq \{(z, t_z) : z \in Z\}$. Then, by \eqref{eq:340}, we have
    \begin{equation*}
        A\sum_{(z, t) \in \CS} \sup\left\{D_\Upsilon(x, y; A_{r,5r}(0) \cap \Upsilon) : x, y \in B_{2t}(z), \ x \xleftrightarrow{B_{2t}(z)} y\right\} \le A\fa_7^{-\zeta/2}A\kK\fa_7^{\xi/2}\alpha\rho\kc_r
    \end{equation*}
    Since $\zeta < \xi$, we conclude that, if $\fa_7 \in (0, \rho)$ is chosen to be sufficiently small, then for each $r > 0$, condition~\eqref{098E} holds with probability at least $1 - (1 - \pp)/100$. By the scale invariance of the law of $\Gamma$, together with the fact that the collection of all possibilities for $r^{-1}(U_r, V_r)$ is a finite set not depending on $r$, if $\fa_8 \in (0, \fa_7)$ is chosen to be sufficiently small, then condition~\eqref{098F} holds with probability at least $1 - (1 - \pp)/100$. This completes the proof.  
\end{proof}

Henceforth fix $\fa_7$ and $\fa_8$ as in \Cref{133}. It remains to deal with condition~\eqref{098G}. Let us first recall the well-known FKG inequality for Poisson point processes. 

\begin{lemma}\label{020}
    Let $X$ be a locally compact Hausdorff space. Write $\SM(X)$ for the Prokhorov space of Radon measures on $X$. Let $F_1, F_2 \colon \SM(X) \to \RR$ be increasing measurable mappings, i.e., for any $\mu, \nu \in \SM(X)$ such that $\mu(A) \le \nu(A)$ for all Borel subsets $A \subset X$, we have $F_1(\mu) \le F_1(\nu)$ and $F_2(\mu) \le F_2(\nu)$. Let $\Xi$ be a Poisson point process on $X$. Then 
    \begin{equation*}
        \BE\lbrack F_1(\Xi)F_2(\Xi)\rbrack \ge \BE\lbrack F_1(\Xi)\rbrack\BE\lbrack F_2(\Xi)\rbrack. 
    \end{equation*}
\end{lemma}

\begin{proof}
    See \cite{MonoPP}. 
\end{proof}

\begin{lemma}\label{135}
    One may choose the parameter $\fa_9$ of \eqref{eq:094} in such a way that for each $r \in \rho^{-1}\SR$, the probability of condition~\eqref{098G} of \Cref{098} is at least $1 - (1 - \pp)/10$. 
\end{lemma}

\begin{proof}
    Since the number of possibilities for $(U_r, V_r)$ is at most a deterministic constant not depending on $r$, it follows from \Cref{271,012} that there exists $\fa \in (0, 1)$ such that conditions~\eqref{035J}, \eqref{035H}, and \eqref{035I} hold with conditional probability at least $\fa$ given $\Upsilon$ and condition~\eqref{098F}. Given conditions~\eqref{098F}, \eqref{035J}, \eqref{035H}, and \eqref{035I}, write $U_j^\diamond$ for the connected component containing $P_j$ of the open subset obtained by removing from $U_j^\star$ the complementary $(\CC \setminus \overline{V_j}, \CC \setminus U_j)$-excursions of $\Gamma_r$. By the Markov property of CLE, the loops of $\Gamma_r$ contained in $U_j^\diamond$ is conditionally a nested $\CLE_\kappa$ in $U_j^\diamond$. Write $\Gamma_j \subset \Gamma_r$ for the non-nested $\CLE_\kappa$ in $U_j^\diamond$ induced by $\Gamma_r$ and $\Upsilon_j$ for the carpet of $\Gamma_j$. We observe that $\Upsilon_j$ is always contained in $\Upsilon_r$ ($\Upsilon_j$ is a proper subset of a connected component of $\Upsilon_r$). Given $U_j^\diamond$, let $\Xi_j$ be a Brownian loop-soup of intensity $c(\kappa)$ in $U_j^\diamond$. Suppose that $\Gamma_j$ and $\Xi_j$ are coupled so that $\Gamma_j$ is given by the outer boundaries of the outermost clusters of the Brownian loops of $\Xi_j$ (cf.~\cite{CLE}). 

    Fix $\delta > 0$ to be chosen in the next paragraph. We observe that there is almost surely a smooth path $L_j$ of finite Euclidean length and $\varepsilon_\ast > 0$ such that $B_{\varepsilon_\ast r}(L_j) \subset U_j^\diamond$ and $B_{\delta r}(x) \cap L_j \neq \emptyset$ for all $x \in \partial U_j \cap \Upsilon_j$. Thus, we may choose $A = A(\delta) > 0$ to be sufficiently large and $\varepsilon_\ast = \varepsilon_\ast(\delta) > 0$ to be sufficiently small so that it holds with probability at least $1 - (1 - \pp)/99$ that condition~\eqref{098F} occurs, and given $\Upsilon$ and condition~\eqref{098F}, it holds with conditional probability at least $\fa/2$ that conditions~\eqref{035J}, \eqref{035H}, and \eqref{035I} occur and 
    \begin{equation}\label{eq:290}
        \parbox{.85\linewidth}{there exists a smooth path $L_j$ of Euclidean length at most $Ar$ such that $B_{\varepsilon_\ast r}(L_j) \subset U_j^\diamond$ and $B_{\delta r}(x) \cap L_j \neq \emptyset$ for all $x \in \partial U_j \cap \Upsilon_j$.}
    \end{equation}
    
    By \Cref{019}, for each $p \in (0, 1)$, there exists a sufficiently small $\delta = \delta(p) > 0$ such that it holds with probability at least $1 - (1 - \pp)/100$ that it holds with conditional probability at least $p$ given $\Upsilon$ that
    \begin{multline}\label{eq:283}
        D_{\Upsilon_r}(x, y; A_{r,5r}(0) \cap \Upsilon_r) \le \rho\kc_r/100, \\
        \forall x, y \in A_{2r,4r}(0) \cap \Upsilon_r \text{ with } x \xleftrightarrow{A_{2r,4r}(0) \cap \Upsilon_r} y \text{ and } \lvert x - y\rvert \le \delta r. 
    \end{multline}
    We observe that, if \eqref{eq:283} holds, then
    \begin{equation}\label{eq:288}
        D_{\Upsilon_r}(x, y; A_{r,5r}(0) \cap \Upsilon_r) \le \rho\kc_r/100, \quad \forall x \in \partial U_j \cap \Upsilon_j, \ \forall y \in \Upsilon_j \text{ with } \lvert x - y\rvert \le \delta r.
    \end{equation}
    Thus, by choosing $p$ to be sufficiently close to one, we conclude that there exists $\fa^\prime \in (0, 1)$ and $\delta > 0$ such that given $\Upsilon$, the complementary $(\CC \setminus \overline{V_j}, \CC \setminus U_j)$-excursions of $\Gamma_r$, and conditions~\eqref{098F}, \eqref{035J}, \eqref{035H}, \eqref{035I}, and \eqref{eq:290}, it holds with conditional probability at least $\fa^\prime$ that condition~\eqref{eq:288} occurs. (Here, we note that, on the event that condition~\eqref{035J} holds, conditions~\eqref{098F}, \eqref{035H}, \eqref{035I}, and \eqref{eq:290} are almost surely determined by $\Upsilon$ and the complementary $(\CC \setminus \overline{V_j}, \CC \setminus U_j)$-excursions of $\Gamma_r$.) Moreover, we observe that, by \Cref{010}, Axiom~\eqref{010E} (monotonicity), condition~\eqref{eq:288} is decreasing with respect to $\Xi_j$. 

    Next, we observe that for each $\varepsilon > 0$, there exists $\fa^\pprime(\varepsilon) \in (0, 1)$ such that, given $U_j^\diamond$, for each $z$ with $B_{\varepsilon r}(z) \subset U_j^\diamond$, it holds with conditional probability at least $\fa^\pprime(\varepsilon)$ that $\Upsilon_j \cap B_{\varepsilon r/2}(z) \neq \emptyset$ and there is no loop of $\Gamma_j$ that crosses between the inner and outer boundaries of $A_{\varepsilon r/2,\varepsilon r}(z)$. Then, by choosing $A^\prime = A^\prime(\varepsilon) > 0$ to be sufficiently large, we may arrange so that, given $U_j^\diamond$, for each $z$ with $B_{\varepsilon r}(z) \subset U_j^\diamond$, it holds with conditional probability at least $\fa^\pprime(\varepsilon)$ that 
    \begin{equation}\label{eq:289}
        \parbox{.87\linewidth}{$\Upsilon_j \cap B_{\varepsilon r/2}(z) \neq \emptyset$, there is no loop of $\Gamma_j$ that crosses between the inner and outer boundaries of $A_{\varepsilon r/2,\varepsilon r}(z)$, and there is a loop in $A_{\varepsilon/2,\varepsilon}(z) \cap \Upsilon_j$ of $D_{\Upsilon_r}$-length at most $A^\prime\kc_{\varepsilon r}$ disconnecting the inner and outer boundaries of $A_{\varepsilon r/2,\varepsilon r}(z)$.}
    \end{equation}
    Moreover, we observe that, by \Cref{010}, Axiom~\eqref{010E} (monotonicity), condition~\eqref{eq:289} is decreasing with respect to $\Xi_j$. 

    One verifies immediately that for each smooth path $L \subset \CC$ of Euclidean length $\ell$, and each $\varepsilon > 0$, we have
    \begin{equation*}
        \#\left\{z \in \left(\frac1{100}\varepsilon\ZZ\right)^2 : \dist(z, L) \le \varepsilon/2\right\} \le 10^{100}\ell/\varepsilon. 
    \end{equation*}
    Thus, by applying \Cref{020} to $\Xi_j$ and the preceding discussion, it follows that for each $\varepsilon \in (0, \varepsilon_\ast]$, given $\Upsilon$, the complementary $(\CC \setminus \overline{V_j}, \CC \setminus U_j)$-excursions of $\Gamma_r$, and conditions~\eqref{098F}, \eqref{035J}, \eqref{035H}, \eqref{035I}, and \eqref{eq:290}, it holds with conditional probability at least 
    \begin{equation*}
        \fa^\prime\fa^\pprime(\varepsilon)^{10^{100}A/\varepsilon}
    \end{equation*}
    that condition~\eqref{eq:288} occurs, and condition~\eqref{eq:289} occurs for all $z \in \left(\frac1{100}\varepsilon r\ZZ\right)^2$ with $\dist(z, L_j) \le \varepsilon r/2$. Henceforth assume that this event occurs. We note that, for each $x_1, x_2 \in \partial U_j \cap \Upsilon_j$, it is possible to connect $B_{\delta r}(x_1)$ and $B_{\delta r}(x_2)$ by the union of the loops described in condition~\eqref{eq:289} over all $z \in \left(\frac1{100}\varepsilon r\ZZ\right)^2$ with $\dist(z, L_j) \le \varepsilon r/2$. Thus, we conclude that
    \begin{equation*}
        D_{\Upsilon_r}(x_1, x_2; A_{r,5r}(0) \cap \Upsilon_r) \le \rho\kc_r/100 + 10^{100} A A^\prime \kc_{\varepsilon r}/\varepsilon + \rho\kc_r/100, \quad \forall x_1, x_2 \in \partial U_j \cap \Upsilon_j. 
    \end{equation*}
    Finally, by \Cref{280}, we may choose $\varepsilon$ to be sufficiently small so that $10^{100} A A^\prime \kc_{\varepsilon r}/\varepsilon \le \rho\kc_r/2$. This completes the proof. 
\end{proof}

\begin{proof}[Proof of \Cref{104}]
    Combine \Cref{132,133,134,135}. 
\end{proof}

\subsection{Forcing a geodesic to take a shortcut}\label{ss:07}

In the present subsection, fix $r \in \rho^{-1}\SR$ and $J_r \subset [1, \#Z_r]_\ZZ$ with $\#J_r = 2$. To lighten notation, we may assume without loss of generality that $J_r = \{1, j\}$. Suppose that the event $\fE_r \cap \fF_r(J_r)$ occurs. For $k \in [1, \#Z_r]_\ZZ$, let the points of $X_k$ be labeled as in \eqref{eq:347}. We observe that the concatenation of 
\begin{itemize}
    \item $\eta_{2,\rR}$, 
    \item the complementary $(\CC \setminus V_2, \CC \setminus U_2)$-excursion of $\Gamma_r$ from $y_{2,\rR}$ to $x_{3,\rR}$, 
    \item $\eta_{3,\rR}$, 
    \item[$\vdots$]
    \item the complementary $(\CC \setminus V_{j - 1}, \CC \setminus U_{j - 1})$-excursion of $\Gamma_r$ from $y_{j - 1,\rR}$ to $x_{j,\rR}$, and
    \item $\eta_{j,\rR}$
\end{itemize}
is a connected arc of a loop of $\Gamma_r$. We shall write $\gamma_r$ for this connected arc. Similarly, the concatenation of 
\begin{itemize}
    \item $\eta_{j + 1,\rR}$, 
    \item the complementary $(\CC \setminus V_{j + 1}, \CC \setminus U_{j + 1})$-excursion of $\Gamma_r$ from $y_{j + 1,\rR}$ to $x_{j + 2,\rR}$, 
    \item $\eta_{j + 2,\rR}$, 
    \item[$\vdots$]
    \item the complementary $(\CC \setminus V_{\#Z_r}, \CC \setminus U_{\#Z_r})$-excursion of $\Gamma_r$ from $y_{\#Z_r,\rR}$ to $x_{1,\rR}$, and
    \item $\eta_{1,\rR}$
\end{itemize}
is a connected arc of a loop of $\Gamma_r$. We shall write $\gamma_r^\prime$ for this connected arc. Moreover, we observe that the concatenation of 
\begin{itemize}
    \item $\eta_{1,\rL}$, 
    \item the complementary $(\CC \setminus V_1, \CC \setminus U_1)$-excursion of $\Gamma_r$ from $y_{1,\rL}$ to $x_{2,\rL}$,
    \item[$\vdots$]
    \item $\eta_{\#Z_r,\rL}$, and
    \item the complementary $(\CC \setminus V_{\#Z_r}, \CC \setminus U_{\#Z_r})$-excursion of $\Gamma_r$ from $y_{\#Z_r,\rL}$ to $x_{1,\rL}$
\end{itemize}
is a loop of $\Gamma_r$. We shall write $\SCL_r$ for this loop. See \Cref{fig:gamma} for an illustration. 

\begin{lemma}\label{096}
    There exists a deterministic constant $\fb > 0$ such that the following is true: Suppose that the event $\fE_r \cup \fF_r(J_r)$ occurs. Let $P_r \colon [0, 1] \to \Upsilon_r$ be a $D_{\Upsilon_r}$-geodesic between two points outside of $B_{4r}(0)$. Suppose that $P_r$ passes through either the tube between $\SCL_r$ and $\gamma_r$ or the tube between $\SCL_r$ and $\gamma_r^\prime$. (By passing through either the tube between $\SCL_r$ and $\gamma_r$ or the tube between $\SCL_r$ and $\gamma_r^\prime$, we mean that $P_r$ (regarded as a path in the prime end closure of the unbounded connected component of $\CC \setminus (\SCL_r \cup \gamma_r \cup \gamma_r^\prime)$) is not homotopic (relative to its endpoints) to a path outside of $B_{4r}(0)$.) Then there are times $0 < \sigma < \tau < 1$ such that 
    \begin{align*}
        P_r([\sigma, \tau]) &\subset B_{4r}(0) \cap \Upsilon_r, \quad D_{\Upsilon_r}(P_r(\sigma), P_r(\tau)) \ge \fb\kc_r,  \quad \text{and}  \\
        \widetilde D_{\Upsilon_r}(P_r(\sigma), P_r(\tau)) &\le M_2 D_{\Upsilon_r}(P_r(\sigma), P_r(\tau)). 
    \end{align*}
\end{lemma}

\begin{proof}
    We may assume without loss of generality that $P_r$ passes through the tube between $\SCL_r$ and $\gamma_r$. In particular, $P_r$ passes through the tube between $\eta_{2,\rL}$ and $\eta_{2,\rR}$. Let $u$, $v$, $\widetilde P$, $s$, $t$, $P_s$, $P_t$ be as in the definition of $\fG_{z_2,\rho r}$ (cf.~\Cref{097}). Then we must have $P_r \cap P_s \neq \emptyset$ and $P_r \cap P_t \neq \emptyset$. Choose times $0 < \sigma < \tau < 1$ such that either $P_r(\sigma) \in P_s$ and $P_r(\tau) \in P_t$ or $P_r(\sigma) \in P_t$ and $P_r(\tau) \in P_s$. We may assume without loss of generality that the former is true. By \Cref{097}, \eqref{097E}, \eqref{097G} and the triangle inequality, we have
    \begin{align}
        \widetilde D_{\Upsilon_r}(P_r(\sigma), P_r(\tau)) &\le \widetilde D_{\Upsilon_r}(u, v) + \widetilde D_{\Upsilon_r}(u, \widetilde P(s)) + \widetilde D_{\Upsilon_r}(\widetilde P(t), v) + \len(P_s; \widetilde D_{\Upsilon_r}) + \len(P_t; \widetilde D_{\Upsilon_r}) \notag \\
        &= \widetilde D_\Upsilon(u, v) + \widetilde D_\Upsilon(u, \widetilde P(s)) + \widetilde D_\Upsilon(\widetilde P(t), v) + \len(P_s; \widetilde D_\Upsilon) + \len(P_t; \widetilde D_\Upsilon) \notag \\
        &\le (1 + 4\lambda)\widetilde D_\Upsilon(u, v). \label{eq:106}
    \end{align}
    Again, by \Cref{097}, \eqref{097E}, \eqref{097G} and the triangle inequality, we have
    \begin{align}
        D_{\Upsilon_r}(P_r(\sigma), P_r(\tau)) &\ge D_{\Upsilon_r}(u, v) - D_{\Upsilon_r}(u, \widetilde P(s)) - D_{\Upsilon_r}(\widetilde P(t), v) - \len(P_s; D_{\Upsilon_r}) - \len(P_t; D_{\Upsilon_r}) \notag \\
        &= D_\Upsilon(u, v) - D_\Upsilon(u, \widetilde P(s)) - D_\Upsilon(\widetilde P(t), v) - \len(P_s; D_\Upsilon) - \len(P_t; D_\Upsilon) \notag \\
        &\ge \frac1{M_1}\widetilde D_\Upsilon(u, v) - \frac1{M_\ast}\left(\widetilde D_\Upsilon(u, \widetilde P(s)) + \widetilde D_\Upsilon(\widetilde P(t), v) + \len(P_s; \widetilde D_\Upsilon) + \len(P_t; \widetilde D_\Upsilon)\right) \notag \\
        &\ge \left(\frac1{M_1} - \frac{4\lambda}{M_\ast}\right)\widetilde D_\Upsilon(u, v). \label{eq:107}
    \end{align}
    Since $\lambda$ may be chosen to be sufficiently small so that $M_2(1/M_1 - 4\lambda/M_\ast) \ge (1 + 4\lambda)$, we conclude from \eqref{eq:106} and \eqref{eq:107} that
    \begin{equation*}
        \widetilde D_{\Upsilon_r}(P_r(\sigma), P_r(\tau)) \le M_2 D_{\Upsilon_r}(P_r(\sigma), P_r(\tau)).
    \end{equation*}
    Finally, by \eqref{eq:340}, \Cref{097}, \eqref{097E}, and \eqref{eq:107}, we have
    \begin{equation*}
        D_{\Upsilon_r}(P_r(\sigma), P_r(\tau)) \ge \left(\frac1{M_1} - \frac{4\lambda}{M_\ast}\right)\widetilde D_\Upsilon(u, v) \ge \left(\frac1{M_1} - \frac{4\lambda}{M_\ast}\right)\alpha^3\kc_{\rho r} \ge \left(\frac1{M_1} - \frac{4\lambda}{M_\ast}\right)\alpha^3\rho^2\kc_r. 
    \end{equation*}
    This completes the proof. 
\end{proof}

\begin{figure}[ht!]
    \centering
    \includegraphics[width=.9\linewidth]{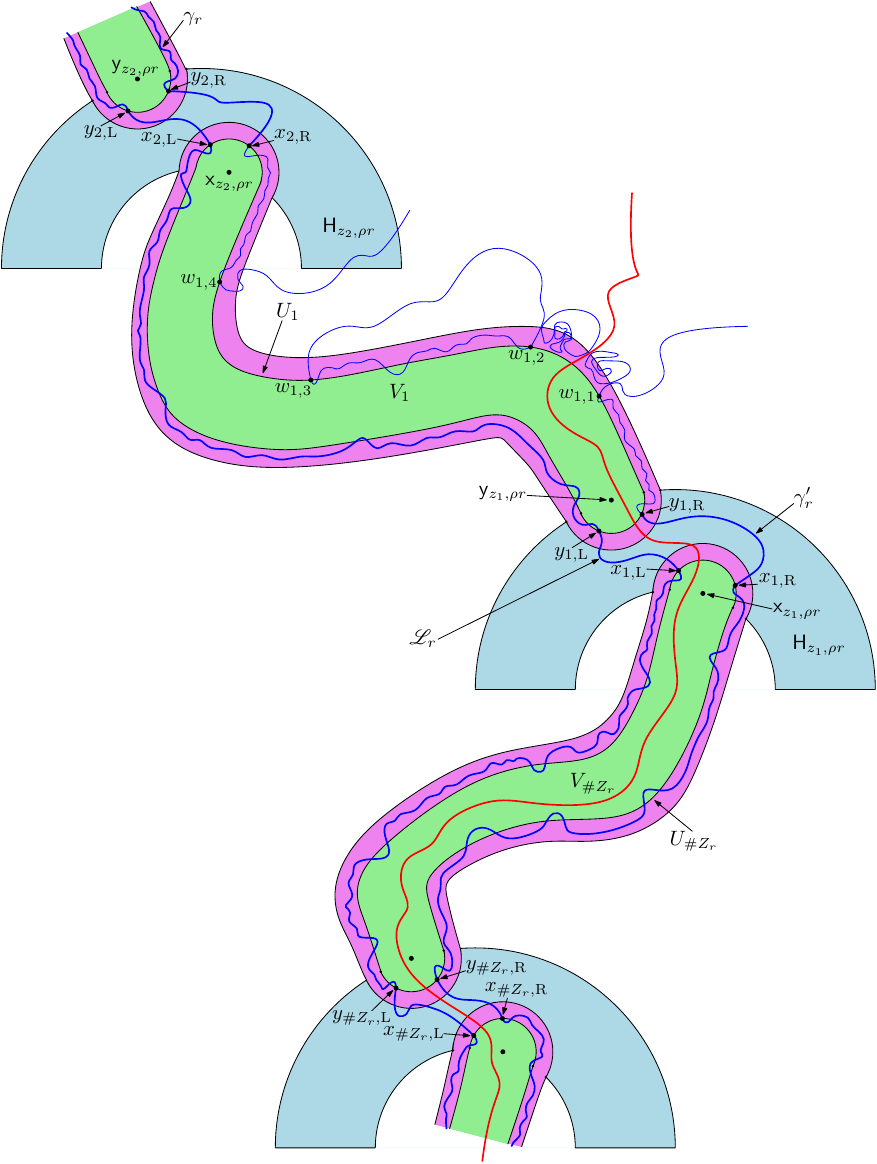}
    \caption{Illustration of $\SCL_r$, $\gamma_r$, and $\gamma_r^\prime$. The blue paths are (connected arcs of) loops of $\Gamma_r$. The red path is a $D_{\Upsilon_r}$-geodesic that passes through the tube between $\SCL_r$ and $\gamma_r^\prime$.}
    \label{fig:gamma}
\end{figure}

\section{Uniqueness}\label{s:02}

In the present section, we use the objects of \Cref{s:00} to prove \Cref{000}. Fix two weak geodesic $\CLE_\kappa$ carpet metrics $D$ and $\widetilde D$ with the same scaling constants $\{\kc_r\}_{r > 0}$. Let $M_\ast$ and $M^\ast$ be as in \eqref{eq:321} and \eqref{eq:322}, respectively. Suppose by way of contradiction that $M_\ast < M^\ast$. Recall that $M_1 = (2M_\ast + M^\ast)/3$, $M_2 = (M_\ast + 2M^\ast)/3$, and $\pp \in (0, 1)$, $\fA > 0$, $0 < \mu < \nu$, $\Lambda \in \NN$ are fixed parameters to be chosen later (cf.~\Cref{223}), in a manner depending only on $D$ and $\widetilde D$. 

Let $z \in \CC$ and $r \in \rho^{-1}\SR$. Recall from \Cref{ss:05} that $\fG_{z,\rho r}$ is the event that roughly speaking says that $B_{\rho r}(z)$ contains a ``shortcut'', and that $\fG_{z,\rho r}$ holds with uniformly positive probability. Recall from \eqref{eq:042} and \eqref{eq:279} that $\fZ_r \subset \partial B_{3r}(0)$ is the deterministic finite subset of ``test points'' and that $Z_r \subset \fZ_r$ is the random subset consisting of points $z$ for which the event $\fG_{z,\rho r}$ occurs. Recall from \eqref{eq:041} and \eqref{eq:044} that $U_r$ and $V_r$ are long narrow ``tubes'' that ``link up'' the half-annuli $\fH_{z,\rho r} \subset A_{(1 - a)r,(1 + a)r}(z)$ for $z \in Z_r$. The random open subsets $U_r$ and $V_r$ are almost surely determined by $Z_r$. Recall from \eqref{eq:040} that $U_r^\star$ is the open subset obtained from the \emph{partial exploration} of $\Gamma$ in $\CC \setminus U_r$ until hitting $V_r$. Recall from \Cref{ss:06} that $\fE_r$ is the event which roughly speaking says that $Z_r$ is large and that $\Upsilon_r$ is a resampling of $\Upsilon$. Recall from \Cref{ss:06,ss:07} that $\fF_r$ is the event which roughly speaking says that after resampling, the ``shortcuts'' of $\fG_{z,\rho r}$ for $z \in Z_r$ are preserved, and there are (connected arcs of) loops $\SCL_r$, $\gamma_r$, $\gamma_r^\prime$ of $\Gamma_r$ that form a ``tube'' that forces the $D_{\Upsilon_r}$-geodesics to take the ``shortcuts''. Recall from \Cref{ss:08} that $\fE_r$ holds with high probability, and that $\fF_r$ almost surely holds with uniformly positive conditional probability given $\fE_r$ and $\Upsilon$.

We shall write
\begin{equation*}
    Z_{z,r} \defeq Z_r(\Upsilon - z) + z; \quad U_{z,r} \defeq U_r(\Upsilon - z) + z; \quad V_{z,r} \defeq V_r(\Upsilon - z) + z; \quad U_{z,r}^\star \defeq U_r^\star(\Upsilon - z) + z, 
\end{equation*}
where $Z_r(\Upsilon - z)$ (resp.~$U_r(\Upsilon - z)$; $V_r(\Upsilon - z) $; $U_r^\star(\Upsilon - z)$) is defined in the same manner as $Z_r$ (resp.~$U_r$; $V_r$; $U_r^\star$) but with $\Upsilon - z$ in place of $\Upsilon$. Let $J_{z,r} \subset [1, \#Z_{z,r}]_\ZZ$ be a subset with $\#J_{z,r} = 2$ (chosen in a manner that is almost surely determined by $\Upsilon$; these will be the indices of the two places where the geodesic will enter and leave the tube). Similarly, let $U_{z,r}^\prime$, $V_{z,r}^\prime$, and $\Upsilon_{z,r}$ be defined in the same manner as $U_r^\prime$, $V_r^\prime$, and $\Upsilon_r$, i.e., $(U_{z,r}^\prime, V_{z,r}^\prime)$ is sampled uniformly from all the possibilities for $(U_{z,r}, V_{z,r})$, independently of everything else, and $\Upsilon_{z,r}$ is a conditionally independent copy of $\Upsilon$ given $\Gamma(\CC \setminus \overline{V_r^\prime}, \CC \setminus U_r^\prime)$ and $\wp$. Again, we will only be interested in the case where $(U_{z,r}, V_{z,r}) = (U_{z,r}^\prime, V_{z,r}^\prime)$. Moreover, we shall assume that $\Upsilon_{z,r}$ for different $z$'s are conditionally independent given $\Upsilon$ (we will only be considering countably many $z$'s simultaneously). We shall write $\fE_{z,r}$ (resp.~$\fF_{z,r}(J_{z,r})$) for the event defined in the same manner as $\fE_r$ (resp.~$\fF_r(J_r)$) but with $\Upsilon - z$ (resp.\ $\Upsilon_{z,r} - z$) in place of $\Upsilon$ (resp.\ $\Upsilon_r$). On the event $\fE_{z,r} \cap \fF_{z,r}(J_{z,r})$, we shall write 
\begin{equation*}
    \SCL_{z,r} \defeq \SCL_r(\Upsilon - z) + z; \quad \gamma_{z,r} \defeq \gamma_r(\Upsilon - z) + z; \quad \gamma_{z,r}^\prime \defeq \gamma_r^\prime(\Upsilon - z) + z,
\end{equation*}
where $\SCL_r(\Upsilon - z)$ (resp.~$\gamma_r(\Upsilon - z)$; $\gamma_r^\prime(\Upsilon - z)$) is defined in the same manner as $\SCL_r$ (resp.~$\gamma_r$; $\gamma_r^\prime$) of \Cref{ss:07} but with $\Upsilon - z$ and $\Upsilon_{z,r} - z$ in place of $\Upsilon$ and $\Upsilon_r$, respectively. 

We observe that whenever $Z$ is a collection of points such that the Euclidean distance between distinct points are large enough, then we may construct a resampling $\Upsilon_Z$ of $\Upsilon$ out of $\{\Upsilon_{z,r}\}_{z \in Z}$. More precisely: Let $\Be \in (\ZZ/(300\Lambda\ZZ))^2$. Then we shall write 
\begin{equation*}
    \ZZ_{\Be,r} \defeq \left\{z \in \CC : z - \frac1{100}\Be r \in (3\Lambda r\ZZ)^2\right\}.
\end{equation*}
Thus, $\ZZ_{\Be,r}$ for $\Be \in (\ZZ/(300\Lambda\ZZ))^2$ form a partition of $(\ZZ/(r/100)\ZZ)^2$ such that each pair of distinct points of $\ZZ_{\Be,r}$ has Euclidean distance at least $3\Lambda r$. For each deterministic subset $Z \subset \ZZ_{\Be,r}$, let $\Upsilon_Z$ be such that
\begin{gather*}
    \left(\CC \setminus \bigcup_{z \in Z} B_{5r}(z)\right) \cap \Upsilon_Z = \left(\CC \setminus \bigcup_{z \in Z} B_{5r}(z)\right) \cap \Upsilon; \\
    B_{5r}(z) \cap \Upsilon_Z = B_{5r}(z) \cap \Upsilon_{z,r} \quad \forall z \in Z. 
\end{gather*}
(Here, we recall that $(\CC \setminus B_{5r}(z)) \cap \Upsilon_{z,r} = (\CC \setminus B_{5r}(z)) \cap \Upsilon$.) Note that $(\Upsilon, \Upsilon_Z)$ and $(\Upsilon_Z, \Upsilon)$ have the same law. (This follows from the fact that $(U_{z,r}, V_{z,r})$'s are independent and independent of everything else.)

The idea of the proof of \Cref{000} is to show that 
\begin{equation}\label{eq:039}
    \parbox{.85\linewidth}{the probability of the event $\overline G_\rr(M^\ast - \delta, b)$ of \Cref{s:03} goes to $0$ as $\delta$ goes to $0$, at a rate which is uniform in $\rr$}
\end{equation}
(cf.~\Cref{225}). This, together with \Cref{031}, will lead to a contradiction of \Cref{008}.

We will call a ball $B_{5r}(z)$ \emph{good} if the event $\fE_{z,r} \cap \fF_{z,r}(J_{z,r})$ occurs, and \emph{very good} if the event $(\fE_{z,r} \cap \fF_{z,r}(J_{z,r}))^\vee$ occurs, which is defined in the same manner as $\fE_{z,r} \cap \fF_{z,r}(J_{z,r})$ but with the roles of $\Upsilon$ and $\Upsilon_{z,r}$ interchanged, occurs. 

To prove \eqref{eq:039}, it suffices to show that with high probability, the $D_\Upsilon$-geodesic $P$ connecting two given compact subsets $K_1$ and $K_2$ has to hit a lot of ``very good'' balls. \Cref{096} will then give us a lot of pairs of times $\sigma < \tau$ such that $\widetilde D_\Upsilon(P(\sigma), P(\tau)) \le M_2 D_\Upsilon(P(\sigma), P(\tau))$, which in turn will show that $\widetilde D_\Upsilon(P(0), P(1))$ is bounded away of $M^\ast D_\Upsilon(P(0), P(1))$. We will show this using an argument based on counting the number of occurrences of certain events (cf.~\Cref{210,330}). More precisely, for $r \in \rho^{-1}\SR$ and a finite collection of points $Z$, we will consider the event $\kF_Z$ which roughly speaking says that the following are true (cf.~\Cref{210}):
\begin{itemize}
    \item Each ball $B_{5r}(z)$ for $z \in Z$ is ``good''. 
    \item The $D_\Upsilon$-geodesic $P$ from $K_1$ to $K_2$ enters $B_r(z)$ for each $z \in Z$. 
    \item The $D_{\Upsilon_Z}$-geodesic $P_Z$ from $K_1$ to $K_2$ passes through the tube between $\SCL_{z,r}$, $\gamma_{z,r}$, and $\gamma_{z,r}^\prime$ (as described in \Cref{ss:07}) for each $z \in Z$. 
\end{itemize}
We will also consider the event $\kF_Z^\vee$ which is defined in the same manner as $\kF_Z$ but with the roles of $\Upsilon$ and $\Upsilon_Z$ interchanged (cf.~\Cref{330}):
\begin{itemize}
    \item Each ball $B_{5r}(z)$ for $z \in Z$ is ``very good''. 
    \item The $D_{\Upsilon_Z}$-geodesic $P_Z$ from $K_1$ to $K_2$ enters $B_r(z)$ for each $z \in Z$. 
    \item The $D_\Upsilon$-geodesic $P$ from $K_1$ to $K_2$ passes through the tube between $\SCL_{z,r}$, $\gamma_{z,r}$, and $\gamma_{z,r}^\prime$ (as described in \Cref{ss:07}) for each $z \in Z$. 
\end{itemize}
Since $(\Upsilon, \Upsilon_Z)$ and $(\Upsilon_Z, \Upsilon)$ have the same law, it follows that $\BP\lbrack\kF_Z\rbrack = \BP\lbrack\kF_Z^\vee\rbrack$. This suggests that the number of $Z$ for which $\kF_Z$ holds should be roughly comparable to the number of $Z$ for which $\kF_Z^\vee$ holds.

Furthermore, we will show that for each $N \in \NN$ and $r \in [\varepsilon^{1 + \nu}, \varepsilon]$, the number of $Z$ with $\#Z \le N$ for which $\kF_Z$ holds increases proportionally to a positive power of $\varepsilon^{-N}$ (cf.~\Cref{201}). Indeed, since the event $\fE_{z,r}$ occurs with high probability, it follows from independence across scales (cf.~\Cref{263}) that with high probability, the $D_\Upsilon$-geodesic $P$ can be covered by balls $B_r(z)$ for which the event $\fE_{z,r}$ occurs. Combining with the fact that $\fF_{z,r}(J_{z,r})$ occurs with uniformly positive conditional probability given $\fE_{z,r}$ and $\Upsilon$ (cf.~\eqref{eq:286}), we have found many $Z_0$'s for which the first two conditions of the definition of $\kF_{Z_0}$ described above are satisfied. In order to produce sets $Z$ for which the third condition described above is also satisfied (i.e., the $D_{\Upsilon_Z}$-geodesic $P_Z$ from $K_1$ to $K_2$ passes through the tube between $\SCL_{z,r}$, $\gamma_{z,r}$, and $\gamma_{z,r}^\prime$ for each $z \in Z$), we start with a set $Z_0$ as above and iteratively construct $Z_j$ by removing the ``bad'' points $z \in Z_{j - 1}$ such that the $D_{\Upsilon_{Z_{j - 1}}}$-geodesic $P_{Z_{j - 1}}$ from $K_1$ to $K_2$ does not pass through the tube between $\SCL_{z,r}$, $\gamma_{z,r}$, and $\gamma_{z,r}^\prime$. We will show that this iterative procedure has to terminate before $Z_j$ becomes too small. More precisely, we will show that $Z_j = Z_{j + 1} = \cdots$ for some $j$ (in which case the event $\kF_{Z_j}$ occurs) such that $\#Z_j$ is at least a constant times $\#Z_0$. See \Cref{ss:16} for details.

By combining the previous two paragraphs with a deterministic observation (cf.~\Cref{204}) and a straightforward calculation (cf.~the end of \Cref{ss:17}), we conclude that, with high probability, many sets $Z$ satisfy the event $\kF_Z^\vee$. In particular, there are numerous ``very good'' balls $B_{5r}(z)$ through which $P$ passes via the tube bounded by $\SCL_{z,r}$, $\gamma_{z,r}$, and $\gamma_{z,r}^\prime$. As discussed earlier, this establishes~\eqref{eq:039}.

\subsection{The core argument}\label{ss:17}

In the present subsection, we state the main estimates that we will prove. Fix $\rr > 0$ and disjoint compact subsets $K_1, K_2 \subset B_{2\rr}(0)$. 

\begin{definition}\label{222}
    Let $\varepsilon \in (0, 1)$. Then we shall write $\kG_\varepsilon^{\rr}(K_1, K_2)$ for the event that the following are true:
    \begin{enumerate}
        \item\label{222A} $K_1 \xleftrightarrow\Upsilon K_2$ and $\widetilde D_\Upsilon(K_1, K_2) \ge M^\ast D_\Upsilon(K_1, K_2) - \varepsilon^{2(1 + \nu)}\kc_{\rr}$. 
        \item\label{222B} For each $x \in B_{2\rr}(0)$, there exists $r \in \rho^{-1}\SR \cap [\varepsilon^{1 + \nu}\rr, \varepsilon\rr]$ and $z \in \left(\frac1{100}r\ZZ\right)^2 \cap B_{2\rr}(0)$ such that $\lvert x - z\rvert \le r$ and the following are true:
        \begin{enumerate}
            \item $\fE_{z,r}(\Upsilon)$ occurs.
            \item There are at most two connected arcs of $\Gamma$ in $A_{5r,\Lambda r}(z)$ connecting its inner and outer boundaries.  
            \item There exists a path $P_z$ in $A_{5r,\Lambda r}(z) \cap \Upsilon$ of $D_\Upsilon$-length at most $\fA\kc_r$ such that every path in $\Upsilon$ that crosses between the inner and outer boundaries of $A_{5r,\Lambda r}(z)$ intersects~$P_z$. 
        \end{enumerate}
    \end{enumerate}
\end{definition}

The key requirement in \Cref{222} is condition~\eqref{222A}. Our goal is to show that if $M_\ast < M^\ast$, then condition~\eqref{222A} is highly unlikely to occur. The motivation for this lies in its later use in \Cref{ss:10}, where it will help produce a contradiction to \Cref{008}. Indeed, \Cref{008} provides a lower bound on the probability that there exist points $x, y \in \overline{B_{\rr}(0)} \cap \Upsilon$ satisfying certain properties such that $\widetilde D_\Upsilon(x, y)$ is approximately $M^\ast D_\Upsilon(x, y)$. We will prove that this lower bound conflicts with our upper bound on the probability of condition~\eqref{222A}. Condition~\eqref{222B} serves as a global regularity assumption; we will establish in \Cref{223} that it holds with high probability. Consequently, an upper bound on $\BP[\kG_\varepsilon^{\rr}(K_1, K_2)]$ will directly yield an upper bound for the probability of condition~\eqref{222A}.

\begin{proposition}\label{200}
    Let $\eta \in (0, 1)$. Suppose that $\dist(K_1, K_2) \ge \eta\rr$ and $\dist(K_1, \partial B_{2\rr}(0)) \ge \eta\rr$. Then it holds with superpolynomially high probability as $\varepsilon \to 0$, at a rate depending only on $\eta$ (but not on $\rr$, $K_1$, $K_2$), that $\kG_\varepsilon^{\rr}(K_1, K_2)$ does not occur. 
\end{proposition}

It is essential for our purposes that the bound in \Cref{200} be uniform in $\rr$, $K_1$, and $K_2$. Indeed, we will eventually choose $K_1$ and $K_2$ to be Euclidean balls whose radii are $\rr$ times a power of $\varepsilon$ (cf.~\Cref{224}). In proving \Cref{200}, we do not need to control the size of $\SR$; it suffices that, for every $r \in \rho^{-1}\SR$, the objects $U_{z,r}$, $V_{z,r}$, $\Upsilon_{z,r}$, $\fE_{z,r}$, and $\fF_{z,r}$ are defined and satisfy the required properties. In fact, decreasing $\SR$ only makes condition~\eqref{222B} in the definition of $\kG_\varepsilon^{\rr}(K_1,K_2)$ less likely to occur, while increasing $\SR$ makes condition~\eqref{222A} less likely, as will become clear in the proof of \Cref{200}. Either way, we will show that the event $\kG_\varepsilon^{\rr}(K_1,K_2)$ is extremely unlikely. The size of $\SR$ (cf.~\Cref{105}) will be used only later, to verify that condition~\eqref{222B} holds with high probability; this, in turn, will force condition~\eqref{222A} to be extremely unlikely.  The proof of \Cref{000} will be given in \Cref{ss:10} using \Cref{200} and may be read without first reading the details of the proof of \Cref{200}.

The proof of \Cref{200} will rely on counting how often certain events occur. 
We now proceed to define these events.

For $\Be \in (\ZZ/(300\Lambda\ZZ))^2$ and $r \in \rho^{-1}\SR$, we shall write 
\begin{equation*}
    \ZZ_{\Be,r}^{\rr}(K_1, K_2) \defeq \bigl\{z \in \ZZ_{\Be,r} \cap B_{2\rr}(0) : B_{\Lambda r}(z) \cap (K_1 \cup K_2) = \emptyset\bigr\}. 
\end{equation*}

\begin{definition}\label{210}
    Let $\Be \in (\ZZ/(300\Lambda\ZZ))^2$, $r \in \rho^{-1}\SR$, and $Z \subset \ZZ_{\Be,r}^{\rr}(K_1, K_2)$. Then we shall write $\kF_Z^{\rr}(K_1, K_2)$ for the event that the following are true:
    \begin{enumerate}
        \item\label{210A} $K_1 \xleftrightarrow\Upsilon K_2$ and $\widetilde D_\Upsilon(K_1, K_2) \ge M^\ast D_\Upsilon(K_1, K_2) - \kc_r$. 
        \item\label{210B} For each $z \in Z$, the event $\fE_{z,r}(\Upsilon)$ occurs.
        \item\label{210C} For each $z \in Z$, there are at most two connected arcs of $\Gamma$ in $A_{5r,\Lambda r}(z)$ connecting its inner and outer boundaries. 
        \item\label{210D} For each $z \in Z$, there exists a path $P_z$ in $A_{5r,\Lambda r}(z) \cap \Upsilon$ of $D_\Upsilon$-length at most $\fA\kc_r$ such that every path in $\Upsilon$ that crosses between the inner and outer boundaries of $A_{5r,\Lambda r}(z)$ intersects $P_z$. 
        \item\label{210E} Write $P \colon [0, 1] \to \Upsilon$ for the $D_\Upsilon$-geodesic from $K_1$ to $K_2$. For each $z \in Z$, the $D_\Upsilon$-geodesic $P$ enters $B_r(z)$. 
    \end{enumerate}
    Before proceeding, for each $z \in Z$, we shall write $s_z$ (resp.~$t_z$) for the first (resp.~last) time at which $P$ hits $U_{z,r}^\star$. Write $i_z, j_z \in [1, \#Z_{z,r}]_\ZZ$ such that $P(s_z) \in \partial U_{i_z}^\star$ and $P(t_z) \in \partial U_{j_z}^\star$. (Note that necessarily $P(s_z) \in \partial U_{i_z}$ and $P(t_z) \in \partial U_{j_z}$.  We also recall that $U_1, \ldots, U_{\#Z_{z,r}}$ are the connected components of $U_{z,r}$, and $U_1^\star, \ldots, U_{\#Z_{z,r}}^\star$ are the connected components of $U_{z,r}^\star$.)  
    \begin{enumerate}
        \setcounter{enumi}5
        \item\label{210F} For each $z \in Z$, the event $\fF_{z,r}(\Upsilon_Z; \{i_z, j_z\})$ occurs. 
        \item\label{210G} Write $P_Z \colon [0, 1] \to \Upsilon_Z$ for the $D_{\Upsilon_Z}$-geodesic from $K_1$ to $K_2$. For each $z \in Z$, the $D_{\Upsilon_Z}$-geodesic $P_Z$ passes through the tube between either $\SCL_{z,r}$ and $\gamma_{z,r}$ or $\SCL_{z,r}$ and $\gamma_{z,r}^\prime$. 
    \end{enumerate}
\end{definition}

An illustration of \Cref{210} is provided in \Cref{fig:kF}. Condition~\eqref{210A} is closely connected to condition~\eqref{222A} from \Cref{222}, while conditions~\eqref{210B}, \eqref{210C}, and \eqref{210D} correspond closely to condition~\eqref{222B} in the same reference. Also, we note that the event $\kG_\varepsilon^{\rr}(K_1, K_2)$ depends only on $\Upsilon$, while the event $\kF_Z^{\rr}(K_1, K_2)$ depends on both $\Upsilon$ and $\Upsilon_Z$. 

\begin{figure}[ht!]
    \centering
    \includegraphics[width=0.8\linewidth]{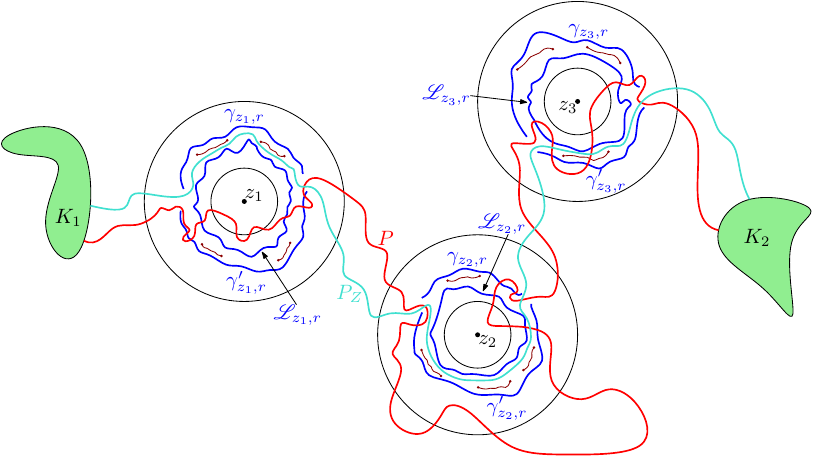}
    \caption{Illustration of the event $\kF_Z^{\rr}(K_1, K_2)$ of \Cref{210} for $Z = \{z_1, z_2, z_3\}$. The red path $P$ is a $D_\Upsilon$-geodesic from $K_1$ to $K_2$. The turquoise path $P_Z$ is a $D_{\Upsilon_Z}$-geodesic from $K_1$ to $K_2$. For each $z \in \{z_1, z_2, z_3\}$, the blue paths $\SCL_{z,r}$, $\gamma_{z,r}$, $\gamma_{z,r}^\prime$ are (connected arcs of) loops of $\Gamma_Z$; $P_Z$ passes through the tube between either $\SCL_{z,r}$ and $\gamma_{z,r}$ or $\SCL_{z,r}$ and $\gamma_{z,r}^\prime$. The short dark red paths are ``shortcuts'' (corresponding to $\widetilde P$ described in \Cref{097}), in the sense that each of them is a $\widetilde D_\Upsilon$-geodesic whose $\widetilde D_\Upsilon$-length is at most $M_1$ times the $D_\Upsilon$-distance between its endpoints. For each $z \in \{z_1, z_2, z_3\}$, since $P_Z$ passes through the tube between either $\SCL_{z,r}$ and $\gamma_{z,r}$ or $\SCL_{z,r}$ and $\gamma_{z,r}^\prime$, it must get close to at least one ``shortcut'' in that tube.}
    \label{fig:kF}
\end{figure}

\begin{definition}\label{330}
    Let $\Be \in (\ZZ/(300\Lambda\ZZ))^2$, $r \in \rho^{-1}\SR$, and $Z \subset \ZZ_{\Be,r}^{\rr}(K_1, K_2)$. Then we shall write $\kF_Z^{\rr}(K_1, K_2)^\vee$ for the event defined in the same manner as $\kF_Z^{\rr}(K_1, K_2)$ but with the roles of $\Upsilon$ and $\Upsilon_Z$ interchanged, i.e., the event that the following are true:
    \begin{enumerate}
        \item\label{330A} $K_1 \xleftrightarrow{\Upsilon_Z} K_2$ and $\widetilde D_{\Upsilon_Z}(K_1, K_2) \ge M^\ast D_{\Upsilon_Z}(K_1, K_2) - \kc_r$. 
        \item\label{330B} For each $z \in Z$, the event $\fE_{z,r}(\Upsilon_Z)$ occurs. 
        \item\label{330C} For each $z \in Z$, there are at most two connected arcs of $\Gamma_Z$ in $A_{5r,\Lambda r}(z)$ connecting $\partial B_{5r}(z)$ and $\partial B_{\Lambda r}(z)$. 
        \item\label{330D} For each $z \in Z$, there exists a path $P_z$ in $A_{5r,\Lambda r}(z) \cap \Upsilon_Z$ of $D_{\Upsilon_Z}$-length at most $\fA\kc_r$ such that every path in $\Upsilon_Z$ that crosses between the inner and outer boundaries of $A_{5r,\Lambda r}(z)$ intersects $P_z$. 
        \item\label{330E} Write $P_Z \colon [0, 1] \to \Upsilon_Z$ for the $D_{\Upsilon_Z}$-geodesic from $K_1$ to $K_2$. For each $z \in Z$, the $D_{\Upsilon_Z}$-geodesic $P_Z$ hits $B_r(z)$. 
    \end{enumerate}
    Before proceeding, for each $z \in Z$, we shall write $s_z$ (resp.~$t_z$) for the first (resp.~last) time at which $P_Z$ hits $U_{z,r}^\star$. Write $i_z, j_z \in [1, \#Z_{z,r}]_\ZZ$ such that $P_Z(s_z) \in \partial U_{i_z}^\star$ and $P_Z(t_z) \in \partial U_{j_z}^\star$.
    \begin{enumerate}
        \setcounter{enumi}5
        \item\label{330F} For each $z \in Z$, the event $\fF_{z,r}(\Upsilon; \{i_z, j_z\})$ occurs.
        \item\label{330G} Write $P \colon [0, 1] \to \Upsilon$ for the $D_\Upsilon$-geodesic from $K_1$ to $K_2$. For each $z \in Z$, the $D_\Upsilon$-geodesic $P$ passes through the tube between either $\SCL_{z,r}$ and $\gamma_{z,r}$ or $\SCL_{z,r}$ and $\gamma_{z,r}^\prime$. (In this situation, $\SCL_{z,r}$, $\gamma_{z,r}$, $\gamma_{z,r}^\prime$ are (connected arcs of) loops of $\Gamma$ instead of $\Gamma_{z,r}$.)
    \end{enumerate}
\end{definition}

Since $(\Upsilon, \Upsilon_Z)$ and $(\Upsilon_Z, \Upsilon)$ have the same law, we have
\begin{equation}\label{eq:214}
    \BP\lbrack\kF_Z^{\rr}(K_1, K_2)\rbrack = \BP\lbrack\kF_Z^{\rr}(K_1, K_2)^\vee\rbrack, \quad \forall \Be \in (\ZZ/(300\Lambda\ZZ))^2, \ \forall r \in \rho^{-1}\SR, \ \forall Z \subset \ZZ_{\Be,r}^{\rr}(K_1, K_2). 
\end{equation}

The following estimate tells us that, given $\Upsilon$ and the event $\kG_\varepsilon^{\rr}(K_1, K_2)$, with high probability, there are many $Z \subset \ZZ_{\Be,r}^{\rr}(K_1, K_2)$ for which $\kF_Z^{\rr}(K_1, K_2)$ occurs. 

\begin{proposition}\label{201}
    For each $\eta \in (0, 1)$, there exists $\kc = \kc(\eta) > 0$ such that for each $N \in \NN$, there exists $\varepsilon_\ast(\eta, N) \in (0, 1)$ such that for each $\varepsilon \in (0, \varepsilon_\ast]$, the following is true: Suppose that $\dist(K_1, K_2) \ge \eta\rr$ and $\dist(K_1, \partial B_{2\rr}(0)) \ge \eta\rr$. Then, on the event $\kG_\varepsilon^{\rr}(K_1, K_2)$, there exists $\Be \in (\ZZ/(300\Lambda\ZZ))^2$ and $r \in \rho^{-1}\SR \cap [\varepsilon^{1 + \nu}\rr, \varepsilon\rr]$ --- which are almost surely determined by $\Upsilon$ --- such that 
    \begin{equation*}
        \BE\!\left\lbrack\#\left\{Z \subset \ZZ_{\Be,r}^{\rr}(K_1, K_2) : \#Z \le N \text{ and } \kF_Z^{\rr}(K_1, K_2) \text{ occurs}\right\} \ \middle\vert \ \Upsilon\right\rbrack \ge \varepsilon^{-\kc N}. 
    \end{equation*}
\end{proposition}

Both the proofs of \Cref{201,204} (stated just below) will be proved via elementary deterministic considerations based on the definitions of the events $\fE_{z,r}$, $\fF_{z,r}$, $\kF_Z^{\rr}(K_1,K_2)$, and $\kF_Z^{\rr}(K_1, K_2)^\vee$. The proof of \Cref{201} is postponed until \Cref{ss:16}. The guiding idea is as follows. We have already observed that condition~\eqref{210A} of the event $\kF_Z^{\rr}(K_1,K_2)$ is closely related to condition~\eqref{222A} of $\kG_\varepsilon^{\rr}(K_1,K_2)$, while conditions~\eqref{210B}--\eqref{210D} of $\kF_Z^{\rr}(K_1,K_2)$ correspond to condition~\eqref{222B} of $\kG_\varepsilon^{\rr}(K_1,K_2)$. Hence, by the pigeonhole principle, on the event $\kG_\varepsilon^{\rr}(K_1,K_2)$ there exist $\Be \in (\ZZ/(300\Lambda\ZZ))^2$, $r \in \rho^{-1}\SR \cap [\varepsilon^{1+\nu}\rr,\varepsilon\rr]$, and a set $\SZ_1 \subset \ZZ_{\Be,r}^{\rr}(K_1,K_2)$ with $\#\SZ_1 = \varepsilon^{-1+o(1)}$ such that conditions~\eqref{210A}--\eqref{210E} of $\kF_Z^{\rr}(K_1,K_2)$ hold for every subset $Z \subset \SZ_1$. Invoking \eqref{eq:286}, we further obtain a subset $\SZ_2 \subset \SZ_1$ with $\BE\!\left\lbrack\#\SZ_2 \mid \Upsilon\right\rbrack = \varepsilon^{-1+o(1)}$ almost surely, and for which conditions~\eqref{210A}--\eqref{210F} of $\kF_Z^{\rr}(K_1,K_2)$ hold for all $Z \subset \SZ_2$. Thus it remains to check~\eqref{210G}. We will show that for each $Z \subset \SZ_2$ there is a subcollection $Z' \subset Z$ with $\#Z' \ge c\,\#Z$ for some deterministic $c>0$ such that condition~\eqref{210G} holds for $\kF_{Z'}^{\rr}(K_1,K_2)$ (hence $\kF_{Z'}^{\rr}(K_1,K_2)$ occurs) (cf.~\Cref{205} and see also the beginning of \Cref{ss:16} for the idea of the construction of $Z'$). Consequently,
\begin{multline*}
    \BE\!\left\lbrack\#\left\{ Z \subset \ZZ_{\Be,r}^{\rr}(K_1,K_2) : \#Z \le N \ \text{and}\ \kF_Z^{\rr}(K_1,K_2)\ \text{occurs} \right\} \ \middle| \ \Upsilon \right\rbrack \\
    \ge \BE\!\left\lbrack\binom{\#\SZ_2}{N} \binom{\#\SZ_2}{(1-c)N}^{-1} \ \middle| \ \Upsilon \right\rbrack \succeq \BE\!\left\lbrack (\#\SZ_2)^{cN} \ \middle| \ \Upsilon \right\rbrack = \varepsilon^{-(c+o(1))N} \quad \text{almost surely}.
\end{multline*}

Our next result provides an unconditional upper bound on the number of sets $Z$ for which the event $\kF_Z^{\rr}(K_1, K_2)^\vee$ occurs.

\begin{proposition}\label{204}
    There exists a deterministic constant $\kC > 0$ such that for each $\Be \in (\ZZ/(300\Lambda\ZZ))^2$, $r \in \rho^{-1}\SR$, and $N \in \NN$, we have
    \begin{equation}\label{eq:237}
        \#\left\{Z \subset \ZZ_{\Be,r}^{\rr}(K_1, K_2) : \#Z \le N \text{ and } \kF_Z^{\rr}(K_1, K_2)^\vee \text{ occurs}\right\} \le \kC^N \quad \text{almost surely}. 
    \end{equation}
\end{proposition}

The proof of \Cref{204} is postponed until \Cref{ss:19}. The intuition behind \Cref{204} is as follows. Fix $Z_0 \subset \ZZ_{\Be,r}^{\rr}(K_1,K_2)$ with $\#Z_0 \le N$ such that $\kF_{Z_0}^{\rr}(K_1,K_2)^\vee$ occurs (if no such $Z_0$ exists, then \eqref{eq:237} is immediate). Let $Z_0'$ be the set of points $z \in \ZZ_{\Be,r}^{\rr}(K_1,K_2)\setminus Z_0$ for which there exists $Z \subset \ZZ_{\Be,r}^{\rr}(K_1,K_2)$ with $\kF_Z^{\rr}(K_1,K_2)^\vee$ occurring. We \emph{claim} that $\#Z_0' \le C\,\#Z_0$ for some deterministic constant $C$ (cf.~\Cref{209}). Indeed, each $z \in Z_0'$ must satisfy the $z$-localized parts of conditions~\eqref{330B}, \eqref{330C}, \eqref{330D}, \eqref{330F}, \eqref{330G} of \Cref{330}, namely:
\begin{itemize}
    \item $\fE_{z,r}(\Upsilon_{z,r})$ occurs.
    \item At most two connected arcs of $\Gamma_{z,r}$ in $A_{5r,\Lambda r}(z)$ connect the inner and outer boundaries.
    \item There exists a path $P_z \subset A_{5r,\Lambda r}(z) \cap \Upsilon_{z,r}$ of $D_{\Upsilon_{z,r}}$-length at most $\fA\kc_r$ such that every path in $\Upsilon_{z,r}$ that crosses between the inner and outer boundaries of $A_{5r,\Lambda r}(z)$ intersects $P_z$.
    \item $\fF_{z,r}(\Upsilon;\{i_z,j_z\})$ occurs for some distinct $i_z,j_z \in [1,\#Z_{z,r}]_\ZZ$.
    \item The $D_\Upsilon$-geodesic $P$ passes through the tube between either $\SCL_{z,r}$ and $\gamma_{z,r}$ or $\SCL_{z,r}$ and $\gamma_{z,r}'$.
\end{itemize}
The last item, together with \Cref{096} (applied with the roles of $\Upsilon$ and $\Upsilon_{z,r}$ interchanged), implies that 
\begin{equation}\label{eq:045}
    \parbox{.87\linewidth}{$\widetilde D_\Upsilon(K_1,K_2)$ is at most $M^\ast D_\Upsilon(K_1,K_2)$ minus a deterministic positive constant multiple of $\#Z_0'$.}
\end{equation}
On the other hand, since $\kF_{Z_0}^{\rr}(K_1,K_2)^\vee$ occurs, condition~\eqref{330A} of $\kF_{Z_0}^{\rr}(K_1,K_2)^\vee$ yields
\begin{equation}\label{eq:046}
    \widetilde D_{\Upsilon_{Z_0}}(K_1,K_2) \,\ge\, M^\ast D_{\Upsilon_{Z_0}}(K_1,K_2) - \kc_r .
\end{equation}
Since $\Upsilon$ and $\Upsilon_{Z_0}$ differ only inside the Euclidean balls centered at $z \in Z_0$, the differences between $\widetilde D_\Upsilon(K_1,K_2)$ and $\widetilde D_{\Upsilon_{Z_0}}(K_1,K_2)$ (resp.~$D_\Upsilon(K_1,K_2)$ and $D_{\Upsilon_{Z_0}}(K_1,K_2)$) are controlled by a deterministic positive constant multiple of $\#Z_0$. Hence, if $\#Z_0'$ were much larger than $\#Z_0$, \eqref{eq:045} and \eqref{eq:046} would be incompatible, proving the \emph{claim} (see also the paragraph after \Cref{209}). Finally, we conclude \Cref{204} by noting that
\begin{equation*}
    \#\left\{ Z \subset \ZZ_{\Be,r}^{\rr}(K_1,K_2) : \#Z \le N \text{ and } \kF_Z^{\rr}(K_1,K_2)^\vee \text{ occurs} \right\} \le 2^{\#Z_0 + \#Z_0'} \le 2^{(1 + C)N}.
\end{equation*}

We now show how \Cref{200} can be proved under the assumptions of \Cref{201,204}. The idea is that, on the event $\kG_\varepsilon^{\rr}(K_1,K_2)$, \Cref{204} yields a deterministic conclusion that contradicts the requirement of \Cref{201} as $N\to\infty$. Consequently, $\kG_\varepsilon^{\rr}(K_1,K_2)$ cannot occur with high probability.

\begin{proof}[Proof of \Cref{200} assuming \Cref{201,204}]
    Since $\SR \subset \{\Lambda^k\}_{k \in \ZZ}$, we have 
    \begin{align}\label{eq:200}
        \#\left(\rho^{-1}\SR \cap [\varepsilon^{1 + \nu}\rr, \varepsilon\rr]\right) &= \#\left(\SR \cap [\varepsilon^{1 + \nu}\rho\rr, \varepsilon\rho\rr]\right) \le \#\left(\SR \cap [(\varepsilon\rho)^{1 + \nu}\rr, \varepsilon\rho\rr]\right) \\
        &\le \nu\log_{\Lambda}(1/(\varepsilon\rho)) \le 2\nu\log_{\Lambda}(1/\varepsilon), \quad \forall \varepsilon \in (0, \rho]. \notag
    \end{align}
    To lighten notation, write $\SR_\varepsilon \defeq \rho^{-1}\SR \cap [\varepsilon^{1 + \nu}\rr, \varepsilon\rr]$. Thus, by \eqref{eq:200} and \Cref{204}, for each $N \in \NN$ and $\varepsilon \in (0, \varepsilon_\ast(\eta, N) \wedge \rho]$, we have
    \begin{align}\label{eq:038}
        &\kC^N \cdot 9 \cdot 10^4 \cdot \Lambda^2 \cdot 2\nu\log_{\Lambda}(1/\varepsilon) \\
        &\ge \kC^N \sum_{\Be \in (\ZZ/(300\Lambda\ZZ))^2} \sum_{r \in \SR_\varepsilon} 1 \notag \\
        &\ge \sum_{\Be \in (\ZZ/(300\Lambda\ZZ))^2} \sum_{r \in \SR_\varepsilon} \BE\!\left\lbrack\#\left\{Z \subset \ZZ_{\Be,r}^{\rr}(K_1, K_2) : \#Z \le N, \ \kF_Z^{\rr}(K_1, K_2)^\vee \text{ occurs}\right\}\right\rbrack \notag \\
        &= \sum_{\Be \in (\ZZ/(300\Lambda\ZZ))^2} \sum_{r \in \SR_\varepsilon} \sum_{\substack{Z \subset \ZZ_{\Be,r}^{\rr}(K_1, K_2) \\ \#Z \le N}} \BP\!\left\lbrack\kF_Z^{\rr}(K_1, K_2)^\vee\right\rbrack. \notag
    \end{align}
    By \eqref{eq:214} and \Cref{201}, we have
    \begin{align}\label{eq:037}
        &\sum_{\Be \in (\ZZ/(300\Lambda\ZZ))^2} \sum_{r \in \SR_\varepsilon} \sum_{\substack{Z \subset \ZZ_{\Be,r}^{\rr}(K_1, K_2) \\ \#Z \le N}} \BP\!\left\lbrack\kF_Z^{\rr}(K_1, K_2)^\vee\right\rbrack \\
        &= \sum_{\Be \in (\ZZ/(300\Lambda\ZZ))^2} \sum_{r \in \SR_\varepsilon} \sum_{\substack{Z \subset \ZZ_{\Be,r}^{\rr}(K_1, K_2) \\ \#Z \le N}} \BP\lbrack\kF_Z^{\rr}(K_1, K_2)\rbrack \quad \notag \\
        &= \sum_{\Be \in (\ZZ/(300\Lambda\ZZ))^2} \sum_{r \in \SR_\varepsilon} \BE\!\left\lbrack\#\left\{Z \subset \ZZ_{\Be,r}^{\rr}(K_1, K_2) : \#Z \le N, \ \kF_Z^{\rr}(K_1, K_2) \text{ occurs}\right\}\right\rbrack \notag \\
        &\ge \varepsilon^{-\kc N} \BP\lbrack\kG_\varepsilon^{\rr}(K_1, K_2)\rbrack. \notag 
    \end{align}
    Combining \eqref{eq:038} and \eqref{eq:037}, for each $N \in \NN$, there exists $\varepsilon_\ast = \varepsilon_\ast(\eta, N) \in (0, 1)$ such that 
    \begin{equation*}
        \BP\lbrack\kG_\varepsilon^{\rr}(K_1, K_2)\rbrack \le 9 \cdot 10^4 \cdot \Lambda^2 \cdot \kC^N \cdot 2\nu\log_{\Lambda}(1/\varepsilon) \cdot \varepsilon^{\kc N}, \quad \forall \varepsilon \in (0, \varepsilon_\ast \wedge \rho].
    \end{equation*}
    This completes the proof. 
\end{proof}

\subsection{Proof of \Cref{201}}\label{ss:16}

Fix $\eta \in (0, 1)$ such that $\dist(K_1, K_2) \wedge \dist(K_1, \partial B_{2\rr}(0)) \ge \eta\rr$. 

For $\Be \in (\ZZ/(300\Lambda\ZZ))^2$, $r \in \rho^{-1}\SR$, and $Z \subset \ZZ_{\Be,r}^{\rr}(K_1, K_2)$, we shall write $E_Z$ (resp.~$F_Z$) for the event that conditions~\eqref{210A}, \eqref{210B}, \eqref{210C}, \eqref{210D}, \eqref{210E} (resp.~\eqref{210A}, \eqref{210B}, \eqref{210C}, \eqref{210D}, \eqref{210E}, \eqref{210F}) of \Cref{210} hold.

It follows directly from the definition of $\kG_\varepsilon^{\rr}(K_1, K_2)$ that, whenever this event occurs, there exists some $\Be \in (\ZZ / (300\Lambda\ZZ))^2$ and $r \in \rho^{-1}\SR$ such that (with high conditional probability) many sets $Z \subset \ZZ_{\Be,r}^{\rr}(K_1, K_2)$ satisfy the event $F_Z$. The main challenge in verifying condition~\eqref{210G} lies in the fact that the $D_{\Upsilon_Z}$-geodesic between $K_1$ and $K_2$ might entirely avoid some of the balls $B_{5r}(z)$ for $z \in Z$. To handle this issue, we will prove that if $Z \subset \ZZ_{\Be,r}^{\rr}(K_1, K_2)$ and $F_Z$ holds, then there exists a subset $Z^\prime \subset Z$ such that $\#Z^\prime$ is at least a fixed positive fraction of $\#Z$, and $\kF_{Z^\prime}^{\rr}(K_1, K_2)$ occurs (cf.~\Cref{205}).

The construction of $Z^\prime$ proceeds as follows. In \Cref{317}, we establish that $D_{\Upsilon_Z}(K_1, K_2)$ is significantly smaller than $D_\Upsilon(K_1, K_2)$. This means that the $D_{\Upsilon_Z}$-geodesic $P_Z$ from $K_1$ to $K_2$ must spend a considerable amount of time in regions where $\Upsilon$ and $\Upsilon_Z$ differ. For each $z \in Z$, there are several types of such regions inside $B_{5r}(z)$ where $P_Z$ could potentially enter; one important example is the tube bounded by $\SCL_{z,r}$, $\gamma_{z,r}$, and $\gamma_{z,r}^\prime$ (see \Cref{fig:bad} for an illustration). In \Cref{318}, we show, roughly speaking, that if $P_Z$ enters $B_{4r}(z)$ but does not travel through the tube between $\SCL_{z,r}$, $\gamma_{z,r}$, and $\gamma_{z,r}^\prime$, then the time it spends inside $B_{\Lambda r}(z)$ cannot be much shorter than the time the $D_\Upsilon$-geodesic $P$ spends there. We then iteratively remove from $Z$ the ``bad’’ points $z$ for which $P_Z$ avoids the tube between $\SCL_{z,r}$, $\gamma_{z,r}$, and $\gamma_{z,r}^\prime$. The key observation is that eliminating these ``bad’’ points does not significantly increase $D_{\Upsilon_Z}(K_1, K_2)$. Hence, throughout the iteration, $D_{\Upsilon_Z}(K_1, K_2)$ remains much smaller than $D_\Upsilon(K_1, K_2)$. This fact ensures that the iterative process must stop before too many points are removed from $Z$. 

We start by proving an upper bound showing that, on the event $F_Z$ (cf.~\Cref{094,317}), the quantity $D_{\Upsilon_Z}(K_1, K_2)$ is at most $D_\Upsilon(K_1, K_2)$ minus a deterministic positive constant multiple of $\kc_r\#Z$. This bound arises because the $D_\Upsilon$-geodesic $P$ must cross each annulus $A_{r,2r}(z)$ for $z \in Z$, which incurs a total cost proportional to $\kc_r\#Z$ (cf.~\Cref{098}, \eqref{098A}). In contrast, the $D_{\Upsilon_Z}$-geodesic $P_Z$ can take advantage of the tube bounded by $\SCL_{z,r}$, $\gamma_{z,r}$, and $\gamma_{z,r}^\prime$ to bypass this cost. Indeed, by \Cref{097},~\eqref{097B} and \Cref{098},~\eqref{035K}, traveling through this tube requires less time than crossing the corresponding annulus $A_{r,2r}(z)$.

\begin{lemma}\label{094}
    Let $\Be \in (\ZZ/(300\Lambda\ZZ))^2$, $r \in \rho^{-1}\SR$, and $Z \subset \ZZ_{\Be,r}^{\rr}(K_1, K_2)$. Suppose that the event $F_Z$ occurs. Let $\{(s_z, t_z)\}_{z \in Z}$ be as in \Cref{210}. Then
    \begin{equation*}
        D_{\Upsilon_Z}(P(s_z), P(t_z)) \le \fa_5\kc_r, \quad \forall z \in Z. 
    \end{equation*}
\end{lemma}

\begin{proof}
    Fix $z \in Z$. Let $i_z$ and $j_z$ be as in \Cref{210}. We observe that there exists $I \in \SI_{i_z}^{\alpha\setminus\beta}$ and $J \in \SI_{j_z}^{\alpha\setminus\beta}$ (cf.~\eqref{eq:102}) such that $P(s_z) \in I \cap \partial U_{z,r}$ and $P(t_z) \in J \cap \partial U_{z,r}$. Note that $P(s_z)$ and $P(t_z)$ are connected inside the tube between $\SCL_{z,r}$, $\gamma_{z,r}$, and $\gamma_{z,r}^\prime$ (cf., e.g., \Cref{fig:gamma}).
    By \eqref{eq:340}, \Cref{097}, \eqref{097B}, and \Cref{098}, \eqref{035K}
    \begin{equation*}
        D_{\Upsilon_Z}(P(s_z), P(t_z)) \le (\#Z_{z,r}) \left(2M_\ast^{-1} a^{-1}\kc_{\rho r} + \rho\kc_r\right) \le 2\pi\fa_5^{-1}\left(2M_\ast^{-1}a^{-1}\kK + 1\right)\rho\kc_r. 
    \end{equation*}
    Since we may choose $\rho = \fa_6$ to be sufficiently small so that $2\pi\fa_5^{-1}(2M_\ast^{-1}a^{-1}\kK + 1)\rho \le \fa_5$, we complete the proof. 
\end{proof}

\begin{lemma}\label{317}
    Let $\Be \in (\ZZ/(300\Lambda\ZZ))^2$, $r \in \rho^{-1}\SR$, and $Z \subset \ZZ_{\Be,r}^{\rr}(K_1, K_2)$. Suppose that the event $F_Z$ occurs. Then
    \begin{equation*}
        D_{\Upsilon_Z}(K_1, K_2) \le D_\Upsilon(K_1, K_2) - \fa_1\kc_r\#Z.
    \end{equation*}
\end{lemma}

\begin{proof}
    Let $\{(s_z, t_z)\}_{z \in Z}$ be as in \Cref{210}. By \Cref{094}, we have
    \begin{equation}\label{eq:230}
        D_{\Upsilon_Z}(P(s_z), P(t_z)) \le \fa_5\kc_r, \quad \forall z \in Z. 
    \end{equation}
    Since $\Upsilon_Z$ coincides with $\Upsilon$ outside of $\bigcup_{z \in Z} U_{z,r}^\star$, it follows that 
    \begin{equation}\label{eq:231}
        \len(P|_{[0, 1] \setminus \bigcup_{z \in Z} [s_z, t_z]}; D_{\Upsilon_Z}) = \len(P|_{[0, 1] \setminus \bigcup_{z \in Z} [s_z, t_z]}; D_\Upsilon). 
    \end{equation}
    Combining \eqref{eq:230} and \eqref{eq:231}, we obtain that 
    \begin{align*}
        D_{\Upsilon_Z}(K_1, K_2) &\le \len(P|_{[0, 1] \setminus \bigcup_{z \in Z} [s_z, t_z]}; D_\Upsilon) + \fa_5\kc_r\#Z \\
        &= D_\Upsilon(K_1, K_2) - \len(P|_{\bigcup_{z \in Z} [s_z, t_z]}; D_\Upsilon) + \fa_5\kc_r\#Z. 
    \end{align*}
    Thus, it suffices to show that
    \begin{equation*}
        \len(P|_{\bigcup_{z \in Z} [s_z, t_z]}; D_\Upsilon) \ge 2\fa_1\kc_r\#Z. 
    \end{equation*}
    Indeed, since $P$ enters $B_r(z)$ for all $z \in Z$. We conclude that there are times $s_z^\prime < t_z^\prime$ such that $P|_{[s_z^\prime, t_z^\prime]}$ forms a $(B_{2r}(z), B_r(z))$-excursion. By \Cref{098}, \eqref{098A}, we have
    \begin{equation*}
        \len(P|_{[s_z^\prime, t_z^\prime]}; D_\Upsilon) \ge 2\fa_1\kc_r, \quad \forall z \in Z. 
    \end{equation*}
    It is clear that the intervals $[s_z^\prime, t_z^\prime]$ for $z \in Z$ are pairwise disjoint and $\bigcup_{z \in Z} [s_z^\prime, t_z^\prime] \subset \bigcup_{z \in Z} [s_z, t_z]$. Thus, 
    \begin{equation*}
        \len(P|_{\bigcup_{z \in Z} [s_z, t_z]}; D_\Upsilon) \ge \sum_{z \in Z} \len(P|_{[s_z^\prime, t_z^\prime]}; D_\Upsilon) \ge 2\fa_1\kc_r\#Z. 
    \end{equation*}
    This completes the proof. 
\end{proof}

We now prove an inequality in the reverse direction of that in \Cref{317}, namely, an upper bound on $D_\Upsilon(K_1, K_2)$ expressed in terms of $D_{\Upsilon_Z}(K_1, K_2)$ (cf.~\Cref{221}). To carry out this proof, we will first need an upper bound on the number of $(B_{\Lambda r}(z), B_{4r}(z))$-excursions that the path $P_Z$ can make.

\begin{lemma}\label{220}
    Let $\Be \in (\ZZ/(300\Lambda\ZZ))^2$, $r \in \rho^{-1}\SR$, and $Z \subset \ZZ_{\Be,r}^{\rr}(K_1, K_2)$. Suppose that the event $F_Z$ occurs. Write $P_Z \colon [0, 1] \to \Upsilon_Z$ for the $D_{\Upsilon_Z}$-geodesic from $K_1$ to $K_2$. For each $z \in Z$, write $\SE_{z,r}(P_Z)$ for the collection of $(B_{\Lambda r}(z), B_{4r}(z))$-excursions of $P_Z$. Then
    \begin{equation*}
        \#\SE_{z,r}(P_Z) \le \fA/\fa_1. 
    \end{equation*}
\end{lemma}

\begin{proof}
    Fix $z \in Z$. Write $n \defeq \#\SE_{z,r}(P_Z)$. Write $\SE_{z,r}(P_Z) = \left\{(\sigma_k^\prime, \sigma_k, \tau_k, \tau_k^\prime) : k \in [1, n]_\ZZ\right\}$ in chronological order, i.e., 
    \begin{equation*}
        \sigma_1^\prime < \tau_1^\prime < \sigma_2^\prime < \tau_2^\prime < \cdots < \sigma_n^\prime < \tau_n^\prime. 
    \end{equation*}
    Since $P_Z|_{[\sigma_1^\prime, \sigma_1]}$ and $P_Z|_{[\tau_n, \tau_n^\prime]}$ intersect the path $P_z$ described in \Cref{210}, \eqref{210D}, it follows that 
    \begin{equation}\label{eq:220}
        \len(P_Z|_{[\sigma_1, \tau_n]}; D_{\Upsilon_Z}) \le \len(P_z; D_{\Upsilon_Z}) = \len(P_z; D_\Upsilon) \le \fA\kc_r. 
    \end{equation}
    On the other hand, by \Cref{098}, \eqref{098A}, 
    \begin{align}\label{eq:221}
        \len(P_Z|_{[\sigma_1, \tau_n]}; D_{\Upsilon_Z}) &\ge \sum_{k = 1}^{n - 1} \len(P_Z|_{[\tau_k, \tau_k^\prime]}; D_{\Upsilon_Z}) + \len(P_Z|_{[\sigma_n^\prime, \sigma_n]}; D_{\Upsilon_Z}) \\
        &= \sum_{k = 1}^{n - 1} \len(P_Z|_{[\tau_k, \tau_k^\prime]}; D_\Upsilon) + \len(P_Z|_{[\sigma_n^\prime, \sigma_n]}; D_\Upsilon) \ge n\fa_1\kc_r. \notag
    \end{align}
    Combining \eqref{eq:220} and \eqref{eq:221}, we complete the proof. 
\end{proof}

\begin{lemma}\label{221}
    Let $\Be \in (\ZZ/(300\Lambda\ZZ))^2$, $r \in \rho^{-1}\SR$, and $Z \subset \ZZ_{\Be,r}^{\rr}(K_1, K_2)$. Suppose that the event $F_Z$ occurs. Then 
    \begin{equation*}
        D_\Upsilon(K_1, K_2) \le D_{\Upsilon_Z}(K_1, K_2) + (\fA^2/\fa_1)\kc_r\#Z.
    \end{equation*}
\end{lemma}

\begin{proof}
    For each $z \in Z$, let $\SE_{z,r}(P_Z)$ be as in \Cref{220}. Since $\Upsilon_Z$ coincides with $\Upsilon$ outside of $\bigcup_{z \in Z} U_{z,r}^\star$, it follows that the $D_{\Upsilon_Z}$-length and the $D_\Upsilon$-length of the restriction of $P_Z$ to 
    \begin{equation*}
        [0, 1] \setminus \bigcup_{z \in Z} \bigcup_{(\sigma^\prime, \sigma, \tau, \tau^\prime) \in \SE_{z,r}(P_Z)} [\sigma, \tau]
    \end{equation*}
    are equal. On the other hand, for each $z \in Z$ and $(\sigma^\prime, \sigma, \tau, \tau^\prime) \in \SE_{z,r}(P_Z)$, the paths $P_Z|_{[\sigma^\prime, \sigma]}$ and $P_Z|_{[\tau, \tau^\prime]}$ intersect the path $P_z$ described in \Cref{210}, \eqref{210D}. Thus, we conclude from \Cref{220} that
    \begin{equation*}
        D_\Upsilon(K_1, K_2) \le D_{\Upsilon_Z}(K_1, K_2) + \sum_{z \in Z} \sum_{(\sigma^\prime, \sigma, \tau, \tau^\prime) \in \SE_{z,r}(P_Z)} \len(P_z; D_\Upsilon) \le D_{\Upsilon_Z}(K_1, K_2) + (\fA^2/\fa_1)\kc_r\#Z.
    \end{equation*}
    This completes the proof. 
\end{proof}

The following estimate tells us the points at which $P$ enters and exits $U_{z,r}^\star$ cannot be too close to each other. 

\begin{lemma}\label{329}
    Let $\Be \in (\ZZ/(300\Lambda\ZZ))^2$, $r \in \rho^{-1}\SR$, and $Z \subset \ZZ_{\Be,r}^{\rr}(K_1, K_2)$. Suppose that the event $E_Z$ occurs. For each $z \in Z$, let $i_z$ and $j_z$ be as in \Cref{210}. Then for each $z \in Z$, the Euclidean distance between $U_{i_z}$ and $U_{j_z}$ is greater than $\fa_3r/2$. 
\end{lemma}

\begin{proof}
    Fix $z \in Z$. Let $s_z$ and $t_z$ be as in \Cref{210}. Since $P|_{[s_z, t_z]}$ enters $B_r(z)$ by \Cref{210}, \eqref{210E}, it follows from \Cref{098}, \eqref{098A} that 
    \begin{equation}\label{eq:249}
        D_\Upsilon(P(s_z), P(t_z)) \ge D_\Upsilon(\partial B_r(z) \cap \Upsilon, \partial B_{2r}(z) \cap \Upsilon) \ge \fa_1\kc_r. 
    \end{equation}
    By \Cref{098}, \eqref{098D} and the definition of $\{U_j\}_{j \in [1, \#Z_{z,r}]_\ZZ}$, each of them has Euclidean diameter at most $100\fa_4r$. Thus, if $\dist(U_{i_z}, U_{j_z}) \le \fa_3r/2$, then $\lvert P(s_z) - P(t_z)\rvert \le \fa_3r$. In this case, it follows from \Cref{098}, \eqref{098B} that $D_\Upsilon(P(s_z), P(t_z)) \le \fa_2\kc_r$ (since $P(s_z) \xleftrightarrow\Upsilon P(t_z)$, no loop of $\Gamma$ can disconnect $P(s_z)$ and $P(t_z)$), in contradiction to \eqref{eq:249}. This completes the proof. 
\end{proof}

The following estimate tells us if a $(B_{\Lambda r}(z), B_{4r}(z))$-excursion of $P_Z$ does not pass through the tube between either $\SCL_{z,r}$ and $\gamma_{z,r}$ or $\SCL_{z,r}$ and $\gamma_{z,r}^\prime$, then the $D_{\Upsilon_Z}$-distance between its endpoints cannot be much shorter than the $D_\Upsilon$-distance. 

\begin{lemma}\label{318}
    Let $\Be \in (\ZZ/(300\Lambda\ZZ))^2$, $r \in \rho^{-1}\SR$, and $Z \subset \ZZ_{\Be,r}^{\rr}(K_1, K_2)$. Suppose that the event $F_Z$ occurs. Write $P_Z \colon [0, 1] \to \Upsilon_Z$ for the $D_{\Upsilon_Z}$-geodesic from $K_1$ to $K_2$. Let $z \in Z$. Let $(\sigma^\prime, \sigma, \tau, \tau^\prime)$ be a $(B_{\Lambda r}(z), B_{4r}(z))$-excursion of $P_Z$. Suppose that $P_Z|_{[\sigma,\tau]}$ does not pass through the tube between either $\SCL_{z,r}$ and $\gamma_{z,r}$ or $\SCL_{z,r}$ and $\gamma_{z,r}^\prime$. Then 
    \begin{equation*}
        D_\Upsilon(P_Z(\sigma), P_Z(\tau)) - D_{\Upsilon_Z}(P_Z(\sigma), P_Z(\tau)) \le 100\fa_2\kc_r. 
    \end{equation*}
\end{lemma}

See \Cref{fig:bad} for an illustration of the proof of \Cref{318}. The idea is as follows.  Write $S_{i_z}$ (resp.~$S_{j_z}$) for the first time at which $P_Z|_{[\sigma, \tau]}$ hits $I$ for some $I \in \SI_{i_z}^{\alpha\setminus\beta}$ (resp.~$I \in \SI_{j_z}^{\alpha\setminus\beta}$). Write $T_{i_z}$ (resp.~$T_{j_z}$) for the last time at which $P_Z|_{[\sigma, \tau]}$ hits $J$ for some $J \in \SI_{i_z}^{\alpha\setminus\beta}$ (resp.~$I \in \SI_{j_z}^{\alpha\setminus\beta}$).  We first show that $P_Z(S_{i_z})$, $P_Z(T_{i_z})$, $P_Z(S_{j_z})$, and $P_Z(T_{j_z})$ are contained in the same connected component of $\Upsilon$ (cf.~\eqref{eq:248}). Since it is clear that the endpoints of each small dashed subpath as illustrated in \Cref{fig:bad} are contained in the same connected component of $\Upsilon$, it follows that all the solid subpaths of $P_Z$ are contained in the same connected component of $\Upsilon$. Thus, roughly speaking, $D_\Upsilon(P_Z(\sigma), P_Z(\tau))$ is bounded above by the $D_\Upsilon$-length (or, equivalently, the $D_{\Upsilon_Z}$-length) of the solid subpaths of $P_Z$ plus the $D_\Upsilon$-distance between the endpoints of the dashed subpaths. The $D_\Upsilon$-distance between $P_Z(S_{i_z})$ and $P_Z(T_{i_z})$ (resp.~$P_Z(S_{j_z})$ and $P_Z(T_{j_z})$) is upper-bounded using \Cref{098}, \eqref{098B}. The $D_\Upsilon$-distances between the endpoints of the small dashed subpaths are upper-bounded using \Cref{098}, \eqref{098C}, \eqref{098E}.

\begin{figure}[ht!]
    \centering
    \includegraphics[width=\linewidth]{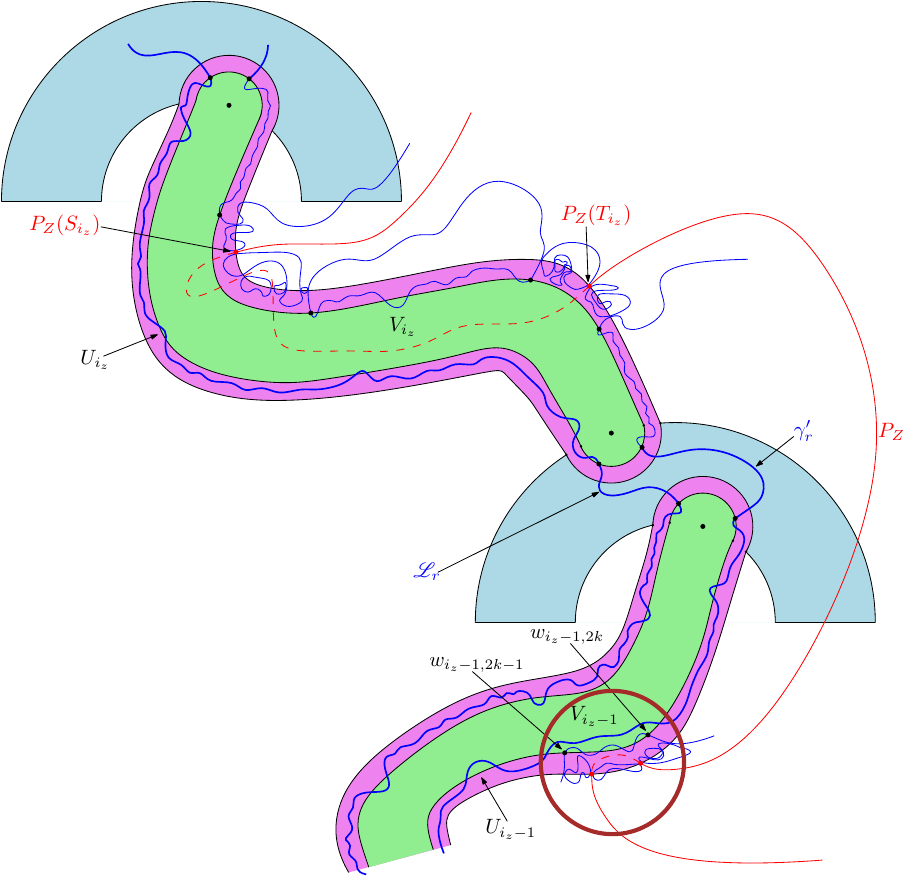}
    \caption{Illustration of the proof of \Cref{318}. The red path $P_Z$ is a $D_{\Upsilon_Z}$-geodesic that enters the ``places'' in $B_{5r}(z)$ where $\Upsilon$ and $\Upsilon_Z$ differ (the dashed subpaths of $P_Z$) but does not passes through the tube between $\SCL_{z,r}$, $\gamma_{z,r}$, and $\gamma_{z,r}^\prime$. In this situation, there are three different types of such ``bad'' subpaths: $P_Z|_{[S_{i_z}, T_{i_z}]}$, $P_Z|_{[S_{j_z}, T_{j_z}]}$ (not shown), and the small dashed subpaths that enters the small ``bad places'' corresponding to $w_{n,2k - 1}$ and $w_{n,2k}$ for some $n \in [1, \#Z_{z,r}]_\ZZ \setminus \{i_z, j_z\}$ and $k \in \NN$. (Recall from \Cref{fig:linkpattern} the definition of the points $w_{n,2k - 1}$ and $w_{n,2k}$.)}
    \label{fig:bad}
\end{figure}

\begin{proof}[Proof of \Cref{318}]
    Let $i_z$ and $j_z$ be as in \Cref{210}. Since $P_Z$ does not pass through the tube between either $\SCL_{z,r}$ and $\gamma_{z,r}$ or $\SCL_{z,r}$ and $\gamma_{z,r}^\prime$, we observe that for each $U_{z,r}^\star$-excursion $(S, T)$ of $P_Z|_{[\sigma, \tau]}$, there are three possibilities:
    \begin{enumerate}
        \item There exist $I, J \in \SI_{i_z}^{\alpha\setminus\beta}$ such that $P_Z(S) \in I \cap \partial U_{z,r}$ and $P_Z(T) \in J \cap \partial U_{z,r}$. 
        \item There exist $I, J \in \SI_{j_z}^{\alpha\setminus\beta}$ such that $P_Z(S) \in I \cap \partial U_{z,r}$ and $P_Z(T) \in J \cap \partial U_{z,r}$. 
        \item\label{it:307} There exists $j \in [1, \#Z_r]_\ZZ \setminus \{i_z, j_z\}$ and $I \in \SI_j^{\alpha\setminus\beta}$ such that $P_Z(S), P_Z(T) \in I \cap \partial U_{z,r}$. 
    \end{enumerate}
As was introduced just after the statement of the lemma, write $S_{i_z}$ (resp.~$S_{j_z}$) for the first time at which $P_Z|_{[\sigma, \tau]}$ hits $I$ for some $I \in \SI_{i_z}^{\alpha\setminus\beta}$ (resp.~$I \in \SI_{j_z}^{\alpha\setminus\beta}$). Write $T_{i_z}$ (resp.~$T_{j_z}$) for the last time at which $P_Z|_{[\sigma, \tau]}$ hits $J$ for some $J \in \SI_{i_z}^{\alpha\setminus\beta}$ (resp.~$I \in \SI_{j_z}^{\alpha\setminus\beta}$). We \emph{claim} that
    \begin{equation}\label{eq:248}
        P_Z(S_{i_z}) \xleftrightarrow\Upsilon P_Z(T_{i_z}) \xleftrightarrow\Upsilon P_Z(S_{j_z}) \xleftrightarrow\Upsilon P_Z(T_{j_z}). 
    \end{equation}
    Note that either $S_{i_z} < T_{i_z} < S_{j_z} < T_{j_z}$ or $S_{i_z} < S_{j_z} < T_{i_z} < T_{j_z}$. We only consider the first case, the second case follows from a similar argument. First, we note that the restriction of $P_Z|_{[\sigma^\prime, \tau^\prime]}$ to the complement of the union of the intervals $[S, T]$ for all $U_{z,r}^\star$-excursions $(S, T)$ of $P_Z|_{[\sigma, \tau]}$ is contained in $\Upsilon$. Moreover, for any times $S$ and $T$ such that condition~\eqref{it:307} holds, it is clear that $P_Z(S) \xleftrightarrow\Upsilon P_Z(T)$. Thus, we conclude that 
    \begin{equation*}
        P_Z|_{[\sigma^\prime, \sigma]} \xleftrightarrow\Upsilon P_Z(S_{i_z}), \quad P_Z(T_{i_z}) \xleftrightarrow\Upsilon P_Z(S_{j_z}), \quad \text{and} \quad P_Z(T_{j_z}) \xleftrightarrow\Upsilon P_Z|_{[\tau, \tau^\prime]}.  
    \end{equation*}
    Since both $P_Z|_{[\sigma^\prime, \sigma]}$ and $P_Z|_{[\tau, \tau^\prime]}$ intersect the path $P_z$ described in \Cref{210}, \eqref{210D}, it follows that $P_Z(S_{i_z}) \xleftrightarrow\Upsilon P_Z(T_{j_z}) \xleftrightarrow\Upsilon P_z$. Thus, it suffices to show that $P_Z(S_{i_z}) \xleftrightarrow\Upsilon P_Z(T_{i_z})$. Suppose by way of contradiction that this is not true. Since the Euclidean diameter of $U_{i_z}$ is at most $100\fa_4r$ (cf.~\Cref{098}, \eqref{098D}), it follows that $\lvert P_Z(S_{i_z}) - P_Z(T_{i_z})\rvert \le 100\fa_4r$. Then, by \Cref{098}, \eqref{098C}, there exists a loop $\SCL \in \Gamma$ with Euclidean diameter at most $\fa_3r/2$, contained in $A_{r,5r}(z)$, and disconnecting $P_Z(S_{i_z})$ and $P_Z(T_{i_z})$. Since $P_Z(S_{i_z})$ is connected to $P_Z|_{[\sigma^\prime, \sigma]}$ in $\Upsilon$, the loop $\SCL$ cannot surround $P_Z(S_{i_z})$. Thus, $\SCL$ must surround $P_Z(T_{i_z})$. However, since $P_Z(T_{i_z}) \xleftrightarrow\Upsilon P_Z(S_{j_z})$, and, by \Cref{329}, the Euclidean distance between $P_Z(T_{i_z})$ and $P_Z(S_{j_z})$ is greater than $\fa_3r/2$, in contradiction to the fact that $P_Z(T_{i_z})$ is surrounded by a loop of $\Gamma$ with Euclidean diameter at most $\fa_3r/2$. This completes the proof of \eqref{eq:248}. 
    
    Let $\CS \subset A_{2r,4r}(0) \times (0, \lambda r]$ be as in \Cref{098}, \eqref{098E}. For each $(w, t) \in \CS$, since there are at most $2N$ connected arcs of $\Gamma$ in $A_{t,2t}(w)$ connecting its inner and outer boundaries, there are at most $N$ $(B_{2t}(w), B_t(w))$-excursions of $\Gamma$. These $(B_{2t}(w), B_t(w))$-excursions divide $B_{2t}(w)$ into at most $N + 1$ connected components. Among these $N + 1$ connected components, at most $N$ of them, which we shall denote by $A_{w,t,1}, \ldots, A_{w,t,N}$, correspond to $\Upsilon$, i.e., $A_{w,t,j} \cap B_t(w) \cap \Upsilon \xleftrightarrow\Upsilon (\CC \setminus B_{2t}(w)) \cap \Upsilon$ for all $j$. Suppose that $S$ and $T$ are times such that condition~\eqref{it:307} holds. It is clear that $P_Z(S)$ and $P_Z(T)$ are contained in the same connected component of $(U_{z,r} \setminus V_{z,r}) \cap \Upsilon$. Denote this connected component by $\SCC$. Since $P_Z$ started from outside of $B_{4r}(z)$, it follows that $\SCC \xleftrightarrow\Upsilon (\CC \setminus B_{4r}(z)) \cap \Upsilon$. Thus, it follows from \Cref{098}, \eqref{098E} that there exists $(w, t)$ such that $\SCC \subset B_t(w)$, and it follows from the above discussion that there exists $j \in [1, N]_\ZZ$ such that $\SCC \subset A_{w,t,j}$. We \emph{observe} that for each pair of points $x, y \in A_{w,t,j} \cap \Upsilon$ with $x \xleftrightarrow\Upsilon (\CC \setminus B_{2t}(w)) \cap \Upsilon$ and $y \xleftrightarrow\Upsilon (\CC \setminus B_{2t}(w)) \cap \Upsilon$, we have $x \xleftrightarrow{B_{2t}(w) \cap \Upsilon} y$. For each $(w, t) \in \CS$ and $j \in [1, N]_\ZZ$, write $S_{w,t,j}$ (resp.~$T_{w,t,j}$) for the first (resp.~last) time at which $P_Z$ hits a connected component $\SCC$ of $(U_{z,r} \setminus V_{z,r}) \cap \Upsilon$ with $\SCC \subset B_t(w)$. It follows from the above \emph{observation} that 
    \begin{equation}\label{eq:250}
        P_Z(S_{w,t,j}) \xleftrightarrow{B_{2t}(w) \cap \Upsilon} P_Z(T_{w,t,j}), \quad \forall (w, t) \in \CS, \ \forall j \in [1, N]_\ZZ. 
    \end{equation}

    Now, we conclude from the discussion of the preceding paragraphs that for each $U_{z,r}^\star$-excursion $(S, T)$ of $P_Z|_{[\sigma, \tau]}$, we have
    \begin{equation*}
        [S, T] \subset [S_{i_z}, T_{i_z}] \cup [S_{j_z}, T_{j_z}] \cup \bigcup_{(w, t) \in \CS} \bigcup_{j = 1}^N [S_{w,t,j}, T_{w,t,j}]. 
    \end{equation*}
    This implies that
    \begin{multline}\label{eq:348}
        D_\Upsilon(P_Z(\sigma), P_Z(\tau)) \le D_{\Upsilon_Z}(P_Z(\sigma), P_Z(\tau)) \\
        + D_\Upsilon(P_Z(S_{i_z}), P_Z(T_{i_z})) + D_\Upsilon(P_Z(S_{j_z}), P_Z(T_{j_z})) + \sum_{(w, t) \in \CS} \sum_{j = 1}^N D_\Upsilon(P_Z(S_{w,t,j}), P_Z(T_{w,t,j})). 
    \end{multline}
    Since $P_Z(S_{i_z}), P_Z(T_{i_z}) \in \partial U_{i_z}$, we have $\lvert P_Z(S_{i_z}) - P_Z(T_{i_z})\rvert \le 100\fa_4r$ (cf.~\Cref{098}, \eqref{098D}). Since $P_Z(S_{i_z}) \xleftrightarrow\Upsilon P_Z(T_{i_z})$ by \eqref{eq:248}, it follows from \Cref{098}, \eqref{098B} that 
    \begin{equation}\label{eq:251}
        D_\Upsilon(P_Z(S_{i_z}), P_Z(T_{i_z})) \le \fa_2\kc_r. 
    \end{equation}
    Similarly, we have
    \begin{equation}\label{eq:252}
        D_\Upsilon(P_Z(S_{j_z}), P_Z(T_{j_z})) \le \fa_2\kc_r. 
    \end{equation}
    By \Cref{098}, \eqref{098E} and \eqref{eq:250}, we have 
    \begin{multline}\label{eq:253}
        \sum_{(w, t) \in \CS} \sum_{j = 1}^N D_\Upsilon(P_Z(S_{w,t,j}), P_Z(T_{w,t,j})) \\
        \le N\sum_{(w, t) \in \CS} \sup\left\{D_\Upsilon(x, y) : x, y \in B_t(w), \ x \xleftrightarrow{B_{2t}(w) \cap \Upsilon} y\right\} \le \fa_2\kc_r. 
    \end{multline}
    Combining \eqref{eq:348}, \eqref{eq:251}, \eqref{eq:252}, and \eqref{eq:253}, we complete the proof. 
\end{proof}

The following lemma is the main input in the proof of \Cref{201}. 

\begin{lemma}\label{205}
    There is a deterministic constant $c \in (0, 1)$ such that for each $\Be \in (\ZZ/(300\Lambda\ZZ))^2$, $r \in \rho^{-1}\SR$, and $Z \subset \ZZ_{\Be,r}^{\rr}(K_1, K_2)$, the following is true: Suppose that the event $F_Z$ occurs. Then there is a subset $Z^\prime \subset Z$ with $\#Z^\prime \ge c\#Z$ such that the event $\kF_{Z^\prime}^{\rr}(K_1, K_2)$ occurs.
\end{lemma}

\begin{proof}
    Set $Z_0 \defeq Z$. Inductively, for each $n \in \NN$, set
    \begin{equation*}
        Z_n \defeq \left\{z \in Z_{n - 1} : P_{Z_{n - 1}} \text{ passes through the tube between either } \SCL_{z,r} \text{ and } \gamma_{z,r} \text{ or } \SCL_{z,r} \text{ and } \gamma_{z,r}^\prime\right\},
    \end{equation*}
    where $P_{Z_{n - 1}}$ denotes the $D_{\Upsilon_{Z_{n - 1}}}$-geodesic from $K_1$ to $K_2$. It suffices to show that there exists a deterministic constant $c \in (0, 1)$ such that $\#Z_n \ge c\#Z$ for all $n \in \NN$. Since $F_Z$ occurs, it follows immediately from the definitions that $F_{Z^\prime}$ occurs for all subsets $Z^\prime \subset Z$. Since $\Upsilon_{Z_n}$ and $\Upsilon_{Z_{n - 1}}$ coincide outside of $\bigcup_{z \in Z_{n - 1} \setminus Z_n} U_{z,r}^\star$, it follows that
    \begin{multline}\label{eq:223}
        D_{\Upsilon_{Z_n}}(K_1, K_2) \le D_{\Upsilon_{Z_{n - 1}}}(K_1, K_2) \\
        + \sum_{z \in Z_{n - 1} \setminus Z_n} \sum_{(\sigma^\prime, \sigma, \tau, \tau^\prime) \in \SE_{z,r}(P_{Z_{n - 1}})} \left(D_{\Upsilon_{Z_n}}(P_{Z_{n - 1}}(\sigma), P_{Z_{n - 1}}(\tau)) - D_{\Upsilon_{Z_{n - 1}}}(P_{Z_{n - 1}}(\sigma), P_{Z_{n - 1}}(\tau))\right)
    \end{multline}
    for all $n \in \NN$, where $\SE_{z,r}(P_{Z_{n - 1}})$ denotes the collection of $(B_{\Lambda r}(z), B_{4r}(z))$-excursions of $P_{Z_{n - 1}}$. By \Cref{318}, we have
    \begin{multline}\label{eq:224}
        D_{\Upsilon_{Z_n}}(P_{Z_{n - 1}}(\sigma), P_{Z_{n - 1}}(\tau)) - D_{\Upsilon_{Z_{n - 1}}}(P_{Z_{n - 1}}(\sigma), P_{Z_{n - 1}}(\tau)) \le 100\fa_2\kc_r, \\
        \forall n \in \NN, \ \forall z \in Z_{n - 1} \setminus Z_n, \ \forall (\sigma^\prime, \sigma, \tau, \tau^\prime) \in \SE_{z,r}(P_{Z_{n - 1}}).
    \end{multline}
    By \Cref{220}, we have
    \begin{equation}\label{eq:225}
        \#\SE_{z,r}(P_{Z_{n - 1}}) \le \fA/\fa_1, \quad \forall n \in \NN, \ \forall z \in Z_{n - 1}. 
    \end{equation}
    Combining \eqref{eq:223}, \eqref{eq:224}, and \eqref{eq:225}, we obtain
    \begin{equation*}
        D_{\Upsilon_{Z_n}}(K_1, K_2) \le D_{\Upsilon_{Z_{n - 1}}}(K_1, K_2) + \frac{100\fA\fa_2}{\fa_1}\kc_r(\#Z_{n - 1} - \#Z_n), \quad \forall n \in \NN. 
    \end{equation*}
    Thus, 
    \begin{align}\label{eq:226}
        D_{\Upsilon_{Z_n}}(K_1, K_2) - D_{\Upsilon_Z}(K_1, K_2) &= \sum_{k = 1}^n \left(D_{\Upsilon_{Z_k}}(K_1, K_2) - D_{\Upsilon_{Z_{k - 1}}}(K_1, K_2)\right) \\
        &\le \sum_{k = 1}^n \frac{100\fA\fa_2}{\fa_1}\kc_r(\#Z_{k - 1} - \#Z_k) \notag \\
        &\le \frac{100\fA\fa_2}{\fa_1}\kc_r\#Z, \quad \forall n \in \NN. \notag 
    \end{align}
    Combining \Cref{317} and \eqref{eq:226}, we obtain
    \begin{equation}\label{eq:227}
        D_{\Upsilon_{Z_n}}(K_1, K_2) \le D_\Upsilon(K_1, K_2) - \left(\fa_1 - \frac{100\fA\fa_2}{\fa_1}\right)\kc_r\#Z, \quad \forall n \in \NN. 
    \end{equation}
    On the other hand, by \Cref{221}, we have
    \begin{equation}\label{eq:228}
        D_\Upsilon(K_1, K_2) \le D_{\Upsilon_{Z_n}}(K_1, K_2) + (\fA^2/\fa_1)\kc_r\#Z_n, \quad \forall n \in \NN. 
    \end{equation}
    Combining \eqref{eq:227} and \eqref{eq:228}, we obtain
    \begin{equation*}
        \#Z_n \ge \frac{\fa_1}{\fA^2}\left(\fa_1 - \frac{100\fA\fa_2}{\fa_1}\right)\#Z, \quad \forall n \in \NN. 
    \end{equation*}
    (Here, we recall from the proof of \Cref{132} that we may have $\fa_1 > 100\fA\fa_2/\fa_1$.) This completes the proof. 
\end{proof}

\begin{proof}[Proof of \Cref{201}]
    Suppose that the event $\kG_\varepsilon^{\rr}(K_1, K_2)$ occurs. It follows immediately from \eqref{eq:340} and \Cref{222}, \eqref{222A} that condition~\eqref{210A} of \Cref{210} holds for all $r \in \rho^{-1}\SR \cap [\varepsilon^{1 + \nu}\rr, \varepsilon\rr]$. 
    
    Since $\dist(K_1, K_2) \ge \eta\rr$ and $\dist(K_1, \partial B_{2\rr}(0)) \ge \eta\rr$, there exists a subpath of $P$ with Euclidean diameter at least $\eta\rr$ that is contained in $B_{2\rr}(0)$. Thus, by \Cref{222}, \eqref{222B}, there exists a deterministic constant $\delta > 0$, depending only on $\eta$, such that there exists a set $\CS$ of pairs $(z, r)$ satisfying the following conditions:
    \begin{enumerate}
        \item $\#\CS \ge \delta/\varepsilon$. 
        \item For each $(z, r) \in \CS$, we have $r \in \rho^{-1}\SR \cap [\varepsilon^{1 + \nu}\rr, \varepsilon\rr]$ and $z \in \left(\frac1{100}r\ZZ\right)^2 \cap B_{2\rr}(0)$. 
        \item For each $(z, r) \in \CS$, conditions~\eqref{210B}, \eqref{210C}, \eqref{210D}, \eqref{210E} of \Cref{210} hold. 
    \end{enumerate}
    By the pigeonhole principle and \eqref{eq:200}, if $\varepsilon \le \rho$, then there exists $\Be \in (\ZZ/(300\Lambda\ZZ))^2$ and $r \in \rho^{-1}\SR \cap [\varepsilon^{1 + \nu}\rr, \varepsilon\rr]$ such that 
    \begin{equation}\label{eq:232}
        \#\left\{z \in \ZZ_{\Be,r}^{\rr}(K_1, K_2) : (z, r) \in \CS\right\} \ge \frac\delta\varepsilon \cdot \frac1{(300\Lambda)^2} \cdot \frac1{2\nu\log_\Lambda(1/\varepsilon)}
    \end{equation}
    Henceforth fix $\Be \in (\ZZ/(300\Lambda\ZZ))^2$ and $r \in \rho^{-1}\SR \cap [\varepsilon^{1 + \nu}\rr, \varepsilon\rr]$ so that \eqref{eq:232} holds. Write 
    \begin{equation*}
        \SZ_1 \defeq \{z \in \ZZ_{\Be,r}^{\rr}(K_1, K_2) : (z, r) \in \CS\}. 
    \end{equation*}
    By the above discussion, the event $E_Z$ occurs for all subsets $Z \subset \SZ_1$ (i.e., conditions~\eqref{210A}, \eqref{210B}, \eqref{210C}, \eqref{210D}, \eqref{210E} of \Cref{210} hold for all subsets $Z \subset \SZ_1$). Write 
    \begin{equation*}
        \SZ_2 \defeq \left\{z \in \SZ_1 : \fF_{z,r}(\Upsilon_Z; \{i_z, j_z\}) \text{ occurs}\right\}. 
    \end{equation*}
    Since the set $\SZ_1$ is almost surely determined by $\Upsilon$, it follows from \eqref{eq:286} and \eqref{eq:232} that
    \begin{equation}\label{eq:233}
        \BE\lbrack\#\SZ_2 \mid \Upsilon\rbrack \ge \fa_9\#\SZ_1 \ge \fa_9 \cdot \frac\delta\varepsilon \cdot \frac1{(300\Lambda)^2} \cdot \frac1{2\nu\log_\Lambda(1/\varepsilon)} \quad \text{almost surely}. 
    \end{equation}
    By definition, the event $F_Z$ occurs for all subsets $Z \subset \SZ_2$. Thus, by \Cref{205}, there is a mapping $Z \mapsto Z^\prime$ from the collection of subsets of $\SZ_2$ to itself such that $\#Z^\prime \ge c\#Z$ and $\kF_{Z^\prime}^{\rr}(K_1, K_2)$ occurs for all subsets $Z \subset \SZ_2$. Thus, 
    \begin{align}\label{eq:234}
        &\#\left\{Z \subset \SZ_2 : \#Z \le N, \ \kF_Z^{\rr}(K_1, K_2) \text{ occurs}\right\} \\
        &\ge \#\left\{Z \subset \SZ_2 : cN \le \#Z \le N, \ \kF_Z^{\rr}(K_1, K_2) \text{ occurs}\right\} \notag \\
        &\ge \#\{Z \subset \SZ_2 : \#Z = N\}\left(\max_{\substack{Z_0 \subset \SZ_2 \\ cN \le \#Z_0 \le N}} \#\{Z \subset \SZ_2 : \#Z = N, \ Z^\prime = Z_0\}\right)^{-1} \notag \\
        &\ge \binom{\#\SZ_2}N \binom{\#\SZ_2}{(1 - c)N}^{-1} \notag \\
        &\succeq (\#\SZ_2)^{cN}. \notag 
    \end{align}
    Combining \eqref{eq:233} and \eqref{eq:234}, we complete the proof. 
\end{proof}

\subsection{Proof of \Cref{204}}\label{ss:19}

The proof of \Cref{204} relies on estimating how many points $z \in \ZZ_{\Be,r}^{\rr}(K_1, K_2)$ could belong to some set $Z \subset \ZZ_{\Be,r}^{\rr}(K_1, K_2)$ for which the event $\kF_Z^{\rr}(K_1, K_2)^\vee$ occurs. For this purpose, we introduce the following definition.

Let $\Be \in (\ZZ/(300\Lambda\ZZ))^2$ and $r \in \rho^{-1}\SR$. Then we shall refer to a point $z \in \ZZ_{\Be,r}^{\rr}(K_1, K_2)$ as $(\Be, r)$-\emph{good} if the following are true:
\begin{enumerate}[label=(\Alph*), ref=\Alph*]
    \item\label{it:200} The event $\fE_{z,r}(\Upsilon_{z,r})$ occurs.
    \item\label{it:201} There are at most two connected arcs of $\Gamma_{z,r}$ in $A_{5r,\Lambda r}(z)$ connecting its inner and outer boundaries. 
    \item\label{it:202} There exists a path $P_z$ in $A_{5r,\Lambda r}(z) \cap \Upsilon_{z,r}$ of $D_{\Upsilon_{z,r}}$-length at most $\fA\kc_r$ such that every path in $\Upsilon_{z,r}$ that crosses between the inner and outer boundaries of $A_{5r,\Lambda r}(z)$ intersects $P_z$. 
    \item\label{it:203} The event $\fF_{z,r}(\Upsilon; \{i_z, j_z\})$ occurs for some distinct $i_z, j_z \in [1, \#Z_{z,r}]_\ZZ$.
    \item\label{it:204} The $D_\Upsilon$-geodesic $P$ passes through the tube between either $\SCL_{z,r}$ and $\gamma_{z,r}$ or $\SCL_{z,r}$ and $\gamma_{z,r}^\prime$. 
\end{enumerate}

\begin{lemma}\label{207}
    Let $\Be \in (\ZZ/(300\Lambda\ZZ))^2$, $r \in \rho^{-1}\SR$, and $Z \subset \ZZ_{\Be,r}^{\rr}(K_1, K_2)$. Suppose that $\kF_Z^{\rr}(K_1, K_2)^\vee$ occurs. Then each $z \in Z$ is $(\Be, r)$-good. 
\end{lemma}

\begin{proof}
    We observe that conditions~\eqref{it:200}, \eqref{it:201}, \eqref{it:202}, \eqref{it:203}, \eqref{it:204} correspond to conditions~\eqref{330B}, \eqref{330C}, \eqref{330D}, \eqref{330F}, \eqref{330G} of \Cref{330}, respectively. Moreover, conditions~\eqref{it:200}, \eqref{it:201}, \eqref{it:202}, \eqref{it:203}, \eqref{it:204} are almost surely determined by $\Upsilon$ and $B_{\Lambda r}(z) \cap \Upsilon_{z,r}$; conditions~\eqref{330B}, \eqref{330C}, \eqref{330D}, \eqref{330F}, \eqref{330G} in \Cref{330} are almost surely determined by $\Upsilon$ and $B_{\Lambda r}(z) \cap \Upsilon_Z$. Thus, \Cref{207} follows immediately from the fact that $B_{\Lambda r}(z) \cap \Upsilon_{z,r} = B_{\Lambda r}(z) \cap \Upsilon_Z$ almost surely.
\end{proof}

The next lemma provides the key ingredient used in the proof of \Cref{204}.

\begin{lemma}\label{209}
    There exists a deterministic constant $C > 0$ such that for each $\Be \in (\ZZ/(300\Lambda\ZZ))^2$ and $r \in \rho^{-1}\SR$, the following is true: Suppose that there are disjoint subsets $Z, Z^\prime \subset \ZZ_{\Be,r}^{\rr}(K_1, K_2)$ such that $\kF_Z^{\rr}(K_1, K_2)^\vee$ occurs and each $z^\prime \in Z^\prime$ is $(\Be, r)$-good. Then $\#Z^\prime \le C\#Z$. 
\end{lemma}

The idea of the proof of \Cref{209} is as follows. Let $P$ be a $D_\Upsilon$-geodesic from $K_1$ to $K_2$. By \Cref{096} (with the roles of $\Upsilon$ and $\Upsilon_{z,r}$ interchanged) together with condition~\eqref{it:204}, the portion of $P$ spent inside the Euclidean balls centered at points $z' \in Z'$ is forced to run near certain ``shortcuts'' so that the $\widetilde D_\Upsilon$-length of this portion falls short of $M^\ast$ times its $D_\Upsilon$-length by at least a deterministic positive constant times $\#Z'$. Since $Z$ and $Z'$ are disjoint, $\Upsilon$ and $\Upsilon_Z$ agree on the balls centered at $z' \in Z'$. Hence the same conclusion applies to the segments of $P$ that lie in $\Upsilon_Z$ (i.e., outside the balls centered at $z \in Z$). We now replace the segments of $P$ lying inside the balls centered at $z \in Z$ by the loops provided by condition~\eqref{330D} of the event $\kF_Z^{\rr}(K_1,K_2)^\vee$, producing a path entirely contained in $\Upsilon_Z$. A parallel estimate shows that the $\widetilde D_{\Upsilon_Z}$-length of this modified path exceeds the $\widetilde D_\Upsilon$-length of $P$ by at most a deterministic positive constant multiple of $\#Z$. Moreover, $D_\Upsilon(K_1,K_2) \le D_{\Upsilon_Z}(K_1,K_2)$ (indeed, by the same reasoning as in \Cref{317}, $D_\Upsilon(K_1,K_2)$ is at most $D_{\Upsilon_Z}(K_1,K_2)$ minus a deterministic positive constant times $\#Z$). Combining these observations yields
\begin{equation*}
    \widetilde D_{\Upsilon_Z}(K_1,K_2) \le M^\ast D_{\Upsilon_Z}(K_1,K_2) + \text{const}\cdot \#Z - \text{const}\cdot \#Z'. 
\end{equation*}
Together with condition~\eqref{330A} of $\kF_{Z_0}^{\rr}(K_1,K_2)^\vee$, this implies that $\#Z'$ is at most a constant multiple of $\#Z$.

\begin{proof}[Proof of \Cref{209}]
    By \Cref{096} (applied with the roles of $\Upsilon$ and $\Upsilon_{z,r}$ interchanged) and condition~\eqref{it:204}, for each $z^\prime \in Z^\prime$, there are times $\sigma_{z^\prime} < \tau_{z^\prime}$ such that 
    \begin{multline}\label{eq:209}
        P([\sigma_{z^\prime}, \tau_{z^\prime}]) \subset B_{4r}(z) \cap \Upsilon, \quad D_\Upsilon(P(\sigma_{z^\prime}), P(\tau_{z^\prime})) \ge \fb\kc_r, \\
        \text{and} \quad \widetilde D_\Upsilon(P(\sigma_{z^\prime}), P(\tau_{z^\prime})) \le M_2 D_\Upsilon(P(\sigma_{z^\prime}), P(\tau_{z^\prime})). 
    \end{multline}
    On the other hand, by a similar argument to the argument applied in the proof of \Cref{221}, for each $z \in Z$ and $(\sigma^\prime, \sigma, \tau, \tau^\prime) \in \SE_{z,r}(P)$, there are times $\sigma^\pprime \in (\sigma^\prime, \sigma)$ and $\tau^\pprime \in (\tau, \tau^\prime)$ such that
    \begin{equation}\label{eq:210}
        \sum_{z \in Z} \sum_{(\sigma^\prime, \sigma, \tau, \tau^\prime) \in \SE_{z,r}(P)} D_{\Upsilon_Z}(P(\sigma^\pprime), P(\tau^\pprime)) \le (\fA^2/\fa_1)\kc_r\#Z,
    \end{equation}
    where $\SE_{z,r}(P)$ denotes the collection of $(B_{\Lambda r}(z), B_{4r}(z))$-excursions of $P$. Since $Z$ and $Z^\prime$ are disjoint, it follows that 
    \begin{equation*}
        [\sigma^\pprime, \tau^\pprime] \cap [\sigma_{z^\prime}, \tau_{z^\prime}] = \emptyset, \quad \forall z \in Z, \ \forall (\sigma^\prime, \sigma, \tau, \tau^\prime) \in \SE_Z, \ \forall z^\prime \in Z^\prime. 
    \end{equation*}
    Write
    \begin{equation*}
        I \defeq [0, 1] \setminus \left(\bigcup_{z \in Z} \bigcup_{(\sigma^\prime, \sigma, \tau, \tau^\prime) \in \SE_Z} [\sigma^\pprime, \tau^\pprime] \cup \bigcup_{z^\prime \in Z^\prime} [\sigma_{z^\prime}, \tau_{z^\prime}]\right). 
    \end{equation*}
    Combining \eqref{eq:209} and \eqref{eq:210}, we obtain that
    \begin{align}\label{eq:211}
        &\widetilde D_{\Upsilon_Z}(K_1, K_2) \\
        &\le \sum_{z \in Z} \sum_{(\sigma^\prime, \sigma, \tau, \tau^\prime) \in \SE_{z,r}(P)} \widetilde D_{\Upsilon_Z}(P(\sigma^\pprime), P(\tau^\pprime)) + \sum_{z^\prime \in Z^\prime} \widetilde D_{\Upsilon_Z}(P(\sigma_{z^\prime}), P(\tau_{z^\prime})) + \len(P|_I, \widetilde D_{\Upsilon_Z}) \notag \\
        &= \sum_{z \in Z} \sum_{(\sigma^\prime, \sigma, \tau, \tau^\prime) \in \SE_{z,r}(P)} \widetilde D_{\Upsilon_Z}(P(\sigma^\pprime), P(\tau^\pprime)) + \sum_{z^\prime \in Z^\prime} \widetilde D_\Upsilon(P(\sigma_{z^\prime}), P(\tau_{z^\prime})) + \len(P|_I, \widetilde D_\Upsilon) \notag \\
        &\le M^\ast \sum_{z \in Z} \sum_{(\sigma^\prime, \sigma, \tau, \tau^\prime) \in \SE_{z,r}(P)} D_{\Upsilon_Z}(P(\sigma^\pprime), P(\tau^\pprime)) + M_2 \sum_{z^\prime \in Z^\prime} D_\Upsilon(P(\sigma_{z^\prime}), P(\tau_{z^\prime})) \notag \\
        &+ M^\ast \len(P|_I, D_\Upsilon) \notag \\
        &\le M^\ast(\fA^2/\fa_1)\kc_r\#Z + M_2 \sum_{z^\prime \in Z^\prime} D_\Upsilon(P(\sigma_{z^\prime}), P(\tau_{z^\prime})) + M^\ast \len(P|_I, D_\Upsilon) \notag \\
        &\le M^\ast(\fA^2/\fa_1)\kc_r\#Z + M^\ast D_\Upsilon(K_1, K_2) - (M^\ast - M_2) \sum_{z^\prime \in Z^\prime} D_\Upsilon(P(\sigma_{z^\prime}), P(\tau_{z^\prime})) \notag \\
        &\le M^\ast(\fA^2/\fa_1)\kc_r\#Z + M^\ast D_\Upsilon(K_1, K_2) - (M^\ast - M_2)\fb\kc_r\#Z^\prime. \notag
    \end{align}
    Since $D_\Upsilon(K_1, K_2) \le D_{\Upsilon_Z}(K_1, K_2)$ on the event $\kF_Z^{\rr}(K_1, K_2)^\vee$ (cf., e.g., \Cref{317} (applied with the roles of $\Upsilon$ and $\Upsilon_Z$ interchanged)), \eqref{eq:211} implies that
    \begin{equation}\label{eq:236}
        \widetilde D_{\Upsilon_Z}(K_1, K_2) \le M^\ast(\fA^2/\fa_1)\kc_r\#Z + M^\ast D_{\Upsilon_Z}(K_1, K_2) - (M^\ast - M_2)\fb\kc_r\#Z^\prime. 
    \end{equation}
    Combining \Cref{330}, \eqref{330A} and \eqref{eq:236}, we obtain
    \begin{equation*}
        (M^\ast - M_2)\fb\kc_r\#Z^\prime \le M^\ast(\fA^2/\fa_1)\kc_r\#Z + \kc_r \le \left(M^\ast\fA^2/\fa_1 + 1\right)\kc_r\#Z. 
    \end{equation*}
    This completes the proof. 
\end{proof}

\begin{proof}[Proof of \Cref{204}]
    Fix $\Be \in (\ZZ/(300\Lambda\ZZ))^2$, $r \in \rho^{-1}\SR$, and $N \in \NN$. Choose $Z_0 \subset \ZZ_{\Be,r}^{\rr}(K_1, K_2)$ with $\#Z_0 \le N$ such that $\kF_{Z_0}^{\rr}(K_1, K_2)^\vee$ occurs. (If such $Z_0$ does not exist, then \eqref{eq:237} holds trivially.) Write 
    \begin{equation*}
        Z_0^\prime \defeq \left\{z \in \ZZ_{\Be,r}^{\rr}(K_1, K_2) \setminus Z_0 : z \text{ is } (\Be, r)\text{-good}\right\}. 
    \end{equation*}
    Then, by \Cref{209}, we have $\#Z_0^\prime \le C\#Z_0$. Let $Z \subset \ZZ_{\Be,r}^{\rr}(K_1, K_2)$ such that $\kF_Z^{\rr}(K_1, K_2)^\vee$ occurs. Then, by \Cref{207}, each $z \in Z$ is $(\Be, r)$-good. Thus, by definition, we must have $Z \subset Z_0 \cup Z_0^\prime$. Thus, 
    \begin{equation*}
        \#\left\{Z \subset \ZZ_{\Be,r}^{\rr}(K_1, K_2) : \#Z \le N \text{ and } \kF_Z^{\rr}(K_1, K_2)^\vee \text{ occurs}\right\} \le 2^{\#Z_0 + \#Z_0^\prime} \le 2^{(1 + C)N}. 
    \end{equation*}
    This completes the proof. 
\end{proof}

\subsection{Proof of \Cref{000}}\label{ss:10}

In the present subsection, we establish \Cref{000}. It is enough to prove that $M_\ast = M^\ast$. To this end, we argue by contradiction and assume that $M_\ast < M^\ast$. Using \Cref{200}, we will then derive a contradiction with \Cref{008,031}. As a first step, we obtain a lower bound for $\#(\rho^{-1}\SR \cap [\varepsilon^{1 + \nu}\rr, \varepsilon\rr])$ under the assumption that $\BP[\underline G_{\rr}(M_1^\prime, b)] \ge b$.

\begin{lemma}\label{105}
    There exists $M_1^\prime = M_1^\prime(\mu, \nu, \Lambda, \alpha) \in (M_\ast, M_1)$ such that for each $b \in (0, 1)$, there exists $\varepsilon_\ast = \varepsilon_\ast(\mu, \nu, \Lambda, \alpha, b, \rho) \in (0, 1)$ such that for each $\rr > 0$ for which $\BP\lbrack\underline G_{\rr}(M_1^\prime, b)\rbrack \ge b$, we have
    \begin{equation*}
        \#\left(\rho^{-1}\SR \cap [\varepsilon^{1 + \nu}\rr, \varepsilon\rr]\right) \ge \mu\log_{\Lambda}(1/\varepsilon), \quad \forall \varepsilon \in (0, \varepsilon_\ast]. 
    \end{equation*}
\end{lemma}

\begin{proof}
    Write $\mu^\prime \defeq (\mu + \nu)/2$. By \Cref{030}, there exists $M_1^\prime = M_1^\prime(\mu, \nu, \Lambda, \alpha) \in (M_\ast, M_1)$ such that for each $b \in (0, 1)$, there exists $\varepsilon_0 = \varepsilon_0(\mu, \nu, \Lambda, \alpha, b) \in (0, 1)$ such that for each $\rr > 0$ for which $\BP\lbrack\underline G_{\rr}(M_1^\prime, b)\rbrack \ge b$, we have
    \begin{equation}\label{eq:326}
        \#\left(\SR \cap [\varepsilon^{1 + \nu}\rr, \varepsilon\rr]\right) \ge \mu^\prime\log_{\Lambda}(1/\varepsilon), \quad \forall \varepsilon \in (0, \varepsilon_0]. 
    \end{equation}
    Since $\SR \subset \{\Lambda^k\}_{k \in \ZZ}$, we conclude from \eqref{eq:326} that
    \begin{align*}
        \#\left(\rho^{-1}\SR \cap [\varepsilon^{1 + \nu}\rr, \varepsilon\rr]\right) &= \#\left(\SR \cap [\varepsilon^{1 + \nu}\rho\rr, \varepsilon\rho\rr]\right) \\
        &= \#\left(\SR \cap [(\varepsilon\rho)^{1 + \nu}\rr, \varepsilon\rho\rr]\right) - \#\left(\SR \cap [(\varepsilon\rho)^{1 + \nu}\rr, \varepsilon^{1 + \nu}\rho\rr]\right) \\
        &\ge \mu^\prime\log_{\Lambda}(1/(\varepsilon\rho)) - \nu\log_{\Lambda}(1/\rho) \\
        &\ge \mu\log_{\Lambda}(1/\varepsilon), \quad \forall \varepsilon \in (0, \varepsilon_0 \wedge \rho]. 
    \end{align*}
    This completes the proof. 
\end{proof}

In the remainder of the present subsection, let $M_1^\prime$ be as in \Cref{105}. We now confirm that the auxiliary condition~\eqref{222B} in the definition of the event 
$\kG_\varepsilon^{\rr}(K_1, K_2)$ holds with high probability for sufficiently small $\varepsilon$.

\begin{lemma}\label{223}
    Let $\alpha_{\mathrm{4A}}$ be as in \eqref{eq:016}. Then we may choose the parameters of \eqref{eq:327} in such a way that 
    \begin{equation}\label{eq:328}
        \alpha_{\mathrm{4A}}\mu > 2(1 + \nu), 
    \end{equation}
    and the following is true: Let $b \in (0, 1)$ and $\rr > 0$. Suppose that $\BP\lbrack\underline G_{\rr}(M_1^\prime, b)\rbrack \ge b$. Then it holds with probability tending to one as $\varepsilon \to 0$ (at a rate which is uniform in $\rr$) that condition~\eqref{222B} in \Cref{222} is true, i.e., for each $x \in B_{2\rr}(0)$, there exists $r \in \rho^{-1}\SR \cap [\varepsilon^{1 + \nu}\rr, \varepsilon\rr]$ and $z \in \left(\frac1{100}r\ZZ\right)^2 \cap B_{2\rr}(0)$ such that $\lvert x - z\rvert \le r$ and the following are true:
    \begin{enumerate}
        \item\label{223A} $\fE_{z,r}(\Upsilon)$ occurs.
        \item\label{223B} There are at most two connected arcs of $\Gamma$ in $A_{5r,\Lambda r}(z)$ connecting its inner and outer boundaries. 
        \item\label{223C} There exists a path $P_z$ in $A_{5r,\Lambda r}(z) \cap \Upsilon$ of $D_\Upsilon$-length at most $\fA\kc_r$ such that every path in $\Upsilon$ that crosses between the inner and outer boundaries of $A_{5r,\Lambda r}(z)$ intersects $P_z$. 
    \end{enumerate}
\end{lemma}

\begin{proof}
    This follows immediately from \Cref{264,263,105}, \Cref{105}, and a union bound. 
\end{proof}

In the remainder of the present subsection, let the parameters of \eqref{eq:327} be as in \Cref{223}. By combining \Cref{200,223}, we obtain the following. 

\begin{lemma}\label{224}
    Let $b, \eta \in (0, 1)$ and $\omega, \rr > 0$. Suppose that $\BP\lbrack\underline G_{\rr}(M_1^\prime, b)\rbrack \ge b$. Then it holds with probability tending to one as $\varepsilon \to 0$ (at a rate which is uniform in $\rr$) that 
    \begin{multline*}
        \widetilde D_\Upsilon(B_{\varepsilon^\omega\rr}(\xx) \cap \Upsilon, B_{\varepsilon^\omega\rr}(\yy) \cap \Upsilon) \le M^\ast D_\Upsilon(B_{\varepsilon^\omega\rr}(\xx) \cap \Upsilon, B_{\varepsilon^\omega\rr}(\yy) \cap \Upsilon) - \varepsilon^{2(1 + \nu)}\kc_{\rr}, \\
        \forall \xx, \yy \in \left(\frac1{100}\varepsilon^\omega\rr\ZZ\right)^2 \cap B_{2\rr}(0) \text{ such that } \lvert\xx - \yy\rvert \ge \eta\rr \text{ and } B_{\varepsilon^\omega\rr}(\xx) \xleftrightarrow\Upsilon B_{\varepsilon^\omega\rr}(\xx). 
    \end{multline*}
\end{lemma}

\begin{proof}
    By applying a union bound, it follows from \Cref{200} that it holds with probability tending to one as $\varepsilon \to 0$ that the event $\kG_\varepsilon^{\rr}\left(B_{\varepsilon^\omega\rr}(\xx), B_{\varepsilon^\omega\rr}(\yy)\right)$ does not occur for all $\xx, \yy \in \left(\frac1{100}\varepsilon^\omega\rr\ZZ\right)^2 \cap B_{2\rr}(0)$ with $\lvert\xx - \yy\rvert \ge \eta\rr$. Combining this with \Cref{223}, we obtain that it holds with probability tending to one as $\varepsilon \to 0$ that condition~\eqref{222A} of the definition of $\kG_\varepsilon^{\rr}\left(B_{\varepsilon^\omega\rr}(\xx), B_{\varepsilon^\omega\rr}(\yy)\right)$ does not occur for all $\xx, \yy \in \left(\frac1{100}\varepsilon^\omega\rr\ZZ\right)^2 \cap B_{2\rr}(0)$ with $\lvert\xx - \yy\rvert \ge \eta\rr$. This completes the proof. 
\end{proof}

Recall the events $\overline G_{\rr}(M, b)$ and $\underline G_{\rr}(M, b)$ from \Cref{025}. Combining \Cref{224} with a few geometric considerations, we derive the following result. 

\begin{lemma}\label{225}
    Let $b, b^\prime \in (0, 1)$ and $\rr > 0$. Suppose that $\BP\lbrack\underline G_{\rr}(M_1^\prime, b)\rbrack \ge b$. Then 
    \begin{equation*}
        \lim_{\delta \to 0} \BP\!\left\lbrack\overline G_{\rr}(M^\ast - \delta, b^\prime)\right\rbrack = 0
    \end{equation*}
    at a rate which is uniform in $\rr$. 
\end{lemma}

\begin{proof}
    Fix $\chi \in (0, 1 - 2/\alpha_{\mathrm{4A}})$ and $\omega > 2(1 + \nu)/\chi$. By \Cref{224}, it holds with probability tending to one as $\varepsilon \to 0$ (at a rate which is uniform in $\rr$) that 
    \begin{multline}\label{eq:238}
        \widetilde D_\Upsilon(B_{\varepsilon^\omega\rr}(\xx) \cap \Upsilon, B_{\varepsilon^\omega\rr}(\yy) \cap \Upsilon) \le M^\ast D_\Upsilon(B_{\varepsilon^\omega\rr}(\xx) \cap \Upsilon, B_{\varepsilon^\omega\rr}(\yy) \cap \Upsilon) - \varepsilon^{2(1 + \nu)}\kc_{\rr}, \\
        \forall \xx, \yy \in \left(\frac1{100}\varepsilon^\omega\rr\ZZ\right)^2 \cap B_{2\rr}(0) \text{ such that } \lvert\xx - \yy\rvert \ge \frac12b^\prime\rr \text{ and } B_{\varepsilon^\omega\rr}(\xx) \xleftrightarrow\Upsilon B_{\varepsilon^\omega\rr}(\xx). 
    \end{multline}
    By \Cref{019}, it holds with probability tending to one as $\varepsilon \to 0$ (at a rate which is uniform in $\rr$) that 
    \begin{equation}\label{eq:239}
        \kc_{\rr}^{-1}D_\Upsilon(x, y) \le \left\lvert\frac{x - y}{\rr}\right\rvert^{\chi}, \quad \forall x, y \in B_{2\rr}(0) \cap \Upsilon \text{ with } x \xleftrightarrow\Upsilon y \text{ and } \lvert x - y\rvert \le \varepsilon\rr. 
    \end{equation}
    Since $\omega > 1 + \nu$ (hence $\alpha_{\mathrm{4A}}(\omega - 1) > 2\omega$ (cf.~\eqref{eq:328})), by \Cref{264}, \eqref{264A} and a union bound, it holds with probability tending to one as $\varepsilon \to 0$ (at a rate which is uniform in $\rr$) that 
    \begin{equation}\label{eq:240}
        \parbox{.85\linewidth}{for each $\xx \in \left(\frac1{100}\varepsilon^\omega\rr\ZZ\right)^2 \cap B_{2\rr}(0)$, there are at most two connected arcs of $\Gamma$ in $A_{\varepsilon^\omega\rr,\varepsilon\rr}(\xx)$ connecting $\partial B_{\varepsilon^\omega\rr}(\xx)$ and $\partial B_{\varepsilon\rr}(\xx)$.}
    \end{equation}
    By \Cref{010}, Axiom~\eqref{010D} (tightness across scales), it holds with probability tending to one as $\varepsilon \to 0$ (at a rate which is uniform in $\rr$) that 
    \begin{equation}\label{eq:241}
        D_\Upsilon(x, y) \le \varepsilon^{-1}\kc_{\rr}, \quad \forall x, y \in B_{2\rr}(0) \cap \Upsilon \text{ with } x \xleftrightarrow\Upsilon y. 
    \end{equation}
    Henceforth assume that $\varepsilon$ is sufficiently small and \eqref{eq:238}, \eqref{eq:239}, \eqref{eq:240}, \eqref{eq:241} hold. Let $x, y \in \overline{B_{\rr}(0)} \cap \Upsilon$ with $x \xleftrightarrow\Upsilon y$ and $\lvert x - y\rvert \ge b^\prime\rr$. Choose $\xx, \yy \in \left(\frac1{100}\varepsilon^\omega\rr\ZZ\right)^2 \cap B_{2\rr}(0)$ such that $x \in B_{\varepsilon^\omega\rr}(\xx)$ and $y \in B_{\varepsilon^\omega\rr}(\yy)$. 

    We \emph{claim} that for any $x^\prime \in B_{\varepsilon^\omega\rr}(\xx)$ and $y^\prime \in B_{\varepsilon^\omega\rr}(\yy)$ such that $x^\prime \xleftrightarrow\Upsilon y^\prime$, we have $x \xleftrightarrow\Upsilon x^\prime$ and $y \xleftrightarrow\Upsilon y^\prime$. Indeed, if the connected component of $\Upsilon$ containing $x$ and $y$ is not the same as the connected component of $\Upsilon$ containing $x^\prime$ and $y^\prime$, then both of these two connected components intersect $\partial B_{\varepsilon^\omega\rr}(\xx)$ and $\partial B_{\varepsilon\rr}(\xx)$. This implies that there are at least four connected arcs of $\Gamma$ in $A_{\varepsilon^\omega\rr,\varepsilon\rr}(\xx)$ connecting $\partial B_{\varepsilon^\omega\rr}(\xx)$ and $\partial B_{\varepsilon\rr}(\xx)$, in contradiction to \eqref{eq:240}. This completes the proof of the \emph{claim}. 

    By \eqref{eq:239} and the above \emph{claim}, we conclude that
    \begin{equation}\label{eq:242}
        \left\lvert D_\Upsilon(B_{\varepsilon^\omega\rr}(\xx) \cap \Upsilon, B_{\varepsilon^\omega\rr}(\yy) \cap \Upsilon) - D_\Upsilon(x, y)\right\rvert \le 2\varepsilon^{\chi\omega}\kc_{\rr}. 
    \end{equation}
    Combining \eqref{eq:238} and \eqref{eq:242}, we obtain that
    \begin{equation*}
        \widetilde D_\Upsilon(x, y) \le M^\ast D_\Upsilon(x, y) - \varepsilon^{2(1 + \nu)}\kc_{\rr} - 2\varepsilon^{\chi\omega}\kc_{\rr}. 
    \end{equation*}
    Since $\chi\omega > 2(1 + \nu)$, it follows that 
    \begin{equation}\label{eq:243}
        \widetilde D_\Upsilon(x, y) \le M^\ast D_\Upsilon(x, y) - 2\varepsilon^{2(1 + \nu)}\kc_{\rr}
    \end{equation}
    whenever $\varepsilon$ is sufficiently small. Combining \eqref{eq:241} and \eqref{eq:243}, we obtain that
    \begin{equation*}
        \widetilde D_\Upsilon(x, y) \le \left(M^\ast - 2\varepsilon^{2(1 + \nu) + 1}\right)D_\Upsilon(x, y). 
    \end{equation*}
    Thus, we conclude that, if $\varepsilon$ is sufficiently small and \eqref{eq:238}, \eqref{eq:239}, \eqref{eq:240}, \eqref{eq:241} are true, then the event $\overline G_{\rr}(M^\ast - 2\varepsilon^{2(1 + \nu) + 1}, b^\prime)$ does not occur. This completes the proof. 
\end{proof}

\begin{proof}[Proof of \Cref{000}]
    It suffices to show that $M_\ast = M^\ast$. Suppose by way of contradiction that $M_\ast < M^\ast$. By applying \Cref{031} with $M_1^\prime$ in place of $M$ and $1$ in place of $\rr$, there exists $\underline{b} \in (0, 1)$ and $\varepsilon_0 \in (0, 1)$ such that for each $\varepsilon \in (0, \varepsilon_0]$, 
    \begin{equation}\label{eq:329}
        \parbox{.80\linewidth}{there are at least $\mu\log_\Lambda(1/\varepsilon)$ values of $\rr \in [\varepsilon^{1 + \nu}, \varepsilon] \cap \{\Lambda^k\}_{k \in \ZZ}$ for which $\BP\lbrack\underline G_{\rr}(M, \underline{b})\rbrack \ge \underline{b}$.}
    \end{equation}
    Let $\overline{b} = b(\mu, \nu, \Lambda)$ be as in \Cref{008}. By applying \Cref{225} with $\underline{b}$ in place of $b$ and $\overline{b}$ in place of $b^\prime$, we conclude from \eqref{eq:329} that there exists $\delta \in (0, M^\ast)$ such that for each $\varepsilon \in (0, \varepsilon_0]$,
    \begin{equation}\label{eq:244}
        \parbox{.80\linewidth}{there are at least $\mu\log_\Lambda(1/\varepsilon)$ values of $\rr \in [\varepsilon^{1 + \nu}, \varepsilon] \cap \{\Lambda^k\}_{k \in \ZZ}$ for which $\BP\lbrack\overline G_{\rr}(M^\ast - \delta, \overline{b})\rbrack < \overline{b}$.}
    \end{equation}
    On the other hand, by applying \Cref{008} with $M^\ast - \delta$ in place of $M$ and $1$ in place of $\rr$, there exists $\varepsilon_1 \in (0, 1)$ such that for each $\varepsilon \in (0, \varepsilon_1]$,
    \begin{equation}\label{eq:245}
        \parbox{.80\linewidth}{there are at least $\mu\log_\Lambda(1/\varepsilon)$ values of $\rr \in [\varepsilon^{1 + \nu}, \varepsilon] \cap \{\Lambda^k\}_{k \in \ZZ}$ for which $\BP\lbrack\overline G_{\rr}(M^\ast - \delta, \overline{b})\rbrack \ge \overline{b}$.}
    \end{equation}
    Since $\mu > \nu/2$ (cf.~\eqref{eq:328}), \eqref{eq:244} and \eqref{eq:245} contradict each other. This completes the proof. 
\end{proof}

\bibliographystyle{alpha}
\bibliography{references}

\end{document}